%% file: modfg.tex
\documentclass[10pt,leqno]{article}
\usepackage{amsmath,amsthm}
\usepackage{amssymb}
\usepackage[all]{xypic}
\usepackage{amsbsy}
\usepackage{makeidx}
\makeindex

\input{macros}

\begin{document}

\title{Quasi-coherent sheaves on the Moduli Stack\\
of Formal Groups}
\author{Paul \ G.\ Goerss\footnote{The
author was partially supported by the National Science
Foundation (USA).}}
\maketitle
\begin{abstract}The central aim of this monograph is to provide
decomposition results for quasi-coherent sheaves on the moduli
stack of formal groups. These results will be based on the geometry
of the stack itself, particularly the height filtration and an analysis of
the formal neighborhoods of the geometric points. The main theorems
are algebraic chromatic convergence results and fracture square
decompositions. There is a major technical hurdle in this story, as the
moduli stack of formal groups does not have the finitness properties required
of an algebraic stack as usually defined. This is not a conceptual problem,
but in order to be clear on this point and to write down a self-contained narrative,
I have included a great deal of discussion of the geometry of the
stack itself, giving various equivalent descriptions.
\end{abstract}
\bigskip

\input{intro} 
\vfill\eject
\tableofcontents

\input{basics}
\input{fglaws}
\input{orbifold}
\input{invardiff}
\input{height}
\input{land}
\input{defs}
\input{chromconv}

\printindex

\input{biblio}
\bigskip

\noindent Department of Mathematics, Northwestern University, Evanston IL
60208

\noindent {\sl pgoerss@math.northwestern.edu}
\enddocument

%% file: macros.tex
\swapnumbers
\newtheorem{thm}{Theorem}[section]
\newtheorem{lem}[thm]{Lemma}
\newtheorem{prop}[thm]{Proposition}
\newtheorem{cor}[thm]{Corollary}
\newtheorem{defn}[thm]{Definition}
\theoremstyle{definition}
\newtheorem{rem}[thm]{Remark}

\newtheorem{exam}[thm]{Example}
\newtheorem{nota}[thm]{Notation}

\newtheorem{thmht}[thm]{Proof of Theorems \ref{diff-ht} and \ref{diff-ht-infty}}
\newtheorem{thmred}[thm]{Proof of the Theorem \ref{ht-n-reduced}}
\newtheorem{thmlaz}[thm]{Proof of the Theorem \ref{refined-lazard}}

\numberwithin{equation}{section}
\numberwithin{figure}{section}

\def\Alg{\operatorname{Alg}}

\def\Aut{\operatorname{Aut}}
\def\cart{{{\mathrm{cart}}}}

\def\AA{{{\mathbb{A}}}}
\def\EE{{{\mathbb{E}}}}
\def\ZZ{{{\mathbb{Z}}}}

\def\QQ{{{\mathbb{Q}}}}

\def\bS{{{\mathbf{S}}}}
\def\SS{\mathbb{S}}

\def\cA{{\cal A}}
\def\cB{{\cal B}}
\def\cC{{\cal C}}

\def\cE{{\cal E}}
\def\cF{{\cal F}}
\def\cG{{\cal G}}
\def\cH{{\cal H}}
\def\cI{{\cal I}}
\def\cJ{{\cal J}}

\def\cL{{\cal L}}
\def\cM{{\cal M}}
\def\cN{{\cal N}}
\def\cO{{\cal O}}
\def\cP{{\cal P}}

\def\cU{{\cal U}}
\def\cV{{\cal V}}

\def\cZ{{\cal Z}}
\def\La{{\Lambda}}
\def\la{{\lambda}}

\def\colim{\operatorname{colim}}
\def\holim{\operatorname{holim}}

\def\Der{\operatorname{Der}}

\def\Ext{{{\mathrm{Ext}}}}
\def\Ga{\Gamma}

\def\Hom{\mathrm{Hom}}
\def\Iso{{{\mathrm{Iso}}}}

\def\Mod{\mathbf{Mod}}
\def\map{\operatorname{map}}

\def\Sets{\mathbf{Sets}}

\def\Tor{\operatorname{Tor}}

\def\longr{{\longrightarrow}}
\def\defeq{\overset{\mathrm{def}}=}
\def\mm{{{\mathfrak{m}}}}
\def\pp{{{\mathfrak{p}}}}

\def\Spec{{{\mathrm{Spec}}}}
\def\GG{{{\mathbb{G}}}}
\def\FF{{{\mathbb{F}}}}
\def\bFF{{{\bar{\FF}_p}}}

\def\Alg{{{\mathbf{Alg}}}}

\def\pre{{{\mathbf{Pre}}}}

\newcommand{\brackets}[4]{{{\left\{ \begin{array}{ll}
				{{#1}}&{{#2}}\\
				\\
				{{#3}}&{{#4}}\end{array}\right.}}}

\def\fg{{\mathbf{fg}}}
\def\fgl{{\mathbf{fgl}}}

\def\tan{{\cal{T}\mathrm{an}}}

\def\Qmod{{{\mathbf{Qmod}}}}
\def\cmpl{{\wedge}}

\def\Aff{{{\mathbf{Aff}}}}
\def\cI{{{\mathcal{I}}}}
\def\cN{{{\mathcal{N}}}}
\def\Spf{{{\mathrm{Spf}}}}
\def\coord{{{\mathrm{Coord}}}}
\def\lie{{{\mathrm{Lie}}}}
\def\shder{{{\mathcal{D}er}}}

\def\VV{{\mathbb{V}}}
\def\red{{{\mathrm{red}}}}
\def\Nil{{{\mathcal{N}il}}}
\def\layer#1{{{\cH(#1)}}}
\def\pnty#1{{{\langle{#1}\rangle}}}
\def\hatlayer#1{{{\widehat{\cH}(#1)}}}
\def\isofgl{{{\mathbf{Isofgl}}}}
\def\crd{{{\mathbf{coord}}}}
\def\pnty#1{{{\langle{#1}\rangle}}}

\def\htimes{{{\hat{\times}}}}
\def\Sym{{{\mathrm{Sym}}}}
\def\rege{{{\epsilon}}}
\def\vare{{{\varepsilon}}}
\def\Gal{{{\mathrm{Gal}}}}

\def\Inf{{{\mathrm{Inf}}}}
\def\gr{{{\mathrm{gr}}}}
\def\lf{{{\mathbf{flv}}}}
\def\Def{{{\mathbf{Def}}}}
\def\Defalt{{{\mathbf{\Def_\ast}}}}
\def\fGa{{{\bar{\FF}_p,\Ga}}}
\def\Art{{{\mathbf{Art}}}}
\def\ga{{{\gamma}}}

\date{}

%% file: intro.tex
For years I have been echoing my betters, especially Mike Hopkins,
and telling anyone who would listen that the chromatic picture
of stable homotopy theory is dictated and controlled by the
geometry of the moduli stack $\cM_\fg$ of smooth, one-dimensional
formal groups. Specifically, I would say that the
height filtration of $\cM_\fg$ dictates a canonical and natural 
decomposition of a quasi-coherent sheaf on $\cM_\fg$, and
this decomposition predicts and controls the chromatic decomposition
of a finite spectrum.  This sounds well, and is even true, but  there
is no single place in the literature where I could send
anyone in order for him or her to get a clear, detailed, unified, and
linear rendition of this story. This document is an attempt to set that right.

Before going on to state in detail what I actually hope to accomplish
here, I should quickly acknowledge that the opening sentences of
this introduction and, indeed, this whole point of view is not original with
me. I have already mentioned Mike Hopkins, and just about everything
I'm going to say here is encapsulated in the table in section 2 of
\cite{GrossHop} and can be gleaned from the notes of various
courses Mike gave at MIT. See, for example, \cite{hop-notes}. Further
back, the  intellectual journey begins, for myself as a homotopy theorist, 
with Quillen's fundamental
insight linking formal groups, complex orientable cohomology
theories, and complex cobordism -- the basic papers are \cite{Q1} and
\cite{Q2}. But the theory of formal groups predates Quillen's work
connecting algebraic topology and the algebraic geometry of
formal groups: there was a rich literature already in place at the
time he wrote his papers. Lazard did fundamental work
in '50s (see \cite{Laz1}), and there was work of Cartier \cite{Cart}
on what happens when you work localized at a prime, and even a thorough
treatment of the deformation theory given by Lubin and Tate \cite{LT}.
In short, Quillen's work opened the door for the importation of a mature
theory in geometry into homotopy theory.

It was Jack Morava, I think, who really had the vision of how this should
go, but the 1970s saw a broad eruption of applications of formal groups to
homotopy theory. The twin towers here are the paper of Miller, Ravenel,
and Wilson \cite{MRW} giving deep computations in the Adams-Novikov
Spectral Sequence and Ravenel's nilpotence conjectures \cite{RAmerJ},
later largely proved by Devinatz, Hopkins, and Smith in \cite{DHS} and
\cite{HS}. This period fundamentally changed stable homotopy theory.
Morava himself wrote a number of papers, most notably \cite{Mor}
(see also Doug Ravenel's Math Review of this paper in \cite{RavMR}),
but there are rumors of a highly-realized and lengthy manuscript on
formal groups and their applications to homotopy theory. If so, it is a loss
that Jack never thought this manuscript ready for prime-time
viewing.\footnote{My standard joke is that if you see this manuscript on eBay
or somewhere, you should let me know. But, of course, it's not a joke.}

Let me begin the account of what you can find here with some indication
of how  stacks come into the narrative. One simple observation,
due originally (I think) to Neil Strickland is that stacks can calculate
homology groups. Specifically, if $E_\ast$ and $F_\ast$ are two
$2$-periodic Landweber exact homology theories and if $G$
and $H$ are the formal groups over $E_0$ and $F_0$ respectively,
then there is a $2$-category pull-back square
$$
\xymatrix{
\Spec(E_0F) \rto \dto & \Spec(F_0) \dto^H\\
\Spec(E_0) \rto_G & \cM_\fg.
}
$$
This can be seen in Lemma \ref{pull-back-for-coord-1} below. I heard
Mike, in his lectures at M\"unster, noting this fact as piquing his interest
in stacks.  Beyond this simple calculation, Strickland should certainly
get a lot of credit for all of this: while the reference \cite{fpfp} never
actually uses the word ``stack'', the point of view is clear and, in fact,
much of what I say here can be found there in different -- and sometimes
not so different -- language.

For computations, especially
with the Adams-Novikov Spectral Sequence, homotopy theorists
worked with the cohomology of comodules over  Hopf algebroids.
A succinct way to define such objects is to say that a Hopf algebroid 
represents an affine groupoid scheme; in particular, Quillen's theorem
mentioned above amounts to the statement that the affine groupoid
scheme arising from the Hopf algebroid of complex cobordism is none
other than the groupoid scheme which assigns to each commutative ring
$A$ the groupoid of formal group laws and their strict isomorphisms over $A$.
Hopf algebroids were and are a powerful computational tool -- as far as
I know, the calculations of \cite{MRW} remain, for the combination
of beauty and technical prowess, in a class with Secretariat's
run at the Belmont Stakes -- but an early and fundamental result was
``Morava's Change of Rings Theorem'', which, in summary, says that
if two Hopf algebroids represent equivalent (not isomorphic) groupoid
schemes, then they have isomorphic cohomology. A  more subtle
observation is that the change of rings results holds under weaker
hypotheses:  the  groupoid schemes need only be equivalent ``locally in
the flat topology''; that is, the presheaves $\pi_0$ of components
and $\pi_1$ of automorphisms induces isomorphic sheaves in the
$fpqc$ topology. (See \cite{HovAmerJ} and \cite{hollander} for
discussions of this result.) In modern language, we prove this
result by combining the following three observations:
\begin{itemize}

\item the category of comodules over a Hopf algebroid is equivalent to the
category of quasi-coherent sheaves on the associated stack;

\item two groupoid schemes locally equivalent in the flat topology have
equivalent associated stacks; and

\item equivalent stacks have equivalent categories of quasi-coherent sheaves.
\end{itemize}
Note that in the end, we have a much stronger result than simply an isomorphism
of cohomology groups -- we have an entire equivalence of categories.

Once we've established an equivalence between the category of
comodules and the category of quasi-coherent sheaves (see Equation \ref{ANSS-E2-2}) we
can rewrite the cohomology of comodules as coherent cohomology
of quasi-coherent sheaves; for example,
$$
\Ext_{MU_\ast MU}^s(\Sigma^{2t}MU_\ast,MU_\ast) \cong
H^s(\cM_\fg,\omega^{\otimes t})
$$
where $\omega$ is the invertible sheaf on $\cM_\fg$ which
assigns to each flat morphism $g:\Spec(R) \to \cM_\fg$ the invariant
differentials $\omega_G$ of the formal group classified by $G$.
Thus, one of our most sensitive algebraic approximations to the
stable homotopy groups of spheres can be computed as the
cohomology of the moduli stack $\cM_\fg$.

There are other reasons for wanting to pass from comodules over
 Hopf algebroids to quasi-coherent sheaves. For example, there are
 naturally occuring stacks which are not canonically equivalent, even
in the local sense mentioned above, to an affine groupoid scheme.
The most immediate example is the moduli stack $\cU(n)$ of formal
groups of height less than or equal to some fixed integer $n \geq 0$.
These stacks have affine presentations, but not canonically; the
canonical presentation is a non-affine 
open subscheme of $\Spec(L)$, where
$L$ is the Lazard ring. Thus the quasi-coherent sheaves on $\cU(n)$
are equivalent to many categories of comodules, but no particular such
category is preferred (except by tradition -- this is one role for the
Johnson-Wilson homology theories $E(n)_\ast$)  and the
quasi-coherent sheaves themselves remain the basic object of study.
The point is taken up in \cite{HS2} and \cite{Nau}.

Here is what I hope to accomplish in these notes.
\begin{itemize}

\item Give a definition of formal group which evidently satisfies
the effective descent condition necessary to produce a moduli
stack. See Proposition \ref{prestack}.
This can be done in a number of ways, but the I have
chosen to use the notion of formal Lie varieties, a concept
developed by Grothendeick to give a conceptual formulation
of smoothness in the formal setting. 

\item A formal group law is equivalent  to a formal group with
a chosen coordinate. The scheme of all coordinates for
a formal group $G$ over a base scheme $S$ is a torsor
$\coord_G$ over $S$ for the group scheme $\La$ which
assigns to each ring $R$ the group of power series invertible
under composition. Using coordinates we can identify
$\cM_\fg$ as the quotient stack of the scheme of formal groups
by the algebraic group $\La$. See Proposition \ref{fg-homotopy-orbit}.
This makes transparent the
fact that $\cM_\fg$ is an algebraic stack (of a suitable sort)
and it makes transparent the equivalence between comodules
and quasi-coherent sheaves.

\item The stack $\cM_\fg$ is not an algebraic stack in the sense
of the standard literature (for example, \cite{Laumon}) because
it does not have a presentation by a scheme locally of finite type -- the
Lazard ring is a polynomial ring on infinitely many generators.
It is, however, pro-algebraic: it can be written as 2-category
(i.e., homotopy) inverse limit of the algebraic stacks $\cM_\fg\pnty{n}$
of $n$-buds of formal groups. This result is inherent in Lazard's original
work -- it is the essence of the $2$-cocycle lemma -- but I learned it
from Mike Hopkins and
it has been worked out in detail by Brian Smithling \cite{smith}.
An important point is that any {\it finitely presented} quasi-coherent
sheaf on $\cM_\fg$ is actually the pull-back of  a quasi-coherent sheaf
on $\cM_\fg\pnty{n}$ for some $n$. See Theorem \ref{fp-modules}.

\item Give a coordinate-free definition of height and the height
filtration. Working over $\ZZ_{(p)}$, the height filtration is
a filtration by closed, reduced substacks
$$
\cdots \subseteq \cM(n) \subseteq \cM(n-1) \subseteq \cdots\subseteq
 \cM(1) \subseteq
\cM_\fg
$$
so that inclusion $\cM(n) \subseteq \cM(n-1)$ is the effective
Cartier divisor defined by a global section $v_n$ of the invertible
sheaf $\omega^{\otimes (p^n-1)}$ over $\cM(n-1)$. This implies,
among other things, that $\cM(n) \subseteq \cM_\fg$ is regularly
embedded, a key ingredient in Landweber Exact Functor
Theorem and chromatic convergence. The height filtration
is essentially unique: working over $\ZZ_{(p)}$,
any closed, reduced substack of $\cM_\fg$ is either
$\cM_\fg$ itself, $\cM(n)$ for some $n$, or $\cM(\infty) =
\cap \cM(n)$. See Theorem \ref{ht-n-reduced}.
This is the geometric content of the Landweber's invariant
prime ideal theorem. The stack $\cM(\infty)$  is not empty
 as the morphism classifying
the additive formal group over $\FF_p$ factors through $\cM(\infty)$.
This point and the next can also be found in Smithling's thesis
\cite{smith}. Some of this material is also in the work
of Hollander \cite{hollander2}.

\item Identity $\layer{n} = \cM(n) - \cM(n+1)$, the moduli stack
of formal groups of exact height $n$, as the neutral gerbe determined
by the automorphism scheme of any height $n$ formal group $\Gamma_n$
over $\FF_p$. See Theorem \ref{diff-ht}.
This automorphism scheme is affine and, if
we choose $\Gamma_n$ to be the Honda formal group of height
$n$, well known to homotopy theorists -- its ring of functions is
the Morava stabilizer algebra (see \cite{Rav}, Chapter 6) and its group of $\FF_{p^n}$ points 
is the Morava stabilizer group. This is all a restatement of 
Lazard's uniqueness theorem for height $n$ formal groups in modern
language; indeed, the key step in the argument is the proof,
essentially due to Lazard, that given any two formal groups $G_1$
and $G_2$ over an $\FF_p$-scheme $S$, then the scheme
$\Iso_S(G_1,G_2)$ of isomorphisms from $G_1$ to $G_2$ is
either empty (if they have different heights) or pro-\'etale and
surjective over $S$ (if they have the same height). See
Theorem \ref{refined-lazard}; we give essentially Lazard's proof,
but similar results with nearly identical statements appear in \cite{Katz}.

\item Describe the formal neighborhood $\hatlayer{n}$ of $\layer{n}$
inside the open substack $\cU(n)$ of $\cM_\fg$ of formal groups
of height at most $n$. Given a choice of $\Ga_n$ of
formal group of height $n$ over the algebraic closure $\bFF$ of
$\FF_p$ the morphism
$$
\Def(\bFF,\Ga_n) \longr \hatlayer{n}
$$
from the Lubin-Tate deformation space to the formal neighborhood
is pro-Galois with Galois group $\GG(\bFF,\Ga_n)$ of the pair
$(\bFF,\Ga_n)$. Lubin-Tate theory identifies $\Def(\bFF,\Ga_n)$
as the formal spectrum of a power series ring; since a power
series ring can have no finite \'etale extensions, we may say
$\Def(\bFF,\Ga_n)$ is the universal cover of $\hatlayer{n}$.
If $\Ga_n$ is actually
defined over $\FF_p$, then $\GG(\bFF,\Ga_n)$ is known to 
homotopy theorists as the big Morava stabilizer group:
$$
\GG(\bFF,\Ga_n) \cong \Gal(\bFF/\FF_p) \rtimes \Aut_{\bFF}(\Ga_n).
$$
From this theory, it is possible to describe what it means to be
a module on the formal neighborhood of a height $n$ point;
that is, to give a definition of the category of ``Morava modules''.
See Remark \ref{qc-hatlayer}.

\item If $\cN \to \cM_\fg$ is a representable, separated, and flat morphism
of algebraic stacks, then the induced height filtration
$$
\cdots \subseteq \cN(n) \subseteq \cN(n-1) \subseteq \cdots 
\subseteq \cN(1) \subseteq
\cN
$$
with $\cN(n) = \cM(n) \times_{\cM_\fg} \cN$
automatically has that the inclusions $\cN(n) \subseteq \cN(n-1)$
are effective Cartier divisors. The Landweber Exact Functor Theorem
(LEFT) 
is a partial converse to this statement. Here I wrote down a proof
due to Mike Hopkins (\cite{hop-notes}) of this fact. Other proofs
abound -- besides the original \cite{landweber}, there's one
due to Haynes Miller \cite{LEFT}, and Sharon Hollander has
an argument as well \cite{hollander2}. The morphism from
the moduli stack of elliptic curves to $\cM_\fg$ which assigns to each
elliptic curve its associated formal group is an example.
It is worth emphasizing that this is a special fact about the moduli
stack of formal groups -- the proof uses that $\cH(n)$ has a 
unique geometric point.

\item Give proofs of the algebraic analogs of the topological 
chromatic convergence and fracture square results for spectra.
Work over $\ZZ_{(p)}$ and let $i_n:\cU(n) \to \cM_\fg$ be
the open inclusion of the moduli stack of formal groups of height
at most $n$. If $\cF$ is a quasi-coherent sheaf on $\cM_\fg$,
we can form the derived push-forward of the pull-back
$R(i_n)_\ast i_n^\ast \cF$. As $n$ varies, these assemble into
a tower of cochain complexes of quasi-coherent sheaves on $\cM_\fg$
and there is a natural map
$$
\cF \longr \holim R(i_n)_\ast i_n^\ast \cF.
$$
Chromatic convergence then says that if $\cF$ is finitely presented,
this is morphism is an equivalence. The result has teeth as the
$\cU(n)$ do not exhaust $\cM_\fg$. To examine the transitions
in this tower, we note that the inclusion
$\cM(n) = \cM_\fg - \cU(n-1) \subseteq \cM_\fg$
is defined by the vanishing of a sheaf of ideals $\cI_n$ 
which is locally generated by  regular sequence.
Then for any quasi-coherent sheaf on $\cM_\fg$ there is
a homotopy Cartesian square (the fracture square)
$$
\xymatrix{
\cF \rto \dto & L(\cF)^\cmpl_{\cI(n)} \dto\\
R(i_{n-1})_\ast i_{n-1}^\ast \cF \rto &R(i_{n-1})_\ast i_{n-1}^\ast(L(\cF)^\cmpl_{\cI(n)})
}
$$
where $L(\cF)^\cmpl_{\cI(n)}$ is the total left derived functor of the
completion of $\cF$. Both proofs use the homotopy
fiber of
$$
\cF \longr  R(i_n)_\ast i_n^\ast \cF
$$
which is the total local cohomology sheaf $R\Gamma_{\cM(n)}\cF$. This can be
analyzed using the fact that $\cM(n) \subseteq \cM_\fg$ is
a regular embedding and Greenlees-May duality \cite{GM};
the requisite arguments can be lifted nearly verbatim from
\cite{all}, but see also \cite{DwyerG} -- the fracture square appears
in exactly this form in this last citation. Chromatic convergence
is less general -- the proof I give here uses that any finitely
presented sheaf can be obtained as a pull-back from the
stack of $n$-buds $\cM_\fg\pnty{m}$ for some $m$. This allows
one to show that the transition map 
$$
R\Gamma_{\cM(n+1)}\cF \to R\Gamma_{\cM(n)}\cF
$$
between the various total
local cohomology sheaves in zero in cohomology for large $n$.
\end{itemize}

This document begins with a compressed introduction to some
of the algebraic geometry we will need. While I can bluff my way
through a lot of algebraic geometry, I am not a geometer either
by inclination or training. There are bound to be minor errors,
but I hope there's nothing egregious. Corrections would be appreciated.
\bigskip

{\bf Acknowledgements:} I hope I've made clear my debt to Mike Hopkins;
in not, let me emphasize it again. Various people have listened
to me talk on this subject in the past few years; in particular, 
Rick Jardine has twice offered me extended forums for this work,
once at the University of Western Ontario, once at the Fields Institute.
Some of my students have listened to me at length as well. Two of them
-- Ethan Pribble \cite{pribble} and Valentina Joukhovitski \cite{valya} --
have written theses on various aspects of the theory.

%% file: basics.tex
\section{Schemes and formal schemes}

This section is devoted entirely to a review of the algebraic
geometry we need for the rest of the paper. It can -- and
perhaps should -- be skipped by anyone knowledgeable
in these matters.

\subsection{Schemes and sheaves}

We first recall some basic definitions about schemes and morphisms
of schemes, then enlarge the category slightly to sheaves in the 
$fpqc$-topology. This is necessary as formal schemes and formal
groups are not really schemes.

Fix a commutative ring $R$.
Schemes over $R$ can be thought of as functors from $\Alg_R$
to the category of sets. We briefly review this material --
mostly to establish language. 

The basic schemes over $R$ are the affine schemes $\Spec(B)$,
where $B$ is an $R$-algebra. As a functor\index{affine scheme}
$$
\Spec(B):\Alg_R \longr \Sets
$$
is the representable functor determined by $R$; that is,
$$
\Spec(B)(A) = \Alg_R(B,A).
$$
If $I \subseteq B$ is an ideal
we have the open subfunctor $U_I \subseteq \Spec(B)$ with
\index{open subfunctor, of an affine scheme}
$$
U_I(A) =\{\ f:B \to A\ |\ f(I)A = A\ \} \subseteq \Spec(B).
$$
This defines the Zariski topology on $\Spec(B)$.
The complement of $U_I$ is defined to be
the closed subfunctor $Z_I = \Spec(B/I)$;
thus,
$$
Z_I(A) =\{\ f:B \to A\ |\ f(I)A = 0\ \} \subseteq \Spec(B).
$$
Note that we can guarantee that
$$
U_I(A) \cup Z_I(A) = \Spec(B)(A)
$$
only if $A$ is a field.

If $X: \Alg_R \to \Sets$ is any functor, we define a subfunctor 
\index{open subfunctor, of an $R$-functor}
$U \subseteq X$ to be {\it open} if the subfunctor
$$
U \times_X \Spec(B) \subseteq \Spec(B)
$$
is open for for all morphisms of functors $\Spec(B) \to X$. Such morphisms
are in one-to-one correspondence with $X(B)$, by the Yoneda Lemma.
A collection of subfunctors $U_i \subseteq X$ is called a cover
if the morphism $\sqcup U_i(\FF) \to X(\FF)$ is onto for all
{\it fields} $\FF$.

As a matter of language, a functor $X:\Alg_R \to \Sets$ will be called an
{\it $R$-functor}\index{$R$-functor}.

\begin{defn}\label{scheme}\index{scheme}An $R$-functor $X$ is a {\it scheme over
$R$} if
it satisfies the following two conditions:
\begin{enumerate}

\item $X$ is a sheaf in the Zariski topology; that is, if $A$ is
an $R$-algebra and $a_1,\ldots,a_n \in A$ are elements so that
$a_1 + \cdots + a_n = 1$, then
$$
\xymatrix{
X(A) \rto & \prod X(A[a_i^{-1}]) \ar@<.5ex>[r] \ar@<-.5ex>[r] &
\prod X(A[a_i^{-1}a_j^{-1}])
}
$$
is an equalizer diagram; and

\item $X$ has an open cover by affine schemes $\Spec(B)$ where
each $B$ is an $R$-algebra

\end{enumerate}
A morphism $X \to Y$ of schemes over $R$ is a natural transformation
of $R$-functors.
\end{defn}

An open subfunctor $U$ of scheme $X$ is itself a scheme; the collection of all 
open subfunctors defines the {\it Zariski topology} on $X$.\index{Zariski topology}

\begin{rem}[{\bf Module sheaves and quasi-coherent sheaves}]\label{mod-shvs}
\index{module sheaves}
There is
an obvious sheaf of rings $\cO_X$ in this topology on $X$ called
the {\it structure sheaf} of $X$. If $U = \Spec(B) \subseteq X$ is an affine
open, then $\cO_X(U) = B$; this defintion extends to other open
subsets by the sheaf condition. A sheaf $\cF$ of $\cO_X$-{\it modules}
on $X$ is a sheaf so that 
\begin{enumerate}

\item for all open $U \subseteq X$, $\cF(U)$ is an $\cO_X(U)$-module;

\item for all inclusions $V \to U$, the restriction map $\cF(U) \to \cF(V)$
is a morphism of $\cO_X(U)$-modules.
\end{enumerate}
We now list some special classes of $\cO_X$-module sheaves.
The following definitions are all in \cite{EGA}, \S 0.5.
Let $X$ be a scheme. For any set $I$ write $\cO_X^{(I)}$ for the
coproduct of $\cO_X$ with itself $I$ times.  This coproduct is
the sheaf associated to the direct sum presheaf.
\begin{enumerate}

\item[QC.] A module sheaf $\cF$ is {\it quasi-coherent} if there is a cover
of $X$ by open subschemes $U_i$ so that for each $i$ there is
an exact sequence of $\cO_{U_i}$-sheaves
$$
\cO_{U_i}^{(J)} \to \cO_{U_i}^{(I)} \to \cF\vert_{U_i} \to 0.
$$
\index{sheaf, quasi-coherent}

\item[LF.]A quasi-coherent sheaf is {\it locally free} if the set $J$ can
be taken to be empty.
\index{sheaf, locally free}

\item[FP.] A quasi-coherent sheaf $\cF$ is {\it finitely presented} if the sets
$I$ and $J$ can be taken to be finite.
\index{sheaf, finitely presented}

\item[FT.] A module sheaf $\cF$ is of {\it finite type} if there is an open cover by subschemes
$U_i$ and, for each $i$, a surjection
$$
 \cO_{U_i}^{(I)} \to \cF\vert_{U_i} \to 0.
$$
with $I$ finite.
\index{sheaf, finite type}

\item[C.] A module sheaf $\cF$ is {\it coherent} if is of finite type and
for all open subschemes $U$ of $X$ and all morphisms
$$
f:\cO_U^n \longr \cF\vert_U
$$
of sheaves, the kernel of $f$ is of finite type.
\index{sheaf, coherent}
\end{enumerate}

There are examples of sheaves of finite type which are not quasi-coherent.
Every coherent sheaf is finitely presented and, hence, quasi-coherent; however,
a finitely presented module sheaf is coherent only if
$\cO_X$ itself is coherent. For affine schemes $\Spec(A)$, this
is equivalent to $A$ being a coherent ring -- every finitely generated
ideal is finitely presented.\index{coherent ring} This will happen if $A$ is a filtered
colimit of Noetherian rings; for example, the Lazard ring $L$.
\end{rem}

If $X_1 \to  Y  \leftarrow X_2$ is a diagram of schemes, the evident
fiber product $X_1 \times_Y X_2$ of functors is again a scheme; furthermore,
if $U =X_1 \to Y$ is an open subscheme, then $U \times_Y  X_2 \to X_2$ 
is also an open subscheme. Thus, if $f: X \to Y$ is a morphism of schemes,
and $\cF$ is a sheaf in the Zariski topology on $X$, we get sheaf a
{\it push-forward} sheaf \index{push-forward sheaf, $f_\ast$} $f_\ast \cF$
on $Y$ with
$$
[f_\ast \cF(U)] = \cF(U \times_Y X).
$$
In particular, $f_\ast \cO_X$ is a sheaf of $\cO_Y$-algebras and
if $\cF$ is an $\cO_X$-module sheaf, $f_\ast \cF$ becomes a
$\cO_Y$-module sheaf. Extra hypotheses are needed 
for $f_\ast(-)$ to send quasi-coherent sheaves to quasi-coherent
sheaves. See Proposition \ref{push-forward-qc} below.

The functor $f_\ast$ from $\cO_X$-modules to $\cO_Y$-modules has
a left adjoint, of course. If $f:X \to Y$ is a morphism of schemes and
$\cF$ is any sheaf on $Y$, define a sheaf $f^{-1}\cF$ on $X$ by
$$
[f^{-1}\cF](U) = \colim \cF(V)
$$
where the colimit is taken over all diagrams of the form
$$
\xymatrix{
U \rto \dto & X \dto\\
V \rto &Y
}
$$
with $V$ open in $Y$. If $\cF$ is an $\cO_Y$-module sheaf, then
$f^{-1}\cF$ is an $f^{-1}\cO_Y$-module sheaf and the {\it pull-back} sheaf
is\index{pull-back sheaf, $f^\ast$}
$$
f^\ast \cF= \cO_X \otimes_{f^{-1}\cO_Y} f^{-1}\cF.
$$
Thus we have an adjoint pair
\begin{equation}\label{star-adjoint}
\xymatrix{
f^\ast:\Mod_Y \ar@<.5ex>[r] &  \ar@<.5ex>[l] \Mod_X:f_\ast.
}
\end{equation}
Here and always, the left adjoint is written on top and from left to right.
If $\cF$ is quasi-coherent, so if $f^\ast \cF$; if $\cO_X$ is
coherent and $\cF$ is coherent, then $f^\ast \cF$ is coherent.

\begin{rem}[{\bf The geometric space of a scheme}]\label{espace-geom}
\index{geometric space of a scheme}
Usually, we define a scheme to be a locally
ringed space with an open cover by prime ideal spectra. This is equivalent
to the definition here, which is essentially that of Demazure and Gabriel.
Since both notions are useful -- even essential -- we show how to
pass from one to the other.

If $X$ is a functor from commutative rings to sets, we define the associated
{\it geometric space} $|X|$ as follows. A point in $|X|$ -- also known as
a {\it geometric point}\index{geometric point of a scheme}of $X$ -- is an equivalence class of morphisms
$f:\Spec(\FF) \to X$ with $\FF$ a field. The morphism $f$ is equivalent
to $f':\Spec(\FF') \to X$ is they agree after some common extension. 
This becomes a topological space with open sets $|U|$ where $U \subseteq
X$ is an open subfunctor.

If $X = \Spec(B)$, then a geometric point of $X$ is an equivalence class of
homomorphisms of commutative rings $g:B \to \FF$; this equivalence
class is determined by the kernel of $g$, which must be a prime ideal.
Furthermore, the open subsets of $|\Spec(B)|$ are exactly the subsets
$D(I)$ where $I \subseteq R$ is an ideal: $D(I)$ is complement of
the closed set $V(I)$ of prime ideals contained in $I$.
Thus $|\Spec(B)|$ is the usual
prime ideal spectrum of $B$. 

If $X = \Spec(B)$, then $|X|$ becomes a locally ringed space, with structure
sheaf $\cO$ the sheaf associated to the presheaf which assigns to each
$D(I)$ the ring $S_I^{-1}R$ where
$$
S_I = \{\ a \in B\ |\ a + \pp \ne \pp\ \hbox{for all}\ \pp \in D(I)\ \}.
$$
The stalk $\cO_x$ of $\cO$ at the point $x$ specified by the prime ideal
$\pp$ is exactly $B_\pp$. If $X$ is a general functor, then there is a
homeomorphism of topological spaces
$$
\colim |\Spec(B)| \mathop{\longr}^{\cong} |X|
$$
where the colimit is over the category of all morphisms $\Spec(B) \to X$.
This equivalence specifies the structure sheaf on $|X|$ as well. Indeed,
if $U \subseteq X$ is an open subfunctor, then, by definition $U \times_X \Spec(B)$
is open in $\Spec(B)$ for all $\Spec(B) \to X$ and $\cO(|U|)$ is determined
by the sheaf condition.

If $|X|$ is a scheme, then $|X|$ has an open cover by open subsets
of the form $V_i = |\Spec(B_i)|$ and, in addition,
$$
(\cO_{|X|})\vert_{V_i} \cong \cO_{|\Spec(B_i)|}.
$$
Whenever a locally ringed space $(Y,\cO)$ has such a cover, we will say that
$Y$ has a cover by prime ideal spectra.

The geometric space functor $|-|$ from $\ZZ$-functors to locally ringed spaces
has a right adjoint $\bS(-)$: if $Y$ is a geometric space, then the $R$-points
of $\bS(Y)$ is the set of morphisms of locally ringed spaces
$$
|\Spec(B)| \longr Y.
$$
The following two statements are the content of the Comparison Theorem
of \S I.1.4.4 of \cite{DG}.
\begin{enumerate}

\item Let $(Y,\cO)$ be a locally ringed space with an open cover 
$V_i$ by prime ideal spectra. Then the adjunction morphism
$|\bS(Y)| \to Y$ is an isomorphism of locally ringed spaces,

\item If $X$ be a functor from commutative rings to sets. Then $|X|$ has
an open cover by prime ideal spectra if and only if $X$ is a scheme and,
in that case, $X \to \bS|X|$ is an isomorphism.
\end{enumerate}
Together these statements imply that adjoint pair $|-|$ and
$\bS(-)$ induce an equivalence of categories between schemes and
locally ringed spaces with an open cover by prime ideal spectra. For
this reason and from now on we use on or the other notion as is
convenient.
\end{rem}

\begin{rem}\label{stalks-morphisms}If $X$ is a scheme and $x$ a geometric 
point of $X$ represented by $f:\Spec(\FF) \to X$, then the stalk $\cO_{X,x}$ of the structure sheaf at
$X$ can be calculated as
$$
\cO_{X,x} \cong \mathop{\colim}_{U \subseteq X} \cO_X(U).
$$
where $U$ runs over all open subshemes so that $f$ factors through $U$.
This is the global sections of $f^{-1}\cO_X$.
If $f$ factors as $\Spec(\FF) \to \Spec(B) \subseteq X$ with $\Spec(B)$
open in $X$, then there is an isomorphism
$$
\cO_{X,x} \cong B_\pp
$$
where $\pp$ is kernel of $R \to \FF$. It is easy to check this is independent
of the choice of $f$. 

If $X \to Y$ is a morphism of schemes, then we have a morphism of
sheaves $f^{-1}\cO_Y \to \cO_X$. If $x \in X$ is a geometric point,
we get an induced morphism of local rings $\cO_{Y,f(x)} \to \cO_{X,x}$.
\end{rem}

\begin{rem}\label{stand-props-schemes-1} We use this paragraph to give
some standard definitions of properties of morphisms of schemes.

1.) A morphism $f:X \to Y$ of
schemes is {\it flat} if for all geometric points of $X$
geometric space, the induced morphism of local rings
$$
\cO_{Y,f(x)} \to \cO_{X,x}
$$
is flat. \index{morphism, flat}The morphism $f$ is {\it faithfully flat}
if it is flat and surjective. Here surjective means $X(\FF) \to Y(\FF)$
is onto for all fields or, equivalently, the induced morphism
of geometric spaces $|X| \to |Y|$ is surjective.\index{morphism, faithfully flat}

2.) A scheme $X$ is called {\it quasi-compact} if every cover by open
subschemes $U_i \subseteq X$ has a finite subcover. A morphism of
schemes $X \to Y$ is quasi-compact if for every quasi-compact
open $V \subseteq Y$, the scheme $V \times_Y X$ is quasi-compact.
\index{morphism, quasi-compact}

3.) A morphism $f:X\to Y$ of schemes is called
{\it quasi-separated} if the diagonal morphism 
$X \to X \times_Y X$ is quasi-compact.
\index{morphism, separated}

4.) A morphism $f:X \to Y$
of schemes is {\it finitely presented} if for all open $U \subseteq Y$,
$f_\ast \cO_X(U)$ is a finitely presented $\cO_Y(U)$-algebra;
that is, $f_\ast \cO_X(U)$ is a quotient of $\cO_Y(U)[x_1,\ldots,x_n]$
by a finitely generated ideal.
\index{morphism, finitely presented}
\end{rem}

Any affine scheme $\Spec(B)$ is quasi-compact as the
subschemes $\Spec(B[1/f])$ form a basis for the Zariski topology.
It follows that every morphism of affine schemes is quasi-compact
and quasi-separated.

The following is in \cite{DG}, Proposition I.2.2.4.

\begin{prop}\label{push-forward-qc}Let $f:X \to Y$ be
a quasi-compact and quasi-separated morphism of schemes.
If $\cF$ is a quasi-coherent $\cO_X$-module sheaf, then
$f_\ast \cF$ is a quasi-coherent $\cO_Y$ module sheaf.
\end{prop}

Thus, if $f$ is quasi-compact and quasi-separated, Equation
\ref{star-adjoint} yields an adjoint pair
pair
$$
\xymatrix{
f^\ast:\Qmod_Y \ar@<.5ex>[r] &  \ar@<.5ex>[l] \Qmod_X:f_\ast .
}
$$

\begin{rem}[{\bf Faithfully flat descent}]\label{ff-descent}\index{faithfully flat
descent} Let $f:X \to Y$ be an morphism of schemes and let 
$$
X^n = X \times_Y \cdots \times_Y X
$$
where the product is taken $n$ times. If $\phi:[m] \to [n]$ is any morphism
in the ordinal number category, define $\phi^\ast :X^n \to X^m$ by the 
pointwise formula
$$
\phi(x_0,\ldots,x_n) = (x_{\phi(0)},\cdots,x_{\phi(n)}).
$$
In this way we obtain a simplicial $R$-functor $X_\bullet$ augmented to $Y$.
This is the {\it coskeleton} of $f$.\index{coskeleton of a morphism}

A {\it descent problem} for $f$ is a pair $(\cF,\psi)$ where $\cF$ is a 
sheaf\index{descent problem} on
$X$ and $\psi:d_1^\ast \cF \to d_0^\ast \cF$ is an isomorphism of  sheaves
on $X \times_Y X$ subject to the cocycle condition
$$
d_1^\ast \psi = d_2^\ast \psi d_0^\ast \psi
$$
over $X \times_Y X \times_Y X$. A solution to a descent problem is a  sheaf
$\cE$ over $Y$ and an isomorphism $\phi_0:f^\ast \cE \to \cF$ over $X$
so that the following diagram commutes
$$
\xymatrix{
d_1^\ast f^\ast \cE \rto^{c} \dto_{d_1^\ast \psi_0}
& d_0^\ast f^\ast \cE \dto^{d_0^\ast \psi_0}\\
d_1^\ast \cF \rto_{\psi} & d_0^\ast \cF
}
$$
where $c$ is the canonical isomorphism obtained from the equation $fd_1 = fd_0$.
If $f:X \to Y$ is flat and $(\cF,\psi)$ is a descent problem with $\cF$ quasi-coherent,
then there is at most one solution with $\cE$ quasi-coherent. If $f$ is faithfully
flat, there is exactly one solution and we get an evident equivalence of categories.
This has many refinements; for example, one could concentrate on algebra sheaves
instead of module sheaves. See Proposition \ref{ff-affine} below.
\end{rem}

\begin{nota}\label{rep-functor}Let $\cC$ be a category, $\cC_0$
a sub-category, and $X \in \cC$. Let $\pre(\cC_0)$ the category
of presheaves (i.e., contravariant functors) from $\cC_0$ to
sets.\footnote{The category of presheaves as defined here is not
a category as it might not have small $\Hom$-sets. There are several
ways to handle this difficulty, one being to bound the cardinality of
all objects in question at a large enough cardinal that all objects of
interest are included. The issues are routine, so will ignore
this problem. The same remark applies to category of $R$-functors.} Then the assignment
$$
X \mapsto \Hom_{\cC_0}(-,X)
$$
defines a functor $\cC \to \pre(\cC_0)$ and we will write $X$ equally
for the object $X$ and the associated representable presheaf. 
In our main examples, $\cC$ will be $R$-functors and $\cC_0$
will be affine schemes or schemes over $R$. In this case, the 
$\cC \to \pre(\cC_0)$ is an embedding; if $\cC$ is $R$-functors
and $\cC_0 = \Aff/R$, it is an equivalence.
\end{nota}

\begin{rem}[{\bf Topologies}]\label{top} In \cite{SGA3-1} Expos\'e IV,
topologies on schemes over a fixed base ring $R$ are defined as follows.

First, if  $X$ is an $R$-functor, then a {\it sieve}\index{sieve} on $X$ is a subfunctor
$F$ of the functor on $R$-functors over $X$ represented by $X$
itself; thus for every
$Y \to X$, $F(Y)$ is either
either empty or one point. The collection $E(F)$ of $Y \to X$ so
that $F(Y) \ne \phi$ has the property that if $Y \in E(F)$ and
$Z \to Y$, then $Z \in E(F)$. The collection $E(F)$ determines 
$F$; conversely any such collection $E$ determines a sieve $F$
with $E(F) = E$.

Next, let $f_i:X_i \to X$ be a collection of morphisms $R$ functors. This
determines a sieve by taking $E$ to be the set of $Y \to X$
which factor through some $f_i$. This collection of morphisms
is the base of resulting sieve; any sieve has at least one base,
for example $E(F)$. Thus, it may be convenient to specify sieves
by families of morphisms.

In \cite{SGA3-1} IV.4.2.,  a topology\index{topology} on the category of $R$-functors
is an assignment, to each $R$-functor, a set of {\it covering sieves}
$J(X)$ subject to the following axioms:
\begin{enumerate}

\item $X \in J(X)$;

\item if $F$ is a sieve for $X$  and $\Spec(B) \times_X F \in
J(\Spec(B))$ for all morphisms of $R$-functors $\Spec(B) \to X$
with affine source, then $F \in J(X)$;

\item (Base change) if $F \in J(X)$ and $Y \to X$ is a morphism of
$R$-functors then $F \times_X Y \in J(Y)$;

\item (Composition) if $F \in J(X)$ and $G \in J(F)$, then $G \in J(X)$;

\item (Saturation) if $F$ is a sieve for $X$, $G$ a sieve for 
$F$ and $G \in J(X)$, then $F \in J(X)$;
\end{enumerate}

\noindent These axioms together imply 
\begin{enumerate}

\item[6.] (Local) If $F_1$ and $F_2$ are in $J(X)$, so is
$F_1 \cap F_2$.
\end{enumerate}

As in \cite{SGA3-1} IV.4.2.3,
these axioms can be reformulated in terms of families of
morphisms: it is equivalent to assign to each $R$-functor
 $X$ a collection $C(X)$ of sets of {\it covering families}
 \index{covering family} of morphisms
$\{ X_i \to X\}$ of $R$-functors with the following
properties:

\begin{enumerate}

\item If $Y \to X$ is an isomorphism, then $\{ Y \to X\} \in C(X)$;

\item If $\{ X_i \to X\}$ is a set of morphisms so that
$\{ \Spec(B) \times_X X_i \to \Spec(B)\} \in C(\Spec(B))$
whenever $\Spec(B) \to X$ is a
morphism from
an affine scheme, then $\{ X_i \to X\} \in C(X)$;

\item (Base change) If $Y \to X$ is a morphism of schemes
and $\{ X_i \to X\} \in
J(X)$, then $\{ Y \times_X X_i \to Y\} \in C(Y)$;

\item (Composition) If $\{ X_i \to X\}\in C(X)$ and $\{ X_{ij} 
\to X_i\}\in C(X_i)$, then
$\{ X_{ij} \to X\}\in C(X)$;

\item (Saturation) If $\{ X_i \to X\}\in J(X)$ and $\{ Y_j \to X\}$ is a set
of morphisms of $R$-functors so that for $i$ there is a $j$ and
a factoring $X_i \to Y_j \to X$ of $X_i \to X$, then $\{ Y_j \to X\}
\in C(X)$.
\end{enumerate}

\noindent These conditions imply
\begin{enumerate}

\item[6.] (Local) If $\{ Y_j \to X\}$ is a set of morphisms so that there exists
a set $\{ X_i \to X \} \in C(X)$ with $\{ Y_j \times_X X_i \to X_i\} \in C(Y_j)$
for all $j$, then $\{ Y_j \to X\} \in C(X)$.
\end{enumerate}
 
 If we are simply given, for each $X$, a collection of morphisms
$J_0(X)$ satisfying the axioms (1), (3), and (4), then we have
a {\it pretopology}; the full topology can  be obtained by completing
in the evident manner using axioms (2) and (5).\index{pretopology}

Notice that while the preceding discussion defines a topology
on $R$-functors we can restrict to a topology on schemes
by simply considering only those $R$-functors which are schemes.
Note, however, that the covering sieves $J(X)$ may contain
$R$-functors which are not schemes.

A category of schemes $\cC$ with a collection $J(X)$ of covering
sieves is called a {\it site}.\index{site} 
For example, if $X$ is a scheme, the {\it Zariski site}\index{site, Zariski} on $X$ is
has base category $\cC$ the set of
open immersions $U \to X$. A covering family for $U$ is a finite
set of open immersions $U_i \to U$ so that $\sqcup U_i \to U$
is surjective. This is a pretopology, and we get the topology
by extending as above. The {\it small \'etale site} on $X$
has base category the \'etale morphisms $U \to X$; a covering
family is a finite family of \'etale maps $U_i \to U$ so that
$\sqcup U_i \to U$ is surjective. The big \'etale topology has
as its underlying category the category of all schemes over
$X$. The covering families remain the same.\index{site, \'etale}

This examples can be produced in another way. Let $P$ be a
fixed property of schemes closed under base change and composition, and
with the property that open immersions have property $P$. We then
define a $P$-cover of an affine scheme $U$ to be a finite
collection of morphisms $U_i \to U$ with affine source and
satisfying property $P$. For a general scheme $X$ a $P$-cover
is a finite collection of morphisms $V_i \to X$ so that for all
affine open subschemes $U$ of $X$, the morphisms $V_i \times_X U \to U$
become a $P$-cover. This is a pretopology and, from this, we get the
$P$-topology. If $P$ is the class of \'etale maps
we recover the \'etale sites; if $P$ is the class open immersions, the
small $P$-site is the Zariski site.

We write $J_\cP(X)$ for the resulting covering sieves. 
These topologies and the $fpqc$ topology about to be defined
are {\it sub-canonical}; \index{topology, sub-canonical}that is, the presheaf of sets represented by
a scheme $X$ is a sheaf; see \cite{SGA3-1} IV. 6.3.1.iii.
\end{rem}

\begin{defn}[{\bf The $fpqc$-topology}]\label{fpqc}The $fpqc$-{\bf topology}
\index{topology, $fpqc$}\index{site, $fpqc$}
on schemes is the topology obtained by taking the class $P$ of
morphisms of schemes to be flat maps. Thus a $fpqc$-cover
of an affine scheme $U$ is a finite collection $U_i \to U$ of
flat morphisms so that $\sqcup U_i \to U$ is surjective and
$fpqc$-cover of an arbitrary scheme $X$ is a finite collection
of morphisms $V_i \to X$ so that $V_i \times_X U \to U$ is
a cover for all affine open $U \subseteq X$. The $fpqc$-site
on $X$ is the category of all schemes over $X$ with the
$fqpc$-topology.
\end{defn} 

A related topology, which we won't use, is the $fppf$-topology
\index{topology, $fppf$}\index{site, $fppf$}
for which we take the class $P$ to be the class of all flat,
finitely presented, and quasi-finite maps.
The acronym $fppf$ stands for ``fid\`element plat de pr\'esentation
finie''; this is self-explanatory.
The acronym $fpqc$ stands for ``fid\`element plat quasi-compact''.  The name
derives from the following result; see \cite{SGA3-1} IV. 6.3.1.v.

\begin{prop}\label{why-fpqc}Let $X$ be a scheme and let $X_i \to X$
be a finite collection of flat, quasi-compact morphisms with the property
that
$$
\sqcup\ X_i \longr X
$$
is surjective. Then $\{ X_i \to X\}$ is a cover for the $fpqc$-topology.
In particular, any flat, surjective, quasi-compact morphism is a cover
for the $fpqc$-topology. 
\end{prop}

\begin{rem}[{\bf Sheaves}]\label{sheaves} Continuing
of synopsis of \cite{SGA3-1} Expos\'e IV, we define and discuss
sheaves. If $X$ is an $R$-functor and $F$ is a sieve on $R$,
then $F$ become a contravariant functor on the category
$\Aff/X$ of affines over $X$. Given a topology on schemes defined
by covering sieves $J(-)$, a {\it sheaf}\index{sheaf} on $X$ is a contravariant functor
$\cF$ on $\Aff/X$ so that for all affines $U \to X$ over $X$ and
all $G \in J(U)$, the evident morphism
$$
\cF(U) \cong \Hom(U,\cF) \longr \Hom(G,\cF)
$$
is an isomorphism. Here $\Hom$ means natural transformations
of contravariant functors. If $G$ is defined by a covering
family $U_i \to U$ of affines, then $\Hom(G,\cF)$ is the equalizer
of 
$$
\xymatrix{
\prod \cF(U_i) \ar@<.5ex>[r] \ar@<-.5ex>[r] & \cF(U_i \times_U U_j)
}
$$
and we recover the more standard definition of a sheaf.
By \cite{SGA3-1} IV.4.3.5, if the topology is generated by
a class of covering families closed under base change, it
is sufficient to check the sheaf condition on those families.

We will be considering only those topologies defined at the end of
Remark \ref{top}; thus all our sieves
will be be obtained by saturation from a class of morphisms
$V \to U$ on affine
schemes which are closed under base change and composition and
contains open immersions. For such topologies, we can restrict
the domain of definition of presheaves to appropriate subcategories
of $\Aff/X$.
\end{rem}

Before proceeding, we need to isolate the following concept.

\begin{defn}\label{super-cartesian}Let $I$ be an indexing category
and let $X = X_\bullet$ be an $I$-diagram of schemes. A {\bf cartesian}
\index{sheaf, cartesian}
$\cO_X$-module sheaf consists of
\begin{enumerate}

\item for each $i\in I$ a quasi-coherent sheaf $\cF_i$ on $X_i$;

\item for each morphism $f:X_i \to X_j$ in the diagram, an
isomorphism
$$
\theta_{f}:f^\ast \cF_j \to \cF_i
$$
of quasi-coherent sheaves
\end{enumerate}
subject to the following compatibility condition:
\begin{enumerate}

\item[{\ }]Given composable arrows $\xymatrix{X_i \rto^f & X_j \rto^g & X_k}$
in the diagram, then we have a commutative diagram
$$
\xymatrix{
f^\ast g^\ast \cF_k \rto^{f^\ast\theta_g} \dto_\cong &
f^\ast \cF_j \dto^{\theta_f}\\
(gf)^\ast \cF_k \rto_{\theta_{gf}} & \cF_i
}
$$
\end{enumerate}
\end{defn}

\begin{rem}[{\bf Quasi-coherent sheaves in other topologies}]\label{mod-sheaves} Let
$X$ be a scheme and consider the topology defined by some
class of morphisms $P$ closed under base change, composition,
and containing open inclusions. We assume further that covering
families are finite and faithfully flat. This includes the Zaraski, \'etale,
and $fpqc$ topologies. Then there is a structure sheaf
$\cO^\cP_X$ on a site with this topology determined by
$$
\cO^\cP_X(\Spec(B) \to X) = B.
$$
Notice that, by the sheaf condition, it is only necessary to specify
$\cO^\cP_X$ on affines.
This a sheaf of rings and we write $\Mod^\cP_X$ for the category
of module sheaves over this sheaf. If $\Mod_X$ is the category
of module sheaves over $\cO_X$ in the Zariski topology
(see \ref{mod-shvs}), then there is an adjoint pair
\begin{equation}\label{epsilon-is-cartesian}
\xymatrix{
\Mod_X \ar@<.5ex>[r]^{\epsilon^\ast} &\ar@<.5ex>[l]^{\epsilon_\ast}
\Mod^\cP_X
}
\end{equation}
with $\epsilon^\ast$ defined by pull-back. The right adjoint $\epsilon_\ast$
is defined by restricting the affine open inclusions $U \to X$ and then
extending by the sheaf condition.

This theory extends well to quasi-coherent sheaves. Define a
module sheaf $\cF \in \Mod_X^\cP$ to be {\it cartesian} if
it is cartesian (as in Definition \ref{super-cartesian}) for the
category of affines over $X$; that is,
given any morphism
$$
\xymatrix@R=5pt{
\Spec(B) \ar[dr] \ar[dd]\\
&X\\
\Spec(A) \ar[ur]
}
$$
in $\Aff/X$, the induced map
$$
B \otimes_A \cF(\Spec(A) \to X) \to \cF(\Spec(B) \to X)
$$
is an isomorphism of $R$-modules. If $\cE \in \Mod_X$ is
quasi-coherent, then $\epsilon^\ast \cE$ is cartesian;
conversely if $\cF$ is cartesian, then $\epsilon_\ast \cF$ is
quasi-coherent. Thus the adjoint pair Equation \ref{epsilon-is-cartesian}
descends to an adjoint pair
\begin{equation}\label{epsilon-is-cartesian-1}
\xymatrix{
\Qmod_X \ar@<.5ex>[r]^{\epsilon^\ast} &\ar@<.5ex>[l]^{\epsilon_\ast}
\Mod^{\cP,\cart}_X.
}
\end{equation}
This is an equivalence of categories; therefore, we drop the 
clumsy notation $\Mod^{\cP,\cart}_X$ and confuse the notion
of a quasi-coherent sheaf with that of cartesian sheaf.
\index{cartesian vs. quasi-coherent}

There are important sheaves in $\Mod^\cP_X$ which are not
cartesian; for example, the sheaf $\Omega_{(-)/X}$ of
differentials over $X$ is quasi-coherent for the Zariski topology,
cartesian for the \'etale topology, but not cartesian
for the $fpqc$ topology.
\end{rem}

We finish this section with a review of an important class of morphisms.

\begin{defn}\label{affine-morphisms}1.) A morphism $f:X \to Y$
of schemes over $R$ is called {\bf affine} if for all morphisms
$\Spec(B) \to Y$, the $R$-functor $\Spec(B) \times_Y X$
is isomorphic to an affine scheme. 
\index{morphism, affine}

2.) A morphism $f:X \to Y$ of schemes
is a {\bf closed embedding} if it is affine and for all flat morphisms
$\Spec(B) \to Y$, the induced morphisms of rings
$$
B = \cO_Y(B) \longr f_\ast \cO_X(B) = \cO_X(\Spec(B) \times_Y X 
\to X)
$$
is surjective.
\index{closed embedding}\index{morphism, closed}

3.) A morphism $f:X \to Y$ of schemes is {\bf separated} if the
diagonal morphism $X \to X \times_Y X$ is a closed embedding.
A scheme over a commutative $R$ is separated if the morphism
$X \to \Spec(R)$ is separated.
\index{morphism, separated}
\end{defn}

If $f:X \to Y$ is an affine morphism of schemes, then the $\cO_Y$ algebra
sheaf $f_\ast \cO_X$ is quasi-coherent.\index{algebra sheaf, quasi-coherent} Conversely, if
$\cB$ is quasi-coherent $\cO_Y$-algebra sheaf, define
a $R$-functor $\Spec_Y(\cB)$ over $Y$ by
$$
\Spec_Y(\cB)(A) = \coprod_{\Spec(A) \to Y} \Alg_A(\cB(\Spec(A) \to Y),A).
$$
Then $q:\Spec_Y(\cB) \to Y$ is an affine morphism of schemes
and $q_\ast \cO_{\Spec_Y(\cB)} \cong \cB$. This gives an equivalence
between the category of quasi-coherent $\cO_Y$-algebras and
the category of affine morphisms over $Y$. Restricting this
equivalence gives a one-to-one correspondence between
closed embeddings $X \to Y$ and ideal sheaves $\cI \subseteq \cO_Y$.

An analogous result with an analogous construction
holds for quasi-coherent sheaves.

\begin{prop}\label{qc-affine}Let $f:X \to Y$ be an affine morphism
of schemes. Then the push-forward functor $f_\ast$ defines
an equivalence of categories between quasi-coherent
sheaves on $X$ and quasi-coherent $f_\ast \cO_X$-module
sheaves on $Y$. In particular, $f_\ast$ is exact.
\end{prop}

If $f:T \to S$ is a morphism of schemes and $X \to S$ is an affine morphism,
the $f^\ast X = T \times_S X$ is also affine. If $f$ is faithfully flat,
we have the following result.

\begin{prop}\label{ff-affine}Let $f:T \to S$ be a faithfully flat morphism
of schemes. The $f^\ast(-)$ defines and equivalence of categories
from the category of schemes affine over $S$ to the category of 
descent problems in schemes affine over $T$.
\end{prop}

\subsection{The tangent scheme}

If $A$ is a commutative ring, let $A(\epsilon) =
A[x]/(x^2)$ be the $A$-algebra of dual numbers. Here we have written
$\epsilon = x + (x^2)$. There is an augmentation $q:A(\epsilon) \to A$ given
by $\epsilon \mapsto 0$. 

Let $R$ be a commutative ring and let $X$ be $R$-functor.
Define the {\it tangent functor}\index{tangent functor}
$\tan_X \to X$ over $X$ to be the functor
$$
\tan_X(A) = X(A(\epsilon))
$$
with the projection induced by the augmentation $q:A(\epsilon) \to A$.
There is a {\it zero section} $s:X \to \tan_X$ induced by the unit
map $A \to A(\epsilon)$. If $X \to S$
is a morphism of $R$-functors, then the relative tangent functor $\tan_{X/S}$
is defined by the pull-back diagram
$$
\xymatrix{
\tan_{X/S} \rto \dto & \tan_X\dto\\
S \rto_-s & \tan_S\\
}
$$
If we let $A(\epsilon_1,\epsilon_2) =
A[x,y]/(x^2,xy,y^2)$, then the
natural $A$-algebra homomorphism $A(\epsilon) \to 
A(\epsilon_1,\epsilon_2)$ given by $\epsilon
\mapsto \epsilon_1 + \epsilon_2$ defines a multiplication over $X$
$$
\tan_{X/S} \times_X \tan_{X/S} \to \tan_{X/S}
$$
so that $\tan_{X/S}$ is an abelian group $R$-functor over
$X$.

If $X \to S$ is a morphism of schemes, then $\tan_{X/S}$ is
an affine scheme over $X$. See Proposition \ref{tan-is-affine}.
We will see this once we have discussed
the connection between the $\tan_{X/S}$ and the sheaf of
differentials $\Omega_{X/S}$. 

Let $X$ be an $R$-functor for some commutative ring $R$.
Define the $\cO_X$-module
presheaf of differential $\Omega_{X/R}$ by the formula
$$
\Omega_{X/R}(\Spec(B) \to X) = \Omega_{B/R}.
$$
This became a quasi-coherent sheaf in the Zariski topology. If
$f:X \to Y$ is a morphism of $R$-functors, define $\Omega_{X/Y}$
by the exact sequence of $\cO_X$-modules (in the Zariski topology)
$$
f^\ast \Omega_{Y/R} \to \cO_{X/R} \to \Omega_{X/Y} \to 0.
$$
Since $\Omega_{B/R} = J(B)/J(B)^2$ where $J(B)$ is the kernel
of the multiplication map
$$
B \otimes_R B \longr B
$$  
this definition can be reformulated as follows. A proof can be found in 
\cite{DG} \S I.4.2.

\begin{lem}\label{sheaf-of-diff}\index{differentials} Let $X \to S$ be a separated morphism of
schemes, so that  diagonal
morphism $:\Delta:X \to X \times_S X$ is a closed embedding. 
Then there is a natural isomorphism $\Omega_{X/S}$
of quasi-coherent sheaves on $X$
$$
\Omega_{X/S} \cong \Delta^\ast \cJ/\cJ^2
$$
where $\cJ$ is the module of the closed embedding $\Delta$.
\end{lem}

If $X$ is not separated, we can still identify the differentials by a 
variation on this method:  if we factor the diagonal
map as a closed embedding followed by an open inclusion
$$
\xymatrix{
X \rto^j & V \rto & X \times_S X
}
$$
then $\Omega_{X/S} \cong i^\ast \cJ/\cJ^2$ where $\cJ$ is the
ideal defining $j$.

Needless to say, there is a close connection between differentials
and derivations. If $R$ is a commutative ring and $M$ is an
$R$-module, the {\it square-zero extension} of $R$ by $M$ is
the $R$-algebra $R \rtimes M$ which is $R \times M$ as an $R$-module
and multiplication
$$
(a,x)(b,y) = (ab, ay + bx).
$$
This has an extension to sheaves.\index{derivations}\index{square-zero
extension}

\begin{defn}\label{derivations} Let $\cF$ be a quasi-coherent
sheaf on a scheme $X$. Then we define the $\cO_X$-algebra
sheaf $\cO_X \rtimes \cF$ on $X$ to be the 
square-zero extension of $\cO_X$. Then a derivation of $X$
with coefficients in $\cF$
is a diagram of sheaves of commutative rings
$$
\xymatrix@C=7pt{
\cO_X \ar[rr]^-f \ar[dr]_{=} &&\cO_X \rtimes \cF\ar[dl]^{p_1}\\
&\cO_X.
}
$$
If $q:X \to S$ is a scheme over $S$, then an $S$-derivation
of $X$ with coefficients in $\cF$ is a derivation of $X$ with
coefficients in $\cF$ so that
$$
q_\ast f: q_\ast \cO_X \longr q_\ast \cO_X \rtimes q_\ast \cF
$$
is a morphism of $\cO_S$-algebra sheaves.
We will write $\Der_S(X,\cF)$ for the set of all $S$-derivations of $X$ with
coefficients in $\cF$.
\end{defn}

\begin{exam}\label{universal-derivation} Suppose $X \to S$
is a separated morphism of schemes. Then, by definition,
$\Delta:X \to X\times_SX$ is a closed embedding; let $\cJ$ be
the ideal of this embedding. 
Write $(X \times_S X)_1 \subseteq X \times_S X$ for the
subscheme defined by the vanishing of $\cJ^2$.
Then the splitting provided by the first projection
$p_1:X \times_S X \to X$ defines an isomorphism
$$
\Delta^\ast \cO_{(X \times_S X)_1} \cong \cO_X\rtimes \Omega_{X/S}.
$$
Then the second projection defines an $S$-derivation of $X$
$$
f_u: \cO_X \longr \cO_X \rtimes \Omega_{X/S}$$
The morphisms $f_u$ or the
resulting morphism $d:\cO_X \to \Omega_{X/S}$ is
called the {\it universal derivation}.\index{derivation, universal}
\end{exam}

The module of $S$-derivations is the global sections of the sheaf
$\shder_S(X,\cF)$ which assigns to each Zariski open $U \subseteq X$ the
module of derivations 
$$
\Der_S(U,\cF\vert_U).
$$
This is an $\cO_X$-module sheaf, although not necessarily quasi-coherent.

\begin{prop}\label{der-and-cotangent} There is a natural isomorphism
of $\cO_X$-module sheaves
$$
\hom_{\cO_X}(\Omega_{X/S},\cF) \longr \shder_S(X,\cF)
$$
given by composing with the universal derivation.
\end{prop}

\begin{proof} The inverse to this this morphism is given as follows.
Let  $f:\cO_X \to \cO_X \rtimes \cF$ be any derivation and let
$f_0$ be the zero derivation; that is, inclusion into the first factor.
Also let $p: \cO_X \rtimes \cF \to \cO_X$ be the projection.
Consider the lifting problem
$$
\xymatrix{
X \dto_p \rto& (X \times_S X)_1 \dto^{\subseteq}\\
\Spec_X( \cO_X \rtimes \cF) \ar@{-->}[ur]\rto_-{f_0 \times f} & X \times_S X.
}
$$
Here we have written $h$ for a morphism when we mean $\Spec_X(h)$. Since
$ \cO_X \rtimes \cF$ is a square-zero extension, this lifting
problem has a unique solution $g$ and that $g$ yields a morphism
$$
\Delta^\ast \cO_{(X\times_S X)_1} \cong \cO_X \rtimes \Omega_{X/S}
\to  \cO_X \rtimes \cF
$$
of $\cO_X$-algebra sheaves over $\cO_X$ as needed.
\end{proof}

The following result follows immediately from the previous proposition
upon taking global sections. Note that if $X = \Spec(B) \to \Spec(R)$,
this amounts to the usual isomorphism
$$
\Mod_B(\Omega_{B/R},M) \cong \Der_R(B,M).
$$

\begin{cor}\label{global-der-cotan}This is a natural isomorphism
of modules over the global sections over $X$
$$
\Mod_X(\Omega_{X/S},\cF)\cong \Der_S(X,\cF)
$$
given by composing with the universal derivation.
\end{cor}

If $\cF$ is a quasi-coherent sheaf of $\cO_X$-modules on a scheme $X$,
we can form the symmetric algebra
$\mathrm{Sym}_{\cO_X}(\cF)$; this is a sheaf
of quasi-coherent $\cO_X$-algebras on $X$ and we denote by
$$
\VV(\cF) \longr X
$$
the resulting affine morphism. \index{symmetric algebra sheaf, $\VV(-)$}If $A$ is an $R$-algebra,
then
$$
\VV(\cF)(A) = \coprod_{\Spec(A) \to X}\Mod_A(\cF(A),A).
$$
The diagonal map
$\cF \to \cF \oplus \cF$ gives $\VV(\cF)$ the structure of
an abelian group scheme over $X$.

Proposition \ref{der-and-cotangent}
implies the following result -- in the latter 
proposition set $\cF=\cO_X$ and note that
$$
\cO_X(\epsilon) = \cO_X \rtimes \cO_X.
$$

\begin{prop}\label{tan-is-affine}\index{tangent scheme}If $X \to S$ is a separated morphism
of schemes, there is a natural isomorphism of
abelian group schemes over $X$
$$
\VV(\Omega_{X/S}) \cong \tan_{X/S}.
$$
\end{prop}

The following standard fact is useful for calculations.

\begin{lem}\label{fund-seq}Let $i:X \to Y$ be a closed embedding of
separated schemes over $S$ defined by an ideal $\cI \subseteq \cO_Y$.
Then there is an exact sequence of sheaves on $X$
$$
\xymatrix{
i^\ast \cI/\cI^2 \rto^{d} & i^\ast \Omega_{Y/S}  \rto &\Omega_{X/S} \rto &0
}
$$
where $d$ is induced by the restriction of the universal derivation.
\end{lem}

\begin{proof} Let $\cF$ be a sheaf of $\cO_X$-modules on $X$. The
statement of the lemma is equivalent to the exactness of the 
sequence
$$
0 \to \shder_S(X,\cF) \to \shder_S(Y,i_\ast \cF) \to
\hom_{\cO_Y}(\cI/\cI^2,i_\ast \cF)
$$
which is easily checked.
\end{proof}

\subsection{Formal Lie varieties}

We next review the notion of a formal Lie variety, which can be
interpreted as a notion of a smooth formal scheme affine over 
a base scheme $S$ with a preferred section. The first concept
(which appeared implicitly in Lemma \ref{fund-seq}) is important
in its own right.

\begin{defn}\label{embed-mod}\index{conormal sheaf}Let
$i:X \to Y$ be
a closed embedding of schemes defined by 
an ideal $\cI \subseteq \cO_Y$. Then the quasi-coherent $\cO_X$-module
$$
\omega_i \defeq i^\ast \cI/\cI^2
$$
is called {\bf conormal sheaf} or the module of the embedding $i$.
\end{defn}

Note that the canonical map $\cI/\cI^2 \to i_\ast \omega_i$ of quasi-coherent
sheaves on $Y$ is an isomorphism.

\begin{defn}\label{inf-nghds}\index{infinitesimal neighborhood} If $i:X \to Y$ is a closed embedding of schemes defined by 
an ideal $\cI$, define  the $n$th {\bf infinitesimal neighborhood}
$$
Y_n = \Inf^n_X(Y) \subseteq Y
$$
of $X$ in $Y$ to be the closed subscheme of $Y$ defined by the ideal
$\cI^{n+1}$.

More generally, suppose that $X \to Y$ is an injection of $fpqc$ sheaves
over some base scheme $S$. Define $\Inf^n_X(Y) \subseteq Y$ to
be the subsheaf with the following sections. If $U \to S$
is quasi-compact, then $[\Inf^n_X(Y)](U)$ is the set of all $a \in Y(U)$
which satisfy the following condition: there is an $fpqc$ cover
$V \to U$ and a closed subscheme $V' \subseteq V$ defined by
an ideal with vanishing $(n+1)$st power so that
$$
a\vert_{V} \in X(V').
$$
\end{defn}

\begin{rem}\label{nat-inf} The proof that the notion of infinitesimal neighborhoods for sheaves generalizes that for schemes is Lemma II.1.02 of \cite{messing}.
This lemma is stated for the $fppf$-topology, but uses only properties
of faithfully flat maps of affine schemes, so applies equally well
to the $fpqc$-topology.
In the same reference, Lemma II.1.03, Messing shows that infinitesimal
neighborhoods behave well with respect to base change. Specifically,
if $X \subseteq Y$ is an embedding of $fpqc$-sheaves over a scheme
$S$ and $f:T \to S$ is a morphism of schemes, then
\begin{equation}\label{inf-pb}
\Inf^n_{f^\ast X}(f^\ast Y) \cong f^\ast \Inf_X^n(Y).
\end{equation}
\end{rem}

If $X \to Y$ is a closed embedding of schemes, we get an ascending chain of closed subschemes
$$
X = Y_0 \subseteq Y_1 \subseteq Y_2 \subseteq \cdots \subseteq Y.
$$
The conormal sheaves $X \to Y_n$ are all canonically isomorphic; hence
this module depends only on $Y_1$. To get an invariant which
depends on $Y_n$, filter $\cO_Y$ by the powers of the ideal
$\cI$ to get a graded $\cO_Y/\cI$-algebra sheaf on $Y$. By Proposition
\ref{qc-affine} this determines  a unique graded $\cO_X$-algebra sheaf 
$\gr_\ast(Y)$ on $X$ with
$$
\gr_k(Y) = i^\ast (\cI^k/\cI^{k+1})
$$
In particular, $\gr_1(Y) = \omega_i$. We immediately have that
$$
\gr_k(Y_n) = \brackets{\gr_k(Y),}{k \leq n;}{0,}{k > n.}
$$

Now suppose $Y$ is a scheme
over $S$ and $e:S \to Y$ is a closed
inclusion and a section of the projection $Y \to S$. Let
us define 
\begin{equation}\label{f-vanish-at-zero}
\cO_Y(e) \subseteq \cO_Y
\end{equation}
to be the ideal sheaf defining this inclusion. It can be thought of as the
sheaf of functions vanishing at $e$.\index{ideal sheaf of a point,
$\cO_Y(e)$} In this case the natural
map of Lemma \ref{fund-seq}
$$
d:\omega_e = e^\ast \cO_Y(e)/\cO_Y(e)^2 \longr e^\ast \Omega_{Y/S}
$$
becomes an isomorphism; indeed, the exact sequence of the proof
collapses to an isomorphism.

\begin{rem}\label{lie-pres} Let $S$ be a scheme, $X$ a sheaf in the $fpqc$-topology over
$S$ and $e:S \to X$ a section of the structure map $X \to S$.
Then we can make the following definitions and constructions.
\begin{enumerate}

\item Let $X_n = \Inf^n_S(X)$. We say $X$ is {\it ind-infinitesmal} if the natural map
$$
\colim \Inf^n_S(X) \to X
$$
is an isomorphism of sheaves.

\item Suppose each of the $X_n$ is representable. Then $\omega_e$
can be defined as the conormal sheaf of any of the embeddings
$S \to X_n$; furthermore $\omega_e \cong e^\ast \Omega_{X_n/S}$
for all $n$.

\item More generally, if each of the $X_n$ is representable define the
graded ring $\gr_\ast (X) = \lim \gr_\ast (X_n)$; then if $k \leq n$
$$
\gr_k(X) \cong \gr_k(X_n).
$$
\end{enumerate}
\end{rem}

\begin{defn}[{\bf Formal Lie variety}]\label{frml-var}\index{formal Lie variety}Let $S$ be a scheme, $X$
a sheaf in the $fpqc$-topology over $S$ and $e:S \to X$ a section
of the structure map $X \to S$. Then $(X,e)$ is a {\bf formal Lie variety}
if
\begin{enumerate}

\item $X$ is ind-infinitesmal and $X_n = \Inf^n_S(X)$ is representable
and affine over $S$ for all $n \geq 0$;

\item the quasi-coherent sheaf $\omega_e$ is locally free of finite rank on $S$;

\item the natural map of graded rings $\Sym_\ast (\omega_e) \to
\gr_\ast(X)$ is an isomorphism.
\end{enumerate}
A morphism $f:(X,e) \to (X',e')$ of formal Lie varieties is morphism
of sheaves which preserves the sections.
\end{defn}

\begin{rem}\label{liev-pb} By Remark \ref{nat-inf} and Equation \ref{inf-pb} formal Lie varieties behave well
under base change. If $(X,e)$ is a formal Lie variety over a base
scheme $S$ and $T \to S$ is a morphism of schemes, then $f^\ast X \to
T$ has an induced section $f^\ast e$ and
$$
f^\ast(X,e) \defeq (f^\ast X,f^\ast e)
$$
is a formal Lie variety. We will often drop the section $e$ from  the
notation.
\end{rem}.

\begin{rem}\label{pwr-series-is-lie} We show that, locally in the Zariski
topology, every formal Lie variety is isomorphic to the formal spectrum
of a power series ring.

1.) Let $S = \Spec(A)$ and let $X$ be the
formal scheme $\Spf(A[[x_1,\cdots,x_t]])$. Thus for an $A$-algebra
$B$,
$$
X(B) = \{\ (b_1,\ldots,b_t)\ |\ \hbox{$b_i$ is nilpotent}\ \} \subseteq
B^n
$$
Let $e:S \to X$ be the zero section, then
$$
X_n = \Spec(A[x_1,\cdots,x_t]/(x_1,\ldots,x_t)^{n+1})
$$
and $\omega_e$ is the sheaf obtained from the free $A$-module of
rank $t$ generated by the residue classes of $x_1,\cdots,x_t$.
It  follows that $(X,e)$ is a formal Lie variety.

2.) Conversely, suppose that $S = \Spec(A)$ and that the global
sections of $\omega_e$ on $S$ is a free $A$-module with
a chosen basis $x_1,\ldots,x_t$. There is an exact sequence
of quasi-coherent $\cO_{X_n}$-sheaves
$$
0 \to \cO_{X_n}(e) \to \cO_{X_n} \to e_\ast \cO_S \to 0
$$
whence an  exact sequences of quasi-coherent
sheaves on $\cO_S$
$$
0 \to q_\ast \cO_{X_n}(e) \to q_\ast \cO_{X_n} \to \cO_S \to 0.
$$
Here we are writing $q:X_n \to S$ for the projection.
Since $q_\ast \cO_{X_n}(e) \to q_\ast \cO_{X_{n-1}}(e)$ is onto for all $n$ and
$q_\ast \cO_{X_1}(e) \cong \omega_e$ we may choose compatible  (in
$n$) splittings $\omega_e \to q_\ast \cO_{X_n}(e)$ and get compatible
maps
$$
\Sym_S(\omega_e) \to q_\ast \cO_{X_n}
$$
which, by Definition \ref{frml-var}.3, induce isomorphisms
$$
\Sym_S(\omega_e)/\cJ^{n+1} \to q_\ast\cO_{X_n}
$$
where $\cJ = \oplus_{k>0}\Sym_k(\omega_e)$ is the augmentation
ideal.
Since the global sections of $\Sym_S(\omega_e)/\cJ^{n+1}$ are
isomorphic to $A[x_1,\cdots,x_t]/(x_1,\cdots,x_t)^{n+1}$ we say
that the choice of the basis for the global sections of $\omega_e$
and the choice of compatible splittings yield an isomorphism
$X \cong \Spf(A[[x_1,\cdots,x_t]])$ of formal Lie varieties. This
isomorphism is very non-canonical, however.

3.) Finally, for a general base scheme $S$, choose an open cover
by affines $U_i =  \Spec(A_i)$ so that the sections of $\omega_e$ over $U_i$
is free. Then, after making suitable choices, we get an isomorphism
$$
U_i \times_S X \cong \Spf(A_i[[x_1,\cdots,x_t]]).
$$
\end{rem}

\begin{rem}[{\bf The tangent scheme of a formal Lie variety}]\label{tan-for}
\index{tangent scheme, formal}
Let $(X,e)$ be a formal Lie variety over a scheme $S$. Then $\tan_{X/S}$
is not necessarily scheme, but only a sheaf in the $fpqc$ topology. We'd
like to give a description of $\tan_{X/S}$ as a formal Lie variety.
Define $\lie_{X/S}$ as the pull-back
$$
\xymatrix{
\lie_{X/S} \rto^\vare \dto & \tan_{X/S} \dto\\
S \rto_e & X.
}
$$
More generally, define $(\tan_{X/S})_n$ by the pull-back diagram
$$
\xymatrix{
(\tan_{X/S})_n \rto \dto & \tan_{X/S} \dto\\
X_n \rto & X,
}
$$
so that $(\tan_{X/S})_1 = \lie_{X/S}$. There are natural maps
$$
\tan_{X_n/S} \to (\tan_{X/S})_n
$$
but these are not in general isomorphisms; however, we do have that
$$
\colim \tan_{X_n/S} \to \colim (\tan_{X/S})_n \to \tan_{X/S}
$$
are all isomorphisms. 

To analyze the sheaves $(\tan_{X/S})_n$
let us write $j_k:X_n \to X_{n+k}$
for the inclusion. Then Lemma \ref{fund-seq} shows that for all $k > 0$,
the natural map
$$
j_{k+1}^\ast \Omega_{X_{n+k+1}/S} \to j_k^\ast \Omega_{X_{n+k}/S}
$$
is an isomorphism of locally free sheaves on $X_n$. Write $(\Omega_{X/S})_n$
for this sheaf.
Again Lemma \ref{fund-seq} shows that there is a surjection
$$
(\Omega_{X/S})_n \longr \Omega_{X_n/S}
$$
but the source is a locally free sheaf and the target is not in general. For example,
if $n=1$, $(\Omega_{X/S})_1 = \omega_e$ but $\Omega_{X_1/S} = 0$.
Now we check, using that $X = \colim X_n$ and Lemma \ref{tan-is-affine} that
there is a natural isomorphism of abelian sheaves over $X_n$
$$
\VV_{X_n}((\Omega_{X/S})_n) \cong (\tan_{X/S})_n.
$$
In particular
$$
\VV_S(\omega_e) \cong \lie_{X/S}.
$$
The natural map $\omega_e = q_\ast e_\ast \omega_e \to q_\ast 
(\Omega_{X/S})_n$ of quasi-coherent sheaves on $S$ defines
a coherent sequence of projections
$$
 (\tan_{X/S})_n \longr \lie_{X/S}
 $$
 and $\vare:\lie_{X/S} \to  (\tan_{X/S})_n$ is a section of this
 projection.
 Local calculations, using Remark  \ref{pwr-series-is-lie},
 now imply that the map $(\tan_{X/S},\vare)$
 is a formal Lie variety over $\lie_{X/S}$; the scheme $(\tan_{X/S})_n$
 is the $n$th infinitesimal neighborhood of $\lie_{X/S}$ in
 $\tan_{X/S}$.

The local calculations are instructive. 
If $S = \Spec(A)$ and suppose $X = \Spf(A[[x_1,\ldots,x_t]]$ with $e:S \to X$
defined by the ideal $I = (x_1,\ldots,x_n)$, then
$$
(\tan_{X/S})_n \cong \Spec((A[[x_1,\ldots,x_]]/I^{n})[dx_1,\ldots,dx_t])
$$
In particular
$$
\lie_{X/S} \cong \Spec(A[dx_1,\ldots,dx_t]).
$$
The projection $(\tan_{X/S})_n \to \lie_{X/S}$ is induced by the
natural inclusion of $A$  into $A[[x_1,\ldots,x_t]]/I^{n}$.

It is worth noting that in the case where $t=1$,
$$
\tan_{X_n/S} = \Spec(A[x,dx]/(x^n,nx^{n-1}dx)).
$$
\end{rem}

\begin{rem}\label{lie-morphisms} Let $f:(X,e_x) \to (Y,e_Y)$ be a morphism
of formal Lie varieties over a fixed base scheme $S$. Then $f$ determines
a sequence of morphims of schemes affine over $S$
$$
f_n:X_n = \Inf_S^n(X) \to \Inf_S^n(Y) = Y_n
$$
with the property $\Inf_S^n(f_k) = f_n$ when $k \geq n$. Conversely,
given any such sequence of morphisms define $f:X \to Y$ by
$f = \colim f_n$; then $f$ is a morphism of formal Lie varieties. Thus
we have a one-to-one correspondence between morphisms of formal Lie varieties
and compatible sequences of morphisms on infinitesimal neighborhoods.
This is the key to following results.
\end{rem}

\begin{lem}\label{iso-sheaf-lie}Let $X$ and $Y$ be two formal Lie varieties
over a scheme $S$ and define the presheaf $\Iso_S(X,Y)$ to the functor
which assigns to each morphism $i:U \to S$ of schemes the set 
of isomorphisms $i^\ast X \to i^\ast Y$ of formal Lie varieties. Then 
$\Iso_S(X,Y)$ is a sheaf in the $fpqc$-topology.
\end{lem}

\begin{proof}Suppose $f:T \to S$ is a quasi-compact and faithfully
flat morphism of schemes and suppose $\phi:f^\ast X \to f^\ast Y$ is
an isomorphism of formal Lie varieties so that $d_1^\ast \phi =
d_0^\ast \phi$ over $T \times_S T$. Then 
$$
\phi_n: f^\ast X_n \longr f^\ast Y_n
$$
also satisfies the sheaf condition. Thus, by faithfully flat descent for
affine schemes (Proposition \ref{ff-affine}), there is a unique
isomorphism of affine schemes $\psi_n:X_n \to Y_n$ so that
$f^\ast \psi_n = \phi_n$. By uniqueness $\Inf_S^n(\phi_k) = \phi_n$.
Set $\psi = \colim \psi_n$. Then $f^\ast \psi = \phi$ as needed.
This argument extends to the entire site by replacing $S$ by
$U$ and $X$ and $Y$ by $i^\ast X$ and $i^\ast Y$ respectively.
\end{proof}

The notion of descent problem was defined in Remark \ref{ff-descent}.
The following result can be upgraded to an equivalence of categories,
as in Proposition \ref{ff-affine}.

\begin{lem}\label{eff-desc-lie}Let $f:T \to S$ be a faithfully flat quasi-compact 
morphism of schemes. Let $(X,\phi)$ be a descent problem in
formal Lie varieties over $T$. Then there is a unique (up to isomorphism)
solution in
formal Lie varieties over $S$.
\end{lem}

\begin{proof}This again follows from faithfully flat descent. We begin
by using Proposition \ref{ff-affine} to get unique (up to isomorphism)
schemes $Y_n$
affine over $S$ and isomorphisms $\phi_n:f^\ast Y_n \to X_n$ solving the
descent problem for $X_n$. Uniqueness
implies that there are unique isomorphisms $S \cong Y_0 \cong S$ and
$\inf_S^n(Y_k) \cong Y_n$. Thus $Y = \colim Y_n$ is the candidate
for the solution to the descent problem. We must verify points
(2) and (3) of Definition \ref{frml-var}.

For (2) we have  that $f^\ast \omega_{e_Y} =
\omega_{e_X}$. Since a quasi-coherent 
sheaf $\cF$ over $Y$ is locally free and finitely
generated if and only if $f^\ast \cF$ is locally free and finitely
generated. (See  \cite{EGAIV}\S 2.6.)
For (3), the map (with $e = e_Y$)
$$
\Sym_\ast (\omega_e) \to \gr_\ast(Y)
$$
is an isomorphism because it becomes an isomorphism after
applying $f^\ast(-)$. Thus point (3) is covered.
\end{proof}

The notion of a category fibered in groupoids is defined in D\'efinition
2.1 of \cite{Laumon}. The associated notion of stack is defined in
D\'efinition 3.1 of the same reference.

Define a category $\cM_\lf$ fibered in groupoids over schemes
as follows. The objects of $\cM_\lf$ are pairs $(S,X)$ where
$S$ a scheme and $X \to S$ is a formal Lie variety over $S$.
A morphism $(T,Y) \to (S,X)$ in $\cM_\lf$ is a pair $(f,\phi)$ where
$f:T \to S$ is a morphism of schemes and $\phi:Y \to f^\ast X$
is an isomorphism of formal Lie varieties over $T$.

\begin{prop}\label{lie-is-stack}\index{stack, formal Lie varieties}The category $\cM_\lf$ fibered in groupoids
is a stack in the $fpqc$-topology.
\end{prop}

\begin{proof} For a category fibered in groupoids to be a stack, isomorphisms
must form a sheaf (Lemma \ref{iso-sheaf-lie}) and the groupoids must
satisfy effective descent (Lemma \ref{eff-desc-lie}). 
\end{proof}

%% file: fglaws.tex
\section{Formal groups and coordinates}

In the section, we introduce formal groups and the moduli stack
$\cM_\fg$ of formal groups -- these are the basic objects of study of this
monograph. Except on extremely rare occasions, ``formal
group'' will mean a commutative
group object in formal Lie varieties of relative dimensions 1 over $S$,
as in Definition \ref{formal-group}. Thus may think of $G$ as
affine and smooth of dimension 1 over $S$.

We will begin with a definition of formal group which does
not depend on a theory of coordinates for formal groups;
however, that theory is important, and we will spend part of
the section working out the details. Specifically, we note that
choices of coordinates amount to sections of scheme
over $S$ and we explore the geometry of that scheme.
The main result is Theorem \ref{coord-affine}, which shows
we are dealing with particularly simple scheme.

Part of this section also explores formal group laws, which are 
particulary familiar to homotopy theorists.

\subsection{Formal groups}

We first note
that the category of formal Lie varieties has products. If $X$ and $Y$
are formal Lie varieties over a scheme $S$, let $X \times_S Y$
be the product sheaf in the $fpqc$ topology. We have
that
$$
X \times_S Y = \colim (X_n \times_S Y_n).
$$

\begin{lem}\label{prod-fl}The sheaf $X \times_S Y$ is a formal
Lie variety and the product of $X$ and $Y$ in the category
for formal Lie varieties.
\end{lem}

\begin{proof}We leave most of this as exercise. The key observations
are that
$$
\Inf_S^n(X \times_S Y) = \cup_{p+q =n}\ \Inf_S^p(X) \times
\Inf_S^q(Y)
$$
and that
$$
\omega_{(e_X,e_Y)} = \omega_{e_X} \oplus \omega_{e_Y}.
$$
\end{proof}

This product has a simple description Zariski locally. (Compare
\ref{pwr-series-is-lie}.)
If we choose an affine open $U = \Spec(A) \to S$ over which
the global sections of $\omega_{e_X}$ and $\omega_{e_Y}$
are free with bases $\{x_1,\cdots,x_m\}$ and $\{y_1,\cdots, y_n\}$
respectively, then there is an isomorphism of formal Lie varieties
$$
(X \times_S Y)\vert_U \cong \Spf(A[[x_1,\cdots,x_m,y_1,\cdots,y_n]]).
$$

\begin{defn}\label{formal-group}\index{formal group}Let $S$ be a scheme. A {\bf formal group}
over $S$ is an abelian group object $(G,e)$ in the category of formal
Lie varieties over $S$ with the property that
$$
\omega_G \defeq \omega_e
$$
is locally free of {\bf rank 1}. A homomorphism of formal groups
is a morphism of group objects.
\end{defn}

If $f:T \to S$ is a morphism of schemes and $G$ is a formal group
over $S$, then $f^\ast G = T \times_S G$ is a formal group over
$T$. If $i:U \to S$ is a Zariski open, we write $G\vert_U$ for $i^\ast G$.

\begin{exam}[{\bf Formal group laws}]\label{fg-local}\index{formal group law}A
formal group $(G,e)$ defines
and is defined by a formal group law Zariski locally. 
This is an expansion of Remark \ref{pwr-series-is-lie}.

In more detail, if
$A$ is a commutative ring, a commutative {\it formal group law} of
dimension $1$ is a power series
$$
F(x_1,x_2) = x_1 +_F x_2 \in A[[x_1,x_2]]
$$
so that 
\begin{enumerate}

\item $0 +_F x = x +_F 0 = x$;

\item $x_1 +_F x_2 = x_2 +_F x_1$;

\item $(x_1 +_F x_2) +_F x_3 = x_1 +_F (x_2 +_F x_3)$.
\end{enumerate}

If we think of a formal group law $F$ as the homomorphism
$F:A[[x]] \to A[[x_1,x_2]]$ sending $x$ to $F(x_1,x_2)$, then
$F$ defines a formal group
$G$ over $S= \Spec(A)$ by setting $G = \Spf(A[[x]])$ with
multiplication
$$
\xymatrix@C=50pt{
G \times_S G \cong  \Spf(A[[x_1,x_2]]) \rto^-{\Spf(F)} &
\Spf(A[[x]]) = G.
}
$$

Conversely, if $G$ is a formal group choose a cover 
$U_i= \Spec(A_i) \to S$ by affines so that for each $i$,
the sections of $\omega_G$ is free of rank 1. A choice of
generator $x$ for these sections defines an isomorphism
$$
G\vert_{U_i} \cong \Spf(A_i[[x]])
$$
and the multiplication on $G$ defines a continuous morphism
of power series
$$
\Delta: A_i[[x]] \longr A_i[[x_1,x_2]].
$$
Then
$$
F_i(x_1,x_2) \defeq \Delta(x)
$$
is a formal group law.
\end{exam}

\begin{exam}[{\bf Homomorphisms}]\label{fgl-local-hom}
\index{formal group law, homomorphism} Homomorphisms
of formal groups are determined by power series, at least
Zariski locally. A {\bf homomorphism} $\phi:F \to F'$ of formal group laws over $R$ is
a power series $\phi(x) \in xA[[x]]$ so that
\begin{equation}\label{fgl-hom}
\phi(x_1 +_{F} x_2) = \phi(x_1) +_{F'} f(x_2).
\end{equation}
A homomorphism is an {\bf isomorphism} if it is invertible
under composition; that is, if $\phi'(0)$ is a unit in $A$.
Any homomorphism of formal group laws induces a homomorphism
of the formal groups defined by the formal group laws.

Conversely, let $\psi:G \to G'$ be a homomorphism of formal groups over
$S$ and choose a cover $U_i = \Spec(A_i)$ so that the global
sections of both
$\omega_{G}$ and $\omega_{G'}$ are free over $A_i$. Choose
a generator $x$ and $y$ for these global sections and let
$F$ and $F'$ be the associated formal group laws over $A_i$. Then
we get a commutative diagram induced by $\psi$
$$
\xymatrix{
\Spf(A_i[[x]]) \rto^\psi \dto_{\Delta_G} & \Spf(A_i[[y]])\dto_{\Delta_G'}\\
\Spf(A_i[[x_1,x_2]] \rto_{\psi \times \psi}& \Spf(A_i[[y_1,y_2]])
}
$$
If we let $\phi_i(x) = \psi^\ast(y) \in A_i[[x]]$, this diagram implies
$\phi_i:F \to F'$ is a homomorphism of formal group laws
\end{exam}

We now introduce the moduli stack $\cM_\fg$ of 
formal groups -- meaning formal Lie groups of dimension 1.
This stack will be algebraic, although not in the sense of
\cite{Laumon}. See Theorem \ref{fg-algebraic} below.

\begin{defn}\label{stack-of-fgs}\index{stack, formal groups}The {\bf moduli stack of
formal groups} $\cM_\fg$ is the following category fibered in groupoids
over schemes. The objects in $\cM_\fg$ are pairs $(S,G)$ where
$S$ is a scheme and $G \to S$ is a (commutative, 1-dimensional)
formal group over $S$. A morphism $(S,G) \to (T,H)$ is
a pair $(f,\phi)$ where $f:S \to T$ is a morphism of schemes
and $\phi:H \to f^\ast G$ is an isomorphism of formal groups.
\end{defn}

Of course, we still must prove the following result.

\begin{prop}\label{prestack}The category $\cM_\fg$ fibered in
groupoids over schemes is a stack in the fpqc topology.
\end{prop}

\begin{proof} The argument exactly as in Proposition
\ref{lie-is-stack}, once we note that the proofs of
Lemmas \ref{iso-sheaf-lie} and \ref{eff-desc-lie}
immediately apply to this case.
\end{proof}

\subsection{Formal group laws}

Here we review some of the classical literature on formal group
laws.

\begin{thm}[{\bf Lazard}]\label{lazard-rep}1.) Let $\fgl$ denote the functor
from commutative rings to sets which assigns to each ring $A$ the set
of formal group laws over $A$. Then $\fgl$ is an affine scheme; indeed,
if $L \cong \ZZ[x_1,x_2,\cdots]$ is the Lazard ring\index{Lazard ring, $L$}, then
$$
\fgl \cong \Spec(L).
$$

2.) Let $\isofgl$ be the functor which assigns to each commutative
ring $A$ the set of isomorphisms $f:F \to F'$ of formal
group laws over $A$. Then $\isofgl$ is an affine scheme; indeed,
if $W \cong L[a_0^{\pm 1},a_1,a_2,\cdots]$, then
$$
\isofgl \cong \Spec(W).
$$
\end{thm}

Put another way, the functor which assigns to any commutative
ring $A$ the groupoid of formal group laws over $A$ and their
isomorphisms is an affine groupoid scheme; that is,
the pair $(L,W)$ is a {\it Hopf algebroid}.

\begin{rem}\label{sym-2-cocyle}\index{symmetric $2$-cocyle}
It is worth noting that the isomorphism $L \cong \ZZ[x_1,x_2,\ldots]$
is not canonical. The difficult part of Lazard's argument is
the symmetric 2-cocycle lemma (\cite{Rav} A.2.12), which we now revisit.
Let 
$$
C_n(x,y) = \frac{1}{d}[(x+y)^n - x^n - y^n]
$$
where $d = p$ if $n$ is a power of $p$ and $d=1$ otherwise.
This is the $n$th homogeneous symmetric $2$-cocycle. Then
Lazard proves that if $F(x_1,x_2)$
is a formal group law over a ring $A$, then there are elements
$b_1,b_2,\ldots$ in $A$ so that
$$
F(x_1,x_2)  \equiv x_1 + x_2 + b_1 C_2(x_1,x_2) + 
b_2 C_3(x_1,x_2) + \cdots
$$
modulo $(b_1,b_2,\ldots)^2$.

The isomorphism $W = L[a_0^{\pm 1},a_1,\ldots]$ depends
only on the usual coordinates on power series.
\end{rem}

We now introduce the prestack $\cM_{\fgl}$ of formal group
laws. It will not be a stack as it does not satisfy effective descent.

Let $\Aff_\ZZ$ be the category of affine schemes over $\Spec(\ZZ)$.
Recall from \cite{Laumon}, Definition 3.1, that a {\it prestack} $\cM$
over $\Aff_\ZZ$
is a category fibered in groupoids over $\Aff_\ZZ$ 
so that isomorphisms between objects form a sheaf in the
$fpqc$ topology.

If $F(x_1,x_2)$ is a formal group law over a ring $A$ and $f:A \to B$
is a ring homomorphism, let $f^\ast F(x_1,x_2)$ be the formal
group law over $B$ obtained by pushing forward the coefficients.
The resulting formal group over $\Spec(B)$ is the pull-back
of the formal group over $\Spec(A)$ defined by $F$; hence, we will
refer to $f^\ast F$ as the pull-back of $F$ along $f$.

\begin{defn}\label{prestack-of-fgls} Define a category
$\cM_{\fgl}$ fibered in groupoids over $\Aff_\ZZ$ as follows. 
The objects are the pairs $(Spec(A),F)$ where $A$ is a commutative
ring and $F$ is a formal group law over $A$. A morphism 
$(\Spec(A),F) \to (\Spec(B),F)$ is a pair $(f,\phi)$ where
$f:B \to A$ is a homomorphism of commutative rings and
$\phi:F \to f^\ast F'$ is an isomorphism of formal group
laws.\index{prestack, formal group laws}
\end{defn}

\begin{lem}\label{mfgl-prestack}The category $\cM_\fgl$ fibered
in groupoids over $\Aff_\ZZ$ is a prestack.
\end{lem}

\begin{proof} Let $S = \Spec(A)$ and let $F$ and $F'$ be
two formal group laws over $S$. We are asking that the functor
which assigns to each morphism $\Spec(f):U=\Spec(R) \to S$ the set
of isomorphisms
$$
\phi:(f^\ast F) \to (f^\ast F')
$$
be a sheaf in the $fpqc$-topology. 
Theorem \ref{lazard-rep}.2 gives that this functor is the
the affine scheme $\Spec(A \otimes_L W \otimes_L A)$.  The
assertion follows from the fact the the $fpqc$ topology
is sub-canonical. See the end of Remark \ref{top}.
\end{proof}

The functor which assigns to each formal group law $F$
over a ring $A$ the associated formal group $G_F$ over
the affine scheme $\Spec(A)$ defines a morphism
$$
\cM_\fgl \longr \cM_\fg
$$
of prestacks over $\Aff_\ZZ$. This is not an equivalence,
but we will see that this
morphism identifies $\cM_\fg$ as the stack associated to the
prestack $\cM_\fgl$. See Theorem \ref{fgl-stack-is-fg}.

The next  result, which I learned from Neil Strickland, is an
indication that stacks have a place in stable homotopy theory.

\begin{lem}\label{pull-back-for-coord-1} Suppose $F_i$, $i=1,2$
are formal group laws over commutative rings $A_i$ respectively.
Let
$$
G_{i} \to S_i = \Spec(A_i),\qquad i =1,2
$$
be the corresponding formal groups. Then the two category
pull-back $S_1 \times_{\cM_\fg} S_2$ is an affine scheme.
Specifically, if $L \to A_i$ classifies $F_i$, then there
is an isomorphism
$$
S_1 \times_{\cM_\fg} S_2 \cong \Spec(A_1 \otimes_L W
\otimes_L A_2).
$$
\end{lem}

\begin{proof}By construction we have a factoring
$$
\xymatrix@C25pt{
S_i \rto^-{F_i} & \cM_\fgl \rto & \cM_\fg.
}
$$
of the morphism classifying $G_i$. By Remark \ref{fgl-local-hom},
the reduction map $\cM_\fgl \to \cM_\fg$ is full and
faithful; hence, the natural map
$$
S_1 \times_{\cM_\fgl} S_2 \to S_1 \times_{\cM_\fg} S_2
$$
is an isomorphism. If $R$ is any commutative ring
$(S_1 \times_{\cM_\fgl} S_2)(R)$ is the trivial groupoid
with object set the triples
$$
(f_1,f_2,\phi:f^\ast_1F_1 \mathop{\longr}^{\cong}
f^\ast_2F_1)
$$
where $f_i:A_i \to R$ are ring homomorphisms. Applying Theorem
\ref{lazard-rep}.2 now implies the result.
\end{proof}

If $G_1$ and $G_2$ are two formal groups over a scheme $S$,
\index{isomorphism sheaf, $\Iso_S(G_1,G_2)$}
let  $\Iso_S(G_1,G_2)$ be the presheaf of sets which assigns to any morphims
$f:U \to S$ with affine source the set of isomorphisms $f^\ast G_1
\to f^\ast G_2$. There is a pull-back diagram
$$
\xymatrix{
\Iso_S(G_1,G_2) \rto \dto & S \times_{\cM_\fg} S\dto\\
S \rto_-\Delta & S \times S.
}
$$
Proposition \ref{prestack} implies that 
$\Iso_S(G_1,G_2)$ is actually a sheaf.
Lemma \ref{pull-back-for-coord-1} immediately implies the following.

\begin{lem}\label{pull-back-for-coord-2}Suppose $F_i$, $i=1,2$
are formal group laws over a single commutative ring $A$ and
let
$$
G_{i} \to S = \Spec(A),\qquad i =1,2
$$
be the corresponding formal groups. Then the sheaf $\Iso_S(G_1,G_2)$
is an affine scheme over $S$. 
Specifically, if $L \to A$ classifies $F_i$, then there
is an isomorphism
$$
\Iso_S(G_1,G_2)  \cong \Spec(A \otimes_{A \otimes A}(A \otimes_L W
\otimes_L A)).
$$
\end{lem}

\subsection{Coordinates}

We now begin to discuss when a formal group can arise
from a formal group law. In the following definition, the
base scheme need not be affine. The sheaves $\cO_{G_n}(e)$
where defined in Equation \ref{f-vanish-at-zero} as the kernel
of the map $\cO_{G_n} \to e_\ast \cO_S$. If $G$ is
a formal group with $n$th infinitesimal neighborhood $G_n$,
then there is an exact sequence 
$$
0 \to q_\ast \cO_{G_n}(e) \to q_\ast \cO_{G_n} \to \cO_S \to 0
$$
of sheaves on $S$. If $S = \Spec(A)$ and $\omega_e$ has 
a generating local section then
$$
H^0(S,q_\ast \cO_{G_n}(e)) \cong xA[x]/(x^{n+1}).
$$

\begin{defn}\label{fgls}  Let $S$ be a scheme and $q:G\to S$ a formal
group over $S$ with conormal sheaf $\omega_e$. Then a {\bf coordinate}
\index{coordinate}for $G$ is a global section
$$
x \in \lim H^0(S,q_\ast \cO_{G_n}(e))
$$
so for all affine morphisms $f:U = \Spec(A) \to S$, $x\vert_U$ generates
the global sections if $(\omega_e)\vert_U$.
\end{defn}

Every formal group has coordinates locally, as in Example \ref{fg-local};
this definition asks for a global coordinate.
\def\CP{{{\mathbb{C}\mathrm{P}}}}

\begin{rem}\label{compare-top}If $E^\ast(-)$ is a complex oriented
$2$-periodic homology theory, the associated formal group is
$\Spf(E^0(\CP^\infty)$ over $\Spec(E^0)$. A coordinate is then
a class $x \in \tilde{E}^0\CP^\infty$ which reduces to 
generator of $\tilde{E}^0\CP^1$. This is the usual
topological definition. See \cite{Rav}, Definition 4.1.1.
\end{rem}

\begin{rem}\label{fgls-rems}1.) Let $(G,x)$ be a formal group law over $S$
with a coordinate $x$. 
Since $x$ provides a global trivialization of the locally free sheaf
$\omega_e$, Definition
\ref{frml-var}.3 allows us to conclude that
\begin{equation}\label{nat-fgl-to-coord}
G_n \cong \Spec_S(\cO_S[x]/(x^{n+1})).
\end{equation}
Equivalently, we have $q_\ast \cO_{G_n} = \cO_S[x]/(x^{n+1})$. 
In particular,
$$
\lim H^0(S,q_\ast\cO_{G_n}) \cong H^0(S,\cO_S)[[x]].
$$

2.) Suppose $F$ is a formal group law over a commutative
ring $A$ and $G_F$ is the associated formal group over
$\Spec(A)$, as in Example \ref{fg-local}. Then, as in
Remark \ref{pwr-series-is-lie}.1, $G_F$
has a preferred coordinate $x$ defined by the definition
$$
G_F = \Spf(A[[x]]).
$$
Then there is an equality of formal group laws\index{coordinates,
and formal group laws}
$$
x_1 +_{(G_F,x)} x_2 = x_1 +_F x_2.
$$
Conversely, if $G$ is a formal group over $\Spec(A)$ with
a coordinate $x$, then Equation \ref{nat-fgl-to-coord}
provides a natural isomorphism ({\it not} an equality) of formal groups over
$\Spec(A)$
$$
G_{F} \mathop{\longr}^{\cong} G.
$$

3.) If $f:(G,x) \to (H,y)$ is a homomorphism of formal groups 
over $S$ with chosen coordinates, then $f$ is defined by a
morphism of $\cO_S$-algebra sheaves
$$
\phi:q_\ast \cO_{H_n} \cong \cO_S[[y]]/(y^{n+1}) \to
\cO_S[[x]]/(x^{n+1}) \cong p_\ast \cO_{G_n}
$$
and, thus, is defined by the power series 
$$
\phi(y) = f(x) \in H^0(S,(\cO_S)[[x]]
$$
which is a homomorphism of formal group laws:
$$
f(x_1) +_{F_H} f(x_2) = f(x_1 +_{F_G} x_2).
$$
Conversely, any such power series defines a homomorphism
of formal group laws.

4.) Suppose we are given a $2$-commuting diagram 
$$
\xymatrix@R=5pt{
T \ar[dr]^-H\ar[dd]_g &\\
&\cM_\fg\\
S \ar[ur]_-G
}
$$
and a coordinate for $x$ for $G$ over $S$. Then there is an induced
coordinate for $H$ over $T$. Let $\phi:H \to g^\ast G$ be the
given isomorphism. Then the coordinate for $H$ is the image of $x$ under
the homomorphisms
$$
\xymatrix{
H^0(S,q_\ast \cO_{G_n}(e)) \rto & H^0(T,g^\ast q_\ast \cO_{G_n}(e)) &
\lto^-{\cong}_-{\phi^\ast} H^0(T,q_\ast \cO_{H_n}(e)).
}
$$
\end{rem} 

The next step is to examine the
geometry of the scheme of possible coordinates for $G$.
We begin with the following result.

\begin{lem}\label{coords-after-ff}Let $q:G \to S$ be a formal
group over a quasi-compact and quasi-separated scheme $S$.
Then there is a quasi-compact and quasi-separated
scheme $T$ and a faithfully flat and quasi-compact
morphism $f:T \to S$ so that
$f^\ast G$ has a coordinate. 
\end{lem}

\begin{proof} Choose a finite cover $U_i \to S$ by affine
open subschemes so that the global
section of $(\omega_e)\vert_{U_i}$ are free. Set
$f:T = \sqcup\ U_i \to S$ to be the evident map.
Then $f$ is faifthully flat, quasi-compact 
and $f^\ast \omega_e$ is 
isomorphic to $\cO_{T}$. Since $T$ is a coproduct
of affines, the map
$$
\lim H^0(T,\cO_{f^\ast G_n}(e)) \to H^0(T,
f^\ast \omega_e)
$$
is onto, and we choose as our coordinate any preimage of
a generator. 
\end{proof}

\begin{defn}\label{prestack-of-coords} Define a category
$\cM_{\crd}$ fibered in groupoids over schemes as follows. The objects
of $\cM_{\crd}$ are pairs $(q:G \to S, x)$ where $G$ is
a formal group over a scheme $S$ and $x$ is a
coordinate for $G$. A morphism in $\cM_{\crd}$ 
$$
(q:G \to S, x) \longr (q':G' \to S', x')
$$
is a morphism of schemes $f:S \to S'$ and an isomorphim
of formal groups $\phi:G \to f^\ast G'$.\index{prestack, coordinate}
\end{defn}
 
 By forgetting the coordinate we get a projection map
 $\cM_\crd \to \cM_\fg$; if we consider this as morphism
 of categories fibered in groupoids
 over affine schemes,  Remark \ref{fgls-rems}.3
 factors this projection as the composite
 $$
 \xymatrix{
 \cM_\crd \rto^i & \cM_\fgl \rto & \cM_\fg.
 }
 $$
 
 \begin{prop}\label{coord-equiv-fgl}The morphism $i:\cM_\crd \to
 \cM_\fgl$ of categories fibered in groupoids is an equivalence
 over affine schemes; that is,
 for all commutative rings $A$, the morphism of groupoids
 $$
 \cM_\crd(A) \longr \cM_\fgl(A)
 $$
 is an equivalence of groupoids.
 \end{prop}
 
 \begin{proof}This is a restatement of Remark \ref{fgls-rems}.3 and
 Remark \ref{fgls-rems}.4.
 \end{proof}

\begin{cor}\label{coord-prestack}The category $\cM_\crd$ fibered in
groupoids over schemes is a prestack.
\end{cor}

\begin{proof}This requires only that if $(G,x)$ and $(H,y)$ are
two objects over a scheme $S$, the the isomorphisms
$\Iso_S((G,x),(H,y))$ form a sheaf. But
$$
\Iso_S((G,x),(H,y)) = \Iso_S(G,H)
$$
where $\Iso_S(G,H)$ is the the sheaf (by Proposition \ref{prestack})
of isomorphisms of formal groups.
\end{proof}

We now give extensions of Lemmas \ref{pull-back-for-coord-2}
and \ref{pull-back-for-coord-1}, in that order.

\begin{prop}\label{diagonal} Let $G_1$ and $G_2$ be two formal
groups over a quasi-compact and quasi-separated scheme $S$.
Then
$$
\Iso_S(G_1,G_2) \to S
$$
is an affine morphism of schemes.
\end{prop}

\begin{proof} We prove this by appealing to Lemma
\ref{pull-back-for-coord-2} and faithfully flat descent.

First, suppose $G_1$ and $G_2$ can be each given a coordinate. Then,
for a fixed choice of coordinate for $G_1$ and $G_2$ and
for any morphism of schemes $f:U= \Spec(A) \to S$, the formal
groups $f^\ast G_i$ over $U$ has an induced coordinate,
and Lemma \ref{pull-back-for-coord-2} shows
$$
f^\ast \Iso_S(G_1,G_2) = \Iso_U(f^\ast G_1,f^\ast G_2) \to U
$$
is affine over $U$; indeed,
\begin{equation}\label{affine-for-iso1}
\Iso_U(f^\ast G_1,f^\ast G_2) \cong \Spec(A \otimes_{A \otimes A}
(A \otimes_L W \otimes_L A)).
\end{equation}
Expanding this thought,  define a presheaf $\cA(G_1,G_2)$ 
of $\cO_S$ algebras by
$$
\cA(G_1,G_2)(f:U \to S) = H^0(U,\Iso_U(f^\ast G_1,f^\ast G_2))
$$
where $f:U \to S$ runs over all flat morphisms with
affine source.
Then Equation \ref{affine-for-iso1} implies that $\cA(G_1,G_2)$ is a
quasi-coherent sheaf of $\cO_S$-algebras. We then have
$$
\Spec_S(\cA(G_1,G_2)) \cong \Iso_S(G_1,G_2)
$$
over $S$. If $f:T \to S$ is any morphism of schemes, then $f^\ast G_i$
also can be given a coordinate, by Remark \ref{fgls-rems}.4, and then Lemma
\ref{pull-back-for-coord-2} implies that 
\begin{equation}\label{iso-pull-coord-1}
f^\ast \cA(G_1,G_2) \cong \cA(f^\ast G_1,f^\ast G_2)
\end{equation}
as quasi-coherent $\cO_T$-algebra sheaves. This is equivalent to the
statement that
\begin{equation}\label{iso-pull-coord-2}
T \times_S \Iso_S(G_1,G_2) \cong \Iso_T(f^\ast G_1,f^\ast G_2).
\end{equation}

For the general case, we appeal to Lemma \ref{coords-after-ff} to choose
an $fpqc$-cover $f:T \to S$ so that $f^\ast G_i$ can each be given 
a coordinate. Then $\Iso_T(f^\ast G_1,f^\ast G_2) \cong
\Spec_T(\cA(f^\ast G_1,f^\ast G_2))$
and Equation \ref{iso-pull-coord-1} (or Equation \ref{iso-pull-coord-2}) yields
an isomorphism of quasi-coherent $\cO_{T \times_S T}$-algebra sheaves
$$
\phi:p_1^\ast \cA(f^\ast G_1,f^\ast G_2) \longr p_2^\ast\cA(f^\ast G_1,f^\ast G_2).
$$
We check that this isomorphism satisfies the cocycle condition and we
get, by faithfully flat descent, a quasi-coherent $\cO_S$-algebra sheaf
$\cA(G_1,G_2)$. Uniqueness of descent and Equation \ref{iso-pull-coord-2}
imply that $\Spec_S(\cA(G_1,G_2) \cong \Iso_S(G_1,G_2)$ 
over $S$.
\end{proof}

\begin{cor}\label{pull-backs-affine-scheme}Let $G \to S$ and
$H \to T$ be formal groups over quasi-compact and quasi-separated
schemes. Then the projection morphism
$$
S \times_{\cM_\fg} T \longr S \times T
$$
is an affine morphism of schemes. In particular $S \times_{\cM_\fg} T$
is a scheme over $S$ and it is an affine scheme over $S$ if
$T$ is an affine scheme.
\end{cor}

\begin{proof}One easily checks that there is an isomorphism
$$
S \times_{\cM_\fg} T \cong \Iso_{S \times T}(p_1^\ast G, p_2^\ast H).
$$
Now we use Proposition \ref{diagonal}.
\end{proof}

In the following definition we are going to have a functor
$F$ on affine schemes over a scheme $S$. We'll write
$F\vert_U$ for $F(U)$ to in order to avoid too many 
parentheses. 

\begin{defn}\label{coord-scheme}1.) Let $G$ be a
formal group over a scheme $S$. Define a functor $\coord(G/S)$ from affine
schemes over $S$ to groupoids as follows.
If $i:U \to S$ is any affine morphism, then the
objects of $\coord(G/S)\vert_U$ are pairs $(i^\ast G,x)$ where  $x$
is a coordinate for $i^\ast G$. The morphisms $f:(i^\ast G,x) \to (i^\ast G,y)$
of $\coord(G/S)\vert_U$ are those morphisms of formal group laws
so that the underlying morphism of formal groups 
$f_0=1:i^\ast G \to i^\ast G$ is the identity.

2.) Let us write $\coord_G \to S$ for the functor of objects 
of the groupoid functor $\coord(G/S)$

3.) By \ref{fgls-rems}.4, a morphism $f:(G,x) \to (G,y)$
so that the underlying morphism of formal groups is the identity amounts to
writing the coordinate $y$ as a power series in $x$. We will call this a
change of coordinates.
\end{defn}

In the following result, note that we have an isomorphism, not
simply an equivalence.

\begin{lem}\label{coord-pull-fgl}Let $G \to S$ be a formal group
over a scheme $S$ and let $S \to \cM_\fg$ classify $G$. Then
there is an isomorphism of groupoids over $S$
$$
\lambda:\coord(G/S) \longr S \times_{\cM_\fg} \cM_\fgl.
$$
\end{lem}

\begin{proof} First we define the morphism.
Let $f:U = \Spec(A) \to S$ be a morphism out of
an affine scheme and let $(f^\ast G,x) \in \coord(G/S)\vert_U$.
Define $\lambda(f^\ast G,x)$ to the the triple
$$
(f:U \to S, F, \phi:G_F \to f^\ast G)
$$
where $F$ is the formal group law determined by $x$ (Remark
\ref{fgls-rems}.2) and $\phi$ is the natural isomorphism from the formal
group determined by $F$ (Example \ref{fgl-local-hom}) to $f^\ast G$.
The inverse of $\lambda$ sends $(f,F,\phi)$
to the pair $(f^\ast G,x)$ where $x$ is the coordinate defined
by $\phi$ (Remark \ref{fgls-rems}.4).
\end{proof}

The next result follows immediately from  Proposition \ref{coord-equiv-fgl}.
Notice only have an equivalence in this case.

\begin{cor}\label{coord-pull-crd}Let $G \to S$ be a formal group
over a scheme $S$ and let $S \to \cM_\fg$ classify $G$. Then
there is an equivalence of groupoids over $S$
$$
\lambda:\coord(G/S) \longr S \times_{\cM_\fg} \cM_\crd.
$$
\end{cor}

In the following result, we will call a groupoid scheme $\cG$
over $S$ affine over $S$ if both the projection maps
$\mathrm{ob}\cG \to S$ and $\mathrm{mor}\cG \to S$
are affine morphisms.

\begin{thm}\label{coord-affine}1.) Let $G\to S$ be a formal group
over a quasi-compact and quasi-separated scheme $S$.
Then $\coord(G/S) \to S$ is a groupoid
scheme affine over $S$. 

2.) For all morphisms $f:T \to S$ of schemes, the groupoid
$\coord(G/S)\vert_T$ is either empty or contractible.\index{coordinate
scheme}

3.) The objects $\coord_G \to S$ of $\coord(G/S)$ form
an affine scheme over $S$.
\end{thm}

\begin{proof} Lemma \ref{coord-pull-fgl} and Theorem \ref{lazard-rep}
imply together that the objects and morphisms of $\coord(G/S)$
are, respectively,
$$
S \times_{\cM_\fg} \Spec(L)
$$
and
$$
S \times_{\cM_\fg} \Spec(W).
$$
Part (1) of the theorem follows from Corollary
\ref{pull-backs-affine-scheme}. Since $\coord_G$ is the scheme
of objects in $\coord(G/S)$, part (3) follows from part (1).

For part (2), if $f^\ast G$ has no coordinate, then $\coord(G/S)\vert_T$
is empty. If, however, $f^\ast G$ has a coordinate, then any two coordinates
are connected by a unique isomorphism, by Remark \ref{fgls-rems}.3,
and the groupoid is contractible.
\end{proof}

We remark that we have shown that the
scheme $\coord_G$ of objects is actually a torsor for an appropriate
group scheme. See Lemma \ref{coord-is-torsor}.

\begin{rem}\label{example-to-clarify}Since the proof of Theorem
\ref{coord-affine} is at the end of a logical thread which winds in
way through most of this section, it might be worthwhile to
consider the example where $S = \Spec(B)$ is affine and
$G$ is a formal group which can be given a coordinate $y$.
Then if $f:\Spec(A) = U \to S$ is any morphism from an affine
scheme, and $(f^\ast G,x) \in \coord(G/S)\vert_U$ is any
coordinate for $G$ over $U$, then $x$ can be written in terms
of $y$; that is,
$$
x = a_0y + a_1y^2 + a_2y^3 + \cdots \defeq a(y)
$$
where $a_i \in A$ and $a_0$ is invertible. From this we see that
the choice of the coordinate $y$ defines an isomorphism
of schemes
$$
\coord_G \cong \Spec(B[a_0^{\pm 1},a_1,a_2,\cdots])
\cong \Spec(B \otimes_L W)
\cong S \times_{\cM_\fg} \Spec(L).
$$
An isomorphism $\phi:(f^\ast G,x_0) \to (f^\ast G,x_1)$ 
in $\coord(G/S)\vert_U$ is determined by a power series
$$
x_1 = \lambda_0x_0 +  \lambda_1 x_0^2 +  \lambda_2 x_0^3 + \cdots
=  \lambda(x_0)
$$
and $x_1 =  \lambda(a(y))$. This shows the choice of the coordinate
$y$ defines an isomorphism of schemes from the morphisms of 
$\coord(G/S)$ to 
$$
\Spec(B[a_0^{\pm 1},a_1,\ldots, \lambda_0^{\pm 1},  \lambda_1,
\ldots]) \cong \Spec(B \otimes_L W \otimes_L W).
$$
\end{rem}

\subsection{$\cM_\fg$ is an $fpqc$-algebraic stack}

We recall the notion of a representable morphism of stacks and
what is means for such a morphism to have geometric properties.
All our stacks are categories fibered in groupoids over affine schemes.

\begin{defn}\label{stacks-over-flat}1.) A morphism $\cN \to \cM$
of stacks is {\bf representable} if for all morphisms $U \to \cM$
with affine source, the $2$-category pull-back $U \times_\cM \cN$
is a scheme.

2.) Let $P$ be some property of morphisms of schemes closed under base change
and let $f:\cN \to \cM$ be a representable morphism of stacks,
then $f$ has property $P$ if the induced morphism 
$$
U \times_\cM \cN \longr U
$$
has property $P$ for all morphisms $U \to \cM$ with affine
source.
\end{defn}

In the situations which arise here, there are times when we only
have to check the property $P$ once. This will happen, for
example, with flat maps. The results is the following.

\begin{lem}\label{check-once} Let $P$ be some property of
morphisms of schemes closed under base change, and suppose
$P$ has the following further property:
\begin{enumerate}

\item[$\bullet$] Let $f:X \to Y$ be a morphism of schemes and let
$g:Z \to Y$ be a faifthfully flat morphism of schemes. Then $f$ has property $P$ if
and only if $Z \times_Y X \to Z$ has property $P$.
\end{enumerate}

Then if $X \to \cM$ is a presentation of $\cM$ by an affine
scheme, a representable morphism of stacks $\cN \to \cM$
has property $P$ if and only if $X \times_\cM \cN \to X$
has property $P$.
\end{lem}

Now we define the notion of algebraic stack used in this monograph.

\begin{defn}\label{alg-stack}Let $Y$ be a scheme and $\cM$ any
stack over $Y$. Then $\cM$ is
an {\bf algebraic stack} in the $fpqc$-topology or more succinctly an
$fpqc$-{\bf stack} if\index{stack, $fpqc$}
\begin{enumerate}

\item the diagonal morphism $\cM \to \cM \times_Y \cM$ is
representable, separated, and quasi-compact; and

\item there a scheme $X$ and a surjective, flat, and quasi-compact morphism
$X \to \cM$. The morphism $X \to \cM$ is called a {\bf presentation} of $\cM$.
\end{enumerate}
\end{defn}

This is a relaxation of the usual definition of algebraic stack (as in \cite{Laumon},
D\'efinition 4.1) where the presentation $X \to \cM$ is required to be
smooth, so in particular flat and locally of finite type. It turns out that
$\cM_\fg$ can be approximated by such stacks, as we see in the
next chapter.

The following result is obtained by combining Propositions
\ref{diagonal-for-mfg} and \ref{lazard-covers} below.

\begin{thm}\label{fg-algebraic}The moduli stack $\cM_\fg$
is an algebraic stack over $\Spec(\ZZ)$ in the $fpqc$-topology. 
Let $\fgl = \Spec(L)$\index{$\fgl$} be
the affine scheme of formal group laws and let $G_F \to \fgl$
be the formal group arising from the universal formal group law.
Then
$$
G_F: \fgl \longr \cM_\fg
$$
is a presentation for $\cM_\fg$.
\end{thm}

Let $\cM$ be a stack and $x_1,x_2:S \to \cM$
be two $1$-morphisms. Then the $2$-category pull-back of
$$
\xymatrix@C=20pt{
&&\cM\dto^\Delta\\
S \ar[rr]^-{(x_1,x_2)} && \cM \times_Y \cM
}
$$
is equivalent to the $fpqc$-sheaf $\Iso_S(x_1,x_2)$ which assigns to
each affine scheme $U \to S$ over $S$ the isomorphisms
$\Iso_{U}(f^\ast x_1,f^\ast x_2)$.

\begin{prop}\label{diagonal-for-mfg}The diagonal morphism
$$
\cM_\fg \longr \cM_\fg \times \cM_\fg
$$
is representable, quasi-compact, and separated.
\end{prop}

\begin{proof} We use Proposition \ref{diagonal}:
for any affine scheme $S$ and any two formal groups
$G_1$ and $G_2$, the morphism
$$
\Iso_S(G_1,G_2) \to S
$$
is an affine morphism of schemes. Hence the diagonal is representable
(\cite{Laumon}, 3.13), quasi-compact, and separated 
(\cite{Laumon}, 3.10).
\end{proof}

\begin{prop}\label{lazard-covers}Let $\fgl = \Spec(L)$ be
the affine scheme of formal group laws and let $G_F \to \fgl$
be the formal group arising from the universal formal group law.
Then
$$
G_F: \fgl \longr \cM_\fg
$$
is surjective, flat, and quasi-compact.
\end{prop}

\begin{proof} The morphism $G_F$ is surjective because every
formal group over a field can be given a coordinate, and hence
arises from a formal group law. To check that it is quasi-compact
and flat, we need to check that for all morphisms
$$
G: \Spec(A) \longr \cM_\fg
$$
with affine, source, the resulting map
$$
\Spec(A) \times_{\cM_\fg} \fgl \to \Spec(A)
$$
is quasi-compact and flat. It is quasi-compact because it is
affine (by Proposition \ref{pull-backs-affine-scheme}). To check
that is flat, we choose an faithfully flat extension $A \to B$
and check that
$$
\Spec(B) \times_{\cM_\fg} \fgl \cong
\Spec(B) \times_{\Spec(A)} \Spec(A) \times_{\cM_\fg} \fgl
\to \Spec(B)
$$
is flat. Put another way, we may assume $G$ has a coordinate and
arises from a formal group law. Then, by Lemma \ref{pull-back-for-coord-1}
$$
\Spec(A) \times_{\cM_\fg} \fgl \cong \Spec(A \otimes_L W)
$$
and $A \to A \otimes_L W \cong A[a_0^{\pm 1},a_1,\ldots]$ is
certainly faithfully flat.
\end{proof}

\begin{thm}\label{coord-stack-is-fg}The $1$-morphism of prestacks
$$
\cM_\crd \longr \cM_\fg
$$
identifies $\cM_\fg$ as the stack associated to the prestack
$\cM_\crd$.
\end{thm}

\def\tM{{{\tilde{\cM}_\crd}}}

\begin{proof} We begin by giving a formal description of the
stack $\tilde{\cM}_\crd$ associated to $\cM_{\crd}$. Then we prove
that there is an appropriate equivalence $\tilde{\cM}_\crd \to \cM_\fg$.

First, we define an equivalence class of coordinates for a formal
group $G$ over $S$ as follows. A representative of this equivalence
class will be a coordinate $x$ for $f^\ast G = T \times_S G \to T$ where 
$f:T \to S$ is a faithfully flat and quasi-compact morphism. 
If $x_1$ and $x_2$ are coordinates for $T_1 \times_S G$ and
$T_2 \times_S G$ respectively, then we say they are equivalent
if $p_1^\ast x_1 = p_2^\ast x_2$ as coordinates for
$(T_1 \times_S T_2) \times_S G$. That this is an equivalence
relation follows from the fact that $\coord_G$ is a sheaf
in the $fpqc$ topology.

Now define the category $\tM$ fibered in groupoids over 
$\Aff_\ZZ$ as follows. 
The objects are pairs $(G \to S,[x])$ where $G$ is a formal
group over $S$ and $[x]$ is an equivalence class of coordinates,
as in the previous paragraph. A morphism $f:(G,[x]) \to (H,[y])$
is given by a morphism $f_0:G \to H$ in $\cM_\fg$. That
$\tM$ is a stack is proved exactly as in Lemma \ref{prestack}.
The projection map $\cM_\coord\to \cM_\fg$ has
an evident factorization
$$
\cM_\crd \longr \tM \longr \cM_\fg.
$$
We will prove that the first map has the universal property
necessary for the associated stack, and we will show the
second map is an equivalence of stacks.

First, we must show that any factorization problem
$$
\xymatrix{
\cM_\crd\dto \rto^\lambda & \cN\\
\tM \ar@{-->}[ur]
}
$$
has a solution $\bar{\lambda}:\tM \to \cN$ so that the triangle $2$-commutes.
To do this,
let $(G\to S,[x])$ be an object in $\tM$ and choose an
$fpqc$ cover $d:T \to S$ so that $d^\ast G$ has a coordinate $x$
representing $[x]$. If we apply $\lambda$ to the effective descent data
$$
\phi:(d_1^\ast d^\ast G,d_1^\ast x) \to (d_0^\ast d^\ast G,d_0^\ast x)
$$
we obtain a object $w \in \cN(S)$ and a unique isomorphism
\begin{equation}\label{unique-iso}
d^\ast w \cong \lambda(d^\ast G,x).
\end{equation}
Set $\bar{\lambda}(G,[x]) = w$. In like manner, $\bar{\lambda}$ can
be defined on morphisms. The unique isomorphisms of Equation
\ref{unique-iso} shows that the resulting diagram 
$2$-commutes.

Second, to show that $\tM \to \cM_\fg$ is an equivalence,
note that for all schemes $S$, the morphism of groupoids
$$
\tM(S) \longr \cM_\fg(S)
$$
is a fibration.That is, given any object $(H,[y])$ in $\tM(S)$
and any morphims $\phi: G \to H$ in $\cM_\fg(S)$, there
is a morphism $\phi:(G:[x]) \to (H,[y])$ in $\tM(S)$ whose underlying
morphism is $\psi$. This follows from Remark \ref{fgls-rems}.5.
If $G \in \cM_\fg(S)$ is a fixed formal group, then the
fiber at $G$ is 
$$
\mathop{\colim}_{T \to S}\coord(G/T)
$$
where $T$ runs over all $fpqc$ covers of $S$. Combining
Theorem \ref{coord-affine}.2 and Lemma \ref{coords-after-ff}
we see that this fiber is contractible.
\end{proof}

The following is an immediate consequence of the previous result
and Proposition \ref{coord-equiv-fgl}.

\begin{thm}\label{fgl-stack-is-fg}The $1$-morphism of prestacks
$$
\cM_\fgl \longr \cM_\fg
$$
identifies $\cM_\fg$ as the stack associated to the prestack
$\cM_\fgl$.
\end{thm}

\subsection{Quasi-coherent sheaves}

Here we define the notion of a quasi-coherent sheaf on
an $fpqc$-algebraic stack and give some preliminary examples
for the moduli stack of formal groups. We then recall
the connection between  quasi-coherent sheaves and comodules over 
a Hopf algebroid and relate the cohomology
of a quasi-coherent sheaf to $\Ext$ in the category of
comodules.

In \ref{mod-sheaves} we noted that if $X$ is scheme, then the category of
quasi-coherent sheaves over $X$ is equivalent to the category
of cartesian $\cO_X$-module sheaves in the $fpqc$-topology.
We will take the latter notion as the {\it definition} of a quasi-coherent
sheaf on an $fpqc$-stack.

The $fpqc$-topology and $fpqc$-site were defined in
Definition \ref{fpqc} and Remark \ref{sheaves} respectively.

\begin{defn}\label{stack-site-fpqc}Let $\cM$ be an $fpqc$-algebraic
stack. We define the $fpqc$ {\bf site} on $\cM$ to have
\index{site, $fpqc$ of a stack}
\begin{enumerate}

\item an underlying category with objects all schemes $U \to \cM$ over
$\cM$ and, as morphisms, all 2-commuting diagram over $\cM$; and

\item for all morphisms $U \to \cM$ in this category we assign
the the $fpqc$-topology on $U$.
\end{enumerate} 
\end{defn}

We often specify sheaves only on affine morphisms $\Spec(A) \to \cM$,
extending as necessary to other morphisms by the sheaf condition.

The structure sheaf on $\cO = \cO_\cM$ is defined by
$$
\cO(\Spec(A) \to \cM) = B.
$$
This is a sheaf of rings and has a corresponding category of 
module sheaves, which we will write as $\Mod_\cM$ or, perhaps,
$\Mod_\fg$ is we have some stack such as the moduli stack of
formal groups.

The notion of a cartesian sheaf can be found 
in Definition \ref{super-cartesian}. 

\begin{defn}\label{qc-mod-on-stack} Let $\cM$ be an $fpqc$-algebraic
stack. A {\bf quasi-coherent} sheaf $\cF$ on $\cM$ is a cartesian
$\cO_\cM$-module sheaf for the category of affines over $\cM$.  In detail we have\index{sheaf, quasi-coherent, on a stack}
\begin{enumerate}

\item for each morphism $u:\Spec(A) \to \cM$ an $A$-module 
$\cF(u)$; 

\item for each $2$-commuting diagram
$$
\xymatrix@R=10pt{
\Spec(B) \ar[dr]^v \ar[dd]\\
&\cM\\
\Spec(A) \ar[ur]_u
}
$$
a morphism of $A$-modules $\cF(u) \to \cF(v)$
so that the induced map
$$
B \otimes_A \cF(u) \to \cF(v)
$$
is an isomorphism.
\end{enumerate}
\end{defn}

\begin{exam}\label{naive-inv-differentials}Here I give an ad hoc construction
of the sheaf of invariant differentials on $\cM_\fg$. A more intrinsic definition
will be given later. See Section 4.2. 

Because of faithfully flat descent, we can define an $\cO_\fg$-module
sheaf $\cF$ on $\cM_\fg$ be specifying 
$$
\cF(G) = \cF(G:\Spec(A) \to \cM_\fg)
$$
for those formal groups for which we can choose a coordinate. Given
such a coordinate $x$ for $G$, we define the invariant differentials
$$
f(x)dx \in A[[x]]dx
$$
to be those continuous differentials which satisfy the identity
$$
f(x+_Fy)d(x+_Fy) = f(x)dx + f(y)dy
$$
where $x+_Fy$ is the formal group law of $G$ with coordinate $x$. The
$A$-module $\omega_G$ of invariant differentials is free of rank 
1 and independent of the choice of coordinate. See Example \ref{dlog}.
Given a $2$-commuting diagram
$$
\xymatrix@R=10pt{
\Spec(B) \ar[dr]^H \ar[dd]_f\\
&\cM_\fg\\
\Spec(A) \ar[ur]_G
}
$$
with isomorphism $\phi:H \to f^\ast G$. Then we have an induced isomorphism
$$
d\phi:f^\ast \omega_G=B \otimes_A \omega_G \to \omega_H.
$$
See Example \ref{calculate-lie} for an explicit formula. Thus we have 
a quasi-coherent sheaf $\omega$ on $\cM_\fg$. This sheaf is locally
free of rank $1$ and we have also have all its tensor powers
$$
\omega^{\otimes n} \defeq \omega^n,\qquad n \in \ZZ.
$$
\end{exam}

\begin{rem}[{\bf Quasi-coherent sheaves and comodules}]\label{comod-qc-sheaves}
Suppose\index{comodules}
$\cM$ is an $fpqc$-algebraic stack with an affine presentation
$\Spec(A) \to \cM$ with the property that
$$
\Spec(A) \times_\cM \Spec(A) \cong \Spec(\Ga)
$$
is also affine. Then we get induced isomorphisms
$$
\Spec(A) \times_\cM \cdots \times_\cM \Spec(A) 
\cong \Spec(\Ga \otimes_A \cdots \otimes_A \Ga)
$$
where the product has $n\geq 2$ factors and the tensor product has
$(n-1)$-factors. The \v Cech nerve of the cover $\Spec(A) \to \cM$
then becomes\index{cobar complex} the diagram of affine
schemes associated to the cobar complex
\begin{equation}\label{cobar2}
\xymatrix{
\cdots\ \Spec(\Ga \otimes_A \Ga) \ar@<1.5ex>[r] \ar[r] \ar@<-1.5ex>[r]& 
\ar@<.75ex>[l] \ar@<-.75ex>[l] \Spec(\Ga) \ar@<.75ex>[r] \ar@<-.75ex>[r] & \lto \Spec(A) \rto & \cM.
}
\end{equation}
The pair $(A,\Ga)$, with all these induced arrows, becomes a {\it Hopf
 algebroid.}\index{Hopf algebroid}
If we set $M = \cF(\Spec(A) \to \cM)$, then one of the arrows
$\Spec(\Ga) \to \Spec(A)$ defines an isomorphism
$$
\Ga \otimes_A M \cong \cF(\Spec(\Ga) \to \cM)
$$
and the other defines a morphism $M \to \Ga\otimes_A M$ which gives the
module $M$ the structure of an $(A,\Ga)$-comodule. This defines a functor
from quasi-coherent sheaves on $\cM$ to $(A,\Ga)$-comodules. This
is an equivalence of categories. See \cite{HovAmerJ} and
\cite{Laumon} Proposition 12.8.2. The proof in the latter case is
carried out in a different topology, but goes through unchanged for the
$fpqc$ topology, as it
is an application of faithfully flat descent. 

A pair $f:\Spec(A) \to \cM$ where $\cM$ is an $fpqc$-stack and $f$ is a presentation so that
$\Spec(A) \times_\cM \Spec(A)$ is itself affine is called a {\it rigidified} stack.
\index{stack, rigidified} Such a
choice leads to the equivalence of categories in the previous paragraph, but any stack
$\cM$ may have many (or no) rigidifications and the Hopf algebroid $(A,\La)$ may
not be in any sense canonical. An example is the moduli stack $\cU(n)$ of formal
groups of height at most $n$. Rigidified stacks are discussed in \cite{pribble} and
\cite{hollander3}.

For the moduli stack $\cM_\fg$, the universal formal group law gives a cover
$\Spec(L) \to \cM_fg$ and we conclude that the category of quasi-coherent sheaves
on $\cM_\fg$ is equivalent to the category of $(L,W)$ comodules, where 
$$
W = L[a_0^{\pm 1},a_1,a_2,\ldots]
$$
as in Remark \ref{sym-2-cocyle}. The structure sheaf $\cO_\fg$ corresponds to the
$L$ with its standard comodule structure given the by the right unit $\eta_R:L \to W$;
the powers of the sheaf of invariant differentials $\omega^n$ correspond to the comodule
$L[n]$ where $\psi:L[n] \to W \otimes_L L[n]$ is given by
$$
\psi(x) = a_0^n\eta_R(x).
$$
\end{rem}

\begin{rem}[{\bf Cohomology}]\label{coh22}
If $\cM$ is an $fpqc$-algebraic stack and $\cF$ is a quasi-coherent sheaf, then the cohomology
$H^\ast (\cM,\cF)$ is obtained by taking derived functors of global sections. If 
$\Spec(A) \to \cM$ is a rigidified stack with corresponding Hopf
algebroid $(A,\Ga)$, then the equivalence
of categories between quasi-coherent sheaves and comodules yields an isomorphism\index{cohomology and comodule Ext}
\begin{equation}\label{coh-is-ext}
H^s(\cM,\cF) \cong \Ext^s_\La(A,M)
\end{equation}
where $M = \cF(\Spec(A) \to \cM)$ is the comodule corresponding to $\cF$. The 
\v Cech nerve of Equation \ref{cobar2} yields the usual cobar complex for computing
Hopf algebroid $\Ext$.
\end{rem}

\def\fglp{{{\mathbf{pfgl}}}}
\def\coordp{{{\mathbf{pcoord}}}}

\subsection{At a prime: $p$-typical coordinates}

When making calculations, especially with the Adams-Novikov spectral
sequence, it is often very convenient to use $p$-typical formal group
laws instead of arbitrary formal group laws. We delve a little into that
theory here. A point to be made is that it is not a formal group which
is $p$-typical, but a formal group law or, equivalently, a coordinate for a
formal group.

If $F$ is a formal group over a ring $A$ in which an integer
$n$ is invertible, the power series
$$
[n](x) \defeq x +_F \cdots +_F x
$$
with the sum taken $n$-times has a unit as its leading coefficient;
hence, it has composition inverse $[1/n](x)$.

Let $A$ be a commutative ring over $\ZZ_{(p)}$ and let
$F$ be a formal group law over $A$. Then, given any integer $n$
prime to $p$ and a primitive $n$th root of unity $\zeta$, we can
form the power series
\begin{equation}\label{p-typ-testers}
f_n(x) = [\frac{1}{n}]_F(x +_F \zeta x +_F \cdots +_F \zeta^{n-1} x).
\end{equation}
Note that this is a power series over $A$. 

More generally, if $S$ is a scheme over $\ZZ_{(p)}$,  $G$ a formal group
over $S$, and  $x$  a coordinate for $G$, then we have a formal
group law $x_1 +_{(F,x)} x_2$ over $H^0(S,\cO_S)$ and we can form 
the power series $f_n(x)$ over $H^0(S,\cO_S)$.

\begin{defn}\label{p-typical}1.) A $p$-{\bf typical formal group law}
\index{coordinate, $p$-typical} $F$
over a commutative ring $A$ is a formal group law $F$ over $A$ so that
$$
f_\ell(x) = 0
$$
for all primes $\ell \ne p$. A homomorphism of $p$-typical formal
group laws is simply a homomorphism of formal groups.

2.) Let $G$ be a formal group over a scheme $S$ over $\ZZ_{(p)}$.
Then a coordinate $x$ for $G$ is $p$-{\bf typical} if 
the associated formal group law over $H^0(S,\cO_S)$ is
$p$-typical.
A morphism $\phi:(G,x) \to (H,y)$
of formal groups with $p$-typical coordinates is simply a homomorphism
of the underlying formal groups.
\end{defn}

The symmetry condition $f_\ell(x)=0$ arises naturally when considering the
theory of Dieudonn\'e modules associated to formal groups. See \cite{Cart}.

\begin{rem}[{\bf Properties of $p$-typical formal group laws}]\label{p-typ-prop}Let 
us record some of the standard properties of $p$-typical coordinates.
A reference, with references to references, can be found in
\cite{Rav}, Appendix 2.
\begin{enumerate}

\item Let $G$ be a formal group over a $\ZZ_{(p)}$-algebra $A$
with a $p$-typical coordinate $x$. Then there are elements
$u_i \in A$ so that
$$
[p]_G(x) = px +_G u_1x^p +_G u_2x^{p^2} +_G \cdots.
$$
Furthermore, the elements $u_i$ determine the $p$-typical
formal group law. However,
the elements $u_i$ depend on the pair $(G,x)$, hence are not
invariant under changes of coordinate. Nonetheless, if $f:A \to B$
is a homomorphism of $\ZZ_{(p)}$-algebras, then 
$$
[p]_{f^\ast G}(x) = px +_G f(u_1)x^p +_G f(u_2)x^{p^2} +_G \cdots.
$$
Thus, this presentation of $[p]_G(x)$ extends to schemes: given
a $p$-typical formal group law $(G,x)$ over a $\ZZ_{(p)}$ scheme,
there are elements $u_i \in H^0(S,\cO_S)$ so that $[p]_G(x)$
can be written as above.

\item Let us write $\fglp$ for the functor which assigns to each
$\ZZ_{(p)}$-algebra $A$ the set of $p$-typical formal group laws
over $A$. Then $\fglp$ is an affine scheme. Indeed, if we write
$V = \ZZ_{(p)}[u_1,u_2,\ldots]$ there is a $p$-typical formal
group law $F$ over $V$  so that
$$
[p]_{F}(y) = py +_{F} u_1y^p +_{F} u_2y^{p^2} +_{F} \cdots.
$$
The evident morphism of schemes $\Spec(V) \to \fglp$ is
an isomorphism.

\item Let $\phi:(G,x) \to (H,y)$ be an isomorphism of formal
groups with $p$-typical coordinates and let $f(x) \in 
R[[x]]$ be the power series determined by $\phi$. Then
there are elements $t_i \in R$ so that
$$
f^{-1}(x) = t_0x +_G t_1x^{p} +_G t_2^{p^2} +_G \cdots.
$$
More is true. If $x$ is a $p$-typical coordinate, then $y$ is
$p$-typical if and only if $f^{-1}(x)$ has this form.
\end{enumerate}
\end{rem}

As in Definitions \ref{prestack-of-fgls} and \ref{prestack-of-coords} and
Lemmas \ref{mfgl-prestack} and \ref{coord-prestack}, we have
prestacks $\cM_\fglp$ of $p$-typical formal group laws
\index{prestack, $p$-typical formal group laws}
and $\cM_\coordp$ of formal groups with $p$-typical coordinates.
We also have the analog of Proposition \ref{coord-equiv-fgl}:
\index{prestack, $p$-typical coordinate}

\begin{prop}\label{coordp-equiv-fglp}The canonical morphism
of prestacks
$$
\cM_\coordp \longr \cM_\fglp
$$
is an equivalence.
\end{prop}

A much deeper result is the following.  If $X$ is a sheaf over $\Spec(R)$
in the $fpqc$-topology and $R \to S$ is a ring homomorphism, we
will write
$$
X \otimes_R S \defeq X \times_{\Spec(R)} \Spec(S).
$$

\begin{thm}[{\bf Cartier's idempotent}]\label{quillen-idem}The canonical
$1$-morphism of categories fibered in groupoids over $\Aff_{\ZZ_{(p)}}$
$$
\cM_\fglp \longr \cM_\fgl \otimes \ZZ_{(p)}
$$
is an equivalence.\index{Cartier idempotent}
\end{thm}

\begin{proof} Let $A$ be a commutative $\ZZ_{(p)}$-algebra.
Cartier's theorem (see, for example, \cite{Rav}A.2.1.18) is usually phrased
as follows: Given any formal group law $F$ over $A$ there
is a $p$-typical formal group law $eF$ over $A$ and an isomorphism
$\phi_F:F \to eF$ of formal group laws; furthermore, if
$F$ is $p$-typical, then $eF = F$ and $\phi$ is the identity.
This implies that if $\psi:F \to F'$ is any isomorphism of formal groups laws, then there is a unique isomorphism $e\psi$ so that the following
diagram commutes:
$$
\xymatrix{
F \rto^{\phi_F} \dto_\psi & eF\dto^{e\psi}\\
F' \rto_{\phi_{F'}} & eF'
}
$$
Rephrased, we see that we have a retraction
$e:\cM_\fgl(A) \to \cM_\fglp(A)$ of the inclusion of groupoids
$\iota:\cM_\fglp(A) \to \cM_\fgl(A)$ and a natural transformation
$\phi:1 \to \iota e$.
\end{proof}

The following is now and immediate consequence of Theorem
\ref{fgl-stack-is-fg} and Theorem \ref{quillen-idem}.
\begin{cor}\label{fglps-to-fgs}The canonical $1$-morphism of prestacks
$$
\cM_\coordp \longr \cM_\fg \otimes \ZZ_{(p)}
$$
identifies $\cM_\fg \otimes \ZZ_{(p)}$ as the stack associated
to the prestack $\cM_\coordp$.
\end{cor}

Similarly $\cM_\coordp \to \cM_\fg \otimes \ZZ_{(p)}$
identifies the target as the stack associated to the prestack
source. Compare Theorem \ref{coord-stack-is-fg}.

The following now follows from Corollary \ref{fglps-to-fgs} and Remark
\ref{p-typ-prop}, parts 2 and 3.

\begin{cor}\label{ptyp-cover}Let  $V = \ZZ_{(p)}[u_1,u_2,\cdots]$
and let $G_F \to \Spec(V)$ be the formal group formed from
the universal $p$-typical formal group law $F$. Then the map
\index{$V$ (ungraded $BP_\ast$)}
$$
\Spec(V) \longr \cM_\fg \otimes \ZZ_{(p)}
$$
classifying $G$ is an $fpqc$-presentation of
$\cM_\fg \otimes \ZZ_{(p)}$. There is an isomorphism of
affine schemes
$$
\Spec(V) \times_{\cM_\fg} \Spec(V) \cong
\Spec(V[t_{0}^{\pm 1},t_1,t_2,\cdots]).
$$
\end{cor}

\begin{rem}[{\bf Gradings and formal group laws}]\label{gradings} There is a natural grading\index{gradings}
on the Lazard ring $L$ and the ring $V =\ZZ_{(p)}[u_1,u_2,\cdots]$
which supports the universal $p$-typical formal group
law. This can be useful for computations.

To get the grading, we put an action of the multiplication group
$\GG_m = \Spec(\ZZ[t^{\pm 1}])$ on scheme $\fgl = \Spec(L)$
of formal group laws as follows. If 
$$
x+_Fy = \sum c_{ij}x^iy^j
$$
is a formal group law over a ring $R$ and 
$\la \in R^\times$ is a unit in $R$, define a new formal
group law $\la F$ over $R$ by
$$
x+_{\la F} y = \la^{-1}((\la x) +_F (\la y)).
$$
This action translates into a coaction
$$
\psi: L \longr \ZZ[t^{\pm 1}] \otimes L
$$
and hence a grading on $L$: $x \in L$ is of degree
$n$ if $\psi(x) =  t^{n} \otimes x$. Then coefficients
$a_{ij}$ of the universal formal group law have degree
$i+j-1$; since
$$
x +_F y = x + y + b_1C_2(x,y) + b_2C_3(x,y) + \cdots
$$
modulo decomposables, we have that $b_i$ has degree $i$.
In particular, $c_{ij}$ is a homogeneous polynomial in
$b_k$ with $k < i +j$.

The same construction applies to $p$-typical formal group
laws and the $p$-series 
$$
[p]_G(x) = px +_G u_1x^p +_G u_2x^{p^2} +_G \cdots
$$
shows that, under the action of $\GG_m$, $u_k$ has degree $p^{k}-1$.
Since the universal $p$-typical formal group is defined over the ring
$V = \ZZ_{(p)}[u_1,u_2,\cdots]$ we have that
the coefficients $c_{ij}$ of the universal $p$-typical
formal group are homogeneous polynomials in the
$u_k$ where $p^k \leq i+j$.

The action of $\GG_m$ extends to the entire groupoid
scheme of formal group laws and their isomorphisms.
If $\phi(x) = \sum_{i \geq 0} a_i x^i$ is an isomorphism
from $F$ to $G$, define
$$
(\la\phi)(x) = \la^{-1}\phi(\la x).
$$
Then $\la \phi$ is an isomorphism from $\la F$ to $\la G$.
If $\phi$ is universal isomorphism over $W = L[a_{0}^{\pm 1},a_1,
\cdots]$, the $a_i$ has degree $i$. More interesting is the
case of $p$-typical formal group laws; if $\phi$ is the
universal isomorphism of $p$-typical formal group laws
over $V[t_{0}^{\pm 1},t_1,t_2,\cdots]$, then
$$
\phi^{-1}(x) = t_0x +_G t_1x^p +_G t_2x^{p^2} +_G \cdots
$$
and we see that the degree of $t_k$ is $p^k-1$. Thus
if $a_i$ is the $i$th coefficient of this power series, we
have that $a_i$ is a homogeneous polynomial in 
$t_k$ and $u_k$ with $p^k \leq i$.

{\bf Warning:}The grading here is not the topological grading;
in order to obtain the usual topological gradings we should
double the degree -- so that, for example, the degree of $v_i$
is $2(p^i-1)$. Also, I'll say nothing about the role of odd degree
elements in comodules -- and there are some subtleties here.
See \cite{LEFT} for a systematic treatment.\index{gradings, caution on}
\end{rem}

%% file: orbifold.tex
\section{The moduli stack of formal groups as a homotopy orbit}

One of the main points of this chapter is to describe the
moduli stack $\cM_\fg$ as the homotopy inverse limit
if the moduli stacks $\cM_\fg\langle n \rangle$ of $n$-buds
for formal groups. This is a restatement of classical results
of Lazard. See  Theorem \ref{chunks-inv}. This has consequences for
the quasi-coherent sheaves on $\cM_\fg$; see Theorem \ref{fp-modules}.

\subsection{Algebraic homotopy orbits}

First some generalities, from \cite{Laumon} \S\S 2.4.2, 3.4.1, and
4.6.1. Let $\La$ be an group scheme over a base
scheme $S$. Let $X \to S$ be a right-$\Lambda$-scheme.
Thus, there is an action morphism
$$
X\times_S \Lambda  \longr X
$$
over $S$ such that the evident diagrams commute. From this data,
we construct a stack $X \times_{\La} E\Lambda$, called the
{\it homotopy orbits}\index{homotopy orbit stack}\index{stack,
homotopy orbits} of the action of $\La$ on $X$, as
follows.\footnote{Under appropriate finiteness hypotheses
which will not apply in our examples, the homotopy orbit
stack can become an algebraic {\it orbifold}.}

Recall that an $\Lambda$-torsor is a scheme $P \to S$ with
a right action of $\Lambda$ so that there is an $fpqc$ cover
$T \to S$ and an isomorphism of $\Lambda$-schemes over $T$
$$
T \times_S \Lambda \cong T \times_S P.
$$
If you want a choice-free way of stating this last, we remark
that this is equivalent to requiring that the natural map
$$
(T \times_S P)\times_T (T \times_S \La) \longr
(T \times_S P)\times_T (T \times_S \La) 
$$
over $(T \times_S \Lambda)$ sending $(x,g)$ to $(xg,g)$ is
an isomorphism.

To define $X \times_\La E\La$ we need to specify a 
category fibered in groupoids.  Suppose $U \to S$ is a scheme over $S$. Define the objects $[X \times_{\La} E\La](U)$ to be pairs
$(P,\alpha)$
where $P\to U $ is a $\Lambda\times_S U$-torsor and
$$
\alpha:P \to U \times_S X
$$
is a $\Lambda$-morphism over $U$. A morphism $(P,\alpha) \to (Q,\beta)$
is an equivariant isomorphism $P \to Q$ so that the evident diagram
over $U \times_S X$ commutes. If $V \to U$ is a morphism of schemes 
over $S$, then the map $[X \times_\La E\La](U) \to [X \times_\La
E\La](V)$ is defined by pull-back. This gives a stack (see \cite{Laumon},
3.4.2) ; we discuss to what extent it is an algebraic stack.

There is a natural map $X \to X \times_{\La} E\La$ defined as follows.
If $f:U \to X$ is a morphism of schemes over $S$ define
$P = U \times_S \La$ and let $\alpha$ be the evident composition
over $U$
$$
\xymatrix{
U \times_S \La \rto^-{f \times \La} & U \times_S X \times_S \La
\rto&
U \times_S X
}
$$
given pointwise by $(u,g) \mapsto (u,f(u)g).$

Note that if  $U \to X \times_{\La} E\La$
classifies $P \to U \times_S X$, then a factoring
$$
\xymatrix{
&X \dto\\
U \ar[ur] \rto & X \times_\La E\La
}
$$
is equivalent to a choice of section of $P \to U$ and hence
an chosen equivariant isomorphism $U \times_S \La \to P$ over $U$.
The notion of an algebraic stack in the $fpqc$ topology was defined
in Definition \ref{alg-stack}.

\begin{prop}\label{homotopy-orbit} Let $\La$ be a
group scheme over $S$ and suppose the structure morphism
$\La \to S$ is flat and quasi-compact. Let $X$ be a scheme over $S$
with a right $\La$-action. Then $X \times_\La E\La$ is an algebraic
stack in the $fpqc$ topology and 
$$
q:X \longr X \times_\La E\La
$$
is an $fpqc$ presesentation. There is a natural commutative
diagram
$$
\xymatrix{
X \times_S \La \ar@<.5ex>[r]^-{d_0} \ar@<-.5ex>[r]_-{d_1}\dto_{\cong}
& X \dto^{=}\\
X \times_{X \times_\La E\La} X
\ar@<.5ex>[r]^-{p_1} \ar@<-.5ex>[r]_-{p_2}& X
}
$$
where $d_0(x,g) = x$ and $d_1(x,g) = xg$ and the vertical isomorphism
sends $(x,g)$ to the triple $(x,xg,g:xg \to x)$.
\end{prop}

\begin{exam}\label{extremes-orbits}There are two evident examples.
First we can take $X = S$ itself with the necessarily trivial
right action, and we'll write 
$$
B\La \defeq S \times_\La E\La.
$$
This is the moduli stack of $\La$-torsors on $S$-schemes or the
{\it classifying stack} of $\La$.\index{stack, classifying}\index{classifying
stack} The
other example sets $X = \La$ with the canonical right action.
Let's assume $\La$ is an affine group scheme over $S$. 
Then the projection map
$$
\La \times_{\La} E\La \to S
$$
is an equivalence. For if $\alpha: P \to U \times_S \La$ is 
any morphism of $\La$-torsors over $U$, then $\alpha$
becomes an isomorphism on some faithfully flat over.
Since $\La \to S$ is affine, $\alpha$ is then an isomorphism
by faithfully flat descent. It follows that the groupoid
$[\La \times_{\La} E\La](U)$ is contractible.
\end{exam}

\begin{rem}\label{cech-for-homotopy-orbit}Note that the \v Cech cover
of $X \times_\La E\La$ that arises from the cover
$X \to X \times_{\La} E\La$ is the standard bar complex
given by the action of $\La$ on $X$. Thus, $X \times_\La E\La$
is that analog of the geometric realization of this bar complex,
whence the name homotopy orbits.
\end{rem}

\begin{rem}\label{qc-comodule}Suppose that $S = \Spec(R)$,
$X = \Spec(A)$ and $\La = \Spec(\Ga)$. Then the group
action $X \times_S \La \to X$ yields a Hopf algebroid
structure on the pair $(A,A \otimes_R \Ga)$. This
is a {\it split} Hopf algebroid. By Remark \ref{comod-qc-sheaves}
the category of quasi-coherent sheaves over $X \times_\La E\La$
is equivalent to the category of $(A,A \otimes_R \Ga)$-comodules.
\end{rem}

\begin{rem}\label{top-exam}Let's compare this construction
of $X \times_\La E\La$ with a construction in simplicial sets. 
Suppose $\La$ is a discrete group (in sets) and $X$ is
a discrete right $\La$-set. Then the simplicial set
$X \times_\La E\La$ is defined to be the nerve of the
groupoid with object set $X$ and morphism set $X \times G$.
However, this groupoid is equivalent to the groupoid with
objects $\alpha:P \to X$ where $P$ is a free and transitive
$G$-set; morphisms are the evident commutative triangles.
This is a direct translation of the construction above.
Equivalent groupoids have weakly equivalent nerves; hence, 
if we are only interested in homotopy type,
we could define $X \times_{\La} E\La$ to be the nerve
of the larger groupoid.
\end{rem}

Next let us say some words about naturality.
This is simpler if we make some assumptions on our group schemes.
A group scheme $\La$ over $S$ is {\it affine over $S$}
if the structure map $q:\La \to S$ is an affine morphism. Since
affine morphisms are closed under composition and base change,
the multiplication map $\La \times_S \La \to \La$ is a morphism
of schemes affine over $S$. Thus the quasi-coherent $\cO_S$-algebra
sheaf $q^\ast \cO_\La$ is a sheaf of Hopf algebras. In most of
our examples, $S = \Spec(A)$ is itself affine; in this case, $\La = \Spec(\Ga)$
for some Hopf algebra $\Ga$ over $A$.

If $\La$ is a group scheme affine over $S$ and $P \to S$ is a $\La$-torsor,
then $P \to S$ is an affine morphism by faithfully flat descent. If
$\phi:\La_1 \to \La_2$ is a morphism of group schemes affine over
$S$ and $P \to S$ a $\La_1$ torsor, let $P \times_{\La_1} \La_2$
be the sheaf associated to the presheaf
$$
A \mapsto (P(A) \times_{S(A)} \La_2(A))/\sim
$$
where $\sim$ is the equivalence relation given pointwise by
$$
(xb,a) \sim (x,ba)
$$
with $x \in P(A)$, $a \in \La_2(A)$, and $b \in \La_1(A)$. 

\begin{lem}\label{push-forward-torsor} Let $\La_1 \to \La_2$
be a morphism of groups schemes affine over $S$ and
Let $P \to S$ be a $\La_1$-torsor. Then
$P \times_{\La_1} \La_2$ is actually a $\La_2$-torsor
over $S$. 
\end{lem}

\begin{proof} If we can choose an isomorphism 
$P \cong \La_1$ over $S$, then we get an induced isomorphism
$P \times_{\La_1} \La_2 \cong \La_2$. More generally,
let $f:T \to S$ be an $fpqc$-cover so that 
$$
T \times_S P \cong T \times_S \La_1.
$$
Then
$$
T \times_S (P \times_{\La_1} \La_2) \cong
(T \times_S P) \times_{T \times_S \La_1} (T \times_S \La_2)
\cong T \times_S \La_2.
$$
Since $\La_2$ is affine over $S$, $T \times_S \La_2$ is affine
over $T$ and faithfully flat descent implies $P \times_{\La_1} \La_2$
is an affine torsor over $S$. 
\end{proof}

Now suppose $X_1$ is a right $\La_1$-scheme, $\La_2$
is a right $\La_2$-scheme and $q:X_1 \to X_2$ is
a morphism of $\La_1$-schemes. Then we get a morphism
of stacks
$$
X_1 \times_{\La_1} E\La_1 \to X_2 \times_{\La_2} E\La_2
$$
sending the pair $(P, \alpha)$ to the pair
$(P \times_{\La_1}\La_2,q\alpha)$; that is, there is a commutative
diagram of $\La_1$-schemes
$$
\xymatrix{
P \rto^\alpha\dto & X_1 \dto^q\\
P \times_{\La_1}\La_2 \rto &X_2.
}
$$

Such morphisms have quite nice properties.
Recall that a morphism of groupoids $f:G \to H$ is a fibration
\index{fibration of groupoids}
if for all $x \in H$, all $y \in G$ and all morphisms
$\phi:x \to f(y)$ in $H$, there is a morphism $\psi: x' \to y$
in $G$ with $f\psi = \phi$. Equivalently the morphism
of nerves $BG \to BH$ is a Kan fibration of simplicial
sets. We will say that a morphism of stacks $\cM \to \cN$ is
a fibration if for all commutative rings $R$, the map
$\cM(R) \to \cN(R)$ is a fibration of groupoids.\footnote{This
begs for a much more extensive and sophisticated discussion.
See \cite{jardine} and \cite{hollander}.}

A topological version of the following result can be found
in Remark \ref{top-fib} below.

\begin{prop}\label{fibration}Suppose $f:\La_1 \to \La_2$ is
a morphism of group schemes affine over $S$, $X_1$ is a $\La_1$-scheme,
$X_2$ is a $\La_2$-scheme, and $q:X_1 \to X_2$ is
a morphism of $\La_1$-schemes. Then
$$
X_1 \times_{\La_1} E\La_1 \longr X_2 \times_{\La_2} E\La_2
$$
is a fibration of algebraic stacks in the $fpqc$ topology.
\end{prop}

\begin{proof}Suppose we are given a diagram (over a base-scheme
$U$ suppressed from the notation)
$$
\xymatrix{
& P \rto^\alpha \dto & X \dto^q\\
Q' \rto_\phi & Q \rto_\beta & Y
}
$$
with (1) $P$ a $\La_1$-torsor and $\alpha$ a $\La_1$-morphism;
(2) $Q'$ and $Q$ both $\La_2$-torsors, $\beta$
$\La_2$-map and $\phi$ is $\La_2$-isomorphism; and (3) $P \to Q$ a morphism of $\La_1$-schemes
with $P \times_{\La_1}\La_2 \cong Q$. Then we take the pull-back
$$
\xymatrix{
Q' \times_QP\rto^-\psi\dto& P\dto\\
Q' \rto_\phi & Q.
}
$$
Then $Q'$ is a $\La_1$-torsor and $\psi$ is a $\La_1$-isomorphism.
Finally, we must check that the natural map 
$(Q'\times_Q P) \times_{\La_1}\La_2 \to Q'$ is an isomorphism of
$\La_2$-torsors. If we can choose isomorphisms $P \cong \La_1$
and $Q \cong\La_2$ this is clear. The general case follows
from faithfully flat descent.
\end{proof}

It is also relatively easy to identify fibers in this setting. We restrict
ourselves to a special case.

\begin{prop}\label{fibers}Suppose $f:\La_1 \to \La_2$ is flat surjective
morphism of group schemes affine over $S$ 
with kernel $K$. Suppose that $X_1$ is a $\La_1$-scheme,
$X_2$ is a $\La_2$-scheme, and $q:X_1 \to X_2$ is
a morphism of $\La_1$-schemes. Then there is a homotopy
pull-back diagram
$$
\xymatrix{
X_1 \times_K EK \rto \dto & X_1 \times_{\La_1} E\La_1\dto\\
X_2 \rto & X_2 \times_{\La_2} E\La_2.
}
$$
\end{prop}

\begin{proof}Let $f:U \to X_2$ be a morphism of schemes.
Then the composition $U \to X_2 \times_{\La_2} E\La_2$
classifies the pair $(U \times_S \La_2,\alpha)$ where $\alpha$
is the composition
$$
\xymatrix@C=17pt{
U \times_S \La_2 \ar[rr]^-{f \times \La_2} &&
U \times_S X_2 \times_S \La_2 \ar[rr] &&
U \times_S X_2.
}
$$
The homotopy fiber at $U$ is the groupoid with objects the
commutative diagrams
$$
\xymatrix{
P \rto^-\beta \dto^g & U \times_S X_1 \dto\\
U \times_{S} \La_2 \rto_-\alpha & U\times_S X_2
}
$$
where $(P,\beta)$ is an object in $[X_1 \times_{\La_1} E\La_1](U)$
and $g$ is a $\La_1$ morphism so that the induced map
$P \times_{\La_1}\La_2 \to U \times_S \La_2$ is an isomorphism.
Let $P'$ be the pull-back of $g$ at inclusions induced by the identity
$U \to U \times_{\La_2} \La_2$. Then $P' \to U \times_S X_2$
is an equivariant morphism from a $K$-torsor to $X_2$.
This defines the functor from the pull-back to $X_1 \times_K EK$.

Conversely, given a $K$-torsor $P$ over $U$ and a $K$-morphism
$P \to X_1$ we can produce a diagram
$$
\xymatrix{
P \times_{U \times_S K} (U\times_S \La_1) \rto\dto^g & U \times_S X_1 \dto\\
P \times_{U \times_S K} (U\times_S \La_2)  \rto & U\times_S X_2.
}
$$
Since $K$ is the kernel of $\La_1 \to \La_2$, projection gives
a natural morphism of $\La_2$-torsors over $U$
$$
P \times_{U \times_S K} (U\times_S \La_2) \to U \times_S \La_2
$$of $\La_2$ torsors over $U$.
This defines the functor back and gives the equivalence of categories.
\end{proof}

\begin{rem}\label{top-fib}In the topological setting of Remark 
\ref{top-exam} we gave two ways to construct
$X \times_\La E\La$. With the smaller, and more usual construction,
a morphism
$$
X_1 \times_{\La_1} E\La_1 \longr X_2 \times_{\La_2} E\La_2
$$
is a fibration only if $\La_1 \to \La_2$ is onto. However,
in the larger construction using transitive and free $\La$-sets,
this morphism is always a fibration, by the same argument
as that given for Proposition \ref{fibration}. Either model allows
us to prove the analog of Proposition \ref{fibers}.
\end{rem}

As a final generality we have:

\begin{prop}\label{torsor-equiv}Suppose $f:\La_1 \to \La_2$ is flat surjective
morphism of group schemes affine over $S$ and let $K$ be the kernel.
Suppose that $X_1$ is a $\La_1$-scheme,
$X_2$ is a $\La_2$-scheme, and $q:X_1 \to X_2$ is
a morphism of $\La_1$-schemes. If $X_1 \to X_2$ is
a $K$-torsor over $X_2$, then
$$
X_1 \times_{\La_1} E\La_1\longr X_2 \times_{\La_2} E\La_2
$$
is an equivalence of algebraic stacks.
\end{prop}

\begin{proof}The hypothesis of $X_1 \to X_2$ means that when the
action is restricted to $K$, then $X_1$ is (after pulling back to
an $fpqc$-cover of $X_1$) isomorphic to $X_2 \times_S K$.
The result follows immediately from Propositions \ref{fibration}
and \ref{fibers},
but can also be proved directly. For if $\alpha: P \to U \times_S X_2$
is some $\La_2$-equivariant morphism from a $\La_2$-torsor
over $U$, then we can form the pull back square
$$
\xymatrix{
Q \rto^-\beta \dto & U \times_S X_1\dto\\
P \rto_-\alpha & U \times_S X_2
}
$$
and $\beta:Q \to U \times_S X_1$ is a $\La_1$-equivariant
morphism from a $\La_1$-torsor over $U$. This defines the necessary
equivalence of categories.
\end{proof}

\subsection{Formal groups}

We now specialize to the case where $S = \Spec(\ZZ)$, 
$\Lambda = \Spec(\ZZ[a_0^{\pm 1},a_1,\ldots])$ is the group scheme of power series invertible under composition.
We set $X = \fgl = \Spec(L)$ where $L$ is the Lazard ring. Thus
for a commutative ring $R$
$$
\La(R) = xR[[x]]^\times
$$
and $X(R) = \fgl(R)$ is
the set of formal group laws over $R$. The group scheme $\La$
acts on $\fgl$ by the formula
$$
(F\phi)(x_1,x_2) = \phi^{-1}(F(\phi(x_1),\phi(x_2)).
$$

In Theorem \ref{coord-affine} we produced, for any formal
group $G$ over an affine scheme $U$, an affine morphism
of schemes
$$
\coord_G \longr S.
$$
The following is essentially a combination of Lemma \ref{coords-after-ff}
and Theorem \ref{coord-affine}.2.

\begin{lem}\label{coord-is-torsor}For a formal group $G$ over
a quasi-compact and quasi-separated scheme $U$, the scheme of
coordinates $\coord_G \to U$ is a $\La$-torsor over $U$.
\index{coordinates, as torsor}
\end{lem}

\begin{proof}The formal group $G$ over $U$ may not
have a coordinate.  However, Lemma \ref{coords-after-ff}
implies that there is an $fqpc$-cover $f:V \to U$ so
that $f^\ast G$ has a coordinate. Reading the proof
of Lemma \ref{coords-after-ff} 
we see that $V$ can be chosen to be affine. Then
$$
V \times_U \coord_G = \coord_{f^\ast G}
$$
is certainly a free right $\La$-scheme over $V$. See Remark
\ref{example-to-clarify} for explicit formulas.
\end{proof}

The following result implies that every $\La$-torsor over $\fgl$
arises in this way from a formal group.

\begin{lem}\label{torsor-over-fgl}Let $S$ be a quasi-compact and
quasi-separated scheme. Let $P \to S$ be a $\La$-torsor
and let $P \to S \times \fgl$ be a morphism $\Lambda$-schemes over
$S$. Then there is a formal group $G \to S$ and
an isomorphism $P \to \coord_G$ of $\La$-torsors
over $S$. This isomorphism is stable under pull-backs in $S$ and
natural in $P$. Furthermore, if $P = \coord_H$, then there
is a natural isomorphism $G \cong H$.
\end{lem}

\begin{proof} We begin with an observation. Let 
$f:U \to S$ be any morphism of schemes so that fiber
$P(U,f)$ of $P(U) \to S(U))$ at $f$
is a free $\La(U)$-set. Then we have a commutative diagram
$$
\xymatrix{
P(U,f) \rto \dto & \fgl(U) \dto \\
P(U,f)/\La(U) = \ast \rto & \fg(U)
}
$$
and the image of the bottom map is a formal group $G_f$ over
$U$. Since the fiber of $\fgl(U) \to \fg(U)$ at $G_f$ is
$\coord_{G_f}(U)$ we have that $G_f$ has a coordinate
and we have  an isomorphism of
free $\La(U)$-sets
\begin{equation}\label{torsor-iso}
P(U,f) \cong \coord_{G_f}(U).
\end{equation}

To get a formal group over $S$ we use descent. Choose a faithfully
flat and quasi-compact map $q:T \to S$ so that fiber $P(T,q)$ 
is a free $\La(T)$-set. This yields a formal group $G_q$
over $T$ as above. Next examine
the commutative diagram
$$
\xymatrix{
P(T) \ar@<.5ex>[r]\ar@<-.5ex>[r] \dto & P(T \times_S T)\dto\\
S(T) \ar@<.5ex>[r]\ar@<-.5ex>[r]  & S(T \times_S T)
}
$$
where the horizontal maps are given by the two projections. Since
the two maps
$$
\xymatrix{
T \times_S T \ar@<.5ex>[r]^-{p_1}\ar@<-.5ex>[r]_-{p_2} & T \rto^q &S
}
$$
are equal the projection maps yield morphisms between fibers
$$
\xymatrix{
P(T,q) \ar@<.5ex>[r]^-{p_1^\ast}\ar@<-.5ex>[r]_-{p_2^\ast}
& P(T \times_S T,qp_1)
}
$$
and hence a unique isomorphism $p_1^\ast G_q \cong p_2^\ast G_q$.
This isomorphism will satisfy the cocycle condition, using uniqueness.
Now descent gives the formal group $G \to S$. Note that if 
$P = \coord_H$, then $G_q = q^\ast H$; therefore, $G \cong H$.

We now define the isomorphism of torsors $P \to \coord_G$ over $S$.
Since both $P$ and $\coord_G$ are sheaves in the $fpqc$ topology,
it is sufficient to define a natural isomorphism $P(U,f) \to \coord_G(U,f)$
for all $f:U \to S$ so that both $P(U,f)$ and $\coord_G(U,f)$ are 
free $\La(U)$-sets. This isomorphism is defined by Equation
\ref{torsor-iso} using the observation that 
$$
\coord_{f^\ast G}(U) = \coord_G(U,f).
$$
\end{proof}

\begin{prop}\label{fg-homotopy-orbit}This morphism
$$
\cM_\fg \longr \fgl \times_\La E\La
$$
is an equivalence of algebraic stacks.
\end{prop}

\begin{proof} Lemma \ref{torsor-over-fgl} at once supplies the
map $\fgl \times_{\La} E\La\to \cM_\fg$ and the needed natural
transformations from either of the two composites to the identity.
\end{proof}

\begin{rem}[{\bf More on gradings}]\label{gradings-revisited-2} In
Remarks \ref{comod-qc-sheaves} and \ref{qc-comodule}
we noted that the category of
quasi-coherent sheaves on $\cM_\fg$ is equivalent to the
category of $(L,W)$-comodules. In Remark \ref{gradings} we noted
that $(L,W)$ has a natural grading. We'd now like to put the
gradings into the comodules and recover the $E_2$-term of
the Adams Novikov Spectral Sequence as the cohomology
of the moduli stack $\cM_\fg$.

Let $\La_0$ be a group scheme with
a right action by another group scheme $H$. Then we can form the semi-direct
product $\La_0 \ltimes H = \La$. To specify a right action of $\La$ on a scheme
$X$ is to specify actions of $\La_0$ and $H$ on $X$ so that for all rings
$A$ and all $x \in X(A)$, $g \in \La_0(A)$, and $u \in H(A)$, we have
$$
x(gu) = (xu)(gu).
$$
We then get a morphism of algebraic stacks
\begin{equation}\label{semi-direct}
X//\La_0 \defeq X \times_{\La_0} \La_0 \longr X \times_G E\La \defeq X//\La
\end{equation}
If $H$ and $\La_0$ are both flat over the base ring $R$, then this
is  a reprentable and flat morphism. We now want to identify the
fiber product $X//\La_0 \times_{X//\La} X//\La_0$.

Let $A$ be a commutative ring and $P_0$ a $\La_0$-torsor over $A$.
If $u \in H(A)$ is an $A$-point of $H$, then we get a new $A$-torsor
$P^u$ with underlying scheme $P$ but a new action defined pointwise
by
$$
x \ast g = x(gu).
$$
Here we have used $\ast$ for the new action and juxtaposition
for the old. If $\alpha:P \to A \otimes X$ is a morphism of
$\La_0$-schemes then we get a new morphism $\alpha^u:P^u \to X$
given pointwise by
$$
\alpha^u(x) = \alpha(x)u^{-1}.
$$
Conjugation by $u$ in $\La$ defines an isomorphism
$\phi_u:P \times_{\La_0} \La \to P^u \times_{\La_0} \La$ of
$\La$-torsors over $A$ so that the following diagram commutes
$$
\xymatrix@R=10pt{
P \times_{\La_0} \La \ar[dr]^\alpha\ar[dd]_{\phi_u}\\
& A \otimes X.\\
P^u \times_{\La_0} \La \ar[ur]_{\alpha^u}
}
$$
Thus we have defined a morphism
$$
X//\La_0 \times H \to X//\La_0 \times_{X//\La} X//\La_0
$$
given pointwise by
$$
((P,\alpha),u) \mapsto ((P,\alpha),(P^u,\alpha^u),\phi_u)
$$
and we leave it to the reader to show that this is an equivalence.

From this equivalence we can conclude that the category of 
quasi-coherent sheaves on $X \times_\La E\La$ is equivalent to the
category of cartesian quasi-cohernet sheaves on the \v Cech nerve
induced by the morphism of Equation \ref{semi-direct}:
\begin{equation}\label{cobar3}
\xymatrix{
\cdots\ X//\La_0 \times H \times H\ar@<1.5ex>[r] \ar[r] \ar@<-1.5ex>[r]&
\ar@<.75ex>[l] \ar@<-.75ex>[l] X//\La_0 \times H \ar@<.75ex>[r] \ar@<-.75ex>[r] &
\lto X//\La_0 \rto & X//\La.
}
\end{equation}

This translates into comdodules as follows. Suppose that $\La_0 = \Spec(\Ga_0)$
and $H = \Spec(K)$ for Hopf algebras $\Ga_0$ and $K$ respectively.  Then
$\La = \Spec(\Ga)$ where $\Ga = \Ga_0 \otimes K$ with the twisted Hopf
algebra structure determined by the action of $H$ on $\La_0$. Suppose 
$X = \Spec(A)$. If $M$ is an $(A,A \otimes K$)-comodule, then
$M \otimes \La_0$ has an induced structure as an $(A,A \otimes K)$-comodule
using the diagonal coaction. We define the category of
$(A,A \otimes K)$-comodules in $(A,A\otimes \La_0)$-comodules
to be those comodules so that the comodule structure map
$$
M \longr M \otimes_A (A \otimes \La_0)
$$
is a morphism of $(A,A \otimes K)$-comodules. 
We have
\begin{enumerate}

\item the category of quasi-coherent sheaves on $X \times_\La E\La$ 
is equivalent to the category of  $(A,A \otimes \Ga)$-comodules; and

\item the category of cartesian sheaves on the \v Cech nerve
of $X//\La_0 \to X//\La$ is equivalent to the category of $(A,A \otimes K)$
comodules in the category of $(A,A\otimes K)$-comodules in
$(A,A\otimes \La_0)$-comodules.
\end{enumerate}

From this we conclude that the category of  $(A,A \otimes \Ga)$-comodules
is equivalent to the category of $(A,A\otimes K)$-comodules in
$(A,A\otimes \La_0)$-comodules.

As example, suppose $H = \GG_m$. Then the action of $\GG_m$
on $G_0$ and $X$ gives a grading to $\La_0$ and $A$ and the
category of $(A,A \otimes K)$
comodules in the category of $(A,A\otimes K)$-comodules in
$(A,A\otimes \La_0)$-comodules is equivalent to the category
of {\it graded} $(A,A \otimes \La_0)$-comodules. Thus
we conclude that the category $(A,A \otimes \La_0[a_0^{\pm 1}])$ comodules
is equivalent to the category of graded $(A,A \otimes \La_0)$-comodules.
In this case it is possible to give completely explicit formulas for
the equivalence. For example, if $M$ is an $(A,A \otimes \La_0[a_0^{\pm 1}])$
comodule, the comodule structure map induces a homomorphism
\index{comdules, graded}
$$
\xymatrix{
M \rto & M \otimes_A ( A \otimes \La_0[a_0^{\pm 1}]) \rto^-\cong &
M  \otimes \La \otimes \ZZ[a_0^{\pm 1}] \rto^-{1 \otimes \epsilon \otimes 1} &
M \otimes \ZZ[a_0^{\pm 1}]
}
$$
which defines the grading and the map
$$
\xymatrix{
M \rto & M \otimes_A ( A \otimes \La_0[a_0^{\pm 1}]) \rto^-{a_0=1}&
M \otimes_A ( A \otimes \La_0)
}
$$
induces the comodule structure.

This equivalence of categories can be used to refine the isomorphism
of Equation \ref{coh-is-ext}. If $\cF$ is a quasic-coherent sheaf
of $X \times_\La E\La$, let $M$ be the associated comodules.
The we have natural isomorphisms --  where we have added asterisks ($\ast$)
to indicate where we are working with graded comodules.
\begin{align}\label{coh-is-ext-graded}
H^s(X \times_G EG,\cF) &\cong \Ext_\La(A,M)\\
                        &\cong \Ext^s_{\La_{0,\ast}}(A_\ast,M_\ast)\nonumber.
\end{align}
In the case of formal groups, we get the grading on the Lazard ring of
this yields the isomorphism of Remark \ref{gradings}; write $L_\ast$
for this graded ring. Then
$$
W_{0,\ast} = L_\ast[a_1,a_2,a_3,\cdots]
$$
represents the functor of strict isomorphisms. The complex cobordism
ring $MU_\ast$ is $L_\ast$ with the grading doubled; likewise,
$MU_\ast MU$ is $W_{0,\ast}$ with the grading doubled. With all
of this done, we can identify
sheaf cohomology with $E_2$-term on the Adams-Novikov spectral
sequence. For example,
\begin{align}\label{ANSS-E2-2}\index{Adams-Novikov spectral sequence}
H^s(\cM_\fg,\omega^t) &\cong \Ext_W(L,L[t])\nonumber\\
                        &\cong \Ext^s_{MU_\ast MU}(MU_\ast,\Omega^{2t}MU_\ast)\\
                        &\cong \Ext^s_{MU_\ast MU}(\Sigma^{2t}MU_\ast,MU_\ast)\nonumber
\end{align}
The extra factor of $2$ arises as part of the topological grading. 
\end{rem}

\def\pnty#1{{{\langle{#1}\rangle}}}

\subsection{Buds of formal groups}

One of the difficulties with the moduli stack $\cM_\fg$ of formal
groups is that it does not have good finiteness properties. We have
written $\cM_\fg$ as $\fgl \times_\La E\La$ and neither the
group $\La$ or the scheme $\fgl$ is  of finite type over $\ZZ$.
However, we can
write $\cM_\fg$ as the homotopy inverse limit of stacks $\cM_\fg\pnty{n}$
which has an affine smooth cover of dimension $n$. 

Let $n \geq 1$ and $\La\pnty{n}$ be the affine group scheme over
$\Spec(\ZZ)$ which
assigns to each commutative ring $R$, the partial power series
of degree $n$
$$
f(x) = a_0x + a_1x^2 + \cdots + a_{n-1}x^{n} \in R[[x]]/(x^{n+1})
$$
with $a_0$ a unit. This becomes a
group under composition of power series. Of course,
$$
\La\pnty{n} = \Spec(\ZZ[a_0^{\pm 1},a_1,\ldots,a_{n-1}]).
$$
Similarly, let $\fgl\pnty{n}$ be the affine scheme of {\it $n$-buds of
formal group laws}\index{buds of formal group laws}
$$
F(x,y) \in R[[x,y]]/(x,y)^{n+1}.
$$
Thus we are requiring that $F(x,0) = x = F(0,x)$, $F(x,y) = F(y,x)$,
and
$$
F(x,F(y,z)) = F(F(x,y),z)
$$
all modulo $(x,y)^{n+1}$. The symmetric $2$-cocyle lemma
\cite{Rav} A.2.12 now implies that\index{Lazard ring for buds, $L\pnty{n}$}
$$
\fgl\pnty{n} = \Spec(\ZZ[x_1,x_2,\cdots,x_{n-1}]) \defeq \Spec(L\pnty{n})
$$
and  modulo $(x_1,\ldots,x_n)^2$, the universal $n$-bud
reads
$$
F_u(x,y) = x+y + x_1C_2(x,y) + \cdots x_{n-1}C_{n}(x,y)
$$
where $C_k(x,y)$ is the $k$th symmetric $2$-cocyle.
The group $\La\pnty{n}$ acts of $\fgl\pnty{n}$.

\begin{defn}\label{buds-stack} The {\bf moduli stack of $n$-buds
of formal groups} is the homotopy orbit stack\index{stack, of buds}
$$
\cM_\fg\pnty{n} = \fgl\pnty{n} \times_{\La\pnty{n}}  E\La\pnty{n}.
$$
\end{defn}

\begin{rem}\label{warning-grading}1.) {\bf Warning:} The stacks $\cM_\fg\pnty{n}$ are not related
to the spectra $BP\pnty{n}$ which appear in chromtatic
stable homotopy -- see \cite{Rav} -- but I was running out
of notation. I apologize for the confusion. The objects $BP\pnty{n}$
will not appear in these notes, although the
cognoscenti should contemplate Lemma \ref{p-buds} below.

2.) Using Remarks \ref{gradings} and \ref{gradings-revisited-2} we see that the category of
quasi-cohernet sheaves is equivalent to the category of graded
comodules over the graded Hopf algebroid $(L\pnty{n}_\ast,W\pnty{n}_{0,\ast})$
where $L\pnty{n}_\ast$ is the ring $L\pnty{n}$ with the degree of $x_i$
equal to $i$ and 
$$
W\pnty{n}_{0,\ast} = L\pnty{n}_\ast[a_1,a_2,\cdots,a_{n-1}]
$$
with the degree of $a_i$ equal to $i$. This will be important later
in the proof of Theorem \ref{fp-modules}. Note that $W\pnty{n}_{0,\ast}$
represents the functor of strict isomorphisms of buds.
\end{rem}

There are canonical maps
$$
\cM_\fg \longr \cM_\fg\pnty{n} \longr \cM_\fg\pnty{n-1}.
$$

\begin{exam}\label{pnty-low} To make your confusion specific\footnote{This is a quote from Steve Wilson. See \cite{primer}.}, note that
$$
\cM_\fg\pnty{1} = B\GG_m = \Spec(\ZZ) \times_{\GG_m} E\GG_m.
$$
This is because $\La_1(R) = R^\times = \GG_m(R)$ is the
group of units in $R$ and, modulo 
$(x,y)^2$, the unique bud of a formal group law is $x + y$.
We also have
$$
\cM_\fg\pnty{2} = \AA^1\times_{\La_2} E\La_2
$$
where $\La_2$ acts on $\AA^1$ by 
$$
(b,a_0x + a_1x^2) \mapsto a_0b - 2(a_1/a_0).
$$
Note that, modulo $(x,y)^3$, any bud of a formal group law
is of the form $x+y + bxy$.
\end{exam}

The following implies that $\cM_\fg\pnty{n}$ is an algebraic
stack in the sense of \cite{Laumon} D\'efinition 4.1. See
also \cite{Laumon}, Exemple 4.6.

\begin{prop}\label{chunks-dim} The morphism
$$
\Spec(L\pnty{n}) \to \cM_\fg\pnty{n}
$$
classifying the universal $n$-bud of a formal group law
is a presentation and smooth of relative dimension $n$.
\end{prop}

\begin{proof}That the morphism is a presentation
follows from Proposition \ref{homotopy-orbit}. 
To see that it is smooth of relative dimension $n$, we must
check that for all morphisms $\Spec(R) \to \cM_\fg\pnty{n}$
the resulting pull-back
$$
\Spec(R) \times_{\cM_\fg\pnty{n}} \Spec(L\pnty{n}) \to \Spec(R)
$$
is smooth of relative dimension $n$.
Since smoothness is local for the $fpqc$ topology, we may
assume that $\Spec(R) \to \cM_\fg\pnty{n}$ classifies
a bud of formal group law. Then
$$
\Spec(R) \times_{\cM_\fg\pnty{n}} \Spec(L\pnty{n}) \cong
\Spec(R[a^{\pm 1}_0,a_1,\cdots,a_{n-1}]) = \Spec(R) \times \La_n
$$
and this suffices.
\end{proof}

Recall that that $n$th symmetric  $2$-cocycle is
$$
C_n(x,y) = \frac{1}{d_n}[(x+y)^n - x^n - y^n].
$$
where
$$
d_n = \brackets{p,}{n=p^k\ \hbox{for a prime $p$};}{1,}{\mathrm{otherwise}.}
$$
Let $\GG_a$ be the additive group scheme and let
$\AA^1\pnty{n}$ be the $\GG_a$ scheme with action
$\AA^1\pnty{n} \times \GG_a \to \AA^1\pnty{n}$ given
by 
$$
(x,a) \mapsto x - d_na.
$$

\begin{lem}\label{trans-down}The morphism $\La\pnty{n} \to \La\pnty{n-1}$
of affine group schemes is flat and surjective with kernel $\GG_a$. Furthermore
there is an equivariant isomorphism of $\GG_a$ schemes over
$\fgl\pnty{n-1}$
$$
\fgl\pnty{n} \cong \fgl\pnty{n-1} \times \AA^1\pnty{n}.
$$
\end{lem}

\begin{proof}The kernel of $\La\pnty{n}(R) \to \La\pnty{n-1}(R)$ is
all power series of the form $\phi_a(x) = x + ax^n$ modulo $(x^{n+1})$.
Since $\phi_a(\phi_{a'}(x)) = \phi_{(a+a')}(x)$ modulo $(x^{n+1})$,
the first statement follows. For the splitting of $\fgl\pnty{n}$ note that if
$\phi_a(x)$ is an isomorphism of buds of formal group laws $F \to F'$,
then
\begin{align*}
F'(x,y) &= F(x,y) + a[x^n - y^n - (x+y)^n]\\
& = F(x,y) - d_naC(x,y).
\end{align*}
Thus the coaction morphism on coordinate rings 
$$
\ZZ[x_1,\ldots, x_n] \longr \ZZ[x_1,\ldots, x_n] \otimes \ZZ[a]
$$
sends $x_i$ to $x_i$ is $i \ne n$ and $x_n$ to 
$$
x_n \otimes 1- 1 \otimes d_n a.
$$
This gives the splitting.
\end{proof}

\begin{prop}\label{chunk-fibration}For all $n \geq 1$ the
reduction map
$$
\cM_\fg\pnty{n} \longr \cM_\fg\pnty{n-1}
$$
is a fibration. If $R$ is any commutative ring in which $d_n$
is a unit, then
$$
\cM_\fg\pnty{n} \otimes R\longr 
\cM_\fg\pnty{n-1} \otimes R
$$
is an equivalence of algebraic stacks.
\end{prop}

\begin{proof} This follows immediately from Example \ref{extremes-orbits},
Propositions \ref{fibration} and
\ref{torsor-equiv}, Lemma \ref{trans-down}, and the following
fact: if $d_n$ is a unit in $A$, then $\AA^1\pnty{n}$ is isomorphic
to $\GG_a$ as a right $\GG_a$-scheme.
\end{proof}

\begin{thm}\label{chunks-inv}
The natural map
$$
\cM_\fg \longr \holim \cM_\fg\pnty{n}
$$
is an equivalence of stacks.
\end{thm}

\begin{proof} We must prove that
for all rings $R$ the natural morphism of groupoids
$$
\cM_\fg(R) \longr \holim \cM_\fg\pnty{n}(R)
$$
is an equivalence. By Proposition \ref{chunk-fibration} we have
that the projection map
$$
\cM_\fg\pnty{n}(R)  \longr \cM_\fg\pnty{n-1}(R)
$$
is a fibration of groupoids for all $n$. Thus we need only show
$\cM_\fg(R) \cong \lim \cM_\fg\pnty{n}(R)$, but
this is obvious.
\end{proof}

The next result is an incredibly complicated way to prove that
every formal group over an algebra over the rationals
is isomorphic to the additive formal group. It proves more,
however, as it also identifies the automorphisms of the additive
formal group. For the proof combine Theorem \ref{chunks-inv}
and Proposition \ref{chunk-fibration}.

\begin{cor}\label{rational}The projection map
$$
\cM_\fg \otimes \QQ
\longr \cM_\fg\pnty{1} \otimes \QQ\simeq
B(\GG_m \otimes \QQ)
$$
is an equivalence.
\end{cor}

When working at a prime $p$, the moduli stacks
$\cM_\fg\pnty{p^n} \otimes \ZZ_{(p)}$
form the significant layers in the tower. These should
have covers by ``$p$-typical buds''; the next result
makes that thought precise. Recall that the universal
$p$-typical formal group law $F$ is defined over the the ring
$V \cong \ZZ_{(p)}[u_1,u_2,\cdots]$. See Corollary
\ref{ptyp-cover}.

\begin{lem}\label{p-buds}Let $V_n = \ZZ_{(p)}[u_1,\ldots,u_n]$
be the subring of $V$ generated by $u_k$, $k \leq n$.
The $p^n$-bud $F_{p^n}$
of the universal $p$-typical formal group law $F$ is defined
over $V_n$ and the morphism
$$F_{p^n}:\Spec(V_n) \to \cM_\fg\pnty{p^n}\otimes\ZZ_{(p)}
$$
classifying this bud is a presentation. Furthermore there is an isomorphism
$$
\Spec(V_n) \times_{\cM_\fg\pnty{n}} \Spec(V_n) \cong
\Spec(V_n[t_0^{\pm 1},t_1,\cdots,t_n]).
$$
\end{lem}

\begin{proof} We use the gradings of Remark \ref{gradings}.
The $n$-bud of a formal group law $G$ is given by the equation
$$
G_n(x,y) = \sum_{i+j \leq n} c_{ij}x^iy^j.
$$
If $F$ is the universal $p$-typical formal group law, we see that
$F_n$ is defined over the subring of $V$ generated by the $u_k$ with
$p^k \leq  n$. Similarly if $\phi(x)$ is the universal isomorphism
of $p$-typical formal group laws, then its bud
$$
\phi_n(x) = \sum_{i=0}^{n-1} a_ix^{i+1}
$$
is defined over the subring of
$V[t_0^{\pm 1},t_1,\cdots]$ generated by $t_k$ and
$u_k$ with $p^k \leq n$.

To show that we have a presentation, suppose $G$ is
a $p^n$ bud of a formal group over a field $\FF$ which
is a $\ZZ_{(p)}$-algebra. Since $\FF$ is a field, we
may assume $G$ arises from the bud of formal group law,
which we also call $G$. Choose any formal group
law $G'$ whose $p^n$-bud is $G$ and choose an
isomorphism $G' \to G''$ where $G''$ is $p$-typical. 
Then the $p^n$-bud of $G''$ is isomorphic to
$G$ and, by the previous paragraph, arises from
a morphism $g:V_n \to \FF$. Thus we obtain the requisite
$2$-commuting diagram
$$
\xymatrix@R=20pt@C=20pt
{
&\Spec(V_n)\dto^{F_{p^n}}\\
\Spec(\FF) \ar[ur]^{\Spec(g)} \rto_-G &
\cM_\fg\pnty{p^n}\otimes \ZZ_{(p)}.
}
$$
A similar argument computes the homotopy pull-back.
\end{proof}

\begin{rem}\label{intrinsic-germ}It is possible to give an intrinsic
geometric definition of an $n$-bud of a formal group in the style
of Definition \ref{frml-var} and Definition \ref{formal-group}.
First an $n$-{\it germ} of a formal Lie variety $X$ over a scheme
$S$ is an affine morphism of schemes $X \to S$ with
a closed section $e$ so that
\begin{enumerate}

\item $X = \Inf^n_S(X)$;

\item the quasi-coherent sheaf $\omega_e$ is locally free of finite rank on $S$;

\item the natural map of graded rings $\Sym_\ast (\omega_e) \to
\gr_\ast(X)$ induces an isomorphism
$$
\Sym_\ast(\omega_e)/\cJ^{n+1} \to \gr_\ast(X)
$$
where $\cJ = \oplus_{k>0}\Sym_k(\omega_e)$ is the augmentation
ideal.
\end{enumerate}
An $n$-bud of a formal groupis  then an $n$-germ $G \to S$ so that
$\omega_e = \omega_G$ is locally free of rank $1$ and there
is a ``multiplication'' map
$$
\Inf_S^n(G \times_S G) \to G
$$
over $S$ so that the obvious diagrams commute.
\end{rem}

\subsection{Coherent sheaves over $\cM_\fg$}

We would like to show that any finitely presented sheaf
over $\cM_\fg$ can be obtained by base change from
$\cM_\fg\pnty{n}$ for some $n$.

Let $m$ and $n$ be integers $0 \leq n < m \leq \infty$
and let
$$
q_{(m,n)} = q:\cM_\fg\pnty{m} \to \cM_\fg\pnty{n}
$$
be the projection. I'll write $q$ for  $q_{(m,n)}$ whenever
possible. Also, I'm writing
$\cM_\fg\pnty{\infty}$ for $\cM_\fg$ itself.

Write $\Qmod_\fg\pnty{n}$ for the quasi-coherent
sheaves on $\cM_\fg\pnty{n}$. We begin by discussing the
pull-back and push-forward functors
$$
\xymatrix{
q^\ast: \Qmod_\fg\pnty{n} \ar@<.5ex>[r]
& \ar@<.5ex>[l] \Qmod_\fg\pnty{m}:q_\ast.
}
$$

By Remark \ref{qc-comodule}, the category of quasi-coherent
sheaves on $\cM_\fg\pnty{n}$ is equivalent to the category
of $(L\pnty{n},W\pnty{n})$-comodules. In fact, if $\cF$ is a quasi-coherent
sheaf, the associated comodule $M$ is obtained by evaluating
$\cF$ at $\Spec(L\pnty{n}) \to \cM_\fg\pnty{n}$, and the comodule
structure is obtained by evaluating $\cF$ on the parallel arrows
$$
\xymatrix{
\Spec(W\pnty{n})
\ar@<.5ex>[r] \ar@<-.5ex>[r] & \Spec(L\pnty{n}) \rto & \cM_\fg\pnty{n}.
}
$$
We will describe
the functors $q_\ast$ and $q^\ast$ by giving a description
on comodules.

Let $\Gamma\pnty{n,m}$ be the group scheme which assigns to each commutative ring $A$  the 
invertible (under composition) power series modulo $(x^{m+1})$
$$
x + a_{n}x^{n+1} + a_{n+1}x^{n+2} + \cdots + a_{n-1}x^m \quad a_i \in R.
$$
Then $\Gamma\pnty{n,m} = \Spec(\ZZ[a_n,a_{n+1},\ldots,a_{m-1}])$
and $\Gamma\pnty{n,m}$
is the kernel of the projection map $\Lambda\pnty{m} \to \Lambda\pnty{n}$.

By Proposition \ref{fibers} there is an
equivalence of algebraic stacks
$$
\Spec(L\pnty{n}) \times_{\cM_\fg\pnty{m}} \cM_\fg\pnty{m} \simeq \fgl\pnty{m}\times_{\Gamma\pnty{n,m}}
E\Gamma\pnty{n,m}.
$$
Let $\cF$ be a quasi-coherent sheaf on $\cM_\fg\pnty{m}$. Then
the value of $q_\ast\cF$ when evaluated at
$\Spec(L\pnty{n}) \to \cM_\fg\pnty{n}$
is $H^0(\Spec(L\pnty{n}) \times_{\cM_\fg\pnty{n}} \cM_\fg\pnty{m},\cF)$.
If $M = \cF(\Spec(L\pnty{m}))$ is the $(L\pnty{m},W\pnty{m})$-comodule equivalent
to $\cF$, then these global sections are the $(L\pnty{n},W\pnty{n})$-comodule $N$
defined by the equalizer diagram
$$
\xymatrix{
N \rto &M \ar@<.5ex>[r] \ar@<-.5ex>[r] & 
\ZZ[a_n,a_{n+1},\ldots,a_{m-1}] \otimes_{L\pnty{m}} M
}
$$ 
where the parallel arrows are given by left inclusion and the coaction map.
The assignment $M \mapsto N$ determines $q_\ast \cF$. 

To describe $q^\ast$, we give the left adjoint to the functor just described
on comodules. If $N$ is a $(L\pnty{n},W\pnty{n})$-comodule, define a
$(L\pnty{m},W\pnty{m})$ comodule $M = L\pnty{m} \otimes_{L\pnty{n}} N$ with coaction map
$$
L\pnty{m} \otimes_{L\pnty{n}} N \to W\pnty{m} \otimes_{L\pnty{m}} \otimes L\pnty{m}
\otimes_{L\pnty{n}} N \cong W\pnty{m} \otimes_{W\pnty{n}} W\pnty{n} \otimes_{L\pnty{n}} N
$$
given by
$$
\eta_R \otimes \psi:L\pnty{m} \otimes_{L\pnty{n}} N \longr W\pnty{m}
\otimes_{W\pnty{n}} W\pnty{n} \otimes_{L\pnty{n}} N.
$$

\begin{prop}\label{flat-split}For all $m$ and $n$, $0 \leq n \leq m
\leq \infty$, the projection morphism
$$
q:\cM_\fg\pnty{m} \longr \cM_\fg \pnty{n}
$$
is faithfully flat.
\end{prop}

\begin{proof} The morphism $q$ is
flat if and only if the functor $\cF \mapsto q^\ast \cF$ is exact. However,
since the ring homomorphism $L\pnty{n} \to L\pnty{m}$ is flat, the equivalent 
functor $N \mapsto  L\pnty{m} \otimes_{L\pnty{n}} N$ on comodules
is evidently exact. The morphism $q$ is now faithfully flat because
it is surjective.
\end{proof}

The notions of finitely presented and coherent sheaves on schemes were defined
in Remark \ref{mod-shvs}.
 
\begin{defn}\label{fp-fg-sheaves}Let $\cF$ be a quasi-coherent sheaf on on
an $fpqc$-algebraic stack $\cM$. Then $\cF$ is {\bf finitely presented}
if there is an $fqpc$-presentation $q:X \to \cM$ so that $q^\ast \cF$
is finitely presented.\index{sheaf, finitely presented, on a stack}
\end{defn}

By examining the definitions, we see that it is equivalent to specify that
there is an $fqpc$-cover $p:Y \to \cM$ and an exact sequence of
sheaves
$$
\cO_{Y}^{(J)} \to \cO_{Y}^{(I)} \to p^\ast \cF \to 0.
$$
with $I$ and $J$ finite. In many of our examples, the cover we have
a cover $X \to \cM$ with $X = \Spec(A)$ with $A$ Noetherian or,
at worst, coherent. In this case, a finitely presented module sheaf
is coherent (see Remark \ref{mod-shvs}). Also $\cF$ is finitely
presented if and only of $\cF(\Spec(A) \to \cM)$ is a finitely 
presented $A$-module.

In the following result, there is experimental evidence to show that
$\cF_0$ might actually by $(q_n)_\ast \cF$, but I don't need this
fact and couldn't find a quick proof.

\begin{thm}\label{fp-modules}Let $\cF$ be a finitely presented
quasi-coherent sheaf on $\cM_\fg$. Then there is an integer $n$,
a quasi-coherent sheaf $\cF_0$ on $\cM_\fg\pnty{n}$ and
an isomorphism
$$
q_n^\ast \cF_0 \to \cF.
$$
is an isomorphism.
\end{thm}

\begin{proof} Using Remark \ref{gradings-revisited-2} and 
Remark \ref{warning-grading}.2, this result is
equivalent to the following statement. 
Let $M$ be a graded comodule over the graded Hopf algebroid 
$(L_\ast,L_\ast[a_1,a_2,\cdots])$ which is finitely presented as an
$L_\ast$-module. Then there is an integer $n$ and a graded comodule
over $(L\pnty{n}_\ast,L\pnty{n}_\ast[a_1,a_2,\cdots,a_{n-1}])$ and an isomorphism
of graded comodules $L_\ast \otimes_{L(\pnty{n}_\ast} M_0 \cong M$.
This we now prove.

If $N$ is a graded module, write $\Sigma^s N$ for
the graded module with $(\Sigma^s N)_k = N_{k-s}$. Let
$$
\oplus\ \Sigma^{t_j} L_\ast \to \oplus\ \Sigma^{s_i}L_\ast \to M \to 0
$$
be any finite presentation. Choose and integer $n$ greater than or
equal to the maximum of the integers $|a-b|$ where 
$$
a,b \in \{\ s_i,t_j\ \}.
$$
Then we can complete the commutative square of $L\pnty{n}_\ast$-modules
$$
\xymatrix{
\oplus\ \Sigma^{t_j} L\pnty{n}_\ast \ar@{-->}[r]^-f\dto 
& \oplus\ \Sigma^{s_i}L\pnty{n}_\ast\dto\\
\oplus\ \Sigma^{t_j} L_\ast \rto &\oplus\ \Sigma^{s_i}L 
}
$$
and, if $M_0$ is the cokerenel of $f$, a morphism of $L\pnty{n}_\ast$-modules
$M_0 \to M$ so that
$$
L_\ast \otimes_{L\pnty{n}_\ast}M_0 \longr M
$$
is an isomorphism. We now need only check that $M_0$ is a
$W\pnty{n}_{0,\ast}$-comodule. But this follows from the same condition
on $n$ we used above to produce $f$.
\end{proof}

%% file: invardiff.tex
\section{Invariant derivations and differentials}

\subsection{The Lie algebra of a group scheme}

We begin with a basic recapitulation of the notion of the Lie
algebra of a group scheme $G$ over a scheme $S$.
The tangent
scheme and the connection between the tangent scheme and differentials
was discussed in \S 1.3.

\begin{defn}\label{lie} Let $G \to S$ be a group scheme over $S$.
 Let $\lie_{G}$ to be the scheme over $S$
obtained by the pull-back diagram\index{Lie algebra}
$$
\xymatrix{
\lie_{G} \rto \dto & \tan_{G/S} \dto\\
S \rto_e & G.
}
$$
\end{defn}

Let $e:S \to G$ be the inclusion of the identity, which we will
assume is closed. If $\omega_e$ is the conormal sheaf
of this embedding, then, by Lemma \ref{fund-seq} we get
a natural isomorphism
$$
d:\omega_e \longr e^\ast \Omega_{G/S} 
$$
and it follows immediately from Proposition \ref{tan-is-affine} that
$$
\lie_{G} \cong \VV(\omega_e).
$$
In particular, $\lie(G) \to S$ is an affine morphism.
See Remark \ref{tan-for} for a similar construction.

\begin{rem}\label{structure-of-lie}The scheme $\lie_{G} \to S$
has a great deal of structure; we'll emphasize those points
which apply most directly here.

1.) Since $\tan_{G/S}$ is an abelian group scheme
over $G$, $\lie_{G}$ is an abelian group scheme over $S$. More than
that, it is an $\AA^1_S$-module; that is, there is a multiplication
morphism of schemes
$$
\AA^1_S \times_S \lie_{G} \longr \lie_{G}
$$
making $\lie_{G}$ into a module over the ring scheme $\AA^1_S$.
This is a coordinate free way of saying that the abelian
group $\lie_{G}(A)$ is naturally an $A$-module.
To get this $A$-module structure, let $a \in A$ and define
$u_a:A(\epsilon) \to A(\epsilon)$ to be the $A$-algebra map determined
by $u_a(\epsilon) = a\epsilon$. Then $\lie_{G}(u_a)$ determines the
multiplication by $a$ in $\lie_{G}(A)$. 

2.) The zero section
$s: G \to \tan_{G/S}$ defines an action of $G$ on $\lie_{G}$
by conjugation; if $x \in G(R)$, this action is written
$$
\mathrm{Ad}(x):\lie_{G} \longr \lie_{G}.
$$
The naturality of the semi-direct product construction shows that there
is a natural isomorphism of group schemes over $G$
$$
\tan_{G/S} \cong G \rtimes_S \lie_{G}.
$$
In particular, if $G$ is commutative we have an isomorphism
\begin{equation}\label{split-tan}
\tan_{G/S} \cong G \times_S \lie_{G}
\end{equation}
which is natural with respect to homomorphisms of abelian
group schemes.

3.) There is a Lie bracket
$$
[\ ,\ ]:\lie_{G} \times_S \lie_{G} \longr \lie_{G}.
$$
Thus, $\lie_{G}$ is an $\AA^1_S$-Lie algebra. If $G$ is commutative --
as is our focus here -- this bracket is zero, so we won't belabor it. 
\end{rem}

\begin{rem}[{\bf Invariant derivations}]\label{inv-derivations}In
Corollary \ref{global-der-cotan} we wrote down a natural isomorphism
between the module $\Der_S(G,\cO_G)$ of
derivations of $G$ over $S$ with coefficients in $\cO_G$ and
the module of sections of $q:\tan_{G/S} \to G$.
If $s'$ is a section of $\lie_{G} \to S$, then we get a section
$$
s = s'\times G: G = S \times_S G \longr \lie_{G} \rtimes G \cong \tan_{G/S}
$$
of $\tan_{G/S} \to G$
and the assignment $s' \mapsto s$ induces an isomorphism from the module
of sections of $\lie_{G}$ to the module of left invariant sections
of $\tan_{G/S}$. The inverse assigns to $s$ the composition
\index{invariant derivations}
$$
\xymatrix{
S \rto^e & G \rto^-{s_0}& \lie_G(S).
}
$$
There is a sheaf version of this which defines an isomorphism from
the local sections of $\lie_{G} \to G$ to an appropriate sheaf
of invariant derivations in $\shder_S(G,\cO_G)$.
\end{rem}



Now let $G \to S$ be a formal group over $S$; we define
$\lie_{G}$ exactly as above:
$$
\lie_{G} = e^\ast \tan_{G/S} \to S.
$$
Let $\vare: \lie_{G} \to \tan_{G/S}$ be the induced map. In Remark
\ref{tan-for} we showed that $(\tan_{G/S},\vare)$ is a formal
Lie variety over $\lie_{G}$ and that there is a natural
isomorphism of abelian group schemes
$$
\VV(\omega_G) \cong \lie_{G}
$$
over $S$. Exactly as in Equation \ref{split-tan} we have an isomorphism
(now as $fpqc$ sheaves)
$$
\tan_{G/S} \cong G \times_S \lie_{G}
$$
over $S$.

\begin{rem}\label{compute-lie}Let $f:G \to H$ be homomorphism
of smooth, commutative formal groups over $S$. In the presence of 
coordinates, it is possible to give a concrete formula for computing
$\lie(f)$ and $\tan(f)$. 

First suppose that we choose can choose a coordinate $y$ for
$G$. Then $y$ determines an isomorphism
$$
\lambda_y: \GG_a \longr \lie_{G}
$$
from the additive group over $S$ to $\lie_{G}$ sending
$a \in \GG_a(R)$ to $\epsilon a \in \lie_{G}$. 

Next suppose that we also choose a coordinate $x$ for $H$. Then
the image of $y$ under $f$ is a power series $f(x)$ and
we get a commutative diagram
$$
\xymatrix{
G \times_S \GG_a \rto \dto_{G \times \lambda_y} &
H \times_S\GG_a \dto^{H \times \lambda_x}\\
\tan_{G/S} \rto_{\tan(f)} & \tan_{H/S}
}
$$
where the top morphism is given pointwise by
$$
(a,b) \mapsto (f(a),bf'(a)).
$$
Restricting to the Lie schemes, we get a commutative diagram
of schemes over $S$
$$
\xymatrix{
\GG_a \rto^{f'(0)} \dto_{\lambda_y} &
\GG_a \dto^{\lambda_x}\\
\lie_{G} \rto_{\lie(f)} & \lie_{H/S}.
}
$$
\end{rem}

Note that we have also effectively proved the following result.

\begin{prop}\label{ga-torsor}Let $G$ be a smooth
one-dimensional, commutative formal groups over $S$.
Then $\lie_{G}$ is a naturally a $\GG_a$-torsor
in the fpqc topology.
\end{prop}

\begin{proof}The scheme $\lie_{G} \to S$ is a $\GG_a$-scheme because
it is an $\AA^1_S$-module. If we choose an fpqc cover $f:T \to S$
so that $f^\ast G$ can be given a coordinate, then we have
just shown, in Remark \ref{compute-lie}, that a choice of
coordinate defines an isomorphism
$$
f^\ast\GG_a \longr \lie_T(f^\ast G) \cong f^\ast\lie_{G}.
$$
\end{proof}

\subsection{Invariant differentials}\label{inv-diff-sec}

Let $q:G \to S$ be a group scheme over $S$ with identity $e:S \to G$. 
Let us assume that $G$ is flat -- and hence faithfully flat -- and quasi-compact
over $S$. Then we have a diagram
$$
\xymatrix{
G \times_S G \ar@<.5ex>[r]^-{p_1} \dto_f \ar@<-.5ex>[r]_-m & G \dto^=\\
G \times_S G \ar@<.5ex>[r]^-{p_1} \ar@<-.5ex>[r]_-{p_2} & G 
}
$$
where $f$ is an isomorphism give pointwise by $f(x,y) = (x,xy)$ and
$m$ is the multiplication map. From this we conclude that we have
a modified version of descent for $q:G \to S$: the category of
quasi-coherent sheaves on $S$ is equivalent to the category of 
quasi-coherent sheaves $\cF$ on $G$ equipped with an isomorphism
$$
(p_1)^\ast \cF \to m^\ast \cF
$$
satisfying a suitable cocycle condition we leave the reader to formulate.

To apply this, we note that we have diagram
$$
\xymatrix{
G \times_S G \ar@<.5ex>[r]^-{p_1} \dto_-{p_2} \ar@<-.5ex>[r]_-m & G \dto^q\\
G \rto_q & S
}
$$
and both the squares are Cartesian. This supplies an isomorphism
$$
(p_1)^\ast \Omega_{G/S} \cong \Omega_{G \times_S G/G} \cong
m^\ast \Omega_{G/S}
$$
which satisfies the necessary cocycle condition. The resulting quasi-coherent
sheaf $\omega_{G}$ on $S$ is the {\it sheaf of invariant differentials}
\index{invariant differentials} on $G$.
Since $\omega_G$ is already the name we've given to the conormal
sheaf of the unit $e:S \to G$ we have to justify this notion. So for the
next sentence, let's write $\omega_G$ for the invariant differentials
and $e^\ast \Omega_{G/S}$ for the conormal sheaf. Then, 
by construction, we have that 
\begin{equation}\label{little-to-big}
q^\ast \omega_{G} \cong \Omega_{G/S}
\end{equation}
from which it follows that
\begin{equation}\label{big-to-little}
e^\ast\Omega_{G/S} = e^\ast q^\ast \omega_{G} \cong \omega_{G}.
\end{equation}
Thus, from now on, we make no distinction between the two.

\begin{exam}\label{affine-invar}This definition is less arcane that it 
seems. Unwinding the proof of faithfully flat descent, we see that there
is an equalizer diagram of sheaves of $S$
$$
\xymatrix{
\omega_{G} \rto & q_\ast \Omega_{G/S} \ar@<.5ex>[r]^-{dp_1 }
\ar@<-.5ex>[r]_-{dm}
& q_\ast \Omega_{(G \times_S G)/G}
}
$$
where I have written $q$ for the canonical projections to $S$.
To be even more concrete, 
suppose $S = \Spec(R)$ and $G$ is affine over $R$; that is,
$G = \Spec(A)$ for some Hopf algebra $A$ over $R$. Then $\omega_{G}$
is determined by the $R$-module $\omega_{A}$ defined by the equalizer
diagram
$$
\xymatrix{
\omega_{A} \rto & \Omega_{A/R} \ar@<.5ex>[r]^-{di_1 }
\ar@<-.5ex>[r]_-{d\Delta}
& \Omega_{(A \otimes_R A)/A}.
}
$$
For example, if $G = \GG_m$, then $A = R[x^{\pm 1}]$ with
$\Delta(x) = x \otimes x$ and we calculate that
$\omega_{A}$ is the free $R$-module on $dx/x$.
\end{exam}

\begin{rem}\label{the-dual-of-inv}As $\lie_{G} \cong \VV(\omega_G)$,
the sheaf {\it dual} to $\omega_{G}$ is the quasi-coherent sheaf which
assigns to each Zariski open $U \subseteq S$ the sections of 
$\lie_{G} \vert_U \to U$. In particular, the global sections of this
dual sheaf are exactly invariant derivations of $G$. If we need
a name for this sheaf we will call it $\mathrm{lie}_{G/S}$.
\end{rem}

These notions extend to formal groups, with a little care. In this
case we don't have a sheaf $\Omega_{G/S}$ defined -- although we could produce
it if need be. However, in Remark \ref{tan-for}, we did define sheaves
$(\Omega_{G/S})_n$ over $G_n$ and we {\it define}
$$
q^\ast \Omega_{G/S} \defeq \lim q_\ast (\Omega_{G/S})_n
\cong \lim q^\ast \Omega_{G_n/S}
$$
over $S$, where $q:G_n \to S$ is any of the projections. Similarly
$$
q_\ast \Omega_{G \times G/G} = \lim q_\ast \Omega_{G_n \times G_n/G_n}.
$$
The following allows us to call $\omega_G$ the sheaf of invariant
differentials for $G$.\index{invariant differentials, formal}

\begin{prop}\label{formal-inv-diff}Let $G \to S$ be a 
formal group over $S$. Then there is an equalizer diagram
of sheaves on $S$
$$
\xymatrix{
\omega_{G} \rto & q_\ast \Omega_{G/S} \ar@<.5ex>[r]^-{dp_1 }
\ar@<-.5ex>[r]_-{dm}
& q_\ast \Omega_{(G \htimes_S G)/G}
}
$$
\end{prop}

\begin{exam}\label{dlog}Suppose that $S = \Spec(A)$ is affine
and that $G$ can be given a coordinate $x$. Then
$\omega_{G}$ is determined by its $S$-module of global
sections over $S$ and we
we have an equalizer diagram of $A$-modules
$$
\xymatrix{
H^0(S,\omega_{G}) \rto & A[[x]]dx \ar@<.5ex>[r]^-{di_1 }
\ar@<-.5ex>[r]_-{d\Delta}
& A[[x,y]]dx.
}
$$
Let's write $F(x,y) = \Delta(x)$ for the resulting formal group law
and $F_x(x,y)$ for the partial derivative of that power series
with respect to $x$. Then
an invariant differential $f(x)dx$ must satisfy the equality
$$
f(x)dx = f(F(x,y))F_x(x,y)dx.
$$
Setting $x= 0$ and then setting $y = x$ we get that
$$
f(x) = \frac{f(0)}{F_x(0,x)}.
$$
Since $F_x(0,0) = 1$, we conclude that $\omega_{G}$ is the
quasi-coherent sheaf on $\Spec(A)$ determined by the free
$A$-module of rank $1$ with generator
$$
\eta =  \frac{dx}{F_x(0,x)}.
$$
\end{exam}

\begin{exam}\label{calculate-lie} Calculating with $\lie_{G}$ and
$\omega_{G}$ is standard, at least locally. Compare Remark
\ref{compute-lie}. Suppose $S = \Spec(A)$
and $f:G \to H$ is a homomorphism of formal groups over $S$.
By passing to a faithfully flat extension, we may as well assume
that $G$ and $H$ can be given coordinates $x$ and $y$ respectively;
then $f$ is determined by a power series $f(x) \in A[[x]]$ and
the induced morphism
$$
df:\omega_{H} \longr  \omega_{G}
$$
is multiplication by $f'(0)$.
\end{exam}

\subsection{Invariant differentials in characteristic $p$}

As a warm-up  for the next section, we will isolate some of the
extra phenomena that occurs when we are working over
a base scheme $S$ which is itself a scheme over $\Spec(\FF_p)$.
In this case there is a Frobenius morphism
$f:S \to S$. Indeed, if $R$ is an $\FF_p$ algebra, the Frobenius
$x \mapsto x^p$ defines a natural morphism $f_R:R \to R$ of
$\FF_p$ algebras and 
$$
f_S(R)= S(f_R):S(R) \longr S(R),
$$
If $X \to S$ is any scheme over $S$, we define $X^{(p)}$ to be the
pull-back
$$
\xymatrix{
X^{(p)} \rto \dto & X \dto\\
S \rto_f & S
}
$$
and the relative Frobenius\index{Frobenius, relative}
$F:X \to X^{(p)}$ to the unique morphism
of schemes over $S$ so that the following diagram commutes
$$
\xymatrix{
X \ar@/^1pc/[rr]^f \rto_F \ar[dr] & X^{(p)}\dto \rto & X \dto\\
&S \rto_f & S.
}
$$
The following is an exercise in definitions and the universal
properties of pull-backs.

\begin{lem}\label{tan-frob}Let $X \to S$ be a scheme over a
scheme $S$ over $\Spec(\FF_p)$.

1.) There is a natural isomorphism $\tan_{X/S}^{(p)} \cong
\tan_{X^{(p)}/S}$.

2.) If $G \to S$ is a group scheme over $S$, then there is a natural
isomorphism $\lie_{G}^{(p)} \cong \lie_{G^{(p)}/S}$.

3.) The relative Frobenius $F: X \to X^{(p)}$ induces the the zero
homomorphism $\tan_{X/S}(F): \tan_{X/S} \to \tan_{X/S}^{(p)}$;
that is, $\tan_{X/S}(F)$ can be factored
$$
\xymatrix{
\tan_{X/S} \rto & X \rto^-F & X^{(p)} \rto^-s & \tan_{X/S}^{(p)}
}
$$
where $s$ is the zero section.

4.) If $G \to S$ is a group scheme, the relative Frobenius
$F: G\to G^{(p)}$ induces the the zero
homomorphism $\lie(F): \lie_{G} \to \lie_{G^{(p)}/S}$;
that is, $\lie(F)$ can be factored
$$
\xymatrix{
\lie_{G} \rto & S  \rto^-s &\lie_{G^{(p)}/S}.
}
$$
\end{lem}

\begin{proof} The first of two of these statements are an exercise
in definitions and the universal properties of pull-backs.
The second two follow from the fact that if $R$ is an
$\FF_p$-algebra, the $f_{R(\epsilon)}(-) = (-)^p:R(\epsilon)
\to R(\epsilon)$ factors
$$
\xymatrix{
R(\epsilon) \rto^{\epsilon=0}& R \rto^{f_R} & R \rto & R(\epsilon).
}
$$
\end{proof} 

While the morphism $\lie(F)$ induced by the relative
Frobenius $F:G \to G^{(p)}$ is the zero map, the relative
Frobenius 
$$
F: \lie_{G} \to \lie_{G}^{(p)} \cong\lie_{G^{(p)}/S}
$$
is not. This is the map on affine schemes over $S$
$$
\VV(\omega_{G}) \longr \VV(\omega_{G})^{(p)} \cong
\VV(\omega_{G^{(p)}})
$$
induced by the Frobenius morphism on algebra sheaves
$$
\Sym_{S}(\omega_{G^{(p)}}) 
\to \Sym_{S}(\omega_{G}).
$$
By restricting to the sub-$\cO_S$-module $\omega_{G^{(p)}}$
of $\Sym_S(\omega_{G^{(p)}})$
we get the map needed for the following result. $\Sym_p(-)$
is the $p$th symmetric power functor.

\begin{lem}\label{lie-smooth-2}Let $G$ be a group scheme over
$S$ and $S$ a scheme over $\FF_p$. Then the $p$th power
map induces a natural homomorphism
of quasi-coherent sheaves over $S$
$$
\omega_{G^{(p)}} \to \mathrm{Sym}_p(\omega_{G})
$$
which, if $G$ is smooth of dimension $1$, becomes an isomorphism
$$
\omega_{G^{(p)}} \cong \mathrm{Sym}_p(\omega_{G}) \cong
\omega_{G}^{\otimes p}.
$$
\end{lem}

\begin{proof}The last statement follows because $\omega_{G}$ is
locally free of rank 1.
\end{proof}

The exact same argument now proves:

\begin{lem}\label{lie-smooth-1}Let $G$ be a formal group over
$S$ and $S$ a scheme over $\FF_p$. Then the $p$th power
map induces a natural homomorphism
of quasi-coherent sheaves over $S$
$$
\omega_{G^{(p)}} \to \mathrm{Sym}_p(\omega_{G})
$$
yields an isomorphism
$$
\omega_{G^{(p)}} \cong \mathrm{Sym}_p(\omega_{G}) \cong
\omega_{G}^{\otimes p}.
$$
\end{lem}

%% file: height.tex
\section{The height filtration}

The theory of formal groups in characteristic zero is quite
simple: in Corollary \ref{rational} we saw that over $\QQ$,
we are reduced to studying the additive formal group law
and its automorphisms. In characteristic $p > 0$ (and hence
over the integers) the story is quite different. Here formal groups
are segregated by height and it is the height filtration which is
at the heart of the geometry of $\cM_\fg$. The point of this
section is to spell this out in detail.

\subsection{Height and the elements $v_n$}

We are going to study formal groups $G$ over schemes $S$
which are themselves schemes over $\Spec(\FF_p)$. In Lemma
\ref{tan-frob} we introduced and discussed the relative Frobenius
$F$ and its effect on tangent and Lie schemes. The following is
a standard lemma for formal groups. The homomorphism
$F:G \to G^{(p)}$ is the relative Frobenius.

\begin{lem}\label{factor-frob}Let $f:G \to H$ be a homomorphism of 
formal groups over $S$
which is a scheme over $\Spec(\FF_p)$. If
$$
0 = \lie (f): \lie_{G} \to \lie_{H}.
$$
then there is a unique morphism $g:G^{(p)} \to H$ so that 
there is a factoring
$$
\xymatrix{
G \rto^{F}\ar[dr]_f& G^{(p)} \dto^g\\
&H
}
$$
\end{lem}

\begin{proof} It follows immediately from the natural decomposition
$\tan_{G/S} \cong G \times \lie_{G}$ that the induced map
$$
\tan(f):\tan_{G/S} \to \tan_{H/S}
$$
is the zero homomorphism as well. Because of the uniqueness
of $g$ it is sufficient to prove the result locally, so choose
an fqpc-cover $q:T \to S$ so that $q^\ast G$ and $q^\ast H$
can each be given a coordinate. As in Remark \ref{compute-lie},
we write
$f$ as a power series $f(x)$ and because $\tan(f) = 0$ we
conclude that $f'(x) = 0$. Because we are working over
$\FF_p$, we may write $f(x) = g(x^p)$ for some
unique $g(x)$ and we let $g$ define the needed homomorphism
$g^\ast :G^{(p)} \to H$.
\end{proof}

Let $G$ be a formal group over $S$, with $S$ a scheme over
$\Spec(\FF_p)$. Since $G$ is commutative the $p$th power map
$$
[p]: G \longr G
$$
is a homomorphism of formal groups over $S$. If $G$ can be given
a coordinate, then Remark \ref{compute-lie} implies that
$\lie([p]) = 0$. More generally, we choose an fpqc cover
$f:T \to S$ so that $f^\ast G$ has a coordinate. Then,
since $f$ is faithfully flat and $f^\ast\lie([p]) = 0$, we
have $\lie([p]) = 0$. Therefore, Lemma \ref{factor-frob}
implies there is a unique homomorphism $V:G^{(p)} \to G$ so
that we have a factoring
$$
\xymatrix{
G \rto^{F} \ar[dr]_{[p]} & G^{(p)} \dto^{V}\\
&G.\\
}
$$
The homomorphism $V$ is called the {\it Verschiebung}.
\index{Verschiebung}.
The induced morphism $\lie(V):\lie_{G}^{(p)} \to
\lie_{G}$ may itself be zero; if so, we obtain a factoring
$$
\xymatrix{
G^{(p)} \rto^{F^{(p)}} \ar[dr]_{V} & G^{(p^2)} \dto^{V_2}\\
&G.\\
}
$$
We may continue if $\lie(V_2) = 0$.

\begin{defn}[The height of a formal group]\label{height}\index{height,
of a formal group}Let
$G$ be a formal group over a scheme $S$ which is itself
a scheme over $\Spec(\FF_p)$. Define $G$ to have
{\bf height} at least $n$ if there is a factoring
$$
\xymatrix{
G \rto^{F} \ar[drrrr]_{[p]} & G^{(p)} \rto^{F^{(p)}} & G^{(p^2)}
\rto^{F^{(p^2)}} & \cdots \rto^{F^{(p^{n-1})}} & G^{(p^n)}
\dto^{V_n}\\
&&&&G.\\
}
$$
Define $G$ to have height exactly $n$ if  $\lie(V_n) \ne 0$.
\end{defn}

Note that a formal group may not have finite height; for example,
if $\hat{\GG}_a$ is the formal additive group, then
$0 = [p]:\hat{\GG}_a \to \hat{\GG}_a$ so it must have
infinite height. It follows from Lazard's uniqueness theorem
 (Corollary \ref{lazard-unique})
that every infinite height formal group is locally isomorphic
to the additive group.

\begin{prop}\label{vn}\index{$v_n$, as section of
$\omega_G^{p^n-1}$}Let $G$ be a formal group over a scheme $S$
which is itself a scheme over $\Spec(\FF_p)$. Suppose that
$G$ has height at least $n$. Then there is a global section
$$
v_n(G) \in H^0(S,\omega_G^{\otimes (p^n-1)})
$$
so that $G$ has height at least $n+1$ if and only if $v_n(G) = 0$.
The element $v_n(G)$ is natural; that is, if $H \to T$ is another
formal group, $f:T \to S$ is a morphism
of schemes and $\phi:H \to f^\ast G$ is an isomorphism of formal
groups, then
$$
f^\ast v_n(G) = v_n(H).
$$
\end{prop}

\begin{proof}Since $G$ is of height at least $n$, we have the
morphism $V_n:G^{(p^n)} \to G$ and $G$ has height at
least $n+1$ if and only if $\lie(V_n) = 0$. This will happen
if and only if the induced map
$$
dV_n:\omega_{G} \longr \omega_{G^{(p^n)}} \cong
\omega_{G}^{\otimes p^n}.
$$
is zero. The last isomorphism uses Lemma \ref{lie-smooth-1}.
Since $\omega_{G}$ is an invertible sheaf, $dV_n$ corresponds to
a unique morphism
$$
v_n(G): \cO_S \longr \omega_{G}^{\otimes (p^n-1)}.
$$
This defines the global section. 
The naturality statement follows from the commutative diagram
$$
\xymatrix{
\lie_{H^{(p^n)}} \rto^-{V_n} \dto_{\lie(\phi^{(p^n)})} &
\lie_{H} \dto^{\lie(\phi)}\\
f^\ast \lie_{G^{(p^n)}}\rto^-{V_n} & f^\ast \lie_{G}
}
$$
\end{proof}

\begin{rem}\label{vn-local}The global section $v_n$ can be computed
locally as follows. Let $S = \Spec(R)$ be affine and suppose
$G \to S$ can be given a coordinate $x$. Then if $G$ is
of height at least $n$, the power series expansion 
of $[p]:G \to G$ gives the $p$-series:
$$
[p](x) = a_nx^{p^n} + a_{2n}x^{2p^n} + \cdots
$$
If $\eta(G,x) = dx/F_x(0,x)$ is the invariant differential associated to this
coordinate, then
$$
v_n(G) = a_n \eta(G,x)^{\otimes p^n-1} \in \omega_{G}^{\otimes (p^n-1)}.
$$
In particular, $v_n(G) = 0$ if and only if $a_n = 0$. 
\end{rem}

We wish to define a descending chain of closed substacks
$$
\cdots \subseteq \cM(3) \subseteq \cM(2) \subseteq \cM(1) \subseteq \cM_{\fg}
$$
with $\cM(n)$ the moduli stack of  formal groups of height greater than
or equal to $n$. Of course, $\cM(n)$ will be defined by the
vanishing of $p,v_1,\ldots,v_{n-1}$, but it's worth dwelling on
the definition so that the behavior of $\cM(n)$ under base change is
transparent.

Let $\cM$ be an $fqpc$-algebraic stack over a base scheme $S$. Recall that an
{\it effective Cartier divisor}\index{Cartier divisor, effective} $D \subseteq \cM$ is  closed subscheme
so that the ideal sheaf $\cI(D) \subseteq \cO_\cM$ defining $D$ is locally free of
rank 1.\footnote{Some authors ({\it cf} \cite{KM}, Chapter 1) require also
that $D$ be flat over $S$. This implies that if $f:T \to S$ is a morphism
of schemes, then $T \times_S D$ is an effective Cartier divisor for $T \times_S \cM$.
But is also means that if $X$ is a scheme with $\cO_X$
torsion free then the closed subscheme obtained from setting
$p=0$ is not an effective Cartier divisor for $X$ over $\ZZ$.}
If we tensor the exact sequence
$$
0 \to \cI(D) \to \cO_\cM \to \cO_D \to 0
$$
of sheaves on $\cM$ with the dual sheaf $\cI(D)^{-1}$, then 
we get an exact sequence
$$
\xymatrix{
0 \rto &\cO_\cM \rto^-s & \cI(D)^{-1}\rto & \cO_D \otimes_{\cO_\cM} \cI(D)^{-1} \rto &0
}
$$
with $s$ a section of $\cI(D)^{-1}$. Conversely, given an invertible sheaf $\cL$, 
a section $s$ of $\cL$, and an exact sequence
$$
\xymatrix{
0 \rto &\cO_\cM \rto^s & \cL \rto& \cL/\cO_\cM \rto & 0
}
$$
then the substack of zeros of $s$ is an effective
Cartier divisor with ideal sheaf defined by the image of the injection
$$
s:\cL^{-1} \longr \cO_\cM.
$$
This establishes a one-to-one correspondence between effective
Cartier divisors and isomorphism classes of pairs $(\cL,s)$ as
above. We will say that the divisor is defined by the pair
$(\cL,s)$. For example
$$
\cM(1) \subseteq \cM_\fg 
$$
is the effective Cartier divisor defined by $(\cO_\fg,p)$. Suppose
$\cM(n)$ has been defined and classifies formal groups of
height at least $n$.

\begin{defn}\label{mn-divisor} 1.) Define the closed substack
$\cM(n+1) \subseteq \cM(n)$ to be the effective Cartier divisor
defined by the pair $(\omega^{p^n-1},v_n)$.\index{$\cM(n)$}

2.) Let $\layer{n} = \cM(n) - \cM(n+1)$ be the open complement
of $\cM(n+1)$ in $\cM(n)$. Then $\layer{n}$ classifies formal
groups of {\bf exact height $n$} or simply of height $n$.
\index{$\layer{n}$}

3.) Let $\cU(n)$ be the open complement of $\cM(n-1)$; then
$\cU(n)$ is the moduli stack of formal groups of height less than
or equal to $n$.\index{$\cU(n)$}
\end{defn} 

Then Proposition \ref{vn} implies that
$\cM(n+1)$ classifies formal groups of height at least $n+1$.
The inclusion $\cM(n) \subseteq \cM_\fg$ is closed; let
$\cI_n \subseteq \cO_\fg$ be the ideal sheaf defining this
inclusion. Thus we have an ascending sequence of ideal
sheaves
$$
0 \subseteq \cI_1 = (p) \subseteq \cI_2 \subseteq \cdots \cO_\fg
$$
and an isomorphism
$$
v_n(G):\omega^{-(p^n-1)} \longr \cI_{n+1}/\cI_n
$$
on $\cM(n)$.

\begin{rem}\label{exact-height}A formal group $G \to S$ has exact height
$n$ if the global section 
$v_n(G) \in H^0(S,\omega_G^{p^n-1})$ is invertible in the
sense that
$$
v_n:\omega_G^{-(p^n-1)} \longr \cO_S
$$
is an isomorphism.
This makes sense even if $n=0$, where would have $p$ invertible
in $H^0(S,\cO_S)$. This {\it defines} the notion of a formal
group of height $0$.
\end{rem}

\begin{rem}\label{vn-local-1} We can follow up Remark \ref{vn-local}
with a local description of $\cI_n$ and the process of defining
$\cM(n)$. If $G \to \Spec(R)$ can be given a coordinate $x$
and $G$ has height a least $n$, then we can write
$$
v_n(G) = u_n \eta(G,x)^{\otimes p^n-1}
$$
where $\eta(G,x) = dx/F_x(0,x)$ is the generator of $\omega_G$.
The choice of generator $\eta(G,x)^{\otimes p^n-1}$ for
$\omega_G^{p^n-1}$ defines an isomorphism
$R \cong \omega_G^{p^n-1}$ and the section
$v_n:R \to  \omega_G^{p^n-1}$ becomes isomorphic
to multiplication by $u_n$. Thus
$$
\cI_{n+1}(G)/\cI_{n}(G) = (u_n)
$$
is the principal ideal generated by $u_n$. Note that $u_n$ is 
not an isomorphism invariant, but the ideal is.

It is tempting to write, for a general formal group $G$ with a coordinate, that
there is an isomorphism
$$
\cI_n(G) = (p,u_1,\ldots,u_{n-1}).
$$
In general, $u_{n-1}$ is only well-defined modulo $\cI_{n-1}(G)$, so
we must be careful with this notation. It is possible to choose
a sequence $p,u_1,\ldots,u_{n-1}$ defining the ideal $\cI_n(R)$, but 
the choices make the sequence unpleasant. In the presence
of a $p$-typical coordinate, the situation improves. See the next remark.
\end{rem}

\begin{rem}\label{p-typical-vn}The form $v_n$ is defined globally
only when $p = v_1 = \cdots v_{n-1} = 0$. But if $G$ is a formal
group with a coordinate over a $\ZZ_{(p)}$ algebra $R$, then Cartier's
theorem gives a
$p$-typical coordinate $x$ for $G$. Let $F$ be the resulting formal group
law for $G$. Then we can write the $p$-series
$$
[p](x) = px +_F  u_1x^p +_F u_2x^{p^2} +_F \cdots
$$
Remark \ref{vn-local} implies that if $p=u_1=\cdots u_{n-1} = 0$,
then
$$
v_n(H) = u_n\eta(G,x)^{\otimes p^{n-1}}
$$
and we really can write $\cI_n(G) = (p,u_1,\ldots,u_{n-1})$.

Note that $v_n(G) = 0$ if and only if $u_n =0$. Since the morphism
$$
\Spec(\ZZ_{(p)}[u_1,u_2,\ldots]) \longr
\cM_{\fg} \otimes\ZZ_{(p)}
$$
classifying the universal $p$-typical formal group is an $fpqc$-cover,
this remark implies the following result.
\end{rem}

\begin{prop}\label{cover-for-mn} For all primes $p$ and all
$n \geq 1$,  there is an  $fpqc$-cover
$$
X(n) \defeq  \Spec(\FF_{(p)}[u_n,u_{n+1},\ldots]) \to \cM(n).
$$
Furthermore,
$$
p_1^\ast(p,u_1,\ldots,u_{n-1}) = 
p_2^\ast(p,u_1,\ldots,u_{n-1}) \subseteq \cO_{X(n)
\times_{\cM(n)} X(n)}
$$
and 
$$
X(n) \times_{\cM(n)} X(n) = \Spec(\FF_p[u_n,u_{n+1},\cdots]
[t_0^{\pm 1},t_1,t_2,\cdots]).
$$
\end{prop}

If $q:\cN \to \cM$ is a representable and flat morphism of algebraic stacks
and $D \subseteq \cM$ is an effective Cartier divisor
defined by
$(\cL,s)$, then
$$
f^\ast D \defeq D \times_\cM \cN \subseteq \cN
$$
is an effective
Cartier divisor defined by $(q^\ast \cL,q^\ast s)$. To see this,
note that because $f$ is flat, we
have an exact sequence 
$$
0 \to f^\ast \cI(D) \longr f^\ast \cO_\cM \longr f^\ast \cO_D \to 0
$$
which is isomorphic to
$$
0 \to f^\ast \cI(D) \longr \cO_\cN \longr \cO_{f^\ast  D} \to 0.
$$
Thus $f^\ast \cI(D) \cong \cI(f^\ast D)$.
From this we can immediately conclude the following.

\begin{prop}\label{LEFT-easy-part}Let $q:\cN \to \cM_\fg$ be
a representable and flat morphism of stacks and define
$$
\cN(n) = \cM(n) \times_{\cM_\fg} \cN.
$$
Then 
$$
\cdots \subseteq \cN(2) \subseteq \cN(1) \subseteq \cN
$$
is a descending chain of closed substacks so that
$$
\cN(n+1) \subseteq \cN(n)
$$
is the effective Cartier divisor defined by $(\omega^{p^n-1},v_n)$.
\end{prop}

This implies that for all $n$ the section $v_n$ defines an injection
$$
v_n: \cO_\cN \longr \omega^{p^n-1}.
$$
If $\cN = \Spec(R) \to \cM_\fg$ classifies a formal group for which
we can choose a coordinate, this implies that each of the
ideals $\cI_n(R)$ is generated by a regular sequence.
The Landweber Exact Functor Theorem \ref{land-exact-2} is a partial
converse to this result.

In these examples, the closed embedding $\cN(n) \subseteq \cN$ is
a regular embedding; that is, the ideal sheaf defining the embedding is
locally generated by a regular sequence.

\subsection{Geometric points and reduced substacks of $\cM_\fg$}

Suppose we now work at a prime $p$, so that $\cM_\fg = \cM_\fg \otimes
\ZZ_{(p)}$. We will show that $\cM_\fg$ has exactly one geometric
point for each height $n$, $0 \leq n \leq \infty$ and use this
to show that the substacks $\cM(n) \subseteq \cM_\fg$ give
a complete list of the reduced substacks of $\cM_\fg$.

We begin with the following definition.

\begin{defn}\label{points}Let $\cM$ be an algebraic stack.

1.) A {\bf geometric}\index{geometric point, of a stack}
 point $\xi$ of $X$ is an equivalence class
of the morphisms $x:\Spec(\FF) \to \cM$ where $\FF$
is a field. Two such morphisms $(x',\FF')$ and $(x'',\FF'')$
are equivalent if $\FF'$ and $\FF''$ have a common 
extension $\FF$ and the evident diagram
$$
\xymatrix{
\Spec(\FF) \rto\dto & \Spec(\FF'') \dto^{x''}\dto\\
\Spec(\FF') \rto_{x'}& \cM
}
$$
$2$-commutes.

2.) The set of geometric points $|X|$ has a topology
with open sets $|\cU|$ where $\cU \subseteq \cM$
is an open substack. When we write $|X|$ we will
mean this set with this topology. This is the
{\it geometric space}\index{geometric space, of a stack}
of the stack.
\end{defn}

The following result implies that the topology of $|\cM_\fg|$ is
quite simple.

\begin{prop}\label{open-ht}Let $\cU \subseteq \cM_\fg$ be
an open substack and suppose that $\cU$ has a geometric
point of height $n$. Then it has a geometric point of height
$k$ for all $k \leq n$.
\end{prop}

\begin{proof}Let $G:\Spec(k) \to U$ represent the geometric
point of height $n$ and $\Spec(L) \to \cM_\fg$ be the cover by
the Lazard ring. Then we have a $2$-commutative diagram
$$
\xymatrix{
&\cU \times_{\cM_\fg} \Spec(L) \rto^-j \dto^q& \Spec(L) \dto\\
\Spec(k) \ar[ur]^F  \rto_G& \cU \rto & \cM_\fg
}
$$
obtained by choosing a coordinate for $G$. The morphism $j$ is
open and the morphism $q$ is flat, as it is the pull-back of a
flat map. Choose an affine open $\Spec(R) \subseteq 
\cU \times_{\cM_\fg} \Spec(L)$ so that the morphism $F$
factors through $\Spec(R)$. Let $G_0$ be the resulting
formal group over $R$.

By localizing $R$ if necessary, we may assume that $R \to k$
is onto. Choose an element $w \in R$ which reduces to $v_n(G) \in k$.
Since $G$ has height $n$, $v_n(G) \ne 0$; thus, $w$ is not nilpotent.
By forming
$R[w^{-1}]$ if necessary, we may assume that $w$ is a
unit. From this we conclude that $\cI_{n+1}(G_0) = R$. Since
$\Spec(R) \to \cM_\fg$ is flat,  Proposition \ref{LEFT-easy-part}
implies that the ideals $\cI_{k}(G_0)$, $k \leq n+1$, is generated by a regular
sequence. (Note that $G_0$ has a canonical coordinate by
construction.)
Let $k \leq n$, $R_k = R/\cI_{k}(G_0)$, and  let $q_k:R \to R_k$ be
the quotient map.
We conclude immediately that $v_k(q_k^\ast G_0)$ is not
nilpotent in $R_k$. Choose a prime ideal $\mathfrak{p}$
in $R_k$ so that $v_k(q_k^\ast G_0) \ne 0$ in
$R/\mathfrak{p}$ and let $K$ be the field of fractions of
of $R/\mathfrak{p}$. Then 
$$
\Spec(K) \to \Spec(R) \to \cU
$$
represents a geometric point of height $k$.
\end{proof}

The importance of the closed substacks $\cM(n)$ is underlined by
the following result. Recall we are working at a prime, so
that $\cM_\fg = \ZZ_{(p)} \otimes \cM_\fg$.

\begin{thm}\label{ht-n-reduced}For all $n$, $1 \leq n \leq\infty$ the algebraic stack
$\cM(n)$ is reduced. Furthermore if $\cN \subseteq \cM_\fg$ is
any closed, reduced substack, then either $\cN = \cM_\fg$
or there is an $n$ so that
$$
\cM(n) = \cN.
$$\index{reduced substacks of $\cM_\fg$}
\end{thm}

Before proving this result, we need to recall what it means for
an algebraic stack to be reduced and how to produce the reduced
substack of a stack, assuming it exists.

Fix an $fpqc$-algebraic stack $\cM$. We define a diagram $\cC$ of closed
substacks of $\cM$ as follows:
\begin{enumerate}

\item An object of $\cC$ is a closed substack $\cN \subseteq \cM$
so that the induced inclusion on geometric points $|\cN| \to |\cM|$
is an isomorphism;

\item A morphism $\cN_1 \subseteq \cN_2$ is an inclusion of
closed substacks.
\end{enumerate}

This diagram $\cC$ of closed substacks is filtered; furthermore it
determines and is determined by a filtered (or cofiltered) diagram
$\{\cI_\cN\}$ of quasi-coherent ideals in $\cO_\cM$. Define
$$
\cI_\red = \colim_{\cC^{\mathrm{op}}}\cI_\cN.
$$
The colimit is taken pointwise and, since tensor products commute
the colimits, $\cI_\red$ is a quasi-coherent ideal. Let
$$
\cM_\red \subseteq \cM\index{substack, reduced}\index{reduced
substack}
$$
be the resulting closed substack. Note that $\cM_\red$ is the
initial closed substack $\cN \subseteq \cM$ so that $|\cN| = |\cM|$.
We say that $\cM$ is reduced if $\cM_\red = \cM$ or, equivalently,
if $\cI_\red = 0$.

The sheaf $\cI_\red$ should be closely related to the ideal of nilpotents
in $\cO_\cM$. Some care
is required here, however. If we define $\Nil_\cM(U) = \Nil_U$ for
any $fpqc$-morphism $U \to \cM$, the resulting ideal
sheaf may not be cartesian in the $fqpc$-topology; thus it is not
evidently quasi-coherent. (If $R\to S$ is a faithfully flat morphism
of rings, then it is not necessarily true that $\Nil_S = S \otimes_R
\Nil_R$.) However it is a sheaf in more restrictive
topologies, such as the ``smooth-\'etale'' used for the algebraic
stacks of \cite{Laumon}. 

\begin{defn}\label{nilpotent-sheaf}Let $\cM$ be an algebraic stack
in the $fpqc$-topology and suppose that $X \to \cM$ is an $fpqc$-presentation
so that
$$
p_1^\ast\Nil_X \cong \Nil_{X \times_\cM X} \cong p_2^\ast\Nil_X
$$
as ideal sheaves in $\cO_{X \times_\cM X}$. Then descent theory
yields a quasi-coherent ideal sheaf $\Nil_\cM \subseteq \cO_\cM$.
This is the {\bf sheaf of nilpotents} for $\cM$.\index{sheaf of nilpotents,
for a stack}
\end{defn}

\begin{rem}\label{rem-nil}1.)  It is not immediately clear that
$\Nil_\cM$ does not depend on the choice of cover $X \to \cM$;
however, this will follow from Proposition \ref{nil-is-red} to follow.

2.) If $\cM$ has a {\it smooth} cover,
then $\Nil_\cM$, when restricted to the smooth-\'etale topology,
agrees with the sheaf $\Nil_\cM$ as defined in $\cite{Laumon}$.

3.) In many of our standard examples, $\Nil_X = 0 = \Nil_{X \times_S X}$.
In particular, this applies to $\cM_\fg$ and $\cM(n)$, by Proposition
\ref{cover-for-mn}. 
\end{rem}

We need the following preliminary result before preceding.

\begin{lem}\label{space-of-closed}Let $\cM$ be an algebraic
stack and $\cN\subseteq \cM$ a closed substack. Let
$X \to \cM$ be an $fqpc$-cover. Then the natural map
$$
|X \times_\cM \cN| \longr |X| \times_{|\cM|} |\cN| 
$$
is an isomorphism.
\end{lem}

\begin{proof}This morphism is onto for a general pull-back; that is,
we don't need $\cN \to \cM$ to be a closed inclusion. To see that
is one-to-one, note that $X \times_\cM \cN$ is equivalent to 
closed subscheme $Y \subseteq X$ and that, hence, the
composite
$$
|Y| = |X \times_\cM \cN| \longr |X| \times_{|\cM|} |\cN|  \to |X|
$$
is an injection.
\end{proof}

\begin{prop}\label{nil-is-red} Suppose that $\cM$ is an algebraic
stack in the $fpqc$ topology and there is an $fpqc$-presentation
$X \to \cM$ so that $\Nil_\cM$ is defined. Then 
$$
\Nil_\cM = \cI_\red.
$$
\end{prop}

\begin{proof} Let $\cM_0 \subseteq \cM$ be
the closed substack defined by $\Nil_X$. Then $X_\red \to
\cM_0$ is a cover. Since $|X_\red| = |X|$ and $|X_\red| \to
|\cM_0|$ is surjective, we can conclude that $|\cM_0| = |\cM|$.
This shows that $\cM_\red \subseteq \cM_0$ or, equivalently,
that $\cI_\red \subseteq \Nil_\cM$.

For the other inclusion, let $\cN \subseteq \cM$ be a closed
inclusion defined by an ideal $\cJ$ and 
suppose $|\cN| = |\cM|$ and let $Y = \cN \times_\cM X \to \cN$
be the resulting cover. Then $Y$ is the closed subscheme of
$X$ defined by $\cJ|_X$ and the natural map
$$
|Y| \longr |X| \times_{|\cM|} |\cN|
$$
is an isomorphism, by Lemma \ref{space-of-closed}.
Thus, $|Y| = |X|$, which implies that $X_\red \subseteq Y$,
or $\Nil_X \subseteq \cJ|_X$. Since $\Nil_X = (\Nil_\cM)|_X$ and
$X$ is a cover $\cM$, this implies that $\Nil_\cM \subseteq \cJ$.
In particular, $\Nil_\cM \subseteq \cI_\red$. 
\end{proof}

\begin{cor}\label{stack-is-red} Suppose that $\cM$ is an algebraic
stack in the $fpqc$ topology and there is an $fpqc$-presentation
$X \to \cM$ so that  $X$ and $X \times_\cM X$ are reduced. Then
$\cM$ is reduced.
\end{cor}

We next begin an investigation of the closed substacks of $\cM_\fg$.
Recall that $\cM(1) = \cM_\fg \otimes \FF_p$.

\begin{prop}\label{closed-sub}Let $\cN \subseteq \cM_{\fg,p}$ be a closed
substack. If $\cN$ has a geometric point of height $n$, then
$$
\cM(n) \subseteq \cN.
$$
\end{prop}

\begin{proof}We begin with the following observation: suppose
that $\cN_1$ and $\cN_2$ are closed substacks of an algebraic
stack, that $\cN_1$ is reduced in the strong sense of 
Proposition \ref{nil-is-red}, and that $|\cN_1| \subseteq
|\cN_2|$. Then $\cN_1 \subseteq \cN_2$. For if
we let $X \to \cM$ be a cover and $Y_i \subseteq X$, $i=1,2$
the resulting closed subscheme which covers of $\cN_i$, then $Y_1$ is
reduced and Lemma \ref{space-of-closed} implies that
$|Y_1| \subseteq |Y_2|$. Then $Y_1 \subseteq Y_2$ and arguing
as the end of the proof of Proposition \ref{nil-is-red}, we
have $\cN_1 \subseteq \cN_2$.

To prove the result, then, we need only show that there is an $n$
so that $|\cM(n)| \subseteq |\cN|$. Thus we must prove that
if $\cN \subseteq \cM_\fg$ is closed and contains  a geometric
point of height $n$, then it contains a geometric point of 
height $k$ for all $k \geq n$. This can be rephrased in terms
of the complementary open $\cU = \cM_\fg - \cN$ as follows:
if $\cU$ does not have a geometric point of height $n$,
it does not have a geometric point of height $k$, $k \geq n$.
Rephrasing this as a positive statement gives
exactly Proposition \ref{open-ht}.
\end{proof}

\begin{thmred}Suppose $\cN \subseteq \cM_\fg$ is closed and reduced.
If $\cN \ne \cM_\fg$, then we have $\cN \subseteq \FF_p \otimes \cM_{\fg} = \cM(1)$.
Indeed, if $\cU = \cM_\fg - \cN$ is not empty, then it must contain
a geometric point of height $0$, by Proposition \ref{open-ht}.
Let $n$ be the smallest integer $1 \leq n \leq \infty$ so that $\cN$
has a geometric point of height $n$. Then Proposition \ref{closed-sub}
implies that $\cM(n) \subseteq \cN$. Furthermore $|\cM(n)| = |\cN|$.
The argument in the first paragraph of  Proposition \ref{closed-sub}
shows that $\cM(n) = \cN$.
\end{thmred}

\subsection{Isomorphisms and layers}

We continue to work at a prime $p$.
In this section we discuss the difference between the closed
substacks $\cM(n)$ and $\cM(n+1)$; that is, we 
discuss the geometry of
$$
\layer{n} \defeq \cM(n) - \cM(n+1).
$$
and the geometry of
$$
\cM(\infty) \defeq \cap_n \cM(n).
$$
In both case we will find that we have stacks of the form
$B\La$ where $\La$ is the group of automorphisms
of some height $n$ formal group law. The group $\La$
is not an algebraic group as it is not finite; however,
it will pro-\'etale in an appropriate sense. See Theorem
\ref{refined-lazard}.

Here is a preliminary result. 

\begin{lem}\label{layer-to-fg-affine}The inclusion
$$
f_n:\cH(n) \longr \cM_\fg
$$
is an affine morphism of algebraic stacks.
\end{lem}

\begin{proof}Suppose $\Spec(R) \to \cM_\fg$ is classifies a formal group
$G$ with a chosen coordinate $x$. Then the $2$-category
pull-back $\Spec(R) \times_{\cM_\fg} \cH(n)$ is the groupoid scheme
which assigns to each commutative ring $S$ the triples
$(f,\Ga,\phi)$ where $f:R \to S$ is a morphism of commutative
rings, $\Ga$ is formal group of exact height $n$ over $S$
and $\phi:\Ga \to f^\ast G$ is an isomorphism of formal groups.
An isomorphism of triples $(f,\Ga,\phi) \to (f,\Ga',\phi')$ is
an isomorphism of formal groups $\psi:\Ga \to \Ga'$ so that
$\phi'\psi=\phi$. Given such a triple, $(f,\Ga,\phi)$, the
existence of $\phi$ forces $f$ to factor as a composition
$$
\xymatrix{
R \rto^-q & R/\cI_n(G))[u_n^{-1}] \rto^-g & S
}
$$
where $[p]_G(x) = u_nx^{p^n} + \cdots$ modulo $\cI_n(G)$.
We now check that the morphism of groupoid schemes
$$
\Spec(R/\cI_n(G))[u_n^{-1}]) \to \Spec(R) \times_{\cM_\fg} \cH(n)
$$
sending $g$ to $(gq,(gq)^\ast G,1)$ is an equivalence. For
more general $G$, we use faithfully flat descent to describe
the pullback as an affine scheme. 
\end{proof}

\begin{rem}\label{first-cover-layer}
From this result and Proposition \ref{cover-for-mn} we have
that there is an $fpqc$-cover
$$
Y(n) \defeq \Spec(\FF_p[u_n^{\pm 1},u_{n+1},u_{n+2},\ldots]) \to \layer{n}.
$$
and that
$$
Y(n) \times_{\layer{n}} Y(n) \cong
\Spec(\FF_p[u_n^{\pm 1},u_{n+1},u_{n+2},\ldots][t_0^{\pm 1},t_1,t_2,\ldots]).
$$
\end{rem}

Now let $S$ be a scheme and let $G_1$ and $G_2$ be
two formal groups over $S$. Define the scheme of isomorphisms from
$G_1$ to $G_2$ by the $2$-category pull-back
$$
\xymatrix{
\Iso_{S}(G_1,G_2) \rto \dto & \cM_\fg \dto^\Delta\\
S \rto_-{G_1 \times G_2} \rto &\cM_\fg \times \cM_\fg.
}
$$
Thus if $f:T \to S$ is a morphism of schemes, then a $T$-point of
$\Iso_S(G_1,G_2)$ is an isomorphism $\phi:f^\ast G_1 \to f^\ast G_2$
of formal groups over $T$. By Proposition \ref{diagonal},
$\Iso_S(G_1,G_2)$ is affine over $S$.

If $G_3$ is another formal group over $S$, then there is a composition
$$
\Iso_{S}(G_2,G_3) \times_S \Iso_{S}(G_1,G_2) \longr \Iso_{S}(G_1,G_3).
$$
In particular, $\Aut_S(G_1) = \Iso_{S}(G_1,G_1)$ acts on the right
on $\Iso_S(G_1,G_2)$. 

Because isomorphisms are locally given by power series, it is fairly clear
that $\Iso_S(G_1,G_2) \to S$ does not have good finiteness properties.
To get well-behaved approximations, let $\cM_\fg\pnty{p^k}$ denote the
moduli stack of $p^k$-buds (Definition \ref{buds-stack}) and define
$\Iso_{S}(G_1,G_2)_k$ by the pull-back diagram
$$
\xymatrix{
\Iso_{S}(G_1,G_2)_k \rto \dto & \cM_\fg\pnty{p^k} \dto^\Delta\\
S \rto_-{G_1 \times G_2} \rto &\cM_\fg\pnty{p^k} \times \cM_\fg\pnty{p^k}.
}
$$
Thus, for $f:T \to S$, a $T$-point of $\Iso_{S}(G_1,G_2)_k$ is an isomorphism
of the $p^k$-buds $\phi:(G_1)_{p^k} \to (G_2)_{p^k}$. 

Let $\Iso_{S}(G_1,G_2)_\infty = \Iso_{S}(G_1,G_2)$
and let $\Iso_{S}(G_1,G_2)_0 = S$; then there is a tower under $\Iso_S(G_1,G_2)$
and over $S$ with transition morphisms
$$
\Iso_S(G_1,G_2)_k \longr \Iso_S(G_1,G_2)_{k-1}.
$$
Pointwise, these maps are fibrations, so we have that 
$$
\Iso_S(G_1,G_2) \to \holim \Iso_S(G_1,G_2)_k
$$
is an equivalence.

The following is a refined version of Lazard's uniqueness theorem. See
Corollary \ref{lazard-unique} below. 

\begin{thm}\label{refined-lazard}Let $S$ be a scheme over $\FF_p$ and
let $G_1$ and $G_2$ be two formal groups of strict height $n$, $1 \leq n <
\infty$ over $S$. Then
$$
\Iso_{S}(G_1,G_2)_1 \longr S
$$
is surjective and \'etale of degree $p^n-1$. For all $k > 1$, the morphism
$$
\Iso_{S}(G_1,G_2)_k \longr \Iso_{S}(G_1,G_2)_{k-1}
$$
is surjective and \'etale of degree $p^n$. Finally, the morphism
$$
\Iso_{S}(G_1,G_2) \longr S
$$
is surjective and pro-\'etale.
\end{thm}

The proof is below in \ref{lazard-proof}.

\begin{cor}[{\bf Lazard's Uniqueness Theorem}]\label{lazard-unique}Let $\FF$ be a field of characteristic\index{Lazard's uniqueness theorem}
$p$ and $G_1$ and $G_2$ two formal groups of strict height $n$. Then there
is a separable extension $f:\FF \to \EE$ so that $f^\ast G_1$ and $f^\ast G_2$
are isomorphic. In particular, if $\FF$ is separably closed, then $G_1$ and
$G_2$ are isomorphic.
\end{cor}

\begin{proof} If the height $n < \infty$, this follows from the surjectivity statement
of Theorem \ref{refined-lazard}. If $n=\infty$, then the $p$-series of $G_i$
must be zero; hence, a choice of $p$-typical coordinate for $G_i$ defines 
an isomorphism from $G_i$ to the additive formal
group.
\end{proof}

\begin{rem}\label{flat-and-iso}If $G_1$ and $G_2$ are two formal
groups over a scheme $S$ classified by maps $G_i:S \to \cM_\fg$, we have
a pull-back diagram
$$
\xymatrix{
\Iso_S(G_1,G_2) \rto \dto & S \times_{\cM_\fg} S\dto\\
S \rto_\Delta & S \times S.
}
$$
If $S = \Spec(A)$ is affine and each of the formal groups $G_i$ can
be given a coordinate, then this writes (by Lemma \ref{pull-back-for-coord-2}) $\Iso_S(G_1,G_2)$
as the spectrum of the ring
$$
A {\otimes_{(A \otimes A)}} A \otimes_L W \otimes_S A.
$$
Thus if $x \in A$ we have (using the standard notation for Hopf algebroids)
$$
x = \eta_R(x)
$$
in this commutative ring. This makes it very unusual that
$$
\Iso_S(G_1,G_2) \longr S
$$
is flat, let alone \'etale. Thus the Theorem \ref{refined-lazard} is something
of a surprise.
\end{rem}

\begin{exam}\label{lazard-refined-calc} We can be very concrete
about the scheme $\Aut_\FF(\Ga_n)$ where 
$\Ga_n$ is of strict height $n$ over over a field $\FF$. The formal
group $\Ga_n$  can be given a coordinate
and we can display $\Aut_\FF(\Ga_n)$ as
$$
\Aut_\FF(\Ga_n) = \Spec(\FF \otimes_L W \otimes_L \FF)
$$
where $L \to \FF$ classifies $\Ga_n$ with a choice of
coordinate. For example, if $1 \leq n < \infty$ and if $\Ga_n$
is the Honda formal group over $\FF_p$ with coordinate so that
$[p](x) = x^{p^n}$, then we have
an isomorphism of Hopf algebras
\begin{equation}\label{ht-n-hopf}
\FF_p \otimes_L W \otimes_L \FF_p =
\FF_p[a_0^{\pm 1},a_1,a_2,\cdots]/(a_i^{p^n} - a_i).
\end{equation}
This is the Hopf algebra analyzed by Ravenel in Chapter 6
of \cite{Rav}, where it is called the Morava stabilizer algebra
\index{Morava stabilizer algebra}. The automorphisms of the the $p^k$ buds
are displayed as
$$
\Aut_\FF(\Ga_n)_{p^k} =  \Spec(\FF_p[a_0^{\pm 1},a_1,\cdots,a_k]/(a_i^{p^n} - a_i).
$$
In the infinite height case, the failure to be \'etale can be easily seen:
if we take $\Ga_\infty = \hat{\GG}_a$ with its standard coordinate, then
\begin{equation}\label{ht-infty-hopf}
\FF_p \otimes_L W \otimes_L \FF_p =
\FF_p[a_0^{\pm 1},a_1,a_2,\cdots].
\end{equation}
This is closely related to the mod $p$ dual Steenrod algebra.
\end{exam}

\begin{thmlaz}\label{lazard-proof} This argument is a rephrasing of an argument I learned from  Neil Strickland \cite{fpfp}. But see also \cite{Katz}.
The properties listed -- \'etale, degree, and surjectivity -- are all local
in the $fpqc$-topology on $S$; thus we may assume that $S =
\Spec(A)$ is affine and that $G_1$ and $G_2$ can be given a simultaneous coordinate
$x$. Furthermore,
since all of these conditions remain invariant under isomorphisms of the formal groups
involved, we may assume that $G_1$ and $G_2$ are $p$-typical. This implies
that we may write the $p$-series of of the formal groups
\begin{align*}
[p]_{G_1}(x) &= u_nx^{p^n} +_{G_1} u_{n+1}x^{p^{n+1}} + \cdots\\
[p]_{G_2}(x) &= u'_nx^{p^n} +_{G_2} u'_{n+1}x^{p^{n+1}} + \cdots
\end{align*}
and, hence, that the Verschiebungs may be written
\begin{align*}
V_{G_1}(x) &= u_nx +_{G_1} u_{n+1}x^{p} + \cdots\\
V_{G_2}(x) &= u'_nx +_{G_2} u'_{n+1}x^{p} + \cdots.
\end{align*}
Because the formal groups have strict height $n$, $u_n$ and $u'_n$ are units.

First assume $k=0$. Then an isomorphism $\phi:(G_1)_p \to (G_2)_p$ of
$p$-typical formal group buds can be written $\phi(x) = b_0x$ modulo $(x^p)$.
Since $V_{G_2}(\phi^{(p^n)}(x)) = \phi(V_{G_1}(x))$ we have
$u'_nb_0^{p^n}x = u_nb_0x$. Since $b_0$, $u_n$, and $u'_n$ are all units
we get an equation
\begin{equation}\label{kiszero}
b_0^{p^n-1} - v = 0
\end{equation}
where $v=u_n/u'_n$ is a unit. Thus,
$$
\Iso_S(G_1,G_2)_1 = \Spec(A[b_0]/(b_0^{p^n-1} - v)).
$$
This is \'etale of degree $p^n-1$ over $\Spec(A)$ since $b_0$ is a unit
in $A[b_0]/(b_0^{p^n-1} - v)$.
Surjectivity follows from  the fact that $A \to A[b_0]/(b_0^{p^n-1} - v)$
is faithfully flat. 

Now assume $k > 0$ and keep the notation above. We make the inductive assumption
that $\Iso_S(G_1,G_2)_{k-1} = \Spec(A_{k-1})$ for some $A$-algebra $A_{k-1}$.
Suppose we have an isomorphism
$$
\phi_0(x):(G_1)_{p^{k-1}} \to (G_2)_{p^{k-1}}
$$
of $p^{k-1}$-buds over some $A$-algebra $R$. We want to lift this
to an isomorphism
$$
\phi:(G_1)_{p^k} \longr (G_2)_{p^k}
$$
so that $\phi \equiv \phi_0$ as isomorphisms of $(G_1)_{p^{k-1}}$ to
$(G_2)_{p^{k-1}}$. We may write $\phi(x) = \phi_0(x) +_{G_1} b_kx^{p^k}$.
Then again we must have
$$
V_{G_2}(\phi^{(p^n)}(x)) = \phi(V_{G_1}(x))
$$
and, equating the coefficients of $x^{p^k}$ we get an equation 
\begin{equation}\label{kisnotzero}
b_k^{p^n}  -vb_k+w =0
\end{equation}
where $v=u_n^{p^k}/u'_n$ is a unit. Thus,
$$
\Iso_S(G_1,G_2)_k = \Spec(A_{k-1}[b_k]/(b_k^{p^n} - vb_k+w)).
$$
This is faithfully flat, \'etale and of degree $p^n$ over $\Spec(A_{k-1})$.
\hfill$\square$
\end{thmlaz}

The projection morphism $\Iso_S(G_1,G_2)_k \to S$ has a right action by the
\'etale group scheme $\Aut_S(G_1)_k \to S$, $1 \leq k \leq \infty$. The
action induces a diagram of schemes over $\Iso_S(G_1,G_2)$
$$
\xymatrix{
\Iso_S(G_1,G_2)_k \times_S \Aut_S(G_1)_k \rto \dto_{p_1}&
\Iso_S(G_1,G_2)_k \times_S \Iso_S(G_1,G_2)_k  \dto^{p_1}\\
\Iso_S(G_1,G_2)_k  \rto_{=}& \Iso_S(G_1,G_2)_k 
}
$$
where the top map is given pointwise by
$$
(\phi,\psi) \mapsto (\phi,\phi\psi).
$$
This map is evidently an isomorphism; hence we have proven the following
result.

\begin{prop}\label{torsor-for-iso}The morphism $\Iso_S(G_1,G_2)_k \to S$
is an $\Aut_S(G_1)_k$-torsor.
\end{prop}

We can specialize this result even further, but fir         st some
notation and definitions.

\begin{rem}\label{group-and-galois}If $X$ is a finite set, define
$X_\ZZ$ to be the scheme $\Spec(\map(X,\ZZ))$.
Then for any category $Y$ fibered in groupoids over affine schemes
we get a new functor $Y \times X_\ZZ $. If $Y = \Spec(R)$, then
$$
Y \times X_\ZZ = \Spec(\map(X,R)) \defeq X_R.
$$
If $G$ is a finite group, the $\GG_\ZZ$ is a finite group scheme over $\ZZ$
and the action of $G$ on itself extends to a right action on $Y \times G_\ZZ$.
\index{group scheme, defined by a group}

If $X = \lim X_k$ is a profinite set, define
$$
X_Z = \lim (X_k)_\ZZ = \Spec(\colim \map(X_k,\ZZ)).
$$
If $G = \lim G_k$ is a profinite group, then $G_\ZZ$ is a profinite group scheme
over $\ZZ$. \index{group scheme, defined by a profinite group}

The notation $G_\ZZ$ is cumbersome; we will drop it if $G$ is evidently
a profinite group.

Now suppose $X \to S$ is a finite and \'etale morphism of schemes;
let $\Aut_S(X)$ denote the automorphisms of $X$ over $S$. This is
finite group. Then $X$ is {\it Galois}\index{Galois morphism}\index{morphism,
Galois} over $S$ if the natural  map
$$
\Aut_S(X) \times_S X \longr X \times_S X
$$
given pointwise by $(\phi,x) \mapsto (x,\phi(x))$ is an isomorphism.
If $X = \lim X_k\to S$ where $\{ X_k \}$ is a tower of finite
and \'etale maps over $S$, then $X$ is {\it  pro-Galois} if there there is a coherent
set of morphisms $\Aut_S(X) \to \Aut_S(X_k)$ so that
$$
\Aut_S(X) = \lim \Aut_S(X_k)
$$
and each of the morphism $X_k \to S$ is Galois.
\end{rem}

\begin{rem}\label{finite-levels-of-morava} Suppose that $\Ga$ is a formal group
of height $n$ over a separably closed field $\FF$ and let $\GG_k(\Ga)$ be the
$\FF$-points of $\Aut_\FF(\Ga)_k$. If $k < \infty$, then $\GG_k(\Ga)$ has order
$p^nk-1$. The equations \ref{kiszero} and
\ref{kisnotzero} imply that the natural map
$$
\GG_k(\Ga)_\FF \longr \Aut_\FF(\Ga)_k,\quad k < \infty
$$
is an isomorphism. Furthermore,
$$
\GG(\Ga) \defeq \GG_\infty(\Ga) \cong \lim \GG_k(\Ga).
$$
This displays $\GG(\Ga)$ as a profinite group. Note that the equations
\ref{kiszero} and \ref{kisnotzero} also imply that
$$
\GG_1(\Ga) \cong \FF_{p^n}^\times\quad\mathrm{and}\quad
\GG_k(\Ga)/\GG_{k-1}(\Ga) \cong \FF_{p^n}.
$$
\end{rem}

\begin{thm}\label{refined-lazard-redux}Let $S$ be a scheme over a separably
closed field $\FF$ and
let $G_1$ and $G_2$ be two formal groups of strict height $n$, $1 \leq n <
\infty$ over $S$. Suppose that $G_1$ obtained by base change from a formal
group $\Ga$ of height $n$ over $\bar{\FF}_p$. Then
for all $k < \infty$, the morphism
$$
\Iso_{S}(G_1,G_2)_k \longr S
$$
is Galois with Galois group $\GG_k(\Ga)$. Finally, the morphism
$$
\Iso_{S}(G_1,G_2) \longr S
$$
is pro-Galois with Galois group $\GG(\Ga)$.
\end{thm}

\begin{proof}Let $f:T \to S$ be any morphism of schemes. Then
$$
T\times_S\Iso(G_1,G_2)_k \cong \Iso_T(f^\ast G_1,f^\ast G_2)_k.
$$
In particular
$$
\Aut_S(G_1)_k \cong S \times_{\Spec(\FF)} \Aut_\FF(\Ga)_k \cong \GG_k(\Ga)_S
$$
and the result now follows from Proposition \ref{torsor-for-iso}.
\end{proof}

The \'etale extensions we produced in the proof of Theorem \ref{refined-lazard} were of a very particular type. See Equations \ref{kiszero} and \ref{kisnotzero}.
This can be rephrased Proposition \ref{rel-frob} below, which can be proved by examining
the proof just given. Here, however, we give a more conceptual proof, based on the following
observation.

 If $R$ is an $\FF_p$-algebra, let us write $f_R:R \to R$ for the Frobenius
 homomorphism sending $x$ to $x^p$. Then $\cM$ is any stack over $\Spec(\FF_p)$,
 we get a Frobenius homomorphism
 $$
 f_\cM:\cM \longr \cM
 $$
 of stacks over $\Spec(\FF_p)$ which, upon evaluating at an $\FF_p$ algebra
 $R$ is given by
 $$
 f_\cM = \cM(f_R):\cM(R) \longr \cM(R).
 $$
 For example if $\cM(1) = \cM_\fg \otimes \Spec(\FF_p)$ is the moduli stack of
 formal group over schemes over $\FF_p$ then
 $$
 f_\cM(G \to S) = G^{(p)} \to S.
 $$

\begin{rem}[{\bf The Frobenius trick}]\label{frob-trick}
\index{Frobenius trick} Let $\layer{n}$ be
the moduli stack of formal groups of exact height $n$, with $1 \leq n < \infty$. For
all formal groups $G$ of exact height $n$ the {\it natural} factoring of the morphism
$[p]:G \to G$ in Definition \ref{height} yields a natural isomorphism
$$
V_n = V_n^G:G^{(p^n)} \longr G.
$$
Thus, if $f_\cN:\layer{n} \longr\layer{n}$ is the Frobenius -- which, as we have just
seen, assigns to each $G \to S$ the formal group $G^{(p)} \to S$ --
we get a natural transformation
$$
V_n:f_\cN^n \longr 1
$$
from $f_\cN^n$ to the identity of $\layer{n}$.
\end{rem}

\begin{prop}\label{rel-frob}Let $0 \leq k_1 \leq k_2 \leq \infty$. Then the
relative Frobenius
$$
\xymatrix@C=0pt{
\Iso_{S}(G_1,G_2)_{k_2} \ar[dr] \ar[rr]^F && \Iso_{S}(G_1,G_2)_{k_2}^{(p)}\ar[dl]\\
&\Iso_{S}(G_1,G_2)_{k_1}
}
$$
is an isomorphism.
\end{prop}

\begin{proof} We will first do the absolute case of $\Iso_S(G_1,G_2) \to S$ --
that is, $k_1 =0$ and $k_2 = \infty$ -- and indicate at the end of the argument what
changes are needed in general.

The scheme $\Iso_S(G_1,G_2)^{(p)}$ over $S$ assigns to each morphism
$f:T \to S$ of schemes the set of isomorphisms
$$
\psi:f^\ast G_1^{(p)} \longr f^\ast G_2^{(p)}.
$$
The relative Frobenius  $F:\Iso_S(G_1,G_2) \to \Iso_S(G_1,G_2)^{(p)}$ over $S$
sends a $T$-point $\phi:f^\ast G_1 \to f^\ast G_2$ to the $T$-point
$\phi^{(p)}:f^\ast G_1^{(p)} \to G_2^{(p)}$. It is this we must show is an isomorphism.
However, if we are given $\psi$, may produce $\phi$ using the following commutative
diagram of isomorphisms
$$
\xymatrix@C=40pt{
f^\ast G_1^{(p^n)} \rto^-{\psi^{(p^{n-1})}} \dto_{V_{G_1}} &
f^\ast G_2^{(p^n)} \dto^{V_{G_2}}\\
f^\ast G_1 \rto_-{\phi} & f^\ast G_2.
}
$$
As $V_{G^{(p)}} = (V_G)^{(p)}$ we may conclude that $\phi^{(p)} = \psi$ from the diagram
$$
\xymatrix@C=40pt{
f^\ast G_1^{(p)} \dto_{\phi^{(p)}} & \lto_{V_{G_1}^{(p)}} f^\ast G_1^{(p^{n+1})} 
\dto^{\psi^{(p^n)}} \rto^{V_{G_1}^{(p)}} & f^\ast G_1^{(p)} \dto^{\psi}\\
f^\ast G_2^{(p)} & \lto^{V_{G_2}^{(p)}} f^\ast G_2^{(p^{n+1})} 
 \rto_{V_{G_2}^{(p)}} & f^\ast G_2^{(p)}.
}
$$

The same proof works in the relative case, but the notation is more complicated. The
scheme
$$
\Iso_S(G_1,G_2)^{(p)}_{k_2} \longr \Iso_S(G_1,G_2)_{k_1}
$$
consists of pairs $(\psi_1,\psi_2)$ where $\psi_1$ and $\psi_2$ are isomorphisms on buds
$$
\psi_2:f^\ast (G_1)^{(p)}_{p^{k_2}} \longr f^\ast (G_2)^{(p)}_{p^{k_2}}
$$
and
$$
\psi_1:f^\ast (G_1)_{p^{k_1}} \longr f^\ast (G_2)_{p^{k_1}}.
$$
so that 
$$
\psi_2 \equiv \psi_1^{(p)} :f^\ast (G_1)^{(p)}_{p^{k_1}} \longr f^\ast (G_2)^{(p)}_{p^{k_1}}.
$$
The relative Frobenius sends
$$
\phi:f^\ast (G_1)_{p^{k_2}} \longr f^\ast (G_2)_{p^{k_2}}
$$
to the pair $(\bar{\phi},\phi^{(p)})$ where $\bar{\phi}$ is the reduction of $\phi$.
The argument given in the absolute case now adapts to show this is an isomorphism.
\end{proof}

\begin{prop}\label{ht-n-points}Let $\cN(n) = \layer{n}$ if $1 \leq n < \infty$ and let $\cN(\infty) = \cM(\infty)$. Let
$$
g:\Spec(A) \to \cN(n),\qquad 1 \leq n \leq \infty
$$
be any morphism. Then $g$ is an $fqpc$-cover. In particular, $g$ is
faithfully flat.
\end{prop}

\begin{proof} If $g_i:\Spec(A_i)
\to \layer{n}$ be any two maps. Then, by Theorem \ref{refined-lazard}
there is a faithfully flat extension $A_1 \otimes A_2 \to B$ so that
the two formal groups over the tensor product become isomorphic
over $B$. Thus, we have a $2$-commutative diagram
$$
\xymatrix{
\Spec(B) \rto^{f_2} \dto_{f_1} & \Spec(A_2) \dto^{g_2}\\
\Spec(A_1) \rto_{g_1} & \layer{n}
}
$$
where both $f_1$ and $f_2$ are faithfully flat.

Now take $g_1$ to be faithfully flat and $g_2$ to be
arbitrary. Then $g_1f_1$ is faithfully flat and, since $f_1$
is faithfully flat, $g_2$ must be faithfully flat as well.
Since
$$
\Spec(A_1) \otimes_{\cM_\fg} \Spec(A_2) \to \Spec(A_1)
$$
is affine, it follows by descent that
$\Spec(A_2) \to \cM(n)[v_n^{\pm 1}]$
is affine as well. In particular, it is quasi-compact.
\end{proof}

The following is now immediate. We will almost always take
$\FF$ to be an algebraic extension of $\FF_p$. 

\begin{cor}\label{field-cover}Let $\cN(n) = \layer{n}$ if $1 \leq n < \infty$ and let $\cN(\infty) = \cM(\infty)$. Let $\FF$ be a field of
characteristic $p$ and $G \to \Spec(\FF)$ any formal
group of height $n$, $1 \leq n \leq \infty$. Then the classifying map for
$G$
$$
\Spec(\FF) \longr \cN(n)
$$
is a cover in the $fpqc$-topology. In particular,
$\layer{n}$ and $\cM(\infty)$ each has a single geometric point.
\end{cor}

Now fix a formal group $\Gamma_n$ over $\FF_p$ of height
$n$; for example, the Honda formal group. If $n = \infty$ we may
as well fix $\Ga_n = \hat{\GG}_a$, the completion of the
additive group. Define $\Aut(\Ga_n)$
to be the group scheme which assigns to each $\FF_p$-algebra
$i:\FF_p \to A$ the automorphisms of the formal group
$i^\ast \Ga_n$ over $A$. See Example \ref{lazard-refined-calc} for
a concrete discussion.

\begin{thm}\label{diff-ht}Let $\layer{n} = \cM(n) - \cM(n+1)$ be the 
open substack of $\cM(n)$ complementary to $\cM(n+1)$. Then
$\layer{n}$ has a single geometric point represented by any formal
group $\Gamma_n$ of height $n$ over ${\FF}_p$.
Furthermore, the map
$$
\layer{n} \longr B\Aut(\Gamma_n)
$$
sending a formal group $G$ of height $n$ over an $\FF_p$-algebra
$A$ to the $\Aut(\Ga_n)$-torsor $\Iso(\Ga_n,G)$ is
an equivalence of stacks.\index{$\layer{n}$, as classifying stack}
\end{thm}

\begin{thm}\label{diff-ht-infty}Let $\cM(\infty) = \cap \cM(n)$. Then
$\cM(\infty)$ has a single geometric point represented by the
formal additive group  $\hat{\GG}_a$  over ${\FF}_p$.
Furthermore, the map
$$
\cM(\infty) \longr B\Aut(\hat{\GG}_a)
$$
sending a formal group $G$ of infinite height over an $\FF_p$-algebra
$A$ to the $\Aut(\hat{\GG}_a)$-torsor $\Iso(\hat{\GG}_a,G)$
is an equivalence of stacks.
\end{thm}

The apparent choice of the formal group $\Ga_n$ makes this result
a bit puzzling. This can be rectified by coming to terms with the
notion of a gerbe. Here we appeal to \cite{Laumon} \S 3.15ff.

\begin{defn}\label{gerbe}1.) Let $S$ be a scheme and $X \to S$
a scheme over $S$. Then a {\bf gerbe} over $X$ is a stack
$q:\cG \to X$ over $X$ with the properties that\index{gerbe}
\begin{enumerate}

\item[i.)] for all affine $U \to S$ and all pairs of morphisms
$x_1,x_2:U \to \cG$ so that $qx_1 = qx_2:U \to X$,
there is an faithfully flat covering $f:V \to U$ by an affine
so that there is an isomorphism $f^\ast x_1 \cong f^\ast x_2$;

\item[ii.)] for all affine $U \to S$ and all $f:U \to X$ over $S$,
 there is an faithfully flat covering $f:V \to U$ by an affine
so that there is a morphism $x:V \to \cG$ with $qx = fx$.
\end{enumerate}

\noindent 2.) A gerbe $q:\cG \to X$ is {\bf neutral} if there is a section
$s:X \to \cG$ of $q$.\index{gerbe, neutral}
\end{defn}

The following is exactly Lemma 3.21 of \cite{Laumon}
\begin{lem}\label{neutral-gerbe}Suppose $q:\cG \to X$ is a
neutral gerbe over $X$. Then a section $s$ of $q$ determines
an equivalence of stacks over $X$
$$
\cG \simeq B\Aut_\cG(s/X)
$$
where $\Aut(s/X)$ is the group scheme which assigns to each
$f:U \to X$ the group $\Aut_\cG(f^\ast x \to U)$. This equivalence
sends $g \in \cG(U)$ to the torsor $\Iso(sq(g),g)$.
\end{lem}

\begin{thmht}Let $\cN(n) = \layer{n}$ if $1 \leq n < \infty$ and let $\cN(\infty) = \cM(\infty)$. We claim that $\cN(n) \to \Spec(\FF_p)$ is
a neutral gerbe. Then the result follows from Lemma \ref{neutral-gerbe}.
The two conditions to be gerbe are easily satisfied in this case:
(1) any two height $n$ formal groups over an $\FF_p$-algebra
$A$ become isomorphic after a faithfully flat extension and
(2) every $\FF_p$ algebra $A$ has a height $n$ formal group
over it. Finally the choice of $\Ga_n$ shows that we have a
neutral gerbe.
\end{thmht}

\begin{rem}\label{morava-stab}The {\it Morava stabilizer group}
$\SS_n$ is defined to be the $\bar{\FF}_p$ points of the
algebraic group $\Aut(\Ga_n)$; that is, if $:i:\FF_p \to \bar{\FF}_p$
is the inclusion, then $\SS_n$ is the automorphisms over $\bar{\FF}_p$
of the formal group $i^\ast\Ga_n$. By Theorem \ref{lazard-unique},
this is independent of the choice of $\Ga_n$. The {\it big Morava
stabilizer group} $\GG_n$ is the group of $2$-commuting diagrams
$$
\xymatrix@C=5pt{
\Spec(\bar{\FF}_p) \ar[rr] \ar[dr]_{i^\ast\Ga_n}
&& \Spec(\bar{\FF}_p)\ar[dl]^{i^\ast\Ga_n}\\
&\cM_\fg.
}
$$
There is a semi-direct product decomposition
$$
\GG_n \cong \mathrm{Gal}(\bar{\FF}_p/\FF_p) \rtimes \SS_n.
$$
\end{rem}

%% file: land.tex
\section{Localizing sheaves at a height $n$ point}

In this section we define and discuss the sheaves $\cF[v_n^{-1}]$
when $\cF$ is an $\cI_n$-torsion quasi-coherent sheaf on
the moduli stack $\cM_\fg$ of formal groups. This is largely
groundwork for the results on chromatic convergence to be
proved later, but we do revisit the Landweber Exact Functor
Theorem here, using a proof due to Mike Hopkins \cite{hop-notes}.
We begin with a discussion of the derived tensor product
and derived completions, which -- by results of Hovey
\cite{hovey-comod} --
have a particularly nice form for the stacks under consideration
here.

\subsection{Derived tensor products and derived completion}

We will want to control the derived tensor product of two
quasi-coherent sheaves on an algebraic stack
$\cM$. While this works particularly well if 
$\cM$ is quasi-compact and separated, for the stacks
encountered in homotopy theory we can do even better: by using
results of Mark Hovey, it is
possible to give completely functorial construction using
resolutions by locally free sheaves.
This is because we will be able to assume
that the category $\Qmod_\cM$ of quasi-coherent
sheaves is generated by the finitely generated locally
free sheaves on $\cM$. There is no reason to expect
this assumption to hold in great generality, of course, but
it holds  when $\cM$ is one the stacks that arises in the chromatic
picture. We will discuss this below in Proposition \ref{adams-generates}.

The following a restatement of some of the results
of \cite{hovey-comod} \S 2, especially Theorem 2.13 of that
paper. Indeed, that result is a statement about the cofibrant
objects in a model category structure on chain complexes
of quasi-coherent sheaves. Weak equivalences can
be deduced from point (2) of the next statement.

\begin{prop}\label{hovey-results}Let $\cM$ be an algebraic
stack so that the finitely generated locally free sheaves generate
the category of $\Qmod_\cM$ of quasi-coherent sheaves
on $\cM$. Then for any chain complex of
quasi-coherent sheaves $\cF$ on $\cM$ there is
a natural quasi-isomorphism of chain complexes
$$
\cP^\cM_\bullet(\cF) = \cP_\bullet \to \cF_\bullet
$$
with the properties that
\begin{enumerate}

\item for all $n$, the sheaf $\cP_n$ is a coproduct of
finitely generated locally free sheaves;

\item  for all finitely generated locally free
sheaves $\cV$ on $\cM$ the morphism of function complexes
$$
\Hom(\cV,\cP_\bullet) \longr \Hom(\cV,\cF_\bullet)
$$
is a quasi-isomorphim.
\end{enumerate}
\end{prop}

Let $\cF$ be a sheaf on $\cM$. In much of the sequel we 
will write $\cF(R)$ for $\cF(\Spec(R) \to \cM)$. We note
that the tensor product of quasi-coherent sheaves behaves
particularly well.

\begin{lem}\label{tens-is-sheaf}Let $\cE$ and $\cF$ be two
quasi-coherent sheaves in the $fpqc$-topology on an algebraic stack
$\cM$. Then presheaf $\cE \otimes \cF = \cE \otimes_{\cO}\cF$
which assigns to each flat and quasi -compact
morphism $\Spec(R) \to \cM$ the tensor product
$$
\cE(R) \otimes_R \cF(R)
$$
is a quasi-coherent sheaf.
\end{lem}

\begin{proof}To see that we actually have a sheaf, let $R \to S$
be faithfully flat extension. Then
$$
\xymatrix{
\cE(R) \otimes_R \cF(R) \rto & \cE(S) \otimes_S \cF(S) 
\ar@<.5ex>[r] \ar@<-.5ex>[r] & \cE(S\otimes_R S) \otimes_{(S\otimes_R S)}
 \cF(S \otimes_R S)
 }
 $$
 is, up to isomorphism,
 $$
 \xymatrix@C=15pt{
\cE(R) \otimes_R \cF(R) \rto & S \otimes_R(\cE(R) \otimes_R \cF(R)) 
\ar@<.5ex>[r] \ar@<-.5ex>[r] & S\otimes_R S \otimes_R (\cE(R) \otimes_{R} \cF(R)).
 }
$$
If we apply $S \otimes _R (-)$ to the later sequence it becomes
exact, as it has a retraction. Since $R \to S$ is faithfully flat,
it was exact to begin with. This proves we have a sheaf; it is quasi-coherent
because (as we have already noted)
$$
S \otimes_R \cE(R) \otimes_R \cF(R) \mathop{\longr}^{\cong} \cE(S) \otimes_S \cF(S).
$$
\end{proof}
 
\begin{defn}\label{tensor-prod}\index{tensor product, derived} Suppose $\cM$ is an algebraic
stack in the $fqpc$-topology so that the finitely generated
locally free sheaves generated $\Qmod_\cM$.
Let $\cE$ and $\cF$ be two quasi-coherent sheaves on
$\cM$. Define their {\bf derived tensor product}
$\cE \otimes^L \cF =\cE \otimes^L_{\cO_\cM} \cF$ to be the chain complex of
quasi-coherent sheaves (for the $fpqc$-topology) with
values at $\Spec(R) \to \cM$ given by
$$
\cE(R) \otimes_R \cP_\bullet(R)
$$
where $\cP_\bullet \to \cF$ is the natural resolution of 
Proposition \ref{hovey-results}. 
\end{defn}

Many of the usual properties of tensor product apply to this
construction. For example, if
$$
0 \to \cE_{1} \to \cE_{2} \to \cE_{3} \to 0
$$
is a short exact sequence of quasi-coherent sheaves,
then we get a distinguished triangle in the derived
category of quasi-coherent sheaves
$$
\cE_{1} \otimes^L \cF 
\to \cE_{2} \otimes^L_{\cO_\cM} \cF
\to \cE_{3} \otimes^L_{\cO_\cM} \cF
\to (\cE_{1} \otimes^L_{\cO_\cM} \cF)[-1].
$$
This definition and the distinguished triangle generalize to 
the case when $\cE$ and $\cF$ are bounded below
complex chain complexes of quasi-coherent sheaves.

Closely related to the derived tensor product is the derived
completion.

\begin{defn}\label{der-completion}\index{completion, derived} Let $\cZ \subseteq \cM$ be
a closed substack defined by a quasi-coherent ideal sheaf
$\cI$. Let $\cF$ be a quasi-coherent sheaf on $\cM$. Then
the {\bf derived completion} of $\cF$ at $\cZ$ by
$$
L(\cF)^\cmpl_\cI = L(\cF)^\cmpl_\cZ = \holim (\cO/\cI^n \otimes^L \cF).
$$
\end{defn}

Thus, if $\Spec(R) \to \cM$ is faithfully flat and quasi-compact,
we can set
$$
L(\cF)^\cmpl_\cZ(R) = \lim P_\bullet (R)/\cI^n(R)\cP_\bullet (R).
$$
This is an $\cO$-module sheaf, but not necessarily quasi-coherent,
as inverse limit and tensor product need not commute. If
$j_n:\cZ^{(n)} \subseteq \cM$ are the infinitesimal neighborhoods
of $\cZ$ defined by the powers of $\cI$, then
$$
L(\cF)^\cmpl_\cZ = \holim (j_n)_\ast (Lj_n)^\ast \cF
$$
where $ (Lj_n)^\ast$ is the total left derived functor of $j_n^\ast$.

We now turn to the question of when the hypotheses
of Proposition \ref{hovey-results} apply. 
There is a classical and useful notion from stable homotopy
theory which guarantees that the
finitely generated locally free sheaves generate the category
of quasi-coherent sheaves.

\begin{defn}\label{adams-cond}\index{Adams condition}1.)
A Hopf algebroid
$(A,\La)$ is an {\bf Adams Hopf algebroid} if the left
unit $A \to \La$ is flat and the $(A,\La)$-comodule
$\La$ can be written as a filtered colimit of comodules
$\La_i$ each of which is a finitely generated projective
$A$-module.

2.) An algebraic stack $\cM$ will be an {\bf Adams stack}
\index{Adams stack}
if there is an $fpqc$-presentation $\Spec(A) \to \cM$
so that
$$
\Spec(A) \times_\cM \Spec(A) \cong \Spec(\La)
$$
is itself affine and the resulting Hopf algebroid $(A,\La)$
is an Adams Hopf algebroid.
\end{defn}

This definition has a curious and unfortunate feature.
We would like to assert
that if $\cM$ has one $fpqc$-presentation $\Spec(A) \to \cM$
so that $(A,\La)$ is an Adams Hopf algebroid then any other $fpqc$
presentation would have the same property. But this is not 
known. See \cite{hovey-comod}, Question 1.4.2.\footnote{This 
problem could be avoided by working with resolutions by appropriately
flat modules; these exist over any quasi-compact and separated
stack. See \cite{all} \S 1. I have chosen to use the Adams condition
only because it fits better with the culture of homotopy theory.} However,
we do have the following rephrasing of Proposition 1.4.4
of \cite{hovey-comod}.

\begin{prop}\label{adams-generates}Let $\cM$ be an Adams
stack. The the finitely generated locally free sheaves
generate the category $\Qmod_\cM$ of quasi-coherent sheaves
on $\cM$.
\end{prop}

We now make good on our claim that most of the stacks in this
monograph are of this kind.

\begin{prop}\label{mn-adams}For all $n$, $0 \leq n \leq \infty$,
the moduli stack $\cM_\fg\pnty{n}$ of $n$-buds of formal
groups is an Adams stack.
\end{prop}

\begin{proof} We show that the evident presentation
$$
\Spec(L\pnty{n}) \longr \cM_\fg\pnty{n}
$$
has the desired property. Recall from Proposition
\ref{homotopy-orbit} that
$$
\Spec(L\pnty{n}) \times_{\cM_\fg\pnty{n}} \Spec(L_n)
\cong \Spec(W\pnty{n}) \cong \fgl\pnty{n} \times \La\pnty{n}.
$$
We use the gradings of Remark \ref{gradings-revisited-2}
and, especially, Remark \ref{warning-grading}.2.

Let 
$$
A_{n,i} \subseteq \ZZ[a_1,\ldots,a_{n-1}]
\subseteq \ZZ[a_0^{\pm 1},a_1,\ldots,a_{n-1}]
$$
be the elements of degree less than or equal to $i$,
let
$$
B_{n,i} = \mathop{\oplus}_{-i \leq s \leq i} A_{n,i}a_0^{s}
\subseteq \ZZ[a_0^{\pm 1},\ldots,a_{n-1}]
$$
and let
$$
W\pnty{n,i} = L\pnty{n} \otimes B_{n,i}.
$$
Then $W\pnty{n,i}$ is a finitely generated free $L\pnty{n}$ module, a
sub-comodule of $W\pnty{n}$, and $\colim_i W\pnty{n,i} = W\pnty{n}$.
\end{proof}

\begin{prop}The following stacks are Adams stacks.
\begin{enumerate}

\item $\cM(n)$, the closed substack of $\cM_\fg \otimes
\ZZ_{(p)}$ of formal groups of height at least $n$;

\item $\layer{n}= \cM(n)[v^{-1}]$, the open substack of $\cM(n)$
of formal groups
of exactly height $n$;

\item $\cU(n)$, the open substack of $\cM_\fg \otimes
\ZZ_{(p)}$ of formal
groups of height at most $n$.
\end{enumerate}
\end{prop}

\begin{proof} Because we have base-changed over $\ZZ_{(p)}$,
we can choose the morphism
$$
\Spec(\ZZ_{(p)}[u_1,u_2,\cdots]) \to \cM_\fg \otimes \ZZ_{(p)}
$$
representing the universal $p$-typical formal group as the
presentation. Then we have presentations
$$
\Spec(\FF_p[u_n,u_{n-1},\cdots]) \to \cM(n)
$$
and
$$
\Spec(\FF_p[u_n^{\pm 1},u_{n-1},\cdots]) \to\layer{n}
$$
and
$$
\Spec(\ZZ_{(p)}[u_1,\ldots,u_{n-1},u_{n}^{\pm 1}]) \to \cU(n).
$$
Then we appeal to Theorem 1.4.9 and Proposition 1.4.11. of 
\cite{hovey-comod}.
\end{proof}

\subsection{Torsion modules and inverting $v_n$}

In the next section on Landweber exactness, and later when we
discuss chromatic convergence, we are are going to 
need some technical lemmas about inverting $v_n$ for $\cI_n$-torsion
sheaves on $\cM_\fg$. We begin with some definitions so that we can
work in some generality with algebraic stacks $\cN$ flat over $\cM_\fg$.
The following definition generalizes the definition of 
regular scale given in \cite{LEFT}.

\begin{defn}\label{scale}\index{scale}Let $\cN$ be an algebraic stack and
$$
0 = \cJ_{0} \subseteq \cJ_1 \subseteq \cJ_2 \subseteq \cdots \subseteq \cO_\cN
$$
be an ascending sequence of ideal sheaves. Then the sequence $\{\cJ_n\}$
forms a {\bf regular scale} for $\cN$
if for all $n$, the ideal sheaf $\cJ_{n+1}/\cJ_{n}$ is locally free of rank
$1$ as a $\cO/\cJ_{n}$ module. A regular scale is a {\bf finite} if 
$\cJ_n = \cO$ for some $n$.
\end{defn}

\begin{rem}\label{scale-reg-seq}Given a regular scale on $\cN$,
let $\cN(n)$ denote that closed substack defined by $\cJ_n$. Then
$\cN(n) \subseteq \cN(n-1)$ is an effective Cartier divisor for
$\cN(n-1)$. An embedding $\cZ \subseteq \cN$ of a closed substack
is called {\it regular} if the ideal defining the embedding is locally
generated by a regular sequence. Thus a regular scale produces
regular embeddings $\cN(n) \subseteq \cN$, but it is does more:
it specifies the terms in the regular sequence modulo
the lower terms.
\end{rem}

\begin{exam}\label{in-form-scale} Fix a prime $p$, let $\cM = \cM_\fg$ and let
$$
0 \subseteq \cI_1 \subseteq \cI_2 \subseteq \cdots \subseteq \cO_\fg
$$
be the ascending chain of ideals giving the closed substacks
$\cM(n) \subseteq \cM_\fg$ classifying formal groups of height greater 
than or equal to $n$. This, of course, is the basic example
of a regular scale. This scale is not finite; however, if we let $i_n:\cU(n) \to \cM_\fg$
be the open substack classifying formal groups of height less than
or equal to $n$, then
$$
i_n^\ast \cI_0 \subseteq i_n^\ast \cI_1 \subseteq i_n^\ast \cI_2
\subseteq \cdots \subseteq i_n^\ast \cO_\fg = \cO_{\cU(n)}
$$
is a finite regular scale as $i_n^\ast \cI_n = i_n^\ast \cI_k = \cO_{\cU(n)}$
for $k \geq n$.

This example can be generalized to stacks $\cN$ representable
and flat over $\cM_\fg$. See Proposition \ref{LEFT-easy-part}.
\end{exam}

We now come to torsion modules and inverting $v_n$. Let $\cN$
be an algebraic stack and let $\{\cJ_n\}$ be a scale for $\cN$.
Let $j_n:\cN(n) \subseteq \cN$ be the closed inclusion
defined by $\cJ_n$ and let $i_{n-1}:\cV(n-1) \to \cM_\fg$ be the open
complement. (The numerology is chosen to agree with case
of $\cI_n \subseteq \cO_\fg$.) Let's write $\cO$ for $\cO_\cN$.

\begin{defn}\label{in-torsion} An $\cO$-module sheaf $\cF$ is
{\bf supported on}\index{support, for a sheaf} $\cN(n)$ if $i_{n-1}^\ast \cF = 0$. 
We also say that $\cF$ is $\cJ_n$-{\bf torsion}\index{torsion sheaf}
\index{sheaf, torsion} if for any flat and
quasi-compact morphism $\Spec(R) \to \cM_\fg$, the $R$-module
$\cF(R)$ is $\cI_n(R)$-torsion.
\end{defn} 

In Definition \ref{in-torsion} we do not assume that $\cF$ is
quasi-coherent; however, the next result shows that the
two notions defined there are equivalent for
quasi-coherent sheaves.

\begin{lem}\label{char-support}Let $\cF$ be a quasi-coherent $\cO $-module
sheaf. Then $\cF$ is supported on $\cN(n)$ if and only
if $\cF$ is $\cJ_n$-torsion. 
\end{lem}

\begin{proof}This is a consequence of the fact that $\cJ_n$ defines
a regular embedding.
For each flat and quasi-compact morphism
$\Spec(R) \to \cM_\fg$, choose -- by passing to a faithfully flat
extension if necessary --  generators $(p,u_1,\ldots,u_{n-1})$
of $\cJ_n(R)$.

First suppose $\cF$ is supported on $\cN(n)$. Then there are commutative
diagrams
$$
\xymatrix{
\Spec(R[u_i^{-1}]) \rto^{\subseteq} \dto & \Spec(R)\dto\\
V(n-1) \rto^{\subseteq} &\cN.
}
$$
Thus $R[u_i^{-1}] \otimes_R \cF(R) \cong \cF(R[u_i^{-1}]) = 0$,
and we may conclude that $\cF(R)$ is $\cJ_n(R)$-torsion.

Conversely, suppose $\cF$ is $\cJ_n$-torsion and $\Spec(R) \to V(n-1)$
is any flat and quasi-compact morphism. Then $\cJ_n(R) = R$, so
$\cJ_n(R)^k = R$ for all $k > 0$. If $x \in \cF(R)$, then
$Rx = \cJ_n(R)^kx = 0$ for some $k$, whence $x=0$. Thus
$i_{n-1}^\ast \cF = 0$.
\end{proof}

In the following result, $\hom$ denote the sheaf of
homomorphisms.

\begin{lem}\label{where-tor-supp}Let $\cF$ be a quasi-coherent
$\cJ_n$-torsion sheaf. Then evaluation defines a natural
isomorphism
$$
\colim \hom_\cO(\cO/\cJ_n^k,\cF) \mathop{\longr}^{\cong} \cF.
$$
If $f_k:\cN(n)_k \subseteq \cN$ is the inclusion of the
$k$th infinitesimal neighborhood
of $\cN(n)$ defined by the vanishing of $\cI_n^k$,  then there
is a quasi-coherent sheaf $\cF_k$ on $\cN(n)_k$ and a natural
isomorphism
$$
(f_k)_\ast \cF_k \cong \hom_\cO(\cO/\cI_n^k,\cF).
$$
\end{lem}

\begin{proof}The first statement can be check locally, and there
it follows from the fact that $\cJ_n$ if finitely generated. For
the second statement, we use the fact that any closed
inclusion is affine (see \ref{affine-morphisms}). From this it follows that
$(f_k)_\ast$ induces an equivalence between the
categories of quasi-coherent $\cO/\cJ_n^k$-modules
on $\cN$ and the category of quasi-coherent modules on
$\cN(n)_k$. See Proposition \ref{qc-affine}.
\end{proof}

Suppose $\cF$ is a quasi-coherent $\cJ_n$-torsion sheaf.
The Lemma \ref{char-support} implies that $i_{n-1}^\ast \cF =0$.
We next consider $i_n^\ast \cF$ or, more exactly, the resulting
push-forward $(i_n)_\ast i_n^\ast \cF$, which is a sheaf
on $\cN$. The next result shows that
the natural map
$$
(i_n)_\ast i_n^\ast \cF \to R(i_n)_\ast i_n^\ast \cF
$$
is an equivalence and gives a local description of $(i_n)_\ast i_n^\ast \cF$.

\begin{prop}\label{no-higher-derived} Let $\cF$ be a quasi-coherent
$\cJ_n$-torsion sheaf on $\cN$.
Let $\Spec(R) \to \cN$ be any flat and quasi-compact
morphism so that $\cJ_n(R)/\cJ_{n-1}(R)$ is free of rank one
over $R/\cJ_{n-1}(R)$. Then we have an isomorphism
$$
[(i_n)_\ast i_n^\ast \cF](R) \cong \cF[u_n^{-1}]
$$
where $u_n \in \cJ_n(R)$ is any element so that
$u_n + \cJ_{n-1}(R)$ generates the $R$-module $\cJ_n(R)/\cJ_{n-1}(R)$.
Furthermore, 
$$
R^s(i_n)_\ast i_n^\ast \cF = 0,\ s > 0.
$$

\end{prop}

\begin{proof} By Lemma \ref{where-tor-supp} and a colimit
argument, we may assume that $\cF = f_\ast \cE$
for some quasi-coherent sheaf $\cE$ on the $k$th
infinitesimal neighborhood $f: \cN(n)_k \to\cN$ of $\cN(n)$.

Consider the pull-back square\footnote{In the case where $\cN = \cM_\fg$ and $\cJ_n = \cI_n$, we have that
$\cN(n)_k \times_{\cN} \cV(n)$ is the $k$th infinitesimal neighborhood
of $\layer{n}$.}
$$
\xymatrix{
\cN(n)_k \times_{\cN} \cV(n) \rto^-{p_2} \dto_{p_1} & \cV(n) \dto^{i_n}\\
\cN(n)_k \rto_{f} & \cN.\\
}
$$
(In the case where $\cN = \cM_\fg$ and $\cJ_n = \cI_n$, we have that
$\cN(n)_k \times_{\cN} \cV(n)$ is the $k$th infinitesimal neighborhood
of $\layer{n}$.) Then
$i_n^\ast f_\ast \cE= (p_2)_\ast p_1^\ast \cE$; thus, we may
conclude that we have an equivalence in the derived category
\begin{equation}\label{compute-back-1}
R(i_n)_\ast i_n^\ast f_\ast \cE \simeq f_\ast R(p_1)_\ast 
p_1^\ast \cE.
\end{equation}
The open inclusion $\cN(n)_k \times_{\cN} \cV(n) \subseteq \cN(n)_k$ is the
complement of the closed inclusion $\cN(n+1) \subseteq \cN(n)_k$. Locally,
this closed inclusion is defined by the vanishing of $u_n$. We see that this
implies
\begin{equation}\label{compute-back-2}
R(p_1)_\ast  p_1^\ast \cE \simeq (p_1)_\ast p_1^\ast \cE \cong
\cE[u_n^{-1}].
\end{equation}
The result now follows because $f_\ast$ is exact.
\end{proof}

\begin{rem}\label{invert-vn-1} Now let $f:\cN \to \cM_\fg$
be a representable and flat morphism of algebraic stacks and
let $\{\cJ_n\} = \{f^\ast \cI_n\}$ be the resulting
scale. See Proposition \ref{LEFT-easy-part}. Regard $v_n$
as a global section of $\omega^{p^n-1}$ considered as
a sheaf over $\cN(n)$. Suppose $\cF$ is actually an 
$\cO/\cJ_n$-module sheaf; that is, suppose $\cF = (j_n)_\ast \cE$
for some quasi-coherent sheaf $\cE$ on $\cN(n)$. Then we can form
the colimit sheaf $\cF[v_n^{-1}]$ of the sequence
$$
\xymatrix{
\cF \rto^-{v_n} & \cF \otimes \omega^{p^n-1} \rto^-{v_n} & \cF \otimes \omega^{2(p^n-1)} \rto^-{v_n} & \cdots.
}
$$
We claim that $\cF[v_n^{-1}] \cong (i_n)_\ast i_n^\ast \cF$.

By Equations \ref{compute-back-1} and \ref{compute-back-2}, and
because $(j_n)_\ast$ is exact, it is sufficient to show
$$
\cE[v_n^{-1}] \cong (p_1)_\ast p_1^\ast \cE.
$$
Since multiplication by
$v_n$ is invertible for sheaves on
$$
\cN(n) \times_{\cN} \cV(n) \cong \layer{n} \times_{\cM_\fg} \cN,
$$
the natural map
$\cE\to (p_1)_\ast p_1^\ast \cE$ factors as a map
$\cE[v_n^{-1}] \to (p_1)_\ast p_1^\ast \cE$. To show this is
an isomorphism we need only work locally. Let
$$
G:\Spec(R) \longr \cM(n)
$$
be a flat and quasi-compact morphism classifying a formal group $G$. 
Taking a faithfully flat extension if needed, we may choose
an invariant derivation $u \in \omega^{-1}_G$ for $G$ generating
the free $R$-module $\omega^{-1}_G$; then the
element
$$
u_n \defeq u^{(p^n-1)}v_n(G) \in R = \omega^{0}_G
$$
generates $\cJ_n(R) = \cJ_n(R)/\cJ_{n-1}(R)$. Then we have 
a commutative diagram
$$
\xymatrix{
\cE(G) \rto^-{v_n} \dto_{=}&
\cE(G) \otimes \omega^{p^n-1} \rto^-{v_n} \dto^{u^{p^n-1}}&
\cE(G) \otimes \omega^{2(p^n-1)} \dto^{u^{2(p^n-1)}}\rto^-{v_n} & \cdots\\
\cE(G) \rto^-{u_n} & \cE(G) \rto^-{u_n} & \cE(G) \rto^-{u_n} & \cdots.
}
$$
Since the vertical maps are isomorphisms, the claim follows from
Proposition \ref{no-higher-derived}.

This observation can be easily be generalized to the case where
$\cF$ is an $\cO/\cJ_n^k$-module sheaf for any $k \geq 1$ because
a power of $v_n$ is a global section of the appropriate power
of $\omega$.
\end{rem}

Because of the previous remark, the following definition does not
create an ambiguity.\index{$\cF[v_n^{-1}]$}\index{inverting $v_n$}

\begin{defn}\label{invert-vn}Let $f:\cN \to \cM_\fg$ be a representable
and flat morphism of algebraic stacks and let $\{\cJ_n\} = \{f^\ast \cI_n\}$
be the resulting scale for $\cN$. If $\cF$ be a quasi-coherent 
$\cJ_n$-torsion sheaf on $\cN$ define
$$
\cF[v_n^{-1}] = (i_n)_\ast i_n^\ast \cF.
$$
\end{defn}

\subsection{LEFT: A condition for flatness}

Let $f:\cN \to \cM_\fg$ be a representable
morphism of algebraic stacks. We would like to give a concrete and easily
checked condition on this morphism to guarantee that it
be flat. This condition is a partial converse to Proposition \ref{LEFT-easy-part}
and a version of the Landweber Exact Functor Theorem (LEFT).
This theorem has a variety of avatars; the one we give
here is due to Hopkins and Miller. See \cite{hop-notes} and \cite{LEFT}.
The original source is \cite{landweber}.


In this section we will work over $\cM_\fg$ over $\Spec(\ZZ)$, rather
than at a given prime. 

Now let $f:\cN \to \cM_\fg$ be a representable, quasi-compact,
and quasi-separated morphism of stacks. The hypotheses on the morphism guarantee that if $\cF$ is a quasi-coherent sheaf on $\cN$, then $f_\ast \cF$ is
a quasi-coherent sheaf on $\cM_\fg$. Compare Proposition
\ref{push-forward-qc}. Let $\cJ_n \subseteq \cO_\cN$
be the kernel of the morphism 
$$
\cO_\cN = f^\ast \cO_\fg \to f^\ast (\cO_\fg/\cI_n).
$$
Thus $\cJ_n$ defines the closed inclusion 
$$
j_n:\cN(n) = \cM(n) \times_{\cM_\fg} \cN \mathop{\longr}^{\subseteq} \cN.
$$
Thus there is a surjection $f^\ast \cI_n \to \cJ_n$ which becomes an isomorphism
if $f$ is flat. Also note that
$$
(j_n)_\ast \cO_{\cN(n)} = \cO_\cN/\cJ_n = \cO_\cN/f^\ast \cI_n.
$$
From this we can conclude that the global section $v_n \in
H^0(\cM(n),\omega^{p^n-1})$ defines a surjection
$$
v_n:\cO_\cN/\cJ_{n-1} \to \cJ_n/\cJ_{n-1} \otimes \omega^{p^n-1}.
$$
This includes the case $n=0$; we set $v_0 = p$.
The basic criterion of flatness is the following. Note
that if $\cN$ is a stack over $\ZZ_{(\ell)}$ for some prime $\ell$, then
the hypotheses automatically true for all prime $p \ne \ell$. This
remark will have a variant for all our other versions of
Landweber exactness below.

\begin{thm}[{\bf Landweber Exactness I}]\label{land-exact-2}Let
$f:\cN \to \cM_\fg$ be
a representable, quasi-compact, and quasi-separated morphism of stacks.
Suppose that for all primes $p$,
\begin{enumerate}

\item $v_n:\cO_\cN/\cJ_{n-1} \to \cJ_n/\cJ_{n-1} \otimes \omega^{p^n-1}$
is an isomorphism, and

\item $\cJ_n = \cO_\cN$ for large $n$.
\end{enumerate}
Then $f$ is flat.
Conversely, if for all primes $p$, $\cJ_n = \cO_\cN$ for some $n$, then $f$ is flat only if
condition (1) holds.
\end{thm}

\begin{rem}\label{land-exact-2-redux} The hypotheses of Theorem \ref{land-exact-2}
imply, in particular, that the ideals $\cJ_n$ form a finite regular scale for
$\cN$; in particular, in the descending chain of closed substacks
$$
\cdots \subseteq \cN(n) \subseteq \cN(n-1) \subseteq \cdots \subseteq
\cN(1) \subseteq \cN
$$
each of the inclusions is that of an effective Cartier divisor and that there
is an $n$ so that $\cN(k)$ is empty for $k > n$. Furthermore, an
inductive argument shows that the natural surjections
$f^\ast \cI_n \to \cJ_n$ are, in fact, isomorphisms. Indeed, if
$f^\ast \cI_{n-1} \cong \cJ_{n-1}$, then we obtain a diagram
$$
\xymatrix@R=10pt{
&f^\ast \cI_n/f^\ast \cI_{n-1}\ar@{>>}[dd] \\
\cO_\cN/f^\ast \cI_{n-1} \ar@{>>}[ur]^{v_n}\ar[dr]_{v_n}^\cong\\
&\cJ_n/f^\ast \cI_{n-1}
}
$$
and we can conclude $f^\ast \cI_n/f^\ast \cI_{n-1} \to \cJ_n/\cJ_{n-1}$ is
an isomorphism.
\end{rem}

By specializing to the affine case and using Remark \ref{vn-local-1},
we obtain a more classical version of Landweber exactness. 

\begin{cor}\label{land-exact-1}Let $g:\Spec(A) \to \cM_\fg$  classify
a formal group $G$ with a coordinate $x$. For all primes $p$,
let $p,u_1,u_2,\ldots$
be elements of $A$ so that the $p$-series can be written
$$
[p](x) = u_kx^{p^k} + \cdots
$$
modulo $(p,u_1,\ldots,u_{k-1})$. Suppose the elements $p, u_1,\ldots$
form a regular sequence and suppose there is some $n$ so that
$$
(p,u_1,\ldots,u_{n-1}) = A
$$
Then $g$ is flat.
\end{cor}

In Landweber's original paper \cite{landweber} the hypothesis that
$\cI_n(G) = A$ for some $n$ was not required. I believe Hollander
also has a way to remove this hypothesis. See \cite{hollander2}.

\begin{rem}\label{LEFT-reformulate}We can reformulate the hypotheses of
Theorem \ref{land-exact-2}
as conditions on the quasi-coherent algebra sheaf $f_\ast \cO_\cN$
on $\cM_\fg$.
As a matter of notation, let's write
$$
\cF/\cI_n \defeq (j_n)_\ast j_n^\ast \cF
$$
for any $\cF$ be a quasi-coherent
sheaf on $\cM_\fg$. We will say that the regular scale $\{\cI_n\}$ acts
{\bf regularly and finitely}\index{regular and finite action} on $\cF$ if for all $n$
$$
v_n:\cF/\cI_n \longr  \cF/\cI_n \otimes \omega^{p^n-1}
$$
is injective and $\cF/\cI_n = 0$ for large $n$. Because have
a pull-back square for all $n$
$$
\xymatrix{
\cM(n) \times_{\cM_\fg} \cN \rto \dto & \cN \dto^f\\
\cM(n) \rto_{j_n} & \cM_\fg
}
$$
we have that $f_\ast(\cO_\cN/\cJ_n) = (f_\ast \cO_\cN)/\cI_n$.

Suppose the hypotheses of Theorem \ref{land-exact-2} hold. Then
$$
v_n:\cO_\cN/\cJ_{n-1} \to \cO_n/\cJ_{n-1} \otimes \omega^{p^n-1}
$$
is injective and $\cO_\cN/\cJ_n = 0$ for large $n$. Since $f_\ast$
is left exact, we have that $\{\cI_n\}$ acts regularly and finitely
on $f_\ast \cO_\cN$.

Conversely, suppose $\{\cI_n\}$ acts regularly and finitely
on $f_\ast \cO_\cN$. We will see below in Theorem \ref{flat-2}
that this implies that $f$ is a flat morphism, which in turn implies
that $f^\ast \cI_n = \cJ_n$ and, hence, that
$$
f^\ast (\cI_n/\cI_{n-1}) = \cJ_n/\cJ_{n-1}.
$$
This, in its turn, implies that the hypotheses of Theorem \ref{land-exact-2}
hold. Thus, Theorem \ref{land-exact-2} is equivalent to the following
result.
\end{rem}

\begin{thm}[{\bf Landweber Exactness II}]\label{land-exact-2-bis1}Let
$f:\cN \to \cM_\fg$ be
a representable, quasi-compact, and quasi-separated morphism of stacks.
Suppose  that for all primes,
the set of ideals $\{\cI_n\}$ acts regularly and finitely on $f_\ast \cO_\cN$.
Then $f$ is flat.

Conversely, if for all primes $p$, $f_\ast \cO_\cN/\cI_n = 0$ for some $n$,
then $f$ is flat only if
the set of ideals $\{\cI_n\}$ acts regularly and finitely on $f_\ast \cO_\cN$.
\end{thm}

This, in turn, is a corollary Proposition \ref{LEFT-easy-part} and  the following result.
Here and in what follows the higher torsion sheaves are defined by
$$
\Tor^{\cO}_s(\cF,\cE) = H_s(\cF \otimes^L_\cO \cE).
$$

\begin{thm}[{\bf Landweber Exactness III}]\label{land-exact-2-bis}Let
$\cF$ be a quasi-coherent sheaf on $\cM_\fg$. 
Suppose that for all primes $p$, the set of ideals $\{\cI_n\}$ acts regularly and finitely on
$\cF$. Then $\cF$ is flat as an $\cO_\fg$ module; that is,
$$
\Tor^{\cO}_s(\cF,\cE) = 0,\quad s>0.
$$
Conversely, if for all primes $p$, $\cF/\cI_n\cF = 0$ for some $n$,
then $\cF$ is flat only if
the set of ideals $\{\cI_n\}$ acts regularly and finitely on $\cF$.
\end{thm}

Theorem \ref{land-exact-2-bis} was  proved by Mike Hopkins in \cite{hop-notes}; the proofs here are the same.

Let $j_n:\cM(n) \to \cM_\fg$ be the
inclusion. 
The first result is this. The argument requires careful
organization of exact sequences.

\begin{prop}\label{flat-1}Suppose that for each prime $p$, the
scale $\{\cI_n\}$ acts regularly and finitely on $\cF$ and that
for each $n$, 
$$
\Tor^{\cO}_s((\cF/\cI_n)[v_n^{-1}],-)=0,\quad s> n.
$$ 
Then $\cF$ is a flat $\cO_\fg$-module sheaf.
\end{prop}

\begin{proof}By hypothesis, we have that for all large
$k$, 
$$
\Tor^{\cO}_s(\cF/\cI_k,-)=0.
$$ 
This begins a downward induction, where the induction hypothesis
is that
$$
\Tor^{\cO}_s((\cF/\cI_{n+1}),-)=0,\qquad s > n+1.
$$ 
The final result is the case $n=-1$. 

We  make the argument for the induction step, using the following
fact: if $\cL$ is any locally free sheaf, then
$$
\Tor_s^{\cO}(\cE \otimes \cL,\cE') \cong \Tor_s^{\cO}(\cE,\cE') \otimes \cL
$$
for any quasi-coherent sheaves $\cE$ and $\cE'$.

Since the scale $\{\cI_n\}$ acts regularly, we have an
exact sequence
$$
0 \to \cF/\cI_n \otimes \omega^{-(p^n-1)} \mathop{\longr}^{v_n}
\cF/\cI_n \to \cF/\cI_{n+1} \to 0.
$$
The induction hypothesis implies that for any quasi-coherent
sheaf $\cE$
$$
v_n: \Tor_s^{\cO}(\cF/\cI_n,\cE) \longr  \Tor_s^{\cO}(\cF/\cI_n,\cE)
\otimes \omega^{p^n-1}
$$
is an injection for $s > n$. 

Now recall that in Remark 
\ref{invert-vn-1} we showed that $\cF/\cI_n[v_n^{-1}]$
can be written as the colimit of the sequence
$$
\xymatrix{
\cF/\cI_n \rto^-{v_n} & \cF/\cI_n \otimes \omega^{p^n-1} \rto^-{v_n} &
\cF/\cI_n \otimes \omega^{2(p^n-1)} \rto^-{v_n} & \cdots.
}
$$
Thus, we have for $s > n$,
$$
0 = \Tor_s(\cF/\cI_n[v_n^{-1}],\cE) \cong \colim
\Tor_s(\cF/\cI_n,\cE) \otimes \omega^{t(p^n-1)}.
$$
Since each of the maps in the sequence is an injection,  the induction
step follows.
\end{proof}

Now we must check the hypothesis of  Proposition
\ref{flat-1} in order to prove Theorem \ref{land-exact-2-bis}.
Recall from Definition \ref{invert-vn}, that if $\cE$ is
any $\cI_n$-torsion sheaf, then
$$
\cE[v_n^{-1}] = (i_n)_\ast i_n^\ast \cE
$$
where $i_n:\cU(n) \to \cM_\fg$ is the inclusion. In the case
where $\cE = \cF/\cI_n$, we have that $\cE$ is the push-forward
of the sheaf $j_n^\ast \cF$ on $\cM(n)$. Since there is a pull-back
diagram
$$
\xymatrix{
\layer{n} \rto^{g_n} \dto_{k_n} & \cM(n) \dto^{j_n}\\
\cU(n) \rto_{i_n} & \cM_\fg
}
$$
we have
$$
i_n^\ast \cF/\cI_n = (k_n)_\ast g_n^\ast j_n^\ast \cF.
$$
Write $f_n:\layer{n} \to \cM_\fg$ for the inclusion
Thus we can conclude that
$$
\cF/\cI_n[v_n^{-1}] = (f_n)_\ast (f_n)^\ast \cF.
$$
The next result then verifies the hypothesis of Proposition
\ref{flat-1}.

\begin{prop}\label{flat-2}Let $\cF$ be any quasi-coherent sheaf
on $\layer{n}$ and let $\cE$ be any quasi-coherent sheaf on
$\cM_\fg$. Then
$$
Tor_s^{\cO}((f_n)_\ast \cF,\cE) = 0,\qquad s > n.
$$
\end{prop}

\begin{proof} Recall from Lemma \ref{layer-to-fg-affine},
that the inclusion $f_n:\layer{n}\to \cM_\fg$ is affine. This
implies that the category of quasi-coherent sheaves
on $\layer{n}$ is equivalent, via the the push-forward
$(f_n)_\ast$, to the category of quasi-coherent $(f_n)_\ast \cO_{\layer{n}}$
modules on $\cM_\fg$. (Compare Proposition \ref{qc-affine}.)
In particular, $(f_n)_\ast$ is exact on
quasi-coherent sheaves. Also, for all quasi-coherent sheaves $\cE$ on
$\cM_\fg$, there is a natural isomorphism
$$
(f_n)_\ast f_n^\ast \cE \cong (f_n)_\ast \cO_{\layer{n}} \otimes_{\cO_\fg}
\cE.
$$
It follows that there is a natural isomorphism
\begin{equation}\label{push-pull-affine}
(f_n)_\ast (\cF) \otimes \cE \cong (f_n)_\ast (\cF \otimes_{\cO_\layer{n}} f_n^\ast \cE)
\end{equation}
which becomes an equivalence of derived sheaves
$$
(f_n)_\ast (\cF) \otimes^L \cE \cong (f_n)_\ast (\cF \otimes_{\cO_\layer{n}}^L L(f_n^\ast) \cE).
$$
By Theorem \ref{diff-ht} we have that the morphism
$\Gamma_n:\Spec(\FF_p)  \to \cH(n)$ classifying any height
$n$ formal group over $\FF_p$ is an $fqpc$-cover;
hence, the category of quasi-coherent sheaves on
$\cH(n)$ is equivalent to  the category of $(\FF_p,\cO_{\Aut(\Gamma_n)})$
comodules. Here we have written
$\Spec(\cO_{\Aut(\Gamma_n)}) = \Spec(\FF_p) \times_{\layer{n}} \Spec(\FF_p)$. From this we have that the functor $\cF \otimes_{\layer{n}}(-)$
is exact, since the corresponding functor on comodules is simply
$$
\cF(\Ga_n:\Spec(\FF_p)\to \layer{n}) \otimes_{\FF_p} (-).
$$
Thus we need only show that 
$$
H_s L(f_n^\ast) \cE = 0
$$
for $s > n$. Since $(f_n)_\ast$ is exact, we need only check
that $H_s (f_n)_\ast L(f_n^\ast) \cE = 0$ for $s > n$.  We need only
check this equation locally, thus we may evaluate at any
morphism
$$
G:\Spec(R) \longr \cM_\fg
$$
classifying a formal group with a coordinate. Applying
the formula of Equation \ref{push-pull-affine} we see that
locally these homology sheaves are given by
$$
\Tor^R_s(u_n^{-1}R/\cI_n(G),\cE(G)).
$$
The result now follows from the fact that $\cI_n(G)$ is locally
generated by a regular sequence of length $n$. 
\end{proof}

\begin{rem}\label{how-formal}Almost all of the argument for
Theorem \ref{land-exact-2-bis} uses only that we have a sequence
of regularly embedded closed substacks $\{\cM(n)\}$ of $\cM_\fg$.
However, in the proof Proposition \ref{flat-2} we used Theorem 
\ref{diff-ht} which, in turn, ultimately depends on Lazard's proof of the
result that, over a separably closed field of characteristic $p$,
all formal groups of height $n$ are isomorphic. Thus, it does not
appear to me that the Landweber exact functor theorem is a
generality -- it seems quite specific to formal groups.
\end{rem}

%% file: defs.tex
\section{The formal neighborhood of a height $n$ point}

In this section we make the following slogan precise:  the formal neighborhood
in $\cM_\fg$ of a height $n$ formal group $\Ga$ over a perfect
field of characteristic $p$ is the Lubin-Tate space of the deformations
of $\Ga$. This is not quite true as stated; a more precise statement
is that Lubin-Tate space is a universal cover the formal neighborhood
and the automorphisms of the formal group are the covering
transformations. The exact result is given in Theorem \ref{lt-is-univ}
below.

Let $\cU_\fg(n) = \cM_{\fg} - \cM(n+1)$ be the open substack of $\cM_\fg$
classifying formal groups of height less than or equal to $n$. Then
$$
\layer{n} = \cU_\fg(n) - \cU_\fg(n-1)
$$
is a closed substack of $\cU_\fg(n)$ defined by the vanishing of the ideal
$\cI_n$. Recall the $\layer{n}$ has a single
geometric point, but that this point has plenty of automorphisms.
See Theorem \ref{diff-ht}. We wish to write down a description of 
the formal neighborhood $\hatlayer{n}$ of $\layer{n} \subseteq \cU_\fg(n)$.
\index{$\hatlayer{n}$}

By definition, $\hatlayer{n}$ is the category fibered in groupoids over
$\Aff_{\ZZ_{(p)}}$ which assigns to each $\ZZ_{(p)}$-algebra
$B$ the groupoid with objects the formal groups $G$ over $B$ so that
\begin{enumerate}

\item $\cI_{n}(G) \subseteq B$ is nilpotent; and

\item $\cI_{n+1}(G) =B$.
\end{enumerate}
Thus, if $q:B \to B/\cI_{n}(G)$ is the quotient map, the formal group $q^\ast G$
has strict height $n$ in the sense that $v_i(G) = 0$ for
$i < n$ and $v_n(G)$ is invertible. A great many examples of such formal
groups can be obtained as deformations of a height $n$ formal group;
thus, we now discuss deformations and Lubin-Tate space.

\subsection{Deformations of height $n$ formal groups over a field}

Fix a formal group $\Ga$ of height $n$
over a perfect field $\FF$ of characteristic $p$. (In practice, $\FF$ will
be an algebraic extension of a the prime field $\FF_p$). Recall
that an Artin local ring is a Noetherian commutative ring
with a unique nilpotent maximal ideal. Let $\Art_\FF$ denote the
category of Artin local rings $(A,\mm)$ so that we can choose
an isomorphism $A/\mm \cong \FF$ from the residue field of 
$A$ to $\FF$. The isomorphism is not
part of the data. Morphisms in $\Art_\FF$ are
ring homomorphisms which induce an isomorphism on residue fields.

\begin{defn}\label{deformation}\index{deformation, of a formal
group}A {\bf deformation} of the pair
$(\FF,\Ga)$ to an object $A$ of $\Art_\FF$ is
a Cartesian square
$$
\xymatrix{
\Ga \rto \dto & G \dto\\
\Spec(\FF) \rto^f & \Spec(A)
}
$$
where $G$ is a formal group over $A$ and $f$ induces an isomorphism
$\Spec(\FF) \cong \Spec(A/\mm)$. 

Deformations become a category $\Def(\FF,\Ga)$
fibered in groupoids over $\Art_\FF$ by setting a morphism
to be a commutative cube
$$
\xymatrix@!0@C=35pt@R=30pt{
&\Ga \ar[rr] \ar'[d][dd] && G'  \ar[dd]\\
\Ga \ar[ur]^= \ar[rr] \ar[dd] && G \ar[ur] \ar[dd]\\
& \Spec(\FF) \ar'[r][rr] &&\Spec(A')\\
\Spec(\FF) \ar[rr] \ar[ur]^-= && \Spec(A)\ar[ur]
}
$$
where the right face is also a Cartesian. 
\end{defn}

\begin{rem}\label{alt-deformation} We can rephrase this as follows.
A deformation of $\Ga$ to $A$ is a triple $(G,i,\phi)$ where $G$ is a formal
group over $A$, $i:\Spec(\FF) \to \Spec(A/\mm)$ is an isomorphism
and $\phi:\Ga \to i^\ast G_0$ is an isomorphism of formal groups over $\FF$.
Here and always we write $G_0$ for the {\it special fiber} of $G$; that is,
the induced formal group over $A/\mm$. There is an isomorphism
of deformations $\psi:(G,i,\phi) \to (G,i',\phi')$ to $A$ if $i = i'$ and
$\psi:G \to G'$ is an isomorphism of formal groups so that
$$
\xymatrix@R=5pt{
&i^\ast G_0 \ar[dd]^{i^\ast \psi_0} \\
 \Ga\ar[ur]^\phi\ar[dr]_{\phi'}\\
&i^\ast G_0'
}
$$
commutes.

In either formulation of a deformation, we note that if $G$ is
a deformation of $\Ga$ to $(A,\mm)$, then $\cI_n(G) \subseteq  \mm$
and $\cI_{n+1}(G) = A$.
\end{rem}

\begin{rem}\label{two-def-formal} Let $R = (R,\mm_R)$ be a complete
local ring so that $R/\mm_R \cong \FF$. We write $\Spf(R)$ equally
for the resulting formal scheme and for the 
the category fibered in groupoids over $\Art_\FF$ that assigns to
each object $(A,\mm)$ of $\Art_\FF$ the discrete groupoid of
all ring homomorphisms which induce an isomorphism
$R/\mm_R \cong A/\mm$; so, in particular,
$f(\mm_R) \subseteq \mm$. This is an abuse of notation, but a mild one,
and should
cause no confusion. Indeed, the formal scheme $\Spf(R)$ 
is the left Kan extension of the functor $\Spf(R)$ on $\Art_\FF$
along the inclusion of $\Art_\FF$ into all rings. 
\end{rem}

\begin{thm}[{\bf Lubin-Tate}]\label{Lubin-Tate}The category fibered in groupoids
$\Def(\FF,\Ga)$ is discrete and representable; that is, there is
a complete local ring $R(\FF,\Ga)$ and a deformation
$$
\xymatrix{
\Ga \rto \dto & H \dto\\
\Spec(\FF) \rto &\Spf(R(\FF,\Ga))
}
$$
of $\Ga$ to $R(\FF,\Ga)$\index{$R(\FF,\Ga)$} so that the induced morphism
$$
\Spf(R(\FF,\Ga)) \longr \Def(\FF,\Ga)
$$
is an equivalence of categories fibered in groupoids over $\Art_\FF$.
\end{thm}

The formal spectrum $\Spf(R(\FF,\Ga)$ is called {\it Lubin-Tate space}.
\index{Lubin-Tate space}

\begin{rem}\label{define-from-spf}
The induced morphism $\Spf(R(\FF,\Ga)) \longr \Def(\FF,\Ga)$ is
not completely trivial to define. Given a homomorphism
$R(\FF,\Ga) \to A$ to an Artin local ring which induces an
isomorphism of residue fields, we are asserting there is
a unique way to complete the back square of the following
diagram so that it commutes.
$$
\xymatrix@!0@C=35pt@R=30pt{
&\Ga \ar[dl]^=\ar[rr] \ar'[d][dd] && f^\ast H\ar[dl] \ar[dd]\\
\Ga  \ar[rr] \ar[dd] && H \ar[dd]\\
& \Spec(\FF) \ar[dl]^=\ar'[r][rr] &&\Spec(A)\ar[dl]\\
\Spec(\FF) \ar[rr] && \Spf(R(\FF,\Ga))
}
$$
Thus, given $f:\Spec(A) \to \Spf(R(\FF,\Ga)$, the universal deformation
$(H,j,\phi_u)$ gets sent to
$$
f^\ast(H,j,\phi_u) \defeq (f^\ast H,f_0^{-1}j,\phi_u)
$$
where
$f_0:\Spec(A/\mm_A) \to \Spf(R(\FF,\Ga)/\mm)$ is the induced
isomorphism and we have written $\phi_u$ for both the universal
isomorphism
$$
\phi_u:\Ga \longr j^\ast H_0
$$
and the induced isomorphism
$$
\phi_u:\Ga \longr (f_0^{-1}j)^\ast (f^\ast H)_0 \cong j^\ast H_0.
$$
In this language, the theorem of Lubin and Tate reads as follows: given
a deformation $(G,i,\phi)$ of $\Ga$ to $A \in \Art_\FF$, there is 
a homomorphism $f:R(\FF,\Ga) \to A$ inducing an isomorphism on
residue fields and a unique isomorphism of deformations
$$
\psi:(G,i,\phi) \longr f^\ast(H,j,\phi_u) .
$$
\end{rem}
\bigskip

The main lemma of Lubin and Tate is to calculate the
deformations of $\Ga$ to the ring of dual numbers $\FF[\rege]$,
where $\rege^2=0$. Indeed, they show for that ring
there is a non-canonical isomorphism
$$
\pi_0\Def(\FF,\Ga)_{\FF[\rege]}\cong (\FF\rege)^{n-1}
$$
where $(-)^{k}$ means the $k$th Cartesian power and
$$
\pi_1(\Def(\FF,\Ga),G)_{\FF[\rege]} = \{1\}
$$
for any deformation $G$.
The general theory of deformations (see \cite{Schles}, Proposition
3.12) then shows that there is a (non-canonical) isomorphism
\begin{equation}\label{lt-eq-1}
\pi_0\Def(\FF,\Ga)\cong \mm^{n-1}
\end{equation}
where we are write $\mm$ for the functor which assigns to an Artin local ring
$A$ its maximal ideal $\mm_A$. For any deformation $G$,
\begin{equation}\label{lt-eq-2}
\pi_1(\Def(\FF,\Ga),G) = \{1\}.
\end{equation}
It immediately follows that the ring $R(\FF,\Ga)$ is a power
series ring. More explicitly, since $R(\FF,\Ga)$ is
local, Noetherian, and $p$-complete, the universal deformation $H$ can
be given a $p$-typical coordinate $x$ for which the $p$-series of $H$
becomes
$$
[p]_H(x) = px +_H u_1x^p +_H \cdots +_H u_{n-1} x^{p^{n-1}} +_H u_nx^{p^n}
+ _H \cdots.
$$
Then there is an isomorphism
\begin{equation}\label{lt-is-power-series}
R(\FF,\Ga) \cong W(\FF)[[u_1,\ldots,u_{n-1}]]
\end{equation}
where $W(\FF)$ is the Witt vectors of $\FF$
and the maximal ideal $\mm = \cI_{n}(H) = (p,u_1,\ldots,u_{n-1})$. Note
that $u_n$ is a unit. This isomorphism is non-canonical as it 
depends on a choice of $p$-typical coordinate.

Equation \ref{lt-eq-2} can be deduced from the following result.

\begin{lem}\label{iso-unramified}Let $(A,\mm)$ be an Artin local ring
with $A/\mm$ of characteristic $p$. Let $G_1$ and $G_2$ be two formal
groups over $B$ so that $(G_1)_0$ and $(G_2)_0$ are of
height $n < \infty$. Then the affine morphism
$$
\Iso_B(G_1,G_2) \longr \Spec(A)
$$
is unramified. In particular, if we are given a choice of isomorphism
$$
\phi:(G_1)_0 \longr (G_2)_0
$$
over $(A/\mm)$,
then there is at most one isomorphism $\psi:G_1 \to G_2$
over $A$ so that $(\psi)_0 = \phi$.
\end{lem}

\begin{proof} By Theorem \ref{refined-lazard} the morphism
$$
\Iso_{A/\mm}((G_1)_0,(G_2)_0) \longr \Spec(A/\mm)
$$
is pro-\'etale; that is, flat and unramified. Since $\Spec(A/\mm) \to
\Spec(A)$ is the unique point of $\Spec(A)$, the main statement follows.
(See Proposition I.3.2 of \cite{Milne}.)
The statement about the unique lifting follows from one of the
characterizations of unramified: there is at most one way to complete
the diagram
$$
\xymatrix{
\Spec(A/\mm) \rto^\phi \dto & \Iso_B(G_1,G_2) \dto \\
\Spec(A) \ar@{-->}[ur] \rto_= & \Spec(A)
}
$$
so that both triangles commute.
\end{proof}

\begin{rem}\label{def-reformed} We can give an alternate
description of deformations in terms of formal group laws. Fix
a coordinate $x$ for $\Ga$ and let $F_\Ga$ be the resulting formal
group law over $\FF$. Define $\Defalt(\FF,\Ga)$ to be the
groupoid valued functor on $\Art_\FF$ which assigns to each
Artin local ring $(A,\mm)$ of $\Art_\FF$ the groupoid with objects
all pairs $(i,F)$ where $i:A/\mm \to \FF$ is an isomorphism
and $F$ is  a formal group law over $A$ so that
$$
i^\ast F_0(x,y) = F_\Ga(x,y) \in \FF[[x,y]].
$$
Here we've written $F_0(x,y)$ for the reduction of $F$
to $A/\mm$. There is a morphism $\psi:(i,F) \to (i',F')$ if $i = i'$
and $\psi:F \to F'$ is an isomorphism of formal group laws
so that
$$
i^\ast \psi_0(x) = x \in \FF[[x]].
$$
The set $\pi_0\Defalt(\FF,\Ga)$ is the the set of $\star$-isomorphism
classes of deformations of the formal group law $F_\Ga$.
\index{$\star$-isomorphisms}

There is a natural transformation of groupoid functors
$$
\Defalt(\FF,\Ga) \longr \Def(\FF,\Ga).
$$
This is a equivalence. It is obviously full and faithful, so 
we need only show that every object in the target is isomorphic
to some object from the source.

To see this, we'll use the notation of Remark \ref{alt-deformation}.
Let $(A,\mm)$ be an Artin local ring over $\FF$
and $G$ a deformation of $\Ga$ to $A$. Since
$\FF$ is a field, $G_0$ can be given a coordinate;
since $A$ is local Noetherian, $G$ can be given a coordinate
which reduces to a chosen coordinate for $G_0$. The isomorphism
$\phi:\Ga \to G_0$ determines an isomorphism of formal group laws
$$
\phi:F_{G_0}(x,y) \longr (i^\ast)^{-1} F_\Ga(x,y).
$$
Lift the power series $\phi(x)$ to a power series $\psi(x) \in A[[x]]$  so that
$\psi_0(x) = \phi(x)$ and define a formal group law $F(x,y)$
over $A$ by requiring that
$$
\psi(x):F_G(x,y) \longr F(x,y)
$$
is an isomorphism. Then $F(x,y)$ is the required formal group
law.
\end{rem}

\begin{rem}\label{spf-lt} Following the example of Remark
\ref{two-def-formal} we extend the notion of deformations
to all commutative rings by left Kan extensions along the
forgetful functor from $\Art_\FF$ to rings. In more detail, let
$B$ be a commutative ring. Define
the category $\Art_\FF/B$ to have objects the morphisms
$$
A \longr B
$$
of commutative rings where $(A,\mm)$ is an Artin local ring
in $\Art_\FF$.
Morphisms in $\Art/B$ are commutative triangles.  Since the
tensor product $A \otimes_\ZZ A'$ of Artin local rings is
an Artin local ring, 
the category $\Art_\FF/B$ is filtered and has a cofinal subcategory
consisting of those morphisms which are injections. 

Define the groupoid $\Def(\FF,\Ga)_B$ of deformations of
$\Ga$ over $B$ to be the colimit
$$
\Def(\FF,\Ga)_B = \colim_{\Art_\FF/B}\Def(\FF,\Ga)_A.
$$
Thus a generalized deformation of $\Ga$ to $B$ is a deformation
of $\Ga$ to an Artin local subring $A \subseteq B$. This is
probably easiest to understand using formal group laws.

Fix  a coordinate of $\Ga$ and let $\Def_\ast(\FF,\Ga)$ be
the groupoid defined in Remark \ref{def-reformed}. Then by
that remark,  there is an equivalence
$$
\colim_{\Art_\FF/B}\Def_\ast(\FF,\Ga)_A \to \Def(\FF,\Ga)_B
$$
and the elements of the source are easily described. The
objects are equivalence classes of pairs $(F,i)$ where
$F$ is a formal group law
$$
F(x,y) = \sum a_{ij}x^iy^j \in B[[x,y]]
$$
so that the coefficients $a_{ij}$ lie in an Artin local subring
$A \subseteq B$ and so that the pair $(F|_A,i) \in \Def_\ast(\FF,\Ga)_A$.
Isomorphisms in $\Def(\FF,\Ga)_B$ must similarly lie over
Artin local subrings.
\end{rem}

The following result extends and follows immediately from
Remark \ref{two-def-formal} and  Theorem \ref{Lubin-Tate}.

\begin{thm}\label{LT-extended}The natural isomorphism of functors
on commutative rings
$$
\Spf(R(\FF,\Ga)) \longr \pi_0\Def(\FF,\Ga)
$$
is an isomorphism and for all deformation $G$ of $\Ga$ over $B$
$$
\pi_1(\Def(\FF,\Ga)_B,G)= \{\ 1\ \}.
$$
\end{thm}

Now let $\layer{n} \subseteq \cU_\fg(n)$ be the closed
substack of formal groups of exact height $n$ and let 
$\hatlayer{n}$ denote its formal neighborhood.
There is a $1$-morphism of groupoid schemes
$$
\Def(\FF,\Ga) \longr \hatlayer{n}
$$
which sends a deformation $(G/B,\phi)$ to the formal group $G$.

Given a formal group $G$ over $B$ so that $\cI_{n}(G)$ is nilpotent and
$\cI_{n+1}(G) = B$, it is not necessarily true that $G$ arises from a
deformation; that would
amount to a choice of $2$-commuting diagram
$$
\xymatrix{
&\Def(\FF,\Ga)\dto\\
\Spec(B) \ar[ur] \rto_-G & \hatlayer{n}.
}
$$
Nonetheless, we have the following two results. Recall from Corollary
\ref{field-cover} that if $\Ga$ is a height $n$ formal group
over a field $\FF$, then the induced map $\Spec(\FF) \to \layer{n}$
is  a presentation. Such a formal group also defines a trivial
deformation; that is, $\Ga$ is itself a deformation of $\Ga$ to
$\FF$.

\begin{prop}\label{pull-over-layer}Let $\Ga$ be a height $n$ formal
group law over an algebraic extension of $\FF_p$. Then the
$2$-commuting square
$$
\xymatrix{
\Spec(\FF) \rto \dto & \Def(\FF,\Ga) \dto \\
\layer{n} \rto & \hatlayer{n}
}
$$
is $2$-category pull-back square.
\end{prop}

\begin{proof}Write $P$ for the $2$-category pull-back.
By  Theorem \ref{LT-extended}, an object in $P$ over a ring
$B$ is a triple $(G,f,\phi)$ where
$G$ is formal group of exact height $n$ over $B$, $f:R(\FF,\Ga) \to B$
is a ring homomorphism so that $B\cdot f(\mm)$ is nilpotent, and
$\phi:G \to f^\ast H$ is an isomorphism of formal groups. Here
$H$ is the universal deformation as in Theorem \ref{Lubin-Tate} and $\mm \subseteq
R(\FF,\Ga)$ is the maximal ideal. A morphism
$\psi: (G,f,\phi)\to (G',f,\phi')$ is an isomorphism $\psi:G \to G'$
so that $\phi'\psi=\phi$. In particular, such a triple has no non-identity
automorphisms and $P$ is discrete.

Given a triple $(G,f,\phi)$, we have that
$$
0 = \cI_n(G) = \cI_n(f^\ast H) = B\cdot f(\mm);
$$
hence, the morphism $f:R(\FF,\Ga) \to B$ factors through $\FF$. Furthermore
$\phi$ itself defines an isomorphism
$$
(G,f,\phi) \to (f^\ast H,f,1).
$$
It follows that there is an equivalence $\Spec(\FF) \to P$ sending $g:\FF \to B$ to the
triple $(f^\ast H,f,1)$ with $f$ the composite $R(\FF,\Ga) \to \FF \to B$.
\end{proof}

\begin{prop}\label{fflat-defs} The morphism $q:\Def(\FF,\Ga) \to \hatlayer{n}$ is representable, flat, and surjective.
\end{prop}

\begin{proof} We first show it is representable; in fact, we will show
that given a diagram
$$
\xymatrix{
\Spec(B) \rto^-G & \hatlayer{n} & \lto \Spf(R(\FF,\Ga))
}
$$
then
$$
p_1:\Spec(B) \times_{\hatlayer{n}} \Spf(R(\FF,\Ga)) \to \Spec(B)
$$
is a formal affine scheme over $B$. By a descent argument,
we may assume that $G$ has a coordinate. Then, arguing as in
Lemma \ref{pull-back-for-coord-1}, we see that  the pull-back is 
equivalent to
$$
\Spf(B \otimes_L W \otimes_L R(\FF,\Ga)),
$$
the formal neighborhood of $B \otimes_L W \otimes_L R(\FF,\Ga)$
at the ideal $\cI_n(p_1^\ast G) = \cI_n(p_2^\ast H)$
where $p_i$, $i=1,2$ are the projections onto the two
factors.

To see that $q$ is flat, we apply
Theorem \ref{land-exact-2} (the Landweber Exact
Functor Theorem) and the description of  $R(\FF,\Ga)$ given
in Equation \ref{lt-is-power-series}. 

For surjectivity, note that if $k$ is any field and $g:\Spec(k) \to \hatlayer{n}$
classifies a formal group $G$, then $\cI_n(G)=0$ and $G$ is a height $n$;
that is, $g$ factors through $\layer{n}$. The result now follows from
the first part and the fact that $\Spec(\FF) \to \layer{n}$ is surjective --
see  Corollary \ref{field-cover}.
\end{proof}

\subsection{The action of the automorphism group}

Let
$\GG(\FF,\Ga) = \Aut(\FF,\Ga)$, the automorphism of the pair
$(\FF,\Ga)$. An element of $\GG(\FF,\Ga)$ is a pull-back diagram
\index{$\GG(\FF,\Ga)$, Morava stabilizer group}
$$
\xymatrix{
\Ga\rto^g\dto & \Ga\dto\\
\Spec(\FF) \rto_\sigma &\Spec(\FF).
}
$$
where $\sigma$ is induced by a field automorphism. For historical and
topological reasons we call this the {\it Morava stabilizer group}
\index{Morava stabilizer group} of the
pair $(\FF,\Ga)$. We may write such a diagram
as a pair $(\sigma,\lambda)$ where $\lambda:\Ga \to \sigma^\ast \Ga$
is the isomorphism induced by $g$.
If $\FF$ is an algebraic extension of $\FF_p$ and $\Ga$ is defined over
$\FF_p$, this yields an isomorphism
$$
\GG(\FF,\Ga) \cong \mathrm{Gal}(\FF/\FF_p) \ltimes \Aut(\Ga)
$$
where $\Aut(\Ga)$ is the group of automorphisms of $\Ga$ defined
over $\FF$.

The group $\GG(\FF,\Ga)$ acts on the groupoid functor
\index{$\GG(\FF,\Ga)$, action on Lubin-Tate space}
 $\Def(\FF,\Ga)$ on the right by sending a diagram
$$
\xymatrix{
\Ga \rto \dto & G \dto\\
\Spec(\FF) \rto & \Spec(A)
}
$$
to the outer square of the diagram
$$
\xymatrix{
\Ga \rto^g \dto & \Ga \rto \dto & G \dto\\
\Spec(\FF) \rto_\sigma & \Spec(\FF) \rto & \Spec(A)
}
$$
This action commutes with the isomorphisms in $\Def(\FF,\Ga)$; thus we obtain
a right action on $\pi_0\Def(\FF,\Ga)$ and hence a left action on $R(\FF,\Ga)$.

\begin{rem}\label{source-of-confusion}If, following Remark \ref{alt-deformation},
we
think of a deformation of $\Ga$ to $A$ as a triple, $(G,i,\phi)$ and
an element of $\GG(\FF,\Ga)$ as a pair $(\sigma,\la)$ as above,
then the action of $(\sigma,\la)$ on $(G,i,\phi)$ yields the triple
$$
(G,i\sigma,(\sigma^\ast\phi)\la).
$$

Because $\pi_0\Def(\FF,\Ga)$ is a set of equivalence classes, we
must take
a little care in interpreting this action on $R(\FF,\Ga)$. See, for
example, \cite{DHAJ}.
\end{rem}

The following is the key lemma about this action.

\begin{lem}\label{input-for-pb} Let $A$ be a Artin local ring 
with residue field isomorphic to $\FF$
and let $(G,i,\phi)$ and $(G',i',\phi')$ be two deformations
of $(\FF,\Ga)$. Suppose there is an isomorphism of
formal groups $\psi:G \to G'$ over $A$. Then there
is a unique pair $(\sigma,\la) \in \Aut(\FF,\Ga)$ so that $\psi$
induces an isomorphism
$$
\psi:(G,i\sigma,\sigma^\ast(\phi)\la) \longr (G',i',\phi').
$$
\end{lem}

\begin{proof} This is simply the assertion that there is a unique
way to fill in left face of the following diagram so that it commutes
$$
\xymatrix@!0@C=35pt@R=30pt{
&\Ga \ar[dl]_g\ar[rr] \ar'[d][dd] &&G_0\ar[dl]_{\psi_0} \ar[dd]\\
\Ga  \ar[rr] \ar[dd] && G_0' \ar[dd]\\
& \Spec(\FF) \ar[dl]_\sigma\ar'[r]^-i[rr] &&\Spec(A/\mm)\ar[dl]\\
\Spec(\FF) \ar[rr]_-{i'} && \Spec(A/\mm)
}
$$
Alternatively, writing down the equations provides both the
pair $(\sigma,\la)$ and its uniqueness. Indeed, we need an equality of
isomorphisms from $\Spec(\FF)$ to $\Spec(A/\mm)$ 
$$
i\sigma= i'
$$
and a commutative diagram of isomorphisms of formal groups
$$
\xymatrix@R=5pt{
&i^\ast G_0 \ar[dd]^{(i\sigma)^\ast \psi_0}\\
\Ga \ar[dr]_-{\phi'}\ar[ur]^-{\sigma^\ast(\phi)\la}\\
& i^\ast G'_0.
}
$$
\end{proof}

The action of the Morava stabilizer group 
is actually continuous in a sense we now make precise.
Assume now that $\FF$ is an algebraic extension of $\FF_p$.
First we notice that the extended automorphism group $\GG(\FF,\Ga)$
is profinite. Define normal subgroups $\GG_k(\FF,\Ga)$ of
$\GG(\FF,\Ga)$ as follows. The group $\GG(\FF,\Ga)$
is the set of Cartesian squares
$$
\xymatrix{
\Ga \rto^\la \dto & \Ga \dto\\
\Spec(\FF) \rto_\sigma & \Spec(\FF)
}
$$
under composition. The subgroup $\GG_i(\FF,\Ga)$ is the set of those
squares so that $\sigma$ is the identity and $\la$ induces the identity
on the $p^k$-bud of the formal group $\Ga$. Then $\GG_0(\FF,\Ga) =
\Aut_\FF(\Ga)$,
$$
\GG(\FF,\Ga)/\GG_0(\FF,\Ga) = \Gal(\FF,\FF_p)
$$
and
$$
\GG(\FF,\Ga) \cong \lim \GG_k(\FF,\Ga).
$$
If $\FF \to \FF'$ is an extension of subfields of $\bFF$, then
we get an injection of groups $\GG_0(\FF,\Ga) \to \GG_0(\FF,\Ga)$
which preserves the subgroups above. Thus the following
result displays $\GG(\FF,\Ga)$ as a profinite group. In
Remark \ref{finite-levels-of-morava} we made the 
following calculation. There is an isomorphism
$$
\GG_0(\bFF,\Ga)/\GG_1(\bFF,\Ga) \cong \FF_{p^n}^\times 
$$
and for $k > 0$ a non-canonical isomorphism
$$
\GG_k(\bFF,\Ga)/\GG_{k+1}(\bFF,\Ga) \cong \FF_{p^n}
$$

\begin{rem}[{\bf The continuity of the action}]\label{continuity} Define
$\Def_k(\FF,\Ga)$ to be the groupoid of triples
$(G,i,\phi)$ where $G$ is a formal group over an Artin local
ring $A$, $i:\Spec(\FF)  \to \Spec(A/\mm)$ is an isomorphism and
$$
\phi: \Ga_{p^k} \longr i^\ast (G_0)_{p^k}
$$
is an isomorphism of $p^k$-buds. The morphisms in $\Def_k(\FF,\Ga)$
are isomorphisms $\psi:G \to G'$ which induce the appropriate commutative
triangle over $\FF$. (Note that $\Def_f(\FF,\GG)$ is {\it not} the deformation
of the buds, as these isomorphisms are defined over the whole group.)
Since every isomorphism of buds over a Noetherian
local ring can be lifted to an isomorphism of the formal groups, we have
that the evident map
$$
\Def(\FF,\Ga) \longr \Def_k(\FF,\Ga)
$$
is surjective on objects and $\GG(\FF,\Ga)$-equivariant; furthermore, the induced
map
\begin{equation}\label{iso-to-buds}
\Def(\FF,\Ga) \longr \lim \Def_k(\FF,\Ga)
\end{equation}
is an isomorphism.
The action of $\GG(\FF,\Ga)$ on $\Def_k(\FF,\Ga)$ factors through the
quotient group $\GG(\FF,\Ga)/\GG_k(\FF,\Ga)$. If we give $\Def_k(\FF,\Ga)$
the discrete topology and $\Def(\FF,\Ga)$ the topology defined by the natural
isomorphism \ref{iso-to-buds}, then the action of $\GG(\FF,\Ga)$ on
$\Def(\FF,\Ga)$ is continuous.
\end{rem}

\begin{lem}\label{pi-defk} There is a natural isomorphism of
functors on Artin local rings
$$
\xymatrix{
[\pi_0\Def(\FF,\Ga)]/\GG_k(\FF,\Ga) \rto^-\cong & \pi_0\Def_k(\FF,\Ga)
}
$$
and for all $(G,i,\phi)$ in $\Def_k(\FF,\Ga)_A$,
$$
\pi_1(\Def_k(\FF,\Ga)_A,G) = \GG_k(\FF,\Ga).
$$
\end{lem}

\begin{proof}This is a direct consequence of Lemma \ref{input-for-pb}.
The natural transformation
$$
\pi_0\Def(\FF,\Ga)/\GG_k(\FF,\Ga) \longr \pi_0\Def_k(\FF,\Ga)
$$
is onto. If $(G,i,\phi)$ and $(G',i,\phi')$ are
two deformations and
$$
\psi:(G,i,\phi) \longr (G',i,\phi')
$$
is an isomorphism in the $\Def_k(\FF,\Ga)$ there is a {\it unique}
automorphism $\la$ of $\Ga$ over $\FF$ so that 
$$
\psi:(G,i,\phi\la) \longr (G',i,\phi')
$$
is an isomorphism in $\Def(\FF,\Ga)$. Note that $\la$ is necessarily
in $\GG_k(\FF,\Ga)$.

Almost the same proof gives the statement about $\pi_1$. Indeed,
if $(G,i,\bar{\phi})$ is any lift of $(G,i,\phi)$ to $\Def(\FF,\Ga)$,
then an automorphism $\psi:(G,i,\phi) \to (G,i,\phi)$ determines
a unique element in $\GG_k(\FF,\Ga)$ so that $\psi$ induces
an isomorphism
$$
\psi:(G,i,\bar{\phi}\la) \longr (G,i,\bar{\phi}).
$$
The assignment $\psi \mapsto \la$ induces the requisite isomorphism. 
\end{proof}

The functor $\Def_k(\FF,\Ga)$ from Artin rings to groupoids can be
extended to all commutative rings using a left Kan extension
as in Remark \ref{spf-lt}. Since
we are taking a filtered colimit, the natural transformation
$\Def(\FF,\Ga) \to \Def_k(\FF,\Ga)$ remains onto for all commutative
rings. We get a natural sequence of maps
$$
\xymatrix{
\pi_0\Def(\FF,\Ga)_B \rto
& \rto^-\cong \lim \pi_0\Def(\FF,\Ga)_B/\GG_k(\FF,\Ga)
& \lim \Def_k(\FF,\Ga)_B.
}
$$
The first map, which is an isomorphism for Artin rings,
is not immediately an isomorphism in this generality
because colimits don't commute with
limits in general; however it is continuous
and, as a result, an injection.

\subsection{Deformations are the universal cover}

We now prove the main result -- that Lubin-Tate space is
the universal cover of $\hatlayer{n}$. The notion of the group
scheme defined by a profinite group $G$ was covered
in Remark \ref{group-and-galois}. In the following result
$\GG(\Ga,\FF)$ is profinite group of automorphisms
of the pair $(\Ga,\FF)$;  that is, the big Morava stabilizer
group.

\begin{thm}\label{galois-decomp} The natural transformations of groupoids
over Artin local rings
$$
\Def(\FF,\Ga) \times \GG(A,\FF) \longr \Def(\FF,\Ga) \times_{\hatlayer{n}}
\Def(\FF,\Ga)
$$
given by
$$
((G,i,\phi),(\sigma,\la)) \mapsto ((G,i,\phi),(G,i\sigma,\phi\la),1:G \to G)
$$
is an equivalence.
\end{thm}

\begin{proof} A typical element in the pull-back is a triple
\begin{equation}\label{typical-def-pb}
((G,i,\phi),(G',i',\phi'),\psi:G \to G')
\end{equation}
where the first two terms are deformations and $\psi$ is any isomorphism
of formal groups. A morphism in the pull-back 
\begin{align*}
(\ga,\ga'): ((G_1,i,\phi),(G_1',i',&\phi'),\psi_1:G_1 \to G_1')\to\\
&((G_2,j,\phi),(G_2',i',\phi'),\psi_2:G_2 \to G_2')
\end{align*}
are isomorphisms $\ga$ and $\ga'$ of deformations so that
$\psi_2\ga = \ga'\psi_1$. Now we apply Lemma \ref{input-for-pb}.
Given the typical element, as in \ref{typical-def-pb}, we get a unique
pair $(\sigma,\la)$ in $\GG(\FF,\Ga)$ so that
\begin{equation}\label{contract-def}
(1,\psi):((G,i,\phi),(G,i\sigma,\phi\la),1_G) \to ((G,i,\phi),(G',i',\phi'),\psi)
\end{equation}
is an isomorphism in the pull-back. The assignment
$$
((G,i,\phi),(G',i',\phi'),\psi) \mapsto ((G,i,\phi),(G,i\sigma,\phi\la),1_G)
$$
becomes a natural
transformation of groupoids sending a morphism $(\ga,\ga')$ to $(\ga,\ga)$.
Then \ref{contract-def}  displays the necessary contraction.
\end{proof}

\begin{defn}\label{auts-over}Let $q:Y \to X$ be a morphism of
categories fibered in groupoids over some base category.
The group $\Aut_Y(X)$\index{$\Aut_Y(X)$} of automorphisms of $Y$ over $X$ 
consists of equivalence classes pairs $(f,\psi)$ where $f:Y \to Y$ is a
1-morphism of groupoids and $\psi:q \to qf$ is a $2$-morphism. 
Two such pairs $(f,\psi)$ and $(f',\psi')$ are equivalent if there
is a 2-morphism $\phi:f \to f'$ so that $\psi'\phi=\psi$. The
composition law reads
$$
(g,\psi)(f,\phi) = (gf,(f^\ast\psi)\phi).
$$
\end{defn}

There is a homomorphism
$$
\GG(\FF,\Ga) \longr  \Aut_{\hatlayer{n}}(\Def(\FF,\Ga))
$$
sending $(\sigma,\la)$ to the pair $(f_{(\sigma,\la)},1)$ where
$f_{(\sigma,\la)}$ is the transformation
$$
(G,i,\phi) \mapsto (G,i\sigma,\phi\la).
$$

\begin{thm}\label{auts-is-morava}This homomorphism
$$
\GG(\FF,\Ga) \longr \Aut_{\hatlayer{n}}(\Def(\FF,\Ga)).
$$
is an isomorphism.
\end{thm}

\begin{proof} That the map is an injection is clear from the definitions.
We now prove 
that it's surjective. Let $(f,\psi)$ be an element of the automorphisms of
$\Def(\FF,\Ga)$ over $\hatlayer{n}$. Let's write
$$
f(G,i,\phi) = (G_f,i_f,\phi_f).
$$
Then $\psi$ gives isomorphism of formal groups $\psi_G:G \to G_f$. By
Lemma \ref{input-for-pb} there is a unique pair $(\sigma,\la) \in \GG(\FF,\Ga)$
so that 
$$
\psi_G:(G,i\sigma,\phi\la) \to (G_f,i_f,\phi_f)
$$
is an isomorphism of deformations. The uniqueness of $(\sigma,\la)$
and this equation give us the needed $2$-morphism
$$
\phi:f_{(\sigma,\la)} \longr f
$$
\end{proof}

We wish to show that the isomorphism of Theorem \ref{auts-is-morava}
is appropriately continuous.

\begin{lem}\label{pro-finite-aut-over}There is a surjective homomorphism of
groups
$$
q_k:\Aut_{\hatlayer{n}}(\Def(\FF,\Ga)) \to  \Aut_{\hatlayer{n}}(\Def_k(\FF,\Ga))
$$
which induces a commutative diagram of groups
$$
\xymatrix{
\GG(\FF,\Ga) \rto \dto &  \Aut_{\hatlayer{n}}(\Def(\FF,\Ga))\dto^{q_k}\\
\GG(\FF,\Ga)/\GG_{k+1}(\FF,\Ga) \rto &  \Aut_{\hatlayer{n}}(\Def_k(\FF,\Ga)).
}
$$
\end{lem}

\begin{proof} We use Theorem \ref{auts-is-morava}. Any automorphism
of $\Def(\FF,\Ga)$ of the form $f_{(\sigma,\la)}$ immediately
induces an automorphism of $\Def_k(\FF,\Ga)$. This defines a morphism
$$
\GG(\FF,\Ga) \longr  \Aut_{\hatlayer{n}}(\Def_k(\FF,\Ga))
$$
which factors through $\GG(\FF,\Ga)/\GG_{k+1}(\FF,\Ga)$. It remains
only to show that it's onto. To show this, we use a variant of the
argument in the proof of Theorem \ref{auts-is-morava}. If $(f,\psi)$ is
an automorphism of $\Def_k(\FF,\Ga)$ over $\hatlayer{n}$, we again write
$$
f(G,i,\phi) = (G_f,i_f,\phi_f).
$$
choose isomorphisms $\bar{\phi}$ and $\bar{\phi}_f$ lifting $\phi$ and
$\phi_f$ respectively. Then Lemma \ref{input-for-pb} supplies 
an element $(\sigma,\la) \in \GG(\FF,\Ga)$ so that 
$$
\psi_G:(G,i\sigma,\bar{\phi}\la) \to (G_f,i_f,\bar{\phi}_f).
$$
The class of $(\sigma,\la)$ in $\GG(\FF,\Ga)/\GG_{k+1}(\FF,\Ga)$
is independent of the choice and $\psi$ supplies the needed
$2$-morphism to show surjectivity.
\end{proof}

The groupoid $\Def_0(\FF,\Ga)$ has a simple description. Indeed,
$\Def_0(\FF,\Ga)$ assigns to each Artin local ring $A$ the pairs
$(G,i)$ where $i:\FF \to \AA/\mm$ is an isomorphism. Since $\FF$
is perfect, the universal property of Witt vectors (\cite{p-divisible} \S III.3)
implies there is a unique homomorphism of rings $W(\FF) \to A$ which
reduces to $\FF$ modulo maximal ideals. Thus, we conclude that
$\Def_0(\FF,\Ga)$ is the functor from groupoids to which assigns to
each Artin local $W(\FF)$-algebra $A$ so that 
$$
W(\FF)/(p) \longr A/\mm
$$
is an isomorphism the groupoid of formal groups $G$ over $A$
so that $\cI_{n}(G)\subseteq \mm$ and $\cI_{n+1}(G) = A$. Thus we have proved:

\begin{lem}\label{def0} There is a natural isomorphism
of categories fibered in groupoids over $\Art_\FF$
$$
\Def_0(\FF,\Ga) \mathop{\longr}^{\cong} W(\FF) \otimes_{\ZZ_p} \hatlayer{n}.
$$
\end{lem}

We now define what it means for a morphism to be Galois in this
setting. Galois morphisms of schemes were defined in 
Remark \ref{group-and-galois}.

\begin{defn}\label{galois-covering}A representable 
morphism $q:X \to Y$ of sheaves of groupoids in the $fpqc$-topology is
{\bf Galois} if $q$ faithfully flat, and if the natural
map
$$
X \times \Aut_Y(X) \longr X \times_Y  X
$$
is an equivalence of groupoids over $X$. 
\end{defn}

The main result of the section is now as follows.

\begin{thm}\label{lt-is-univ}Let $\FF = \bFF$ be the algebraic closure
of the prime field and let $\Ga$ be any height $n$-formal group
over $\FF_p$. Then
$$
q:\Def(\FF,\Ga) \longr \hatlayer{n}
$$
is Galois with Galois group 
$$
\GG(\FF,\Ga) = \Gal(\FF/\FF_p) \ltimes \Aut_{\bFF}(\Ga).
$$
The discrete groupoid
$\Def(\FF,\Ga) \simeq \Spf(R(\FF,\Ga))$ itself has no
non-trivial \'etale covers, so the morphism $q$ is 
the universal cover.
\end{thm}

\begin{proof}To get that $q$ is Galois,
combine Proposition \ref{fflat-defs}, Theorem \ref{galois-decomp}, Theorem
\ref{auts-is-morava},
Lemma \ref{pro-finite-aut-over}, and Lemma \ref{def0}.
That $R(\FF,\Ga)$ has
no non-trivial \'etale extensions follows from the fact that this
ring is complete, local, and has an algebraically closed residue
field.
\end{proof}

\begin{rem}\label{reinterpret-lt-ring}All of these results can
be rewritten in terms of the Lubin-Tate ring $R(\FF,\Ga)$ of
Theorem \ref{Lubin-Tate} if
we wish. For example, we can define a homomorphism
$$
\GG(\FF,\Ga) \longr \Aut_{\hatlayer{n}}(\Spf(R(\FF,\Ga))
$$
as follows. We refer to Remark \ref{define-from-spf}.
Let $(H,j,\phi_u)$ be the universal deformation over $\Spf(R(\FF,\Ga))$
and let $(\sigma,\la)$ be in $\GG(\FF,\Ga)$. Then we get a new
deformation $(H,j\sigma,\phi_u\la)$ over $\Spf(R(\FF,\Ga)$, 
classified by a map
$$
f=f_{(\sigma,\la)}:\Spf(R(\FF,\Ga)) \to \Spf(R(\FF,\Ga)).
$$
Thus there is a unique isomorphism of deformations
$$
\psi = \psi_{(\sigma,\la)}:(H,j,\phi_u) \to f^\ast(H,j,\phi_u).
$$
The pair $( f_{(\sigma,\la)},\psi_{(\sigma,\la)})$ now produces the
$2$-commuting diagram
$$
\xymatrix@C=0pt
{
\Spf(R(\FF,\Ga)) \ar[dr] \ar[rr]^{f_{(\sigma,\la)}} && \Spf(R(\FF,\Ga))\ar[dl]\\
&\hatlayer{n}.
}
$$
and the assignment
$$
(\sigma,\la) \longmapsto (f_{(\sigma,\la)}, \psi_{(\sigma,\la)})
$$
defines a group homomorphism
$$
\GG(\FF,\Ga) \longr \Aut_{\hatlayer{n}}(\Spf(R(\FF,\Ga))).
$$
Theorem \ref{auts-is-morava} then becomes the following result.
\end{rem}

\begin{prop}\label{auts-is-morava-aff}This homomorphism
$$
\GG(\FF,\Ga) \longr \Aut_{\hatlayer{n}}(\Spf(R(\FF,\Ga))).
$$
is an isomorphism.
\end{prop}

Theorem \ref{lt-is-univ} then reads as follows:

\begin{thm}\label{lubin-tate-fnbhd} Let $\Ga$ be a formal
group of height $n$ over $\FF_p$
and let
$$
q:\Spf(R(\bar{\FF}_p,\Ga)) \longr \hatlayer{n}
$$
classify a universal deformation of $\Ga$ regarded as a formal
group over $\bFF$. Then $q$ is the universal
cover of the formal neighborhood $\hatlayer{n}$ of $\Ga$; specifically, $q$
is pro-\'etale and Galois with
Galois group the big Morava stabilizer group
$$
\GG(\bar{\FF}_p,\Ga)) \cong \mathrm{Gal}(\bar{\FF}_p/\FF_p) \rtimes \Aut_\bFF(\Ga).
$$
\end{thm}

\subsection{Morava modules}

We add two remarks intended to clarify what it means to be a comodule
over Morava $E$-theory
$$
E_n \defeq E(\FF_{p^n},\Ga_n).
$$
The conceptual difficulty is that we define
$$
(E_n)_\ast E_n = \pi_\ast L_{K(n)} (E_n\wedge E_n)
$$
where $L_{K(n)}$ is localization at Morava $K$-theory. As such, the usual translation
from homotopy theory to comodules needs some modification. The 
appropriate concept is that of a Morava module, and we will give
some exposition of this in Remark \ref{qc-hatlayer}. To simplify  matters
we pass to $E(\bFF,\Ga)$.

\begin{rem}\label{relation-to-homotopy} Theorem \ref{galois-decomp} 
is the restatement of the well-known calculation 
of the homology cooperations in Lubin-Tate theory. There is a $2$-periodic
homology theory $E(\fGa)$ with $E(\fGa)_0 = R(\fGa)$ and whose associated
formal group is a choice of universal deformation of $\Ga$. Then $E(\fGa)$ is
Landweber exact and\index{$(E_n)_\ast E_n$, and Lubin-Tate space}
\begin{align*}
E(\fGa)_0 E(\fGa) &\defeq \pi_0L_{K(n)}(E(\fGa) \wedge E(\fGa))\\
&\ \cong \map(\GG(\fGa),R(\fGa))
\end{align*}
where $\map(-,-)$ is the set of continuous maps. Proofs of this
statement can be found in \cite{Strick2} and \cite{HoveyHHA}; indeed,
the argument given here for Theorem  \ref{galois-decomp}  is very
similar to Hovey's.
\end{rem}

\def\RF{{{R(\bFF,\Ga)}}}
\def\GF{{{\GG(\bFF,\Ga)}}}

\begin{rem}[{\bf Morava modules}]\label{qc-hatlayer}\index{Morava modules}
Theorem \ref{lubin-tate-fnbhd} allows us
to interpret quasi-coherent sheaves on $\hatlayer{n}$ as quasi-coherent
sheaves on $\Spf(R(\bFF,\Ga))$ with a suitable $\GG(\FF,\Ga)$ action.
Let's spell this out in more detail.

Let $\mm = \cI_{n}(H) \subseteq R(\bFF,\Ga)$, where $H$ is any choice of
the universal deformation. Recall that a a quasi-coherent sheaf on 
$\Spf(\RF)$ is determined by a tower
$$
\cdots \to M_k \to M_{k-1} \to \cdots \to M_1
$$
where $M_k$ is an $\RF/\mm^k$-module, $M_k \to M_{k-1}$ is a
$\RF/\mm^k$-module homomorphism and
$$
\RF/\mm^{k-1} \otimes_{\RF/\mm^k} M_k \longr M_{k-1}
$$
is an isomorphism. Under appropriate finiteness conditions, this tower
is determined by its inverse limit $\lim M_k$ regarded as a continuous
$\RF$-module. 

A quasi-coherent sheaf on
$$
\Spf(\RF) \times_{\hatlayer{n}}\Spf(\RF) \cong \Spf(\map(\GF,\RF))
$$
has a similar description as  modules over the tower
$$
\{\map(\GF,\RF/\mm^n)\}.
$$

A {\it Morava module} is a tower of $\RF$-modules
$$
\cdots \to M_k \to M_{k-1} \to \cdots \to M_1
$$
so that
\begin{enumerate}

\item $M_k$ is an $\RF/\mm^k$-module and the induced map
$$
\RF/\mm^{k-1} \otimes_{\RF/\mm^k} M_k \longr M_{k-1}
$$
is an isomorphism;

\item $M_k$ has a continuous $\GF$-action, where $M_k$ has the
discrete topology;

\item the action of $\GF$ is twisted over $\RF$ in the sense that
if $a \in \RF$, $x \in M_k$, and $g \in \GF$, then
$$
g(ax) = g(a)g(x).
$$
\end{enumerate}

Now Theorem \ref{lubin-tate-fnbhd}  implies there in equivalence of categories
between quasi-coherent sheaves on $\hatlayer{n}$ and Morava modules.
\end{rem}

%% file: chromconv.tex
\section{Completion and chromatic convergence}

In this section we give the recipe for recovering a coherent
sheaf on $\cM_\fg$ (over $\ZZ_{(p)}$) from its restrictions
to each of the open substacks of formal groups of height
less than or equal to $n$. This has two steps: passing from
one height to the next via a fracture square (Theorem
\ref{fracture-square-1}) and then taking a derived
inverse limit (Theorem \ref{chrom-conv}). The latter
theorem has particular teeth as the union of the open
substacks of finite height is not all of $\cM_\fg$.

Students of the homotopy theory literature will see that, in the end,
our arguments are not so different from the Hopkins-Ravenel
Chromatic Convergence of \cite{orange}. Much of the algebra
here can be reworked in the language of comodules and, as such,
it can be deduced from the work of Hovey and Strickland \cite{HS2}.

\subsection{Local cohomology and scales}

We begin by recalling some notation from Definition \ref{scale}
and Proposition \ref{LEFT-easy-part}.
Let $f:\cN \to \cM_\fg$ be
a representable, separated, and flat morphism of algebraic stacks.
We will confuse the ideal sheaves $\cI_n$ defining the height filtration
with the pull-backs $f^\ast \cI_n$, which induce the height
filtration on $\cN$. Thus, we let
$$
0 = \cI_0 \subseteq \cI_1 \subseteq \cI_2 \subseteq \cdots \subseteq \cO_\cN
$$
denote the resulting scale on $\cN$.
Let $\cN(n) = \cM(n) \times_{\cM_\fg} \cN \subseteq \cN$ be
the closed substack defined by $\cI_n$ and
let $\cV(n-1)$ be the open complement. We will write $i_n:\cV(n) \to \cN$
and $j_n:\cN(n) \to \cN$ for the inclusions. Finally, let's write
$\cO$ for $\cO_\cN$. 

If $\cF$ is a quasi-coherent $\cI_n$-torsion sheaf, we defined (in
\ref{invert-vn})
$$
\cF[v_n^{-1}] = (i_n)_\ast i_n^\ast \cF.
$$
The notation was justified in Remark \ref{invert-vn-1}. A local
description of this sheaf was given in Proposition \ref{no-higher-derived}.

We wish to recursively define quasi-coherent sheaves $\cO/\cI_n^\infty$
on $\cN$
by setting $\cO/\cI_{0}^\infty = \cO$ and then defining
$\cO/\cI_{n+1}^\infty$ by the short exact sequence\index{$\cO/\cI_n^\infty$}
\begin{equation}\label{k-infty}
0 \to \cO/\cI_{n}^\infty \to \cO/\cI_{n}^\infty[v_n^{-1}]
\to \cO/\cI_{n+1}^\infty \to 0.
\end{equation}
In order to do this, we must prove the following lemma.
In the process,
give local descriptions on these sheaves. See Equations
\ref{concrete-1} and \ref{concrete-2}.

\begin{lem}\label{inject-at-zero}For all $n \geq  0$, the sheaf
$\cO/\cI_{n}^\infty$ is an $\cI_{n}$-torsion sheaf and the unit of
the adjunction
$$
\cO/\cI_{n}^\infty \to (i_n)_\ast i_n^\ast \cO/\cI_{n}^\infty =
\cO/\cI_{n}^\infty[v_n^{-1}]
$$
is injective.
\end{lem}

\begin{proof} Both statements are local, so can be proved by
evaluating on an affine morphism $\Spec(R) \to \cM$ which
is flat and quasi-compact. By taking a faithfully flat extension
if necessary,
we may assume that are elements $u_n \in R$
so that
$$
u_n + \cI_n(R) \in \cO/\cI_{n}(R) \cong 
 R/(u_0,\cdots,u_{n-1})
$$
is a generator of $\cI_n(R)/\cI_{n-1}(R)$.
Since we have a scale, multiplication by $u_n$ on
$R/(u_0,\cdots,u_{n-1})$ is injective. Define
$R$-modules $R/(u_0^\infty,\cdots,u_{n-1}^\infty)$
inductively by beginning with $R$ and by insisting there
be a short exact sequence
$$
0 \to R/(u_0^\infty,\cdots,u_{n-1}^\infty) \to
R/(u_0^\infty,\cdots,u_{n-1}^\infty)[u_n^{-1}] 
\to R/(u_0^\infty,\cdots,u_{n}^\infty)
\to 0.
$$
Then inductively we have, using Proposition \ref{no-higher-derived}
\begin{equation}\label{concrete-1}
\cO/\cI_{n}^\infty(R) = R/(u_0^\infty,\cdots,u_{n-1}^\infty)
\end{equation}
and 
\begin{equation}\label{concrete-2}
(i_n)_\ast i_n^\ast \cO/\cI_{n}^\infty(R) =
R/(u_0^\infty,\cdots,u_{n-1}^\infty)[{u}_n^{-1}].
\end{equation}
The result now follows.
\end{proof}

We note that Proposition \ref{no-higher-derived} also implies:

\begin{lem}\label{no-higher-derived-1}For all $n > 0$ and all $s > 0$
$$
R^s(i_n)_\ast i_n^\ast \cO/\cI^{\infty}_{n} = 0.
$$
\end{lem}

\begin{rem}[{\bf Triangles and fiber sequences}]\label{shift-notation}In the rest of the section, we are
going to use a shift functor on (co-)chain complexes of sheaves
determined\index{triangles vs fiber sequences}
by the following equation. If $C$ is a cochain complex and
$n$ is an integer, then\index{$C[n]$}
$$
H^sC[n] = H^{s+n}C.
$$
If $C$ is a chain complex,  then we regard it as a cochain complex
by the equation $H^sC = H_{-s}C$; thus, $H_sC[n] = H_{s-n}C$.
A distinguished triangle of cochain complexes
$$
A \to B \to C \to A[1]
$$
induces a long exact sequence in cohomology
$$
\cdots \to H^sA \to H^sB \to H^sC \to H^sA[1] = H^{s+1}A \to \cdots
$$
To shorten notation we may revert to the homotopy theory conventions
and say that $A \to B \to C$ is a fiber sequence in cochain complexes.

If $M$ is a sheaf, we may regard it as a cochain complex in degree
zero; hence
$$
H^sM[-n] = \brackets{M,}{s=n;}{0,}{s \ne n.}
$$
\end{rem}

We now introduce local cohomology, which will be an important tool
for the rest of this section.

\begin{defn}\label{loc-coh-def}\index{local cohomology}
Let $Z \subseteq \cN$ be any closed substack with open complement
$i:U \to \cN$. If $\cF$ is a quasi-coherent sheaf 
on $\cN$, define the derived {\bf local cohomology sheaf} of
$\cF$ by the the distinguished triangle
\begin{equation}\label{local-coh-def}
R\Gamma_Z(\cN,\cF) \to \cF \to Ri_\ast i^\ast \cF \to R\Gamma_Z(\cN,\cF)[1].
\end{equation}
Put another way, $R\Gamma_Z(\cN,\cF)$ is the homotopy fiber
of the map $\cF \to Ri_\ast i^\ast \cF$.
\end{defn}

If $\cN$ is understood, we may write $R\Ga_{Z}\cF$ for $R\Ga_Z(\cN,\cF)$;
if $Z$ is defined by an ideal sheaf $\cI \subseteq \cO$, we may write
$R\Ga_\cI \cF$ for $R\Ga_Z(\cN,\cF)$.

The {\it local cohomology} of $\cF$ at $Z$ is then the graded cohomology sheaf
$$
H^\ast_Z(\cN,\cF) \defeq H^\ast R\Gamma_Z(\cN,\cF).
$$
If $V \to \cN$
is an open morphism in our topology, then
$$
\Gamma_{Z}(\cN,\cF)(V) =
H^0_{Z}(\cN,\cF)(V)
$$
is the set of sections $s \in \cF(V)$ which vanish when restricted
to $\cF(U \times_\cN V)$. If $\cI$ is locally generated
by a regular sequence, then we can give the following local
description of $\Gamma_Z(\cN,\cF)$. Let $\Spec(R) \to \cN$
be any morphism so that $\cI(R)$ is generated by a regular
sequence $u_0,\ldots,u_{n-1}$. Then there is an exact sequence
\begin{equation}\label{begin-Koszul}
\Gamma_Z(\cN,\cF)(R) \to \cF(R) \to \prod_i \cF(R)[u_i^{-1}].
\end{equation}
This has the following consequence.  See Corollary 3.2.4 of
\cite{all} for a generalization.

In the next result and what follows, $\hom$ denotes the
sheaf of homomorphisms and $\Hom$ denotes its
global sections.

\begin{lem}\label{local-to-ext} Suppose that the ideal $\cI \subseteq \cO_\cN
= \cO$ defining
the closed substack $Z \subseteq \cN$ is locally generated by
a regular sequence. Then for any quasi-coherent sheaf $\cF$ on
$\cN$ there is a natural equivalence
$$
\colim_k R\hom(\cO/\cI^k,\cF) \mathop{\longr}^{\simeq}
R\Ga_{Z}(\cN,\cF).
$$
\end{lem}

\begin{proof}Before taking derived functors, we note that there
is certainly a natural map
$$
\colim_k \Hom(\cO/\cI^k,\cF) \longr
\Ga_{Z}(\cN,\cF)
$$
given by evaluating at the unit. We first prove that this is an isomorphism;
for this it is sufficient to work locally. Let $\Spec(R) \to \cM$
where $\cI(R)$ is generated by the regular sequence $u_0,\cdots,
u_{k-1}$. Then the exact sequence of \ref{begin-Koszul} implies
that $x \in \cF(R)$ is in $\Gamma_Z(\cM,\cF)(R)$ if and only
if for all $i$ there is a $t_i$ so that $u_i^{t_i}x = 0$. This
yields the desired (underived) isomorphism.
Since colimit is exact on 
filtered diagrams, the derived version follows.
\end{proof}

We now set $Z = \cN(n+1)$, defined by $\cI_{n+1}$, so that $\cV(n) = \cN - \cN(n+1)$ and
there is  a distinguished triangle
$$
R\Ga_{\cN(n+1)}\cF \to \cF \to R(i_{n})_\ast i_n^\ast \cF \to 
R\Ga_{\cN(n+1)}\cF[1].
$$
If $\cF$ is a quasi-coherent $\cI_n$-torsion sheaf, then Proposition
\ref{no-higher-derived} applies and $(i_n)_\ast i_n^\ast \cF = \cF[v_n^{-1}]$.

The exact sequence defining $\cO/\cI_n^\infty$ and
Lemmas \ref{inject-at-zero} and \ref{no-higher-derived-1} imply
following result.

\begin{lem}\label{no-higher-derived-redux}For all $n \geq   1$ there
is an isomorphism in the derived category
$$
R\Gamma_{\cN(n)}(\cN,\cO/\cI_{n-1}^\infty) \cong \cO/\cI_{n}^\infty[-1].
$$
\end{lem}

We also have the following key calculation.

\begin{prop}\label{vanishing-n} For all $n \geq 1$ there is an equivalence
in the derived category of quasi-coherent sheaves
$$
R\Gamma_{\cN(n)}(\cN,\cO) \simeq \cO/\cI_n^\infty[-n].
$$
That is,
$$
H^s_{\cN(n)}(\cN,\cO) \cong \brackets{0,}{s\ne n;}{\cO/\cI_n^\infty,}{s=n.}
$$
\end{prop}

\begin{proof}We proceed by induction to show that 
$$
R\Gamma_{\cN(n)}(\cN,\cO/\cI_{n-k}^\infty) \simeq \cO/\cI_{n}^\infty[-k].
$$
Lemma \ref{no-higher-derived-redux} is the case $k=1$. To get the inductive
case, we have an exact sequence
$$
0 \to \cO/\cI_{n-(k+1)}^\infty \to (i_{n-k})_\ast i^\ast_{n-k} \cO/\cI_{n-(k+1)}^\infty
\to \cO/\cI_{n-k}^\infty \to 0.
$$
Hence we need to show that
$$
R\Ga_{\cN(n)}  (i_{n-k})_\ast i^\ast_{n-k}\cO/\cI_{n-(k+1)}^\infty = 0,
$$
or equivalently that
$$
 (i_{n-k})_\ast i^\ast_{n-k}\cO/\cI_{n-(k+1)}^\infty
  \to R(i_{n-1})_\ast i^\ast_{n-1}  (i_{n-k})_\ast i^\ast_{n-k}\cO/\cI_{n-(k+1)}^\infty
$$
is an equivalence. Consider the sequence of inclusions
$$
\xymatrix{
\cV(n-k) \ar@/_1pc/[rr]_{i_{n-k}} \rto^-f & \cV(n-1) \rto^-{i_{n-1}} & \cN
}
$$
We easily check that $i_{n-1}^\ast (i_{n-k})_\ast = f_\ast$; since $i_{n-1}^\ast$
is exact we have an equivalence
$$
R(i_{n-k})_\ast i^\ast_{n-k}\cO/\cI_{n-(k+1)}^\infty
  \to R(i_{n-1})_\ast i^\ast_{n-1}(i_{n-k})_\ast i^\ast_{n-k}\cO/\cI_{n-(k+1)}^\infty
$$
The result now follows from Lemma \ref{no-higher-derived-1}.
\end{proof}

\begin{thm}\label{local-to-tensor}Let $\cF$ be a quasi-coherent sheaf
on $\cN$. Then there are natural equivalences in the derived category
$$
R\Gamma_{\cN(n)}(\cN,\cF) \simeq \cO/\cI_n^\infty[-n] \otimes_\cO^L 
\cF$$
\end{thm}

\begin{proof} This follows immediately from Lemma \ref{local-to-ext}
and Proposition \ref{vanishing-n}; indeed, since $\cO/\cI_n^k$ is
locally finitely presented
\begin{align*}
R\Gamma_{\cN(n)}(\cN,\cF) &\simeq \colim R\hom(\cO/\cI_n^k,\cF)\\
	&\simeq \colim R\hom(\cO/\cI_n^k,\cO) \otimes^L \cF\\
	&\simeq R\Gamma_{\cN(n)}(\cN,\cO) \otimes^L \cF.
\end{align*}
\end{proof}

Another consequence of Lemma \ref{local-to-ext} and Proposition
\ref{vanishing-n} is the following result.

\begin{prop}\label{finite-iso}There is an equivalence
$$
\colim_k R\hom(\cO/\cI_n^k,\cO) \mathop{\longr}^{\cong}
\cO/\cI  _n^\infty[-n].
$$
\end{prop}

We also will be interested in what happens if we vary $n$. Consider the sequence
of inclusions
$$
\xymatrix{
\cV(n-1) \ar@/_1pc/[rr]_{i_{n-1}} \rto^-f & \cV(n) \rto^-{i_{n}} & \cN
}
$$
Recall that $\cV(n)$ is the complement of $\cN(n+1)$. In the case
where $\cN = \cM_\fg$, $\cN(n)=\cM(n)$ classifies formal groups
of height at least $n$, $\cV(n) = \cU(n)$ classifies formal groups
of height at most $n$ and $\layer{n} = \cM(n) \cap \cU(n)$ classifies
formal groups of exact height $n$.

\begin{lem}\label{rel-local} For all quasi-coherent $\cF$ on $\cN$, there 
are fiber sequences of cochain complexes of quasi-coherent
sheaves
$$
R(i_n)_\ast i_n^\ast R\Ga_{\cN(n)}\cF \to R(i_n)_\ast (i_n)^\ast \cF \to
 R(i_{n-1})_\ast (i_{n-1})^\ast \cF
$$
and
$$
R\Ga_{\cN(n+1)}\cF \to R\Ga_{\cN(n)}\cF \to R(i_n)_\ast i_n^\ast R\Ga_{\cN(n)}\cF
$$
\end{lem}

\begin{proof}The fiber sequence which defines local cohomology
(see Definition \ref{loc-coh-def}) yields that these sequences
are equivalent; so, we prove the first.

For any quasi-coherent sheaf on $\cN$, we have a fiber sequence
\begin{equation}\label{before-i-n}
R\Ga_{\cH(n)}(\cV(n),i_n^\ast \cF) \to i_n^\ast \cF 
\to Rf_\ast i_{n-1}^\ast \cF.
\end{equation}
Here we have taken the liberty of writing $\cH(n)$ for $\cV(n) \cap
\cN(n)$ and we have used $f^\ast i_n^\ast \cong
i_{n-1}^\ast$.

Next note that the adjoint to the equivalence $R(i_{n-1})_\ast \cong
R(i_n)_\ast Rf_\ast$ yields  a commutative diagram
$$
\xymatrix{
i_n^\ast R\Ga_{\cN(n)}(\cN,\cF) \rto\dto  & i_n^\ast \cF\rto\dto^{=} &
i_{n}^\ast R(i_{n-1})_\ast i_{n-1}^\ast \cF\dto \\
R\Ga_{H(n)}(\cV(n),i_n^\ast \cF)  \rto & i_n^\ast \cF\rto &
Rf_\ast i_{n-1}^\ast \cF.
}
$$
Finally, for all quasi-coherent sheaves
$\cE$ on $\cV(n-1)$, the natural map
$$
i_{n}^\ast R(i_{n-1}^\ast) \cE\longr Rf_\ast \cE 
$$
is an equivalence; indeed, we easily check that
$i_{n}^\ast (i_{n-1})_\ast\cE \to f_\ast \cE$ is an isomorphism and
then we use that $i_n^\ast$ is exact. From this we conclude that
$$
R\Ga_{H(n)}(\cV(n),i_n^\ast \cF)  \longr i_n^\ast R\Ga_{\cN(n)}(\cN,\cF)
$$
is an equivalence. We feed this into Equation \ref{before-i-n} and
apply $R(i_n)_\ast$ to get the result.
\end{proof}

Applying the second of the fiber sequences of Lemma \ref{rel-local} to
$\cF = \cO$ itself and using Theorem \ref{local-to-tensor},
we get the fiber sequence
\begin{equation}\label{k-infty-shifted}
\cO/\cI_{n+1}^\infty[-n-1] \to \cO/\cI_n^\infty[-n] \to\cO/\cI_n^\infty[v^{-1}][-n]
\end{equation}
which is the evident shift of the defining sequence \ref{k-infty}. From this
we obtain the following result.

\begin{lem}\label{local-to-tensor-bound}There is a natural commutative diagram
$$
\xymatrix{
R\Gamma_{\cN(n+1)}(\cN,\cF) \dto_\simeq \rto &
R\Gamma_{\cN(n)}(\cN,\cF) \dto^ \simeq \\
\cO/\cI_{n+1}^\infty[-n-1] \otimes_\cO^L \cF\rto &
\cO/\cI_n^\infty[-n] \otimes_\cO^L \cF
}
$$
where the bottom morphism is the boundary morphism induced by the short
exact sequence
$$
0 \to \cO/\cI_n^\infty \to  \cO/\cI_n^\infty[v_n^{-1}] \to  \cO/\cI_{n+1}^\infty \to 0.
$$
\end{lem}

\subsection{Greenlees-May duality}

There is a remarkable duality between local cohomology and completion
first noticed by Greenlees and May \cite{GM} and globalized in 
\cite{all}. Similar results appear in \cite{DG}, which also has
the general version of the fracture square we will write down below
in  Theorem \ref{fracture-square}. The techniques of \cite{all} apply directly
to the case of a qausi-compact and separated stack $\cN$ and
the closed substacks $\cN(n) \subseteq \cN$ arising from a scale. The main
result we'll use is the following. Derived completion was defined
in Definition \ref{der-completion}.

\begin{prop}\label{compl-hom}For all quasi-coherent sheaves $\cF$ on $\cN$
there is a natural equivalence
$$
L(\cF)^\cmpl_{\cN(n)} \simeq R\hom(\cO/\cI_n^\infty[-n],\cF).
$$
\end{prop}

This result is actually equivalent to an apparently stronger result --
{\it Greenlees-May duality}:\index{Greenlees-May duality}

\begin{thm}\label{compl-hom-gen} Let $\cE$ and $\cF$ be two 
chain complexes of quasi-coherent sheaves on $\cN$ . Then there
is a natural equivalence
$$
R\hom (R\Ga_{\cN(n)}\cE,\cF) \simeq R\hom(\cE,L(\cF)^\cmpl_{\cN(n)}).
$$
\end{thm}

Certainly Theorem \ref{compl-hom-gen} implies Proposition
\ref{compl-hom} by setting $\cE = \cO_\cN$ and applying
Proposition \ref{vanishing-n}. Conversely, Theorem \ref{local-to-tensor}
gives a natural isomorphism
\begin{align*}
R\hom (R\Ga_{\cN(n)}\cE,\cF)
&\cong R\hom(\cO/\cI_n^\infty[-n] \otimes^L\cE,\cF)\\
&\cong R\hom(\cE,R\hom(\cO/\cI_n^\infty[-n],\cF)).
\end{align*}
Hence Theorem \ref{compl-hom-gen} follows from Proposition
\ref{compl-hom}. 

The argument to prove Proposition \ref{compl-hom} goes exactly as in
\cite{all}; hence we will content ourselves with giving an outline.

Lemma \ref{local-to-ext} allows us to define a natural map
$$
\Phi:L(\cF)^\cmpl_{\cN(n)} \longr R\hom(R\Ga_{\cN(n)}(\cN,\cO),\cF)
$$
as follows. First note that for $\cO$-module sheaves $\cE$ and $\cF$, there
is a natural map 
\begin{equation}\label{op-eval}
\cE \otimes \cF \longr \Hom(\Hom(\cE,\cO),\cF)
\end{equation}
given pointwise by sending $x \otimes y$ to the homomorphism
$\phi_{x\otimes y}$ with
$$
\phi_{x\otimes y}(f) = f(x)y.
$$
The morphism of Equation \ref{op-eval} can be derived to an  morphism
$$
\cE \otimes^L \cF \longr R\Hom(R\Hom(\cE,\cO),\cF).
$$
Now $\Phi$ is defined as the composition
\begin{align*}
L(\cF)^\cmpl_{\cN(n)} = \holim (\cF \otimes^L \cO/\cI_n^k)
\to &\holim R\hom(R\hom(\cO/\cI_n^k,\cO),\cF)\\
\cong\ &R\hom(\colim R\hom(\cO/\cI_n^k,\cO),\cF)\\
\cong\ &R\hom(R\Ga_{\cN(n)}(\cN,\cO),\cF).
\end{align*}

Proposition \ref{compl-hom} now can be restated as

\begin{prop}\label{compl-hom-1}For all quasi-coherent sheaves $\cF$,
the natural map
$$
\Phi:L(\cF)^\cmpl_{\cN(n)} \longr R\hom(R\Ga_{\cN(n)}(\cN,\cO),\cF)
$$
is an equivalence.
\end{prop}

\def\bu{{{\mathbf{u}}}}

The first observation is that the
question is local; that is, it is sufficient to show that there
$\Phi$ is an equivalence when evaluated at any flat and quasi-compact
morphism $\Spec(R) \longr \cM$ for which $\cI_n(R)$ is 
generated by a regular sequence. This follows readily from the definition
of completion (\ref{der-completion}) and the remarks immediately
afterwards. If we write
 $I = \cI_n(R)$ and $M = \cF(R)$, then
we are asking that the map
$$
\Phi_V: L(M)^\cmpl_I \longr R\hom(R\Gamma_I(R),M)
$$
be an equivalence. This is exactly what Greenlees and May prove.
There is a finiteness condition in the argument which is worth emphasizing:
for all $i$, the $R$-module $R/(u_0^i,\cdots,u_{n-1}^i)$ has a
finite resolution by finitely generated free $R$-modules.
The usual such resolution is the Koszul complex, which 
we now review.

\def\bx{{{\mathbf{u}}}}

Let $R$ be a
commutative ring and let $u \in R$. Define $K(u)$ to the chain
complex
$$
\xymatrix{
R \rto^u & R
}
$$
concentrated in degrees $0$ and $1$. If $\bx = (u_0,\ldots,u_{n-1})$
is an ordered $n$-tuple of elements in $R$, define the
{\it Koszul complex}\index{Koszul complex} to be
$$
K(\bx) = K(u_0) \otimes \cdots \otimes K(u_{n-1}).
$$ 
Note that if $\bx$ is a regular sequence in $R$ and $I$ is the
ideal generated by $u_0,\ldots,u_{n-1}$, then $K(\bx)$ is the
Kozsul resolution of $R/I$ and
$$
H_s(K(\bx) \otimes M) \cong \Tor_s^R(R/I,M).
$$

Now fix the $n$-tuple $\bx$ and define $\bx^i = (u_0^i,\ldots,u_{n-1}^i)$.
The commutative squares
$$
\xymatrix{
R \rto^{u_j^i} \dto_{u_j} & R\dto^{=}\\
R \rto^{u_j^{i-1}} & R
}
$$
combine to give morphisms $f_i:K(\bx^i) \to K(\bx^{i-1})$. Thus
if the element of $\bx$ form a regular sequence,\footnote{Or, more
generally, if the elements of $\bx$ are {\it pro-regular} as in
\cite{GM}.} then a simple bicomplex arguments shows that
for any $R$-module $M$ there is an homology isomorphism
$$
L(M)^\cmpl_I \simeq \holim_j (K(\bx^j) \otimes M).
$$
This equivalence is natural in $M$, although it doesn't look
very natural in $R$ or $I$.

The dual complex
$$
K^\ast(\bx) \defeq \Hom_R(K(\bx),R)
$$
is a chain complex concentrated in degrees $s$, $-n \leq s \leq 0$. Note
that if the $u_i$ form a regular sequence
\begin{equation}\label{homology-dual}
H_sK^\ast(\bx) = \Ext_R^{-s}(R/(u_0,\ldots,u_{n-1}),R) \cong
\brackets{R/(u_0,\ldots,u_{n-1}),}{s=-n;}{0,}{s\ne-n.}
\end{equation}
The dual of the maps $f_i$ give maps $f_i^\ast:K^\ast(\bx^{i-1})
\to K(\bx^i)$. Define 
$$
K^\ast(\bx^\infty) = \colim K^\ast(\bx^i).
$$
We have have natural homology equivalences, assuming the elements
in $\bx$ form a regular sequence:
\begin{align*}
K^\ast(\bx^\infty) \otimes M &\simeq \colim K^\ast(\bx^i) \otimes M\\
&\simeq \colim \Hom_R(K(\bx^i),M)\\
&\simeq\colim R\hom_R(R/(x_0^i,\ldots,x_{n-1}^i),M)\\
&\simeq R\Ga_I(M).
\end{align*}
More is true, because $K(\bx^i)$ is finitely generated as a
chain complex of $R$-modules we have that the map
of Equation \ref{op-eval}
$$
K(\bx^i) \otimes M \longr \Hom(K^\ast(\bx^i),M)
$$
is an isomorphism, natural in $M$. Then the local version of
Greenlees-May duality follows:
\begin{align*}
L(M)^\cmpl_I &\simeq \holim (K(\bx^j) \otimes R)\\
& \simeq \holim  \Hom(K^\ast(\bx^i),M)\\
& \simeq R\hom(\colim K^\ast(\bx^i),M)\\
& \simeq R\hom (R\Ga_I(R),M).
\end{align*}
It is an exercise is bicomplexes to show that this is, up to natural
homology equivalence, the map $\Phi$ of Proposition \ref{compl-hom-1}.

\begin{rem}\label{kozsul-localcoh}The isomorphism
$$
H_{-s}(K^\ast(\bx^\infty) \otimes M) \cong H^s_I(R,M)
$$
developed above extends the exact sequence of Equation \ref{begin-Koszul}.
Indeed, $K^\ast(\bx^\infty)$ is exactly the chain complex\index{local cohomology
complex}
$$
M \to \prod_i M[u_i^{-1}] \to \prod_{i_1<i_2} M[u_{i_1}^{-1}u_{i_2}^{-1}] 
\to \cdots \to M[u_0^{-1}\cdots u_{k-1}^{-1}] \to 0\ \cdots.
$$
\end{rem}

\subsection{Algebraic chromatic convergence}

We now supply the two results we promised: a fracture square
for reconstructing quasi-coherent sheaves for the completions
and a decomposition of a coherent sheaf as a homotopy
inverse limit.

We begin with a preliminary calculation. Compare Corollary 5.1.1 of
\cite{all}.

\begin{thm}\label{local-doesnt-see-compl}Let $\cF$ be a quasi-coherent
sheaf on $\cN$. Then the natural map
$$
R\Ga_{\cN(n)}\cF \longr R\Ga_{\cN(n)}L(\cF)^\cmpl_{\cN(n)}
$$
is an equivalence.
\end{thm}

\begin{proof} The question is local (again) and, therefore, reduces to
the following assertion. Let $\bx = (u_0,\ldots,u_{n-1})$ be
a regular sequence in $R$, let $I$ be the ideal generated by
this regular sequence, and let $P$ be a projective $R$-module. Then
$$
K^\ast(\bx^\infty) \otimes P \longr K^\ast(\bx^\infty) \otimes (P)^\cmpl_I
$$
is an equivalence. Indeed, if we apply homology to the map
$$
K^\ast(\bx^i) \otimes P \longr K^\ast(\bx^i) \otimes (P)^\cmpl_I
$$
then, by Equation \ref{homology-dual} we obtain the maps
$$
\Ext_R^s(R/(u_0^i,\ldots,u_{n-1}^i),P) \longr
\Ext_R^s(R/(u_0^i,\ldots,u_{n-1}^i),(P)^\cmpl_I).
$$
Both source and target are zero if $s\ne n$ and if $s=n$ we have
the map
$$
P/(u_0^i,\ldots,u_{n-1}^i) \to (P)^\cmpl_I/(u_0^i,\ldots,u_{n-1}^i)
$$
which is an isomorphism.
\end{proof}

This result has the following fracture square as a consequence. Recall that 
the open inclusion $i_{n-1}:\cV(n-1)\to \cN$ is complementary
to the closed inclusion $j_n:\cN(n) \to \cN$.

\begin{thm}[{\bf The fracture squares}]\label{fracture-square}\index{fracture
square}Let $\cF$ be a quasi-coherent 
sheaf on $\cN$. Then there is a homotopy cartesian square in the
derived category
$$
\xymatrix{
\cF \rto \dto &L(\cF)^\wedge_{\cN(n)}\dto\\
R(i_{n-1})_\ast i_{n-1}^\ast \cF \rto & R(i_{n-1})_\ast i_{n-1}^\ast 
L(\cF)^\wedge_{\cN(n)}.
}
$$
\end{thm}

\begin{proof}The induced morphisms on fibers of the vertical
maps is exactly
$$
R\Ga_{\cN(n)}\cF \longr R\Ga_{\cN(n)}L(\cF)^\cmpl_{\cN(n)}
$$
which is an equivalence by Proposition \ref{local-doesnt-see-compl}.
\end{proof}

\begin{rem}\label{fracture-square-1} An important special case is worth isolating.
Let $\cF$ be a quasi-coherent sheaf on $\cN$ and consider
the sequence of inclusions
$$
\xymatrix{
\cV(n-1) \ar@/_1pc/[rr]_{i_{n-1}} \rto^-f & \cV(n) \rto^-{i_n} & \cN
}
$$
Applying Theorem \ref{fracture-square} to a complex of sheaves of the
form $R(i_n)_\ast i_n^\ast \cF$ where $\cF$ is quasi-coherent on
$\cN$, we get a homotopy cartesian square\
\begin{equation}\label{frac-chrom}
\xymatrix{
R(i_n)_\ast i_n^\ast \cF \rto \dto &L(R(i_n)_\ast i_n^\ast \cF)^\wedge_{\cN(n)}\dto\\
R(i_{n-1})_\ast i_{n-1}^\ast \cF \rto & R(i_{n-1})_\ast i_{n-1}^\ast 
L(R(i_n)_\ast i_n^\ast \cF)^\wedge_{\cN(n)}.
}
\end{equation}
The right hand vertical column of this diagram seems excessively complicated,
but expected to those familiar with the results of \cite{666} \S 7.3.
The topological
analog of these calculations supplies a fracture square of spectra
$$
\xymatrix{
L_nX \rto \dto & L_{K(n)} X \dto\\
L_{n-1}X \rto & L_{n-1}L_{K(n)}X.
}
$$
The connection to completion is somewhat less than straightforward and
given by the equations\index{fracture square, in homotopy}
$$
L_{K(n)}X \simeq L_{K(n)}L_nX = \holim S/I \wedge L_nX
$$
where $\{S/I\}$ is a suitable family of type $n$ complexes. 

Despite the unwieldy nature of the diagram of \ref{frac-chrom}, the induced
map on the homotopy fibers of the vertical map actually simplies somewhat,
as the following result shows. Compare Proposition
\ref{local-doesnt-see-compl}.
\end{rem}

\begin{prop}\label{rel-fibers}For all quasi-coherent sheaves $\cF$
on $\cN$, the natural map
$$
R\Ga_{\cN(n)}\cF \to R\Ga_{\cN(n)}R(i_n)_\ast i_n^\ast \cF
$$
is an equivalence.
\end{prop}

\begin{proof}This follows from the fact that
$$
R(i_{n-1})_\ast i_{n-1}^\ast \cF \to R(i_{n-1})_\ast i_{n-1}^\ast R(i_n)_\ast i_n^\ast \cF
$$
is an equivalence, which in turn follows from the fact that
$$
i_{n-1}^\ast (i_n)_\ast = f^\ast
$$
and the fact that $i_{n-1}^\ast$ is exact.
\end{proof}

Now let's specialize to the case where $\cN = \cM_\fg$
itself and
$i_n:\cU(n) \to \cM_\fg$ be the inclusion of the open
moduli substack of formal groups of height at most $n$.
Then we will show that if $\cF$ is a {\it coherent} sheaf on
$\cM_\fg$, then the natural map
$$
\cF \to \holim R(i_n)_\ast i_n^\ast \cF
$$
is an isomorphism in the derived category of quasi-coherent
sheaves. This is an algebraic analog of chromatic convergence.
There is something to prove here as the open substacks
$\cU(n)$ do not exhaust $\cM_\fg$; indeed, the
morphism
$$
\GG_a:\Spec(\FF_p) \longr \cM_\fg
$$
classifying the additive formal group (which has infinite
height) does not factor through $\cU(n)$ for any $n$.

The proof is below in Theorem \ref{chrom-conv}. The observation
that drives the argument in this: recall from 
Theorem \ref{fp-modules} that if $\cF$ is a coherent sheaf
on $\cM_\fg$, then there is an integer $r$ and a coherent
sheaf $\cF_0$ on the moduli stack of buds $\cM_\fg\pnty{p^r}$
so that $\cF \cong q^\ast \cF_0$. Thus we begin with the next
computation.

\begin{thm}\label{zero-map-homology}Let $\cF$ be a quasi-coherent
sheaf on $\cM_\fg\pnty{p^r}$. Then for all $n > r$ and all $s$,
the map on local cohomology groups
$$
H^s_{\cM(n+1)}(\cM_\fg,q^\ast \cF) \to
H^s_{\cM(n)}(\cM_\fg,q^\ast \cF)
$$
is zero.
\end{thm}

\begin{proof} We apply Lemma \ref{local-to-tensor-bound} and show that
the induced map
$$
\cO/\cI_{n+1}^\infty \otimes^L_\cO \cF \to \cO/\cI^\infty_n[1] \otimes^L_\cO \cF
$$
is zero in homology.
It is sufficient to prove this after evaluation
at any affine presentation $f:X \to \cM_\fg$. Let
$$
X = \Spec(\ZZ_{(p)}[u_1,u_2,\ldots]) \defeq \Spec(V)
$$
and let $f$ classify the formal group obtained from the universal
$p$-typical formal group law. Similarly, let 
$$
X_r = \Spec(\ZZ_{(p)}[u_1,u_2,\ldots,u_r]) \defeq
\Spec(V_r) \to \cM_\fg\pnty{p^r}
$$
classify the resulting bud. This, too, is a presentation, by Lemma \ref{p-buds}.
Let $M = \cF(X_r \to \cM_\fg\pnty{p^r})$. Then 
$$
V \otimes_{V_r} M \cong (q^\ast\cF)(X \to \cM_\fg)
$$
and we are trying to calculate
$$
\Tor_s^V(V/(p^\infty,\ldots,u_n^\infty),V \otimes_{V_r} M)
\to 
\Tor_{s-1}^V(V/(p^\infty,\ldots,u_{n-1}^\infty),V \otimes_{V_r} M).
$$
Since $V$ is a free $V_r$-module,
to see this homomorphism is zero it is sufficient to note that
$$
V/(p^\infty,\ldots,u_{n-1}^\infty) \to
V/(p^\infty,\ldots,u_{n-1}^\infty)[u_n^{-1}]
$$
is split injective as a $V_r$-module as long as $n > r$.
\end{proof}

\begin{cor}\label{vanish-homology} Let $\cF$ be a quasi-coherent
sheaf on $\cM_\fg\pnty{p^r}$. Then
$$
H_s(\cO/\cI_n^\infty \otimes^L q^\ast \cF) = 0
$$
for $s > r$. 
\end{cor}

\begin{proof}Again one can work locally, using the presentations
of the previous proof. We prove the result by induction on $n$.
If $n \leq r$ the chain complex $K^\ast(\bx^\infty)$ of 
$V$-modules
supplies a resolution of length $n$ of $V/(p^\infty,\ldots,u_{n-1}^\infty)$
by flat $V$-modules; therefore,
$$
H_s(\cO/\cI_n^\infty \otimes^L q^\ast \cF) = 0,\qquad s > n.
$$
So we may assume $n > r$. Then the previous result and the
induction hypothesis imply that
$$
H_s(R(i_n)_\ast i_n^\ast \cO/\cI_{n-1}^\infty\otimes^L q^\ast \cF)
\cong 
H_s (\cO/\cI_{n}^\infty\otimes^L q^\ast \cF)
$$
for $s > r$. Evaluated at $\Spec(V) \to \cM_\fg$ this
is an isomorphism
$$
\Tor_s^V(V/(p^\infty,\ldots,u_{n-1}^\infty)[u_n^{-1}],V \otimes_{V_r} M)
\cong 
\Tor_{s}^V(V/(p^\infty,\ldots,u_{n}^\infty),V \otimes_{V_r} M).
$$
Since $n > r$, we have
\begin{align*}
\Tor_s^V(V/(p^\infty,\ldots,&u_{n-1}^\infty)[u_n^{-1}],V \otimes_{V_r} M)\\
&\cong
\Tor_s^V(V/(p^\infty,\ldots,u_{n-1}^\infty),V \otimes_{V_r} M)[u_n^{-1}]. 
\end{align*}
Since $s > r$, the latter group is zero by the induction hypothesis.
\end{proof}

\begin{thm}[{\bf Chromatic Convergence}]\label{chrom-conv}\index{chromatic
convergence, algebraic}Let $\cF$ be a coherent sheaf
on $\cM_\fg$. Then the natural map
$$
\cF \longr \holim R(i_n)_\ast i_n^\ast \cF
$$
is a quasi-isomorphism.
\end{thm}

\begin{proof}There are distinquished triangles
$$
R\Gamma_{\cM(n)}\cF \to \cF \to R(i_n)_\ast i_n^\ast \cF\to 
R\Gamma_{\cM(n)}\cF[1];
$$
therefore, it is sufficient to show that
$$
\holim R\Gamma_{\cM(n)}\cF \simeq 0.
$$
But this follows from Theorem \ref{zero-map-homology}.
\end{proof}

\begin{rem}\label{conv-gen}The  chromatic
convergence result holds in slightly greater generality: if $\cF_0$
is any quasi-coherent sheaf on $\cM_\fg\pnty{p^r}$ for some
$r < \infty$, then
the natural map
$$
q^\ast\cF_0 \longr \holim R(i_n)_\ast i_n^\ast q^\ast\cF_0
$$
is a quasi-isomorphism. I also point out that Hollander  \cite{hollander4}
has a proof that
works if we only assume the the quasi-coherent sheaf $\cF$
has finite projective dimension in an appropriate sense.
\end{rem}

\begin{thm}\label{coho-conv}Let $\cF_0$ be a quasi-coherent sheaf
on $\cM_\fg\pnty{p^r}$ and let $\cF = q^\ast\cF_0$ be the 
pull-back to $\cM_\fg$. Then the natural map
$$
H^s(\cM_\fg,\cF) \longr H^s(\cU(n),i_n^\ast\cF)
$$
is an isomorphism for $s < n-r$ and injective for $s=n-r-1$.
\end{thm}

\begin{proof}The failure of this map to be an isomorphism is
measured by the long exact sequence
\begin{align*}
\cdots \to H^s(\cM_\fg,R\Gamma_{\cM(n+1)}\cF) \to
H^s(\cM_\fg,&\cF) \\
\to H^s(\cU(n),i_n^\ast\cF)
&\to H^{s+1}(\cM_\fg,R\Gamma_{\cM(n+1)}\cF) \to \cdots
\end{align*}
where $H^\ast(\cM_\fg,R\Gamma_{\cM(n+1)}\cF)$ is the
hyper-cohomology of the derived local cohomology
sheaf $R\Gamma_{\cM(n+1)}\cF$. This can be computed
via the spectral sequence
$$
H^p(\cM_\fg,H^qR\Gamma_{\cM(n+1)}\cF)
\Longrightarrow 
H^{p+q}(\cM_\fg,R\Gamma_{\cM(n+1)}\cF).
$$
The isomorphism of Theorem \ref{local-to-tensor}
$$
H^qR\Gamma_{\cM(n+1)}\cF \cong H_{n+1-q}
(\cO/\cL_{n+1}^\infty \otimes^L \cF)
$$
and Corollary \ref{vanish-homology} now give the result.
\end{proof}

\begin{rem}\label{compare-hom-cc}\index{chromatic convergence, in homotopy}
The Hopkins-Ravenel chromatic
convergence results of \cite{orange} says that if $X$ is a $p$-local
finite complex, then there is a natural weak equivalence
$$
X \mathop{\longr}^{\simeq} \holim L_nX
$$
where $L_nX$ is the localization at the Johnson-Wilson theory
$E(n)_\ast$. For such $X$, $BP_\ast X$ is a finitely presented
comodule and, as in \cite{HS2}, we can interpret Theorem \ref{chrom-conv} as
saying that there is an isomorphism
$$
BP_\ast X \cong R\lim BP_\ast L_nX
$$
where $R\lim$ is an appropriate total derived functor of inverse
limit in comodules. Because homology and inverse limits
do not necessarily commute, this is not, in itself, enough to 
prove the Hopkins-Ravenel result; some more homotopy theoretic
data is needed. 
\end{rem}